\documentclass{memo-l}

\usepackage{times,latexsym,amssymb}
\usepackage{amsmath,amsthm,bm}
\usepackage{graphics,epsfig,graphicx,multicol,psfrag}
\usepackage{caption}
\usepackage{color}
\usepackage{bm}
\usepackage[colorlinks,pdfpagelabels,pdfstartview = FitH,bookmarksopen
= true,bookmarksnumbered = true,linkcolor = blue,plainpages =
false,hypertexnames = false,citecolor = red,pagebackref=false]{hyperref}

\usepackage{amsbsy}
\usepackage{amstext}
\usepackage{amssymb}
\usepackage{esint}
\usepackage{stmaryrd}
\usepackage[dvipsnames]{xcolor}
\allowdisplaybreaks
\newtheorem{theorem}{Theorem}[chapter]
\newtheorem{lemma}[theorem]{Lemma}
\newtheorem{proposition}[theorem]{Proposition}
\newtheorem{corollary}[theorem]{Corollary}

\theoremstyle{definition}
\newtheorem{definition}[theorem]{Definition}

\theoremstyle{remark}
\newtheorem{remark}[theorem]{Remark}

\numberwithin{section}{chapter}
\numberwithin{equation}{chapter}

\renewcommand{\proofname}{\rm\bfseries{Proof.}}
\newcommand{\mint}{- \mskip-19,5mu \int}

\def\N{\mathbb{N}}
\def\R{\mathbb{R}}
\def\L{\mathcal{L}}
\def\H{\mathcal{H}}

\def\bg{{\boldsymbol\gamma}}

\renewcommand{\d}{\mathrm{d}}
\newcommand{\dx}{\mathrm{d}x}
\newcommand{\dy}{\mathrm{d}y}

\newcommand{\dt}{\mathrm{d}t}
\newcommand{\ds}{\mathrm{d}s}

\renewcommand{\epsilon}{\varepsilon}

\DeclareMathOperator{\spt}{spt}

\DeclareMathOperator{\Div}{div}

\DeclareMathOperator{\dist}{dist}
\DeclareMathOperator{\loc}{loc}
\renewcommand{\epsilon}{\varepsilon}
\newcommand{\eps}{\varepsilon}
\renewcommand{\rho}{\varrho}

\def\eqn#1$$#2$${\begin{equation}\label#1#2\end{equation}}

\newcommand{\noi}{\noindent}

\newcommand{\dsty}{\displaystyle}

\newcommand{\al}{\alpha}

\newcommand{\be}{\beta}

\newcommand{\gm}{\gamma}
\newcommand{\dl}{\delta}
\newcommand{\Dl}{\Delta}
\newcommand{\lm}{\lambda}
\newcommand{\Lm}{\Lambda}

\newcommand{\varep}{\varepsilon}
\newcommand{\vp}{\varphi}
\newcommand{\sig}{\sigma}

\newcommand{\z}{\zeta}

\newcommand{\rr}{\mathbb{R}}
\newcommand{\rn}{\rr^N}

\newcommand{\bl}[1]{\mathbf{#1}}

\newcommand{\dvg}{\operatorname{div}}

\newcommand{\essup}{\operatornamewithlimits{ess\,sup}}

\newcommand{\osc}{\operatornamewithlimits{osc}}

\newcommand{\pl}{\partial}

\newcommand{\power}[2]{\bm{#1^{\mbox{\unboldmath{\scriptsize$#2$}}}}}

\def\Xint#1{\mathchoice
    {\XXint\displaystyle\textstyle{#1}}%
    {\XXint\textstyle\scriptstyle{#1}}%
    {\XXint\scriptstyle\scriptscriptstyle{#1}}%
    {\XXint\scriptscriptstyle\scriptscriptstyle{#1}}%
    \!\int}
\def\XXint#1#2#3{\setbox0=\hbox{$#1{#2#3}{\int}$}
    \vcenter{\hbox{$#2#3$}}\kern-0.5\wd0}
\def\bint{\Xint-}
\def\dashint{\Xint{\raise4pt\hbox to7pt{\hrulefill}}}

\def\Xiint#1{\mathchoice
    {\XXiint\displaystyle\textstyle{#1}}%
    {\XXiint\textstyle\scriptstyle{#1}}%
    {\XXiint\scriptstyle\scriptscriptstyle{#1}}%
    {\XXiint\scriptscriptstyle\scriptscriptstyle{#1}}%
    \!\iint}
\def\XXiint#1#2#3{\setbox0=\hbox{$#1{#2#3}{\iint}$}
    \vcenter{\hbox{$#2#3$}}\kern-0.5\wd0}
\def\biint{\Xiint{-\!-}}

\def\XXiiint#1#2#3{\setbox0=\hbox{$#1{#2#3}{\iint}$}
    \vcenter{\hbox{$#2#3$}}\kern-0.5\wd0}

\makeindex

\begin{document}

\frontmatter

\title{Hölder Continuity of the Gradient of Solutions to\\ Doubly Non-Linear Parabolic Equations}


\author{Verena B\"{o}gelein}
\address{Fachbereich Mathematik, Universit\"at Salzburg,
Hellbrunner Str. 34, 5020 Salzburg, Austria}
\curraddr{}
\email{verena.boegelein@plus.ac.at}
\thanks{The first author was supported in part by the FWF-Project P31956-N32 “Doubly nonlinear evolution equations”}

\author{Frank Duzaar}
\address{Fachbereich Mathematik, Universit\"at Salzburg,
Hellbrunner Str. 34, 5020 Salzburg, Austria}
\curraddr{}
\email{frankjohannes.duzaar@plus.ac.at}

\author{Ugo Gianazza}
\address{Dipartimento di Matematica ``F. Casorati",
Universit\`a di Pavia,
via Ferrata 5, 27100 Pavia, Italy}
\curraddr{}
\email{gianazza@imati.cnr.it}
\thanks{The third author was supported in part by Grant  2017TEXA3H\_002 ``Gradient flows, Optimal Transport and Metric Measure Structures".}

\author{Naian Liao}
\address{Fachbereich Mathematik, Universit\"at Salzburg,
Hellbrunner Str. 34, 5020 Salzburg, Austria}
\curraddr{}
\email{naian.liao@plus.ac.at}
\thanks{The fourth author was supported by the FWF-Project P36272-N “On the Stefan type problems” and the FWF-Project P31956-N32 “Doubly nonlinear evolution equations”}

\author{Christoph Scheven}
\address{Fakult\"at f\"ur Mathematik, 
Universit\"at Duisburg-Essen, 45117 Essen, Germany}
\curraddr{}
\email{christoph.scheven@uni-due.de}

\date{}



\dedicatory{To the memory of Emmanuele DiBenedetto }

\begin{abstract}
This paper is devoted to studying the local behavior of non-negative weak solutions to the doubly non-linear parabolic equation
\begin{equation*}
    \partial_t u^q - 
	\Div\big(|D u|^{p-2}D u\big)
	=
	0 
\end{equation*}
in a space-time cylinder.
H\"older estimates are established for the gradient of its weak solutions in the super-critical fast diffusion regime $0<p-1< q<\frac{N(p-1)}{(N-p)_+}$ where $N$ is the space dimension. 
Moreover,   decay estimates are obtained for weak solutions and their gradient in the vicinity of possible extinction time. 
Two main components towards these regularity estimates are a time-insensitive Harnack inequality that is particular about this regime, and Schauder estimates for the parabolic $p$-Laplace equation. 
\end{abstract}

\maketitle

\tableofcontents


\mainmatter

\chapter{Introduction}\label{Intro}
\section{Regularity for doubly non-linear parabolic equations}\label{sec:intro-intro}

Porous medium type equations and more generally doubly non-linear parabolic equations arise in numerous applications of physics. In Chapter~\ref{SEC:model} we will explain  a physical model for the flow of fluids in porous media that leads to the doubly non-linear parabolic equation considered in this paper. Despite their importance,  many fundamental questions are still largely open in the field and a complete mathematical understanding of these equations is yet to be established.

Although some parts of the paper are dealing with more general PDEs, the main result is concerned with non-negative weak solutions to the prototype doubly non-linear parabolic equation
\begin{equation}\label{doubly-nonlinear-prototype}
    \partial_t u^q - 
	\Div\big(|D u|^{p-2}D u\big)
	=
	0\quad \mbox{in $ E_T$,}
\end{equation}
where $p>1$, $q>0$ and $E_T:= E\times(0,T]$ is a space-time cylinder over an open and bounded set $E\subset\R^N,\,N\in\mathbb{N}$. 
A mathematical interest of the equation lies in its singularity or/and degeneracy at points where either $u=0$ or $Du=0$. 
Moreover, it covers the porous medium equation ($p=2$),  the parabolic $p$-Laplace equation ($q=1$) and Trudinger's equation ($q=p-1$). Monographs \cite{DB, Vazquez, WZYL-01} deal with various topics about this equation. The theme of this work concerns  the {\it fast diffusion} regime $q>p-1$ where solutions to \eqref{doubly-nonlinear-prototype} show an infinite speed of propagation as in the case of the heat equation.

Existence of weak solutions to \eqref{doubly-nonlinear-prototype} can be established in the parabolic Sobolev-space 
\begin{equation*}  
	C\big( [0,T]; L^{q+1}(E)\big) \cap L^p \big(0,T; W^{1,p} (E)\big),
\end{equation*}
see~Section~\ref{S:existence} for more discussion. Whereas uniqueness of weak solutions was known only in very limited cases; the issue will be discussed further in Section~\ref{S:CP}. 

The investigation of regularity of weak solutions was initiated by Ivanov in a series of papers. Local boundedness of weak solutions has been studied in~\cite{Ivanov-1995-2}. Regarding this topic, we refer to Section~\ref{sec:int-bound} for a more detailed exposition. The issue of H\"older continuity of weak solutions requires more delicate analysis and have attracted a number of authors. The well-studied ranges are the borderline case $p-1=q$, the doubly singular case, i.e. $p-1<q$ and $1<p<2$, as well as the doubly degenerate case, i.e. $q<p-1$ and $p>2$. See~\cite{Ivanov-1989, Ivanov-1994, Ivanov-1995-3, Porzio-Vespri, Urbano-08, KSU-12, BDL-21, BDLS-22, Liao-Schaetzler, Vespri-Vestberg}, and also Figure~\ref{Fig:hold}, where the white regions give a graphical representation of the well-studied ranges.
\begin{center}
\begin{figure}
\psfragscanon
\psfrag{p}{$\scriptstyle p$}
\psfrag{q}{$\scriptstyle q$}
\psfrag{pzero}{$\scriptstyle p=1$}
\psfrag{qzero}{$\scriptstyle q=0$}
\psfrag{ptwo}{$\scriptstyle p=2$}
\psfrag{line}{$\scriptstyle q=p-1$}
\begin{center}
\includegraphics[width=.75\textwidth]{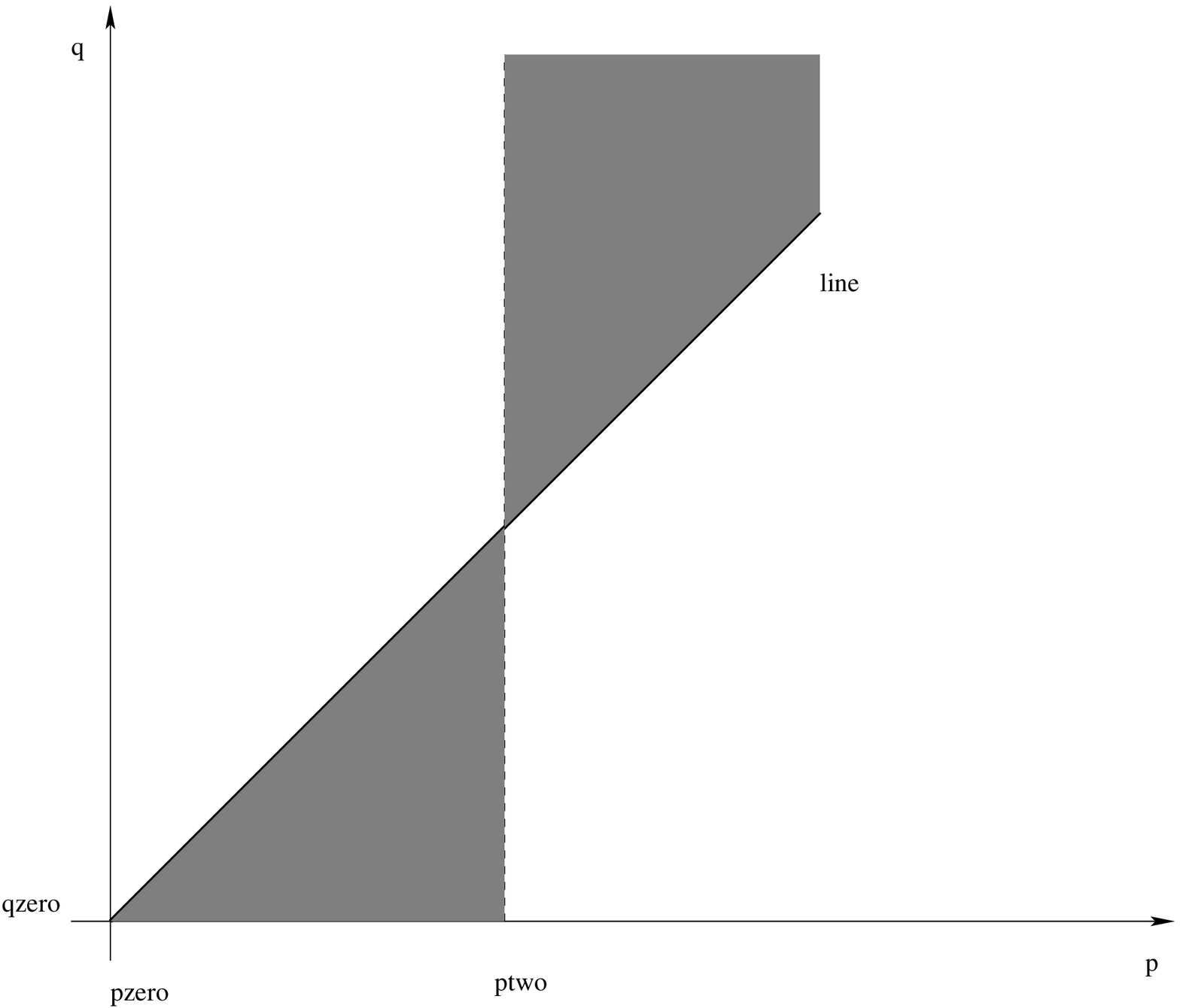}
\caption{\label{Fig:hold}
}
\end{center}
\end{figure}
\end{center}

The theory becomes even more fragmented when it comes to gradient regularity. 
The well-studied case is the parabolic $p$-Laplace equation starting from DiBenedetto \& Friedman \cite{DiBenedetto-Friedman, DiBenedetto-Friedman3}. Indeed, the gradient of weak solutions is H\"older continuous.  

On the other hand, gradient regularity for solutions to the porous medium equation is more involved and much less understood. 
In the slow diffusion regime ($p=2$ \& $q<1$), only the one-dimensional case is reasonably understood. Aronson studied the Lipschitz continuity of $x\mapsto u^{1-q}(x,t)$ 
in \cite{Aronson-69}. For $0<q\le\frac12$ DiBenedetto~\cite{DB-1d} investigated the Lipschitz continuity of $t\mapsto u^{1-q}(x,t)$. The full range $0<q<1$ has later been established by B\'enilan~\cite{Benilan} and Aronson \& Caffarelli~\cite{Aronson-Caffarelli}. 
In the same regime, the optimal regularity in the multi-dimensional case is still a major open problem. 
 Caffarelli \& V\'azquez \& Wolanski \cite{CVW} proved for a solution $u$ to the Cauchy problem that $D u^{1-q}$ becomes bounded after some waiting time. This result was localized and improved for $0<q\le\frac12$ by Gianazza \& Siljander \cite{Gianazza-Siljander} who were, roughly speaking, able to quantify the waiting time. More precisely, if at some point $(x_o,t_o)$ the average $a^{q}=\bint_{B_r(x_o)} u^{q}(x,t_o) \,\dx$ is strictly positive, then $Du^{1-q}$ is locally bounded in a certain cylinder, whose center lies at time $t_o+a^{q-1} r^2$. 

In the fast diffusion regime, ($p=2$ \& $q>1$) the picture is clearer. In the super-critical range $1<q<\frac{N}{(N-2)_+}$ DiBenedetto \& Kwong \& Vespri \cite{DBKV} proved that weak solutions are locally analytic in the space variable and at least Lipschitz-continuous in time up to the extinction time. They presented very precise, quantitative, regularity estimates up to the boundary. See also \cite{JX} for recent contributions.

As for the doubly non-linear equation \eqref{doubly-nonlinear-prototype}, to our knowledge, the only attempts were made by Ivanov \& Mkrtychyan \cite{Ivanov-Mk}, Ivanov \cite{Ivanov-1996}, and Savar\'e \& Vespri \cite{Savare}. More details will be given in Section~\ref{sec:reg:opt}.

Our main aim is to establish the  gradient H\"older regularity of weak solutions to \eqref{doubly-nonlinear-prototype} in the {\it super-critical fast diffusion} regime $0<p-1< q<\frac{N(p-1)}{(N-p)_+}$. 
Our approach relies on a ``time-insensitive" Harnack inequality that is particular about this regime and Schauder estimates for the parabolic $p$-Laplace equation, both of which are of independent interest.
The thrust of our contribution is to generate precise, quantitative, local estimates regarding weak solutions as well as their gradient that dictate their behavior. Moreover, it turns out that such estimates are sharp as evidenced by various explicit examples.

\section{The main  regularity  theorem}\label{S:intro:main:thm}
As already mentioned, the main result of the present paper deals with the gradient regularity of weak solutions $u$
to doubly non-linear parabolic equations \eqref{doubly-nonlinear-prototype} in the super-critical fast diffusion range. The formulation of the result is based, among other things, on the pointwise value $u(x_o,t_o)>0$. For this reason, it is convenient to state the result for continuous weak solutions. 

\begin{theorem}\label{THM:REGULARITY-INTRO}
Assume that $0<p-1< q<\frac{N(p-1)}{(N-p)_+}$.
There are constants $\boldsymbol\gm,\widetilde{\boldsymbol\gm}>1$ and $\alpha_o\in(0,1)$, which  depend only on $N$, $p$ and $q$, such that if $u$ is a 
non-negative,  continuous, weak solution to \eqref{doubly-nonlinear-prototype} in $E_T$,
if $u_o:=u(x_o,t_o)>0$, and if the set inclusion
\[
\widetilde{\bg}Q_o
:=
K_{\widetilde{\bg}\rho}(x_o) \times \big(t_o - u_o^{q+1-p}(\widetilde{\boldsymbol\gm}\rho)^p, t_o + u_o^{q+1-p}(\widetilde{\boldsymbol\gm}\rho)^p\big)\subset E_T
\]
holds, then we have the gradient bound
\begin{equation}\label{grad-est-intro}
    \sup_{Q_o}|Du|
    \le
    \frac{\boldsymbol\gm u_o}{\rho},
\end{equation}
the Lipschitz estimate
\begin{equation}\label{diff-u-intro}
    |u(z_1)-u(z_2)|
    \le
    \boldsymbol\gm u_o\,
    \Bigg[\frac{|x_1-x_2|}{\rho}+\sqrt{\frac{|t_1-t_2|}{u_o^{q+1-p}\rho^p}} \,\Bigg],
\end{equation}
and the gradient H\"older estimate
\begin{equation}\label{C1alph-intro}
    |Du(z_1) - Du(z_2)|
    \le
    \frac{\boldsymbol\gm u_o}{\rho}\,
    \Bigg[\,\frac{|x_1-x_2|}{\rho}+\sqrt{\frac{|t_1-t_2|}{
      u_o^{q+1-p}\rho^p}}\, \Bigg]^{\alpha_o},
\end{equation}
for all $z_1=(x_1,t_1), z_2=(x_2,t_2)\in Q_o$. 
\end{theorem}
Figure \ref{Fig:har} illustrates the region where Theorem~\ref{THM:REGULARITY-INTRO} holds. 
We emphasize that Theorem~\ref{THM:REGULARITY-INTRO} is false for $q=p-1$ and $q=\frac{N(p-1)}{(N-p)_+}$; see Section~\ref{sec:reg:opt-1}.

\begin{center}
\begin{figure}
\psfragscanon
\psfrag{p}{$\scriptstyle p$}
\psfrag{q}{$\scriptstyle q$}
\psfrag{pzer}{$\scriptstyle p=1$}
\psfrag{qzero}{$\scriptstyle q=0$}
\psfrag{pn}{$\scriptstyle p=N$}
\psfrag{line}{$\scriptstyle q=p-1$}
\psfrag{curve}{$\scriptstyle q=\frac{N(p-1)}{N-p}$}
\begin{center}
\includegraphics[width=.75\textwidth]{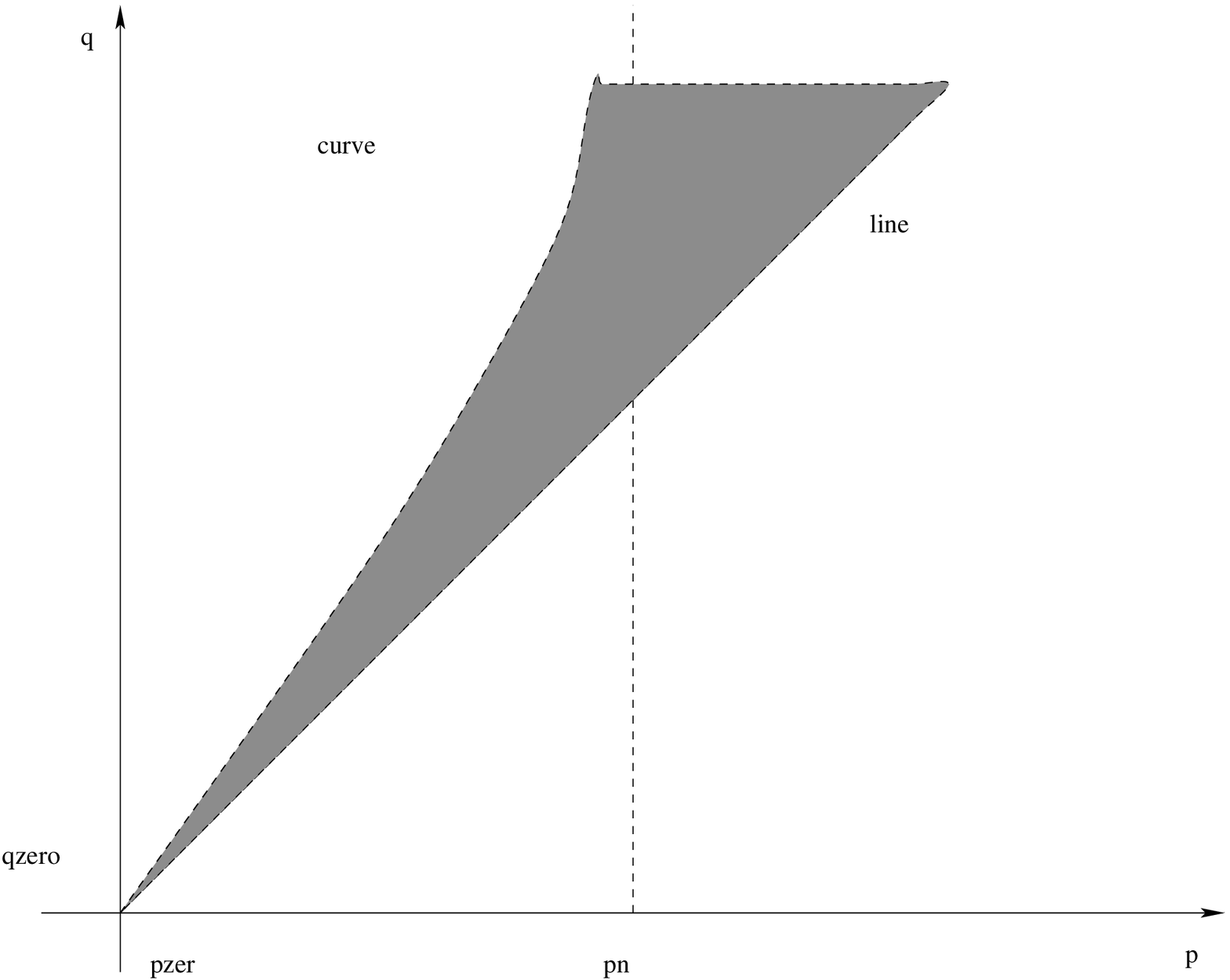}
\caption{\label{Fig:har}
}
\end{center}
\end{figure}
\end{center}

\begin{remark}\label{Rmk:semicontin}
In the statement of gradient regularity, the continuity of $u$ is assumed in order to give an unambiguous meaning to $u(x_o,t_o)$. It suffices to state these estimates for the {\it upper semicontinuous regularization} $u^*$ of $u$, which is, for locally bounded solutions, uniquely determined and verifies $u=u^*$~a.e. in $E_T$. See~\cite[Theorem~2.3]{Liao-JMPA-21} for details. 
\end{remark}

Based on Theorem~\ref{THM:REGULARITY-INTRO}, the gradient estimates can also be formulated in a generic compact subset $\mathcal{K}$ of $E_T$. In particular, $Du$ is H\"older continuous over $\mathcal{K}$, and precise estimates can be obtained in terms of $\sup_{\mathcal{K}}u$ and $\dist(\mathcal{K},\pl E_T)$.

\begin{corollary}\label{COR:GRAD-REG}
Under the assumptions of Theorem~\ref{THM:REGULARITY-INTRO} there exists $\alpha_1\in(0,1)$, depending only on $N,p$ and $q$,  such that the following holds.  Let
  $\mathcal{K}\subset E\times(0,T)$ be a compact subset and define
  \begin{equation*}
    \rho_o:=\inf_{\genfrac{}{}{0pt}{2}{(x,t)\in \mathcal K,}{ (y,s)\in\partial E_T}}\big(|x-y|+|t-s|^{\frac1p}\big)
  \end{equation*}
  and $\mathcal{M}:=\max\{1,\sup_{\mathcal{K}} u\}$. Then we have the $L^\infty$-gradient estimate
  \begin{equation}\label{grad-bound-K}
    \sup_{\mathcal{K}}|Du|\le \bg\frac{\mathcal{M}^{\frac{q+1}{p}}}{\rho_o},
  \end{equation}
  as well as the H\"older  gradient estimate
  \begin{equation}\label{grad-holder-K}
    |Du(x_1,t_1)-Du(x_2,t_2)|
     \le
    \bg\frac{\mathcal{M}^{\frac{q+1}{p}}}{\rho_o}
    \Bigg[\mathcal{M}^{\frac{q+1-p}{p}}\frac{|x_1-x_2|}{\rho_o}
    +\sqrt{\frac{|t_1-t_2|}{\rho_o^p}}\, \Bigg]^{\alpha_1}
\end{equation} 
for every $(x_1,t_1),$ $(x_2,t_2)\in \mathcal{K}$.     
\end{corollary}

For  $p=2$ our result reduces to the case of the singular porous medium equation in the range $1<q<\frac{N}{(N-2)_+}$. With $q=\frac1m$ 
and the substitution $w=u^\frac1m$ this corresponds to the bound $\frac{(N-2)_+}{N}<m<1$ for the standard form $\partial_tw-\Delta w^m=0$ of the porous medium equation. Therefore, Theorem \ref{THM:REGULARITY-INTRO} recovers
as a special case the gradient estimates of DiBenedetto \& Kwong \& Vespri \cite[Theorem 2.1]{DBKV} for weak solutions to the singular porous medium equation.

A few words about the proof (detailed in Chapter~\ref{sec:regularity}) and the underlying ideas are in order here. 
The first observation is general in nature, and reflects a common principle in the treatment of nonlinear parabolic equations with degeneracy as in \eqref{doubly-nonlinear-prototype}. One crucial building block in the study of the problem is the choice of the correct geometry in space and time: only the proper choice of geometry will lead to homogeneous estimates. In present situation it turns out that the natural cylinders for our purposes are given by  $Q_o:=K_\rho(x_o)\times
\big( t_o-u_o^{q+1-p}\rho^p, t_o+u_o^{q+1-p}\rho^p\big)$, provided $u_o>0$. The geometry reflected by such cylinders is \emph{intrinsic} as it varies according to the value of the solution at the center. Once the correct
{\it intrinsic geometry} 
is identified, one can turn to the actual regularity assertions.

The starting point is the fact that, in the described parameter range a Harnack inequality without waiting-time phenomenon holds true. Therefore, from the assumption $u_o>0$ we may conclude that $u$ is controlled from above and below in terms of $u_o$ on the symmetric, intrinsic, space-time cylinder $Q_o$ (backward-forward cylinder).
More precisely, if $\bg_{\rm h}$ denotes the Harnack constant, then $\bg_{\rm h}^{-1}\le u/u_o\le \bg_{\rm h}$ on $Q_o$. 
Rescaling $u/u_o$ in time and space to
the symmetric unit cylinder $\mathcal Q_1:= K_1(0)\times (-1,1)$ leads to a weak solution $\tilde u$ to the doubly non-linear parabolic equation \eqref{doubly-nonlinear-prototype} in $\mathcal Q_1$ with values in $[\bg_{\rm h}^{-1}, \bg_{\rm h}]$. This allows us to pass from   $\tilde u$ to $v:=\tilde u^q$.  Then  $v(y,s)$ is a weak solution to the parabolic equation
\begin{equation*}
	\partial_s v - 
	\Div\Big(\big(\tfrac1q\big)^{p-1}\tilde u^{(p-1)(1-q)}|Dv|^{p-2}Dv\Big)
	=
	0
	\quad\mbox{in $\mathcal{Q}_1$,}
\end{equation*}
with values in the interval $[\bg_{\rm h}^{-q}, \bg_{\rm h}^{q}]$. Furthermore, denoting
\begin{align*}
	a(y,s)
	:=
	\big(\tfrac1q\big)^{p-1}\tilde u(y,s)^{(p-1)(1-q)},
\end{align*}
the above equation can be interpreted as a parabolic equation of $p$-Laplace type with measurable and bounded coefficients that are also uniformly bounded away from zero. This allows us to apply
the by now classical $C^{0,\alpha}$-regularity results from \cite{DB} to deduce that $v$ is $\alpha$-Hölder continuous in $\frac12\mathcal{Q}_{1}$ with some Hölder exponent $\alpha\in(0,1)$. Together with the lower and upper bound for $v$, such a fact implies that the coefficient $a$ is also $\alpha$-Hölder continuous. In particular, this shows that $v$ is a weak solution to a parabolic equation of $p$-Laplace type with H\"older continuous coefficients in $\frac12\mathcal{Q}_{1}$. This reduces questions about Lipschitz and gradient H\"older regularity of $v$ to whether or not corresponding Schauder estimates for weak solutions can be established. Since the exponent $p$ can take values close to $1$ in the described parameter range, boundedness of weak solutions must be assumed in the context of Schauder estimates. However, in the range of $p$ and $q$ required, this is indeed the case. Theorem~\ref{theorem:schauder:intro} 
ensures  $\alpha_o$-H\"older continuity of $Dv$ for some  H\"older exponent $\alpha_o\in(0,1)$ together with the quantitative gradient bound 
$$
    \sup_{\frac14\mathcal Q_1}|Dv|\le \boldsymbol\gm,
$$
and the quantitative $\alpha_o$-H\"older gradient estimate 
\begin{equation*}
    \big|Dv(y_1,s_1) - Dv(y_2,s_2)\big|
    \le
    \boldsymbol\gm\Big[ |y_1-y_2|+\sqrt{|s_1-s_2|}\Big]^{\al_o}
\end{equation*}
for any $(y_1,s_1), (y_2,s_2)\in \tfrac14 \mathcal{Q}_{1}$.
The asserted inequalities for the original solution $u$ are obtained by rescaling $v$ to $u$.

 To summarize, the proof strategy consists of three major steps. In the first step, we prove that weak solutions are bounded in the parameter range under consideration. In the second step we show that for non-negative weak solutions an elliptic Harnack-type inequality holds. This means that the Harnack inequality holds without waiting-time (backward-forward Harnack inequality). Thus, in a neighborhood near a point $(x_o,t_o)$ with $u_o=u(x_o,t_o)>0$ weak solutions are controlled from below and from above by $u_o$.
In the third and last step, one  substitutes $v=u^q$. The resulting equation of $v$ can be interpreted as a parabolic $p$-Laplace equation with measurable coefficients, so that  the classical result about Hölder regularity of weak solutions can be directly applied to it, and consequently the coefficients (essentially power functions of the solution) are identified as being themselves Hölder continuous. At this point the desired gradient regularity can be deduced from suitable Schauder estimates for parabolic $p$-Laplace equations with Hölder continuous coefficients. The peculiarity (and new feature) in this step is that we need to establish these Schauder estimates for any $p>1$. In particular, $p$ may be close to $1$.
Needless to say, Schauder estimates  for $p$ close to 1, can only be true under the additional assumption that weak solutions are bounded, which, however, is granted for our application here. 

A unique feature of weak solutions to \eqref{doubly-nonlinear-prototype} in the fast diffusion regime is that they might become extinct abruptly.
As a direct application of the gradient estimates from Theorem~\ref{THM:REGULARITY-INTRO} and the integral Harnack inequality from Theorem~\ref{Thm:bd:2}, decay estimates for weak solutions at their extinction time can be established.

\begin{corollary}\label{Cor:12:4}
Let $p$ and $q$ be as in Theorem \ref{THM:REGULARITY-INTRO}. Then there exist
 constants $\bg>1$ and $\bg_o>1$ depending only on $N$, $p$ and $q$, such that if $u$ is a non-negative weak solution
to the doubly non-linear parabolic equation \eqref{doubly-nonlinear-prototype} in $E_T$ and if $T$ is an extinction time for $u$, i.e.~$u(\cdot,T)=0$ a.e.~in $E$,
then, for any $x_o\in E$ and $t_o\in(\tfrac12T,T)$ we have
\begin{align*}\label{ext-speed}
    u(x_o,t_o)
    &
    \le \bg\bigg[\frac{T-t_o}{d^p(x_o)}\bigg]^{\frac{1}{q+1-p}},\\
    \big|Du(x_o,t_o)\big|
    &
    \le \frac{\bg}{d(x_o)}\bigg[ \frac{T-t_o}{d^p(x_o)}\bigg]^{\frac{1}{q+1-p}}.
\end{align*}
Here, $d(x_o):=\dist(x_o,\partial E)$ denotes the usual Euclidean distance of $x_o$ to the boundary $\partial E$. Moreover, for any $r\in(0,d(x_o)/\bg_o)$
we have the following decay estimates for the oscillation of $u$ and its gradient, i.e.
\begin{align*}
   \osc_{K_r(x_o)} u(\cdot, t_o) 
   &\le 
   \bg\frac{ r}{d(x_o)}\bigg[\frac{T-t_o}{d^p(x_o)}\bigg]^{\frac{1}{q+1-p}},\\    \osc_{K_r(x_o)} Du(\cdot, t_o) 
   &\le
\frac{\bg}{d(x_o)}\bigg[\frac{r}{d(x_o)}\bigg]^{\al_o}\bigg[\frac{T-t_o}{d^p(x_o)}\bigg]^{\frac{1}{q+1-p}}.
\end{align*}
Here,  $\al_o\in(0,1)$ is the H\"older exponent from Theorem~\ref{THM:REGULARITY-INTRO}. 
\end{corollary}

Corollary~\ref{Cor:12:4} will be proved in Section~\ref{sec:ext-time}, where we will also provide an estimate for the extinction time of the solution to a Cauchy-Dirichlet problem.

\section{Boundedness, expansion of positivity }\label{sec:int-bound}

Besides being important building blocks in the general theory developed here, {\em boundedness of solutions} and the so-called {\em expansion of positivity} property are significant results by themselves. We study them in a fairly general setting. Indeed,
in this section we consider non-negative solutions to quasi-linear, parabolic partial differential equations of the form 
\begin{equation}  \label{Eq:1:1f}
	\partial_t u^q-\dvg\bl{A}(x,t,u, Du) = 0\quad \mbox{in $ E_T$.}
\end{equation}
 Here the function $\bl{A}(x,t,u,\xi)\colon E_T\times\rr^{N+1}\to\rn$ is Carath\'eodory (cf.~\S~\ref{sec:preliminaries}) and subject to the structure conditions
\begin{equation}\label{Eq:1:2}
	\left\{
	\begin{array}{c}
		\bl{A}(x,t,u,\xi)\cdot \xi\ge C_o|\xi|^p \\[5pt]
		|\bl{A}(x,t,u,\xi)|\le C_1|\xi|^{p-1}%
	\end{array}
	\right .
	\qquad \mbox{for a.e.~$(x,t)\in E_T$, $\forall\,u\in\rr$, $\forall\,\xi\in\rn$,}
\end{equation}
where $C_o$ and $C_1$ are given positive constants. The set of parameters $\{p, q, N, C_o,C_1\}$ is referred to as the structural data.

Both issues we deal with here
present similar aspects, as far as the previous literature is concerned. 

The boundedness results which we discuss next, are a consequence of the energy estimates we will state in Proposition~\ref{Prop:2:1} and De Giorgi's iteration. For all the details of the proof, see Chapter~\ref{sec:boundedness}.

\begin{theorem}[quantitative $L^\infty$-bound] \label{THM:BD:1}
 Assume that $0<p-1<q$ and let $r\ge1$ satisfy
\begin{equation}\label{def:lambda-r}
    \lm_r:=N(p-q-1)+rp>0.
\end{equation}
Let $u$ be a non-negative weak sub-solution to \eqref{Eq:1:1f} with \eqref{Eq:1:2} in $E_T$. 
In the case $r>m:=p\frac{N+q+1}{N}$, we assume additionally that $u$ is  locally bounded.
Then on any parabolic cylinder $Q_{\rho, s}=K_\rho(x_o)\times(t_o-s,t_o]\Subset E_T$ 
we have
\[
    \essup_{Q_{\frac12\rho, \frac12 s}} u 
    \le  
    \boldsymbol \gm  \Big(\frac{\rho^p}{s} \Big)^{\frac{N}{\lm_{r}}}  
    \bigg[
    \biint_{Q_{\rho, s}}u^{r}\,\dx\dt\bigg]^{\frac{p}{\lm_{r}}} 
    +
    \boldsymbol \gm 
    \Big(\frac{s}{\rho^p}\Big)^{\frac1{q+1-p}},
\]
with a constant $\boldsymbol\gm$ that depends only on the data. 
\end{theorem}


\begin{remark}\label{rm:boundedness}\upshape
We first note that $u\in L^m_{\loc}(E_T)$ by the embedding of Lemma~\ref{lem:Sobolev}. 
The proof of Theorem~\ref{THM:BD:1} will distinguish three cases. In the first two cases, we will deal with exponents $r\le m$ and show that local boundedness of weak sub-solutions is inherent in the notion of solution, 
whereas in the third case, we will examine exponents $r>m$ and establish the quantitative $L^\infty$-estimate by assuming that sub-solutions are locally bounded {\it a priori}. Note also that
there exists some $r\ge1$ such that $r\leq m$ and $\lm_r>0$ if and only if $\lambda_m>0$. Since 
\[
    \lambda_m =
    (N+p)(m-q-1)
    =
    \frac{N+p}{N}\,\lambda_{q+1},
\]
this is equivalent to $q+1<m$  on the one hand and to $\lambda_{q+1}>0$, respectively 
$$
q<\frac{N(p-1)+p}{(N-p)_+}
$$
on the other hand.
Therefore, in this case, Theorem~\ref{THM:BD:1} can be applied with $r=q+1$ and without assuming local boundedness of $u$.   
\end{remark}

In view of Remark~\ref{rm:boundedness} we obtain the following corollary of Theorem~\ref{THM:BD:1}. See also Figure~\ref{Fig:bound}, where the shaded region gives the admissible ranges of parameters $p$ and $q$ in order for $u$ to be bounded.

\begin{corollary}\label{Cor:bdd}
Assume that $0<p-1<q<\frac{N(p-1)+p}{(N-p)_+}$. Then, any non-negative weak sub-solution to \eqref{Eq:1:1f} with \eqref{Eq:1:2} in $E_T$ is locally bounded.
\end{corollary}

\begin{center}
\begin{figure}
\psfragscanon
\psfrag{p}{$\scriptstyle p$}
\psfrag{q}{$\scriptstyle q$}
\psfrag{pzer}{$\scriptstyle p=1$}
\psfrag{qzer}{$\scriptstyle q=0$}
\psfrag{pn}{$\scriptstyle p=N$}
\psfrag{line}{$\scriptstyle q=p-1$}
\psfrag{curve}{$\scriptstyle q=\frac{N(p-1)+p}{N-p}$}
\begin{center}
\includegraphics[width=.75\textwidth]{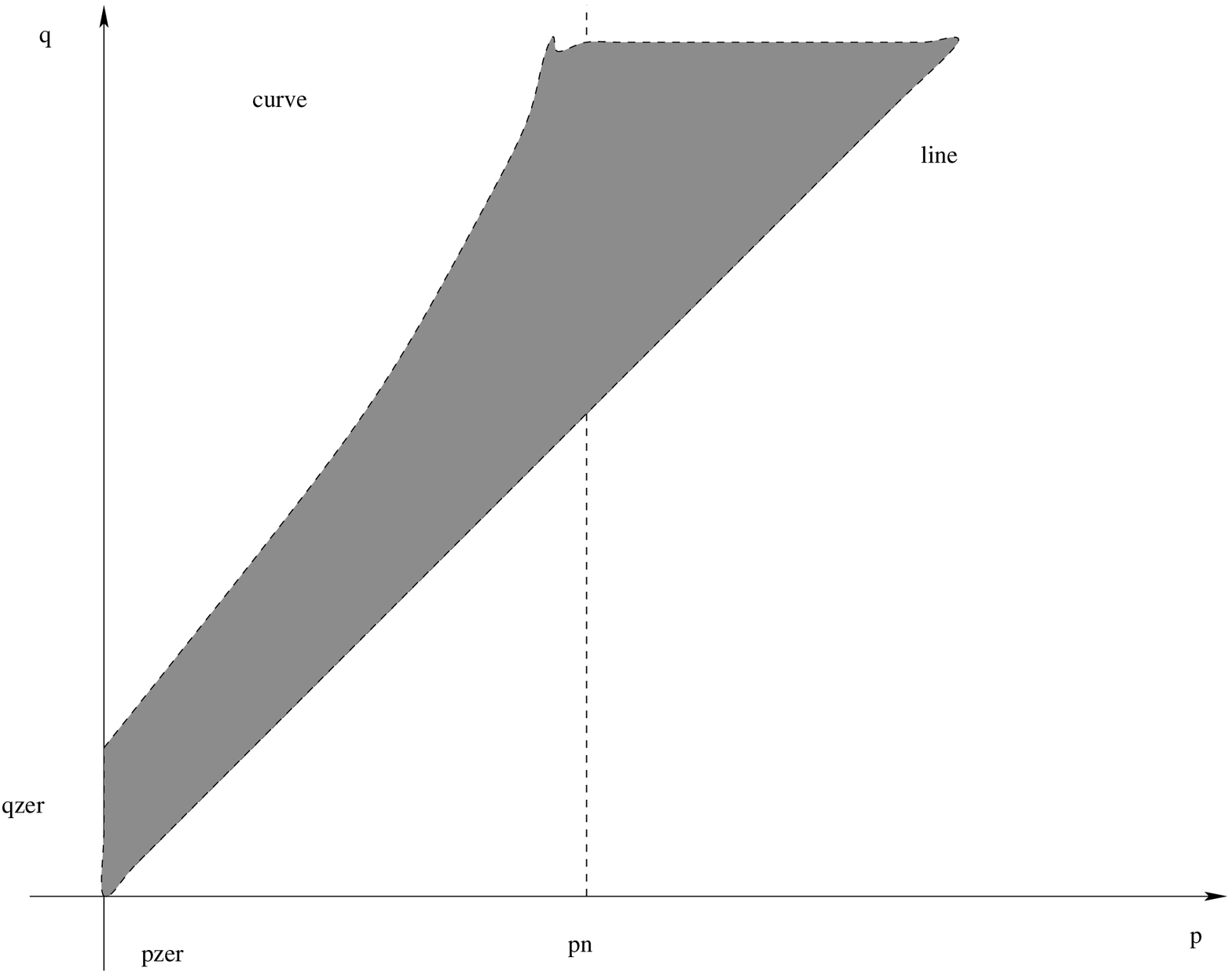}
\caption{\label{Fig:bound}}
\end{center}
\end{figure}
\end{center}

Further details about previous results in the literature, and the optimality of our estimates will be given in Section~\ref{S:bd-opt}; here it suffices to remark that 
the novelty of Theorem~\ref{THM:BD:1} consists in being, to the best of our knowledge, the first result where explicit {\em quantitative} boundedness estimates are provided for solutions to doubly non-linear parabolic equations with the full quasi-linear structure as in \eqref{Eq:1:1f}--\eqref{Eq:1:2}. 

Moreover, the full range $0<p-1<q$ is dealt with in Theorem~\ref{THM:BD:1}, and 
Corollary~\ref{Cor:bdd} clearly singles out the ranges of values of $p$ and $q$, such that the sheer notion of solution ensures its local boundedness, without the need to assume it a priori, and avoiding any kind of further technical limitations on $p$ and $q$. 

It had already been pointed out in \cite[Theorem~3.1]{FSV-14} that one needs to work with different energy estimates for different ranges of parameters, but in our opinion the statement here is terser, and also the proofs are more streamlined than the existing ones. 

Coming to the \textit{expansion of positivity}, as pointed out in \cite[Chapter~4, Section~1]{DBGV-mono}, it lies at the heart of any form of Harnack estimate, both for elliptic and for parabolic partial differential equations.

It is a property of non-negative super-solutions; if we limit ourselves to the parabolic setting for simplicity, it states that the information on the measure of the positivity set of a super-solution $u$ at the time level $s$ over the cube $K_{\varrho}(y)$ translates into an expansion of the positivity set both in space (from $K_{\varrho}(y)$ to $K_{2\varrho}(y)$) and in time (from $s$ to $s_1>s$, where the precise value of $s_1$ depends on the specific equation under consideration).

When dealing with \eqref{Eq:1:1f}--\eqref{Eq:1:2}, we have the following statement.
\begin{proposition}\label{PROP:EXPANSION}
Let $u$ be a non-negative,  local, weak super-solution to \eqref{Eq:1:1f} -- \eqref{Eq:1:2} in $E_T$. Assume that $0<p-1\le q$.
Suppose for some constants $M>0$ and $\al\in(0,1)$, there holds
	\begin{equation*}
		\Big|\big\{ u(\cdot, t_o) \ge M \big\}\cap K_{\varrho}(x_o)\Big|
		\ge
		\al  |K_\varrho |.
	\end{equation*}
Then there exist constants $\dl,\,\eta\in(0,1)$ depending only on the data and $\al$, such that 
\begin{equation*}
	u\ge\eta M
	\quad
	\mbox{a.e.~in $ K_{2\varrho}(x_o) \times\big( t_o+\tfrac12 \dl M^{q+1-p}\varrho^p,
	t_o+\dl M^{q+1-p}\varrho^p\big],$}
\end{equation*}
provided 
$$K_{8\rho}(x_o)\times\big(t_o, t_o+\dl M^{q+1-p}\varrho^p\big]\subset E_T.$$
\end{proposition}
\begin{remark}\label{Rmk:1:4}
\upshape
The constants $\eta$ and $\dl$ are stable as $q-1+p\downarrow0$. Therefore, Proposition~\ref{PROP:EXPANSION} recovers \cite[Proposition~B.1]{BDL-21}.
The dependence of $\eta(\al)$ can be traced by 
\[
\eta=\exp\Big\{-\boldsymbol\gm_1 \al^{-\boldsymbol\gm_3 -1} \exp\big\{(q+1-p) \boldsymbol\gm_2\al^{-\boldsymbol\gm_3}\big\}\Big\}
\]
for some $\boldsymbol\gm_1,\,\boldsymbol\gm_2>1$  depending only on the data, while $\boldsymbol\gm_3=\tfrac1{p-1}\big(2p+1+N\frac{p+1}{p}\big)$. 
\end{remark}

Per se, the result is not completely new. It was first given in \cite[Proposition~5]{FSV-14} for non-negative super-solutions to doubly non-linear equations written as
\begin{equation*}
u_t-\Div(|u|^{m-1}|Du|^{p-2}Du)=0,
\end{equation*}
for $p>1$, $2<p+m<3$. 
The  expansion of positivity was proven under the same extra assumption about the gradient $Du$ which we have already mentioned above when discussing local boundedness; once more, in \cite[Theorem~6.4]{Vespri-Vestberg} such a hypothesis was shown to be unnecessary.

Coming back to our formulation, if the equation is just the prototype \eqref{doubly-nonlinear-prototype}, Proposition~\ref{PROP:EXPANSION} was proven in 
\cite[Section~4]{Henriques-22} under the same conditions for $p$ and $q$ as we have here. 

In all these three instances, the proof relies on a proper adaptation of the techniques used in \cite[Chapter~4]{DBGV-mono} in the proof of the expansion of positivity for non-negative super-solutions to the singular parabolic $p$-Laplacian and porous medium equation.

Besides considering \eqref{Eq:1:1f} in its full generality under conditions \eqref{Eq:1:2}, here the other important novelty is represented by the technical tools we employ in the proof; indeed, we get back to the approach discussed in \cite[Chapter~4]{DB}, based on a proper use of the logarithmic function, and we suitably adapt it to our context. It is worth noticing that such a function is a fundamental tool in Moser's proof of the Harnack inequality for linear parabolic equations in divergence form with bounded and measurable coefficients (see \cite[Lemma~2]{moser-71}), and it is quite remarkable that it plays such an important role in two very different frameworks; it is somehow suggestive that the common diffusive feature of the two equations is more basic than their technical differences.

Finally, as pointed out in Remark~\ref{Rmk:1:4}, we pay particular attention to the stability of the estimates, so that we can recover the limiting case $p-1=q$.

The proof of Proposition~\ref{PROP:EXPANSION}, together with the development of all the necessary technical tools is given in Chapter~\ref{sec:Expansion}.
\section{Harnack's inequality}

Our first result concerning the Harnack inequality is as follows. For the same reason as Remark~\ref{Rmk:semicontin}, it is stated for continuous weak solutions. We recall Figure~\ref{Fig:har} which illustrates the range where the Harnack inequality holds. 

\begin{theorem}[Harnack inequality]\label{THM:HARNACK:0}
Let $0<p-1< q<\frac{N(p-1)}{(N-p)_+}$.
There are constants $\boldsymbol\gm>1$ and $\sig\in(0,1)$, which  depend
only on the data
$p$, $q$, $N$, $C_o$ and $C_1$, so that 
if $u$ is a  non-negative,  continuous, weak solution to the general parabolic equation
\eqref{Eq:1:1f}  in $E_T$ such that  the coercivity and growth conditions \eqref{Eq:1:2} hold, that $u(x_o,t_o)>0$, and  that the set inclusion 
\begin{equation} \label{Eq:set-incl}
K_{8\rho}(x_o) \times \big(t_o -  \mathcal{M}^{q+1-p}(8\rho)^p, t_o + \mathcal{M}^{q+1-p}(8\rho)^p\big)\subset E_T, 
\end{equation}
with
\[
    \mathcal{M}:=\sup_{K_{\rho}(x_o)}u(\cdot, t_o)
\]
hold, then for all
\[
    (x,t)\in K_{\rho}(x_o) \times \big(t_o - \sig [u(x_o,t_o)]^{q+1-p}\rho^p, t_o + \sig [u(x_o,t_o)]^{q+1-p}\rho^p\big),
\]
we have
\[
    \boldsymbol\gm^{-1} u(x_o,t_o) \le u(x,t) \le \boldsymbol\gm u(x_o,t_o).
\]
\end{theorem}

Theorem~\ref{THM:HARNACK:0} allows for a full general quasi-linear structure of $\bl{A}(x,t, u,\xi)$ in the doubly non-linear equations. However, the required set inclusion \eqref{Eq:set-incl} sometimes poses certain constraints in applications. 
Our second result weakens this requirement, yet only for the model equation; 
see Remark~\ref{Rmk:Harnack-cp} for possible generalizations.

\begin{theorem}[Harnack inequality:  model equation]\label{THM:HARNACK}
Let $0<p-1< q<\frac{N(p-1)}{(N-p)_+}$. There are constants $\boldsymbol\gm>1$ and $\sig\in(0,1)$, which  depend only on $N$, $p$ and $q$, such that if $u$ is a 
non-negative,  continuous, weak solution to the model equation \eqref{doubly-nonlinear-prototype} in $E_T$, that $u(x_o,t_o)>0$, and that the set inclusion
\[
K_{8\rho}(x_o) \times \big(t_o - \boldsymbol\gm [u(x_o,t_o)]^{q+1-p}(8\rho)^p, t_o + \boldsymbol\gm [u(x_o,t_o)]^{q+1-p}(8\rho)^p\big)\subset E_T
\]
holds, then for all 
\[
(x,t)\in K_{\rho}(x_o) \times \big(t_o - \sig [u(x_o,t_o)]^{q+1-p}\rho^p, t_o + \sig [u(x_o,t_o)]^{q+1-p}\rho^p\big),
\]
we have 
\[
\boldsymbol\gm^{-1} u(x_o,t_o) \le u(x,t) \le \boldsymbol\gm u(x_o,t_o).
\]
\end{theorem}

This kind of Harnack inequality presents two unique features that differ from the usual parabolic Harnack inequality. First of all, pointwise information has to be read in an intrinsic geometry induced by the particular scaling of the partial differential equation. Second, the usual waiting-time phenomenon is defied by  the forward-backward nature of the Harnack inequality. That is, a strong ``elliptic" feature appears, the role of time diminishes  and pointwise information can be read not only forward but also backward. We refer the reader to the monograph \cite{DBGV-mono} for a comprehensive discussion of this issue. In particular, our Theorem~\ref{THM:HARNACK:0} includes Theorems~1.2 \& 16.2 in \cite[Chapter~6]{DBGV-mono} as special cases and offers a unified perspective in a larger framework. A Harnack inequality for the complete slow diffusion range can be found in \cite{BHSS}. 

Harnack-type estimates for doubly non-linear equations of fast diffusion type have been considered by a number of authors, cf.~\cite{FSV-15, Vespri-Vestberg}. However, the previous works have used different notions of solution, and have required a time derivative of solutions, which is quite unnatural for a parabolic theory. In a recent work, a similar result as Theorem~\ref{THM:HARNACK:0} is obtained, cf.~\cite[Theorem~9.5]{Vespri-Vestberg}. However, their range of $p$ and $q$ is not optimal for the Harnack inequality to hold.

Our Harnack inequality dispenses with any requirement on the time derivative of solutions and holds true with an optimal range of $p$ and $q$. 
The optimality will be discussed in \S~\ref{S:Harnack-opt}. 

An interesting consequence of Theorem~\ref{THM:HARNACK:0} is the following Liouville-type result.

\begin{corollary}\label{Cor:Liouville}
Let $0<p-1< q<\frac{N(p-1)}{(N-p)_+}$. If $u$ is a non-negative weak solution to the general parabolic equation \eqref{Eq:1:1f} in any sub-domains of $\rn\times\rr$, then it must be a constant on $\rn\times\rr$.
\end{corollary}

As is well known for parabolic equations, a one-sided bound of solutions is generally insufficient to assert a Liouville-type result. Nevertheless, Corollary~\ref{Cor:Liouville} holds true thanks to the elliptic character of solutions as demonstrated in our Harnack inequality. The range of $p$ and $q$ for such a result is optimal; see \S~\ref{S:Harnack-opt} for examples. 

 
A key component for the Harnack inequality is an integral Harnack inequality.

\begin{theorem}\label{Thm:bd:2}
Let $0<p-1< q<\frac{N(p-1)}{(N-p)_+}$ and define $\lm_q:=N(p-1-q)+qp>0$.
There exists a positive constant $\boldsymbol\gm$ depending only on the data, such that
whenever $u$ is a non-negative weak solution to \eqref{Eq:1:1} with \eqref{Eq:1:2} in $E_T$, then
 \[
    \essup_{Q_{\frac12\rho, \frac12 s}} u\le  \boldsymbol \gm  \Big(\frac{\rho^p}{s} \Big)^{\frac{N}{\lm_{q}}}  
    \bigg[\inf_{t\in[t_o-s, t_o]}\bint_{K_{\rho}(x_o)\times\{t\}}u^{q}\, \dx\bigg]^{\frac{p}{\lm_{q}}} +\boldsymbol \gm \Big(\frac{s}{\rho^p}\Big)^{\frac1{q+1-p}},
\]
provided $Q_{\rho, s}=K_\rho(x_o)\times(t_o-s,t_o]$ is included in $E_T$. 
\end{theorem}

This result generalizes Theorem~2.1 for the parabolic $p$-Laplacian type equations and Theorem~17.1 for the porous medium type equations in \cite[Chapter~6]{DBGV-mono}. See also \cite{FSV-13,Vespri-Vestberg}. The novelty is that a unified approach is devised and a larger class of equations are dealt with. This kind of integral Harnack inequality also readily yields decay estimates of solutions near their possible extinction time.

A sizeable part of the proofs of Theorems~\ref{THM:HARNACK:0} \& \ref{THM:HARNACK} overlaps.
Yet, in obtaining Theorem~\ref{THM:HARNACK}, we rely on an interesting comparison principle for doubly non-linear equations of the form $\pl_t u^q - \dvg\bl{A}(x,u,Du)=0$, cf.~Proposition~\ref{Prop:CP}. Loosely speaking, we can compare two non-negative solutions if one of them vanishes on the lateral boundary. An important feature of the comparison principle, in contrast to previous results, is that no time derivative of solutions is required. Our approach relies on the doubling of the time variable developed by Otto~\cite{Otto}. We stress that while a general comparison principle holds when $q=1$ (cf.~Proposition~\ref{prop:comparison-plapl}), the same result remains elusive for doubly non-linear equations. See a recent contribution~\cite{LL-19} in this connection.

The effort poured so far notwithstanding, a complete theory of Harnack inequality for doubly non-linear equations is yet to be established. First of all, one would like to understand if Theorem~\ref{THM:HARNACK} holds true under the full general quasi-linear structure conditions \eqref{Eq:1:2}. In other words, one wonders if the Harnack inequality is a \emph{structural} property of differential operators and does not rely on solving any PDEs. A more challenging task is to explore what governs the local behavior of solutions beyond the borderline $q=\frac{N(p-1)}{(N-p)_+}$. To understand these issues would shed new light on unexplored mathematical structures and physical behavior of systems modeled by these equations.


\section{Schauder estimates for parabolic \texorpdfstring{$p$}{p}-Laplacian equations}

The Harnack inequality described in the previous section reduces the
analysis of the doubly non-linear equation to the treatment of
parabolic $p$-Laplace equations with H\"older continuous  coefficients.
More precisely, in this section we consider equations of the type
\begin{equation}\label{p-laplace-intro}
  \partial_tu-\Div\big(a(x,t)(\mu^2+|Du|^2)^{\frac{p-2}{2}}Du\big)=0
  \qquad\mbox{in $E_T$},
\end{equation}
for parameters $p>1$ and $\mu\in[0,1]$,  where the coefficients $a\in L^\infty(E_T)$ satisfy
\begin{equation}\label{prop-a-intro}
\left\{
  \begin{array}{c}
    C_o\le a(x,t)\le C_1,\\[6pt]
    |a(x,t)-a(y,t)|\le C_1|x-y|^\alpha,
  \end{array}
  \right.
\end{equation}
for a.e.~$x,y\in E$ and $t\in(0,T)$, with a H\"older exponent
$\alpha\in(0,1)$ and structural constants $0<C_o\le C_1$. 
For bounded weak solutions to this equation, we establish local H\"older regularity of the spatial gradient.
The precise result, including the corresponding quantitative estimates of Schauder type, is stated in the following theorem. 

\begin{theorem}[Schauder estimate for parabolic $p$-Laplace equations]\label{theorem:schauder:intro}
  Let $u$ be a bounded weak solution to~\eqref{p-laplace-intro} for
  parameters $p>1$ and $\mu\in[0,1]$, under
  assumptions~\eqref{prop-a-intro}.
  Then the spatial gradient satisfies 
  \begin{equation*}
    Du\in C^{\alpha_o,\alpha_o/2}_{\mathrm{loc}}(E_T,\R^N),
  \end{equation*}
  for a H\"older exponent $\alpha_o\in(0,1)$ that depends only on
  $N,p,C_o,C_1$ and $\alpha$.  
  Moreover, for any subset $\mathcal{K}\subset E_T$ 
  with $\rho:= \tfrac14\mathrm{dist}_{\mathrm{par}}(\mathcal{K},\partial_\mathrm{par} E_T)>0$, we have
  the quantitative local $L^\infty$-gradient estimate
  \begin{equation}
    \label{gradient-sup-bound-intro}
    \sup_{\mathcal{K}}|Du|
    \le
    \boldsymbol\gm \bigg[\frac{\boldsymbol \omega}{\rho}
    +
    \Big(\frac{\boldsymbol \omega}{\rho}\Big)^{\frac{2}{p}}
    +\mu
    \bigg]
    =:\boldsymbol\gm_o^{-1}\lambda,
  \end{equation}
  and the quantitative local gradient H\"older estimate
  \begin{equation}\label{gradient-holder-bound-intro}
    |Du(z_1) - Du(z_2)|
    \le
    \boldsymbol\gm\lambda
    \bigg[\frac{d_\mathrm{par}^{(\lambda)}(z_1,z_2)}{\min\{1,\lambda^{\frac{p-2}{2}}\}\rho}\bigg]^{\alpha_o}
    \quad \mbox{ for any $z_1,z_2\in\mathcal{K}$,} 
  \end{equation}
  for constants $\boldsymbol\gm=\boldsymbol\gm (N,p,C_o,C_1,\alpha)\ge1$ and
  $\boldsymbol\gm_o=\boldsymbol\gm_o(N,p,C_o,C_1)\ge1$. 
  Here, we abbreviated $\boldsymbol \omega:=\osc_{E_T} u$ and used the
  \emph{intrinsic parabolic distance} defined by 
  \begin{equation*}
    d_{\rm par}^{(\lambda)}(z_1,z_2):=|x_1-x_2|+\sqrt{\lambda^{p-2}|t_1-t_2|},
   \end{equation*}
   for $z_1,z_2\in E_T$ and $\lambda>0$.  
\end{theorem}

H\"older regularity of the gradient was first established by 
DiBenedetto \& Friedman for solutions of the model equation without coefficients in the groundbreaking works \cite{DiBenedetto-Friedman, DiBenedetto-Friedman2,
DiBenedetto-Friedman3}. These works cover the super-critical range $p>p_\ast:=\frac{2N}{N+2}$, while \cite{Choe:1991} contains extensions for arbitrary exponents $p>1$, but still without coefficients. 
Systems of parabolic $p$-Laplace type with H\"older continuous coefficients were treated in 
\cite{Chen-DB-89, Kuusi-Mingione, Misawa-Schauder}. Furthermore, the case of variable exponents $p(x,t)> p_\ast$ was included in \cite{Boegelein-Duzaar:p(z)}.

The main new feature of our result is 
that we cover equations with H\"older continuous coefficients for arbitrary exponents $p>1$.  Below the critical exponent $p_\ast$, weak solutions need not be locally bounded, cf. the counterexamples in Section \ref{S:bd-opt}. These counterexamples demonstrate that Schauder-type estimates cannot
be expected to hold for arbitrary weak solutions
to~\eqref{p-laplace-intro} if
$1<p\le\frac{2N}{N+2}$. Therefore, we restrict ourselves to bounded
solutions. This is sufficient for applications in our main theorem as discussed in Section~\ref{S:intro:main:thm}. 

The proof of the theorem is divided into two steps. As 
a first step, in Chapter \ref{sec:grad-bound}, we consider equations with differentiable coefficients. In order to avoid further technical difficulties stemming from the degeneracy of the equation, we restrict ourselves to the case $\mu>0$ in this chapter. However, it is crucial that all constants are independent of $\mu$. The second step is the generalization to equations with H\"older continuous coefficients, which we establish in Chapter~\ref{sec:Schauder} by considering a comparison problem with frozen coefficients.

The results for the case of differentiable coefficients in Chapter \ref{sec:grad-bound} are classical for exponents $p>\frac{2N}{N+2}$, for which boundedness of the gradient follows by Moser's iteration scheme. In order to include the sub-critical range $1<p\le\frac{2N}{N+2}$, we use ideas due to DiBenedetto \cite[Chapter VIII, Lemma 4.1]{DB}, see also \cite{DiBenedetto-Friedman3, Choe:1991}. In fact, in the sub-critical range, the start of Moser's iteration requires the local integrability of the gradient to a sufficiently large power $m>p$, i.e. $|Du|\in L^m_{\mathrm{loc}}(E_T)$. Following \cite[Chapter VIII, Lemma 4.1]{DB}, we raise the integrability of $|Du|$ by a clever manipulation of the energy estimate for second order derivatives in~\eqref{energy-est}. 
This is a unified treatment for all exponents $1<p\le 2$ among bounded solutions and yields the integrability $|Du|\in L^m_{\mathrm{loc}}(E_T)$ for every $m>1$, cf. Lemma~\ref{lem:Lq-est}. With this result at hand, Moser's iteration can be implemented 
to obtain local boundedness of the gradient. The resulting gradient estimates are stated on a backward parabolic cylinder of the type 
\begin{equation}\label{cyl:p-lap-gradient}
Q_\rho^{(\lambda)}(x_o,t_o)=B_\rho(x_o)\times(t_o-\lambda^{2-p}\rho^2,t_o], 
\end{equation}
where the parameter $\lambda>0$ can be chosen arbitrarily. In the later application, however, this parameter will be chosen in dependence of the size of $|Du|$, so that homogeneous estimates can be established within \eqref{cyl:p-lap-gradient}. 
Once the boundedness of the gradient is established
in the case of differentiable coefficients, the H\"older continuity of the gradient then follows for all exponents $p>1$ from the theory developed in the seminal works by DiBenedetto \& Friedman \cite{DiBenedetto-Friedman, DiBenedetto-Friedman2,
DiBenedetto-Friedman3}, see also \cite{Chen,Choe:1991} 
and the monograph \cite[Chapter~IX]{DB}.

The aim of Chapter~\ref{sec:Schauder} is to generalize these regularity results to equations with H\"older continuous coefficients. The strategy for this is to compare the given solution with a solution to a related problem with frozen coefficients. 
This is also the underlying idea in the earlier work \cite{Boegelein-Duzaar:p(z)}, which covers variable exponents $p(x,t)>\frac{2N}{N+2}$. However, in the present work we develop  
a slightly different approach than that in
\cite{Boegelein-Duzaar:p(z)}, which allows us to include the sub-critical 
case, and meanwhile, contains some simplifications also in the super-critical case. First we prove {\it a priori} estimates under the additional assumption that the gradient is bounded. 
The virtual advantage of this approach is that we are able to work with cylinders $Q_{\rho}^{(\lambda)}$ of a fixed geometry given 
by the parameter $\lambda$, which can be chosen in dependence on $\|Du\|_{L^\infty}$. 
This represents a significant simplification compared to \cite{Boegelein-Duzaar:p(z)}, in which the authors were forced to use a sequence of cylinders $Q_{\rho_i}^{(\lambda_i)}$ with varying parameters $\lambda_i$ that take into account a possible blow-up of the gradient. 
In order to work with cylinders of a fixed geometry, it is crucial that the article \cite{BDLS-Tolksdorf} provides an improved version of a Campanato type estimate for differentiable coefficients that holds on cylinders $Q_r^{(\lambda)}$ with a fixed geometry given by $\lambda>0$ and arbitrary $r\in(0,\rho]$, cf. Lemma~\ref{lem:campanato}. In Chapter~\ref{sec:Schauder}, we prove that this Campanato type estimate 
can be extended to solutions of the general problem under the additional assumption of boundedness of the gradient. This yields the desired {\it a priori} estimates. Finally, we show by an approximation argument that the {\it a priori} estimate can be extended to the general case, which completes the proof of Theorem~\ref{theorem:schauder:intro}.  

We note that in the super-critical case $p>\frac{2N}{N+2}$, the method described above can also be used to derive gradient bounds that are independent of the oscillation of the solution and depend on integrals of $|Du|^p$ instead. The necessary changes are indicated in Remark~\ref{rem:super-critical}. Consequently, in the super-critical range, our approach is not limited to the treatment of bounded solutions.

\chapter[A model for the flow of fluids in porous media]{A model for the flow of fluids in porous media}\label{SEC:model}

We present a simple description of the flow of fluids in a special class of porous media, namely, those considered in hydrogeology, such as soils, sands, porous rocks, fractured crystalline rocks, etc. Indeed, the aim of this Chapter is to focus on doubly non-linear parabolic equations, of the type we will deal with in the following, and show how they represent a valid tool when providing the physical description of such a kind of motion of fluids. We will not take into account different porous media, like, just to name an example, coffee powder, when it is flooded with hot water in a coffee machine. 

There is no presumption of either completeness or novelty: we try and present here in a unified way some results, which are widely known among people working in the field from an applied point of view, but perhaps, not so much to pure mathematicians. We dwell on few specific problems, but there are other issues, which give rise to the same class of equations we are interested in, and which we will not touch upon.  A more thorough discussion of the issues we deal with, can be found in \cite{Barenblatt}. For
somehow different models, the interested reader can refer to \cite{SW1,SW2}. More references can also be found in \cite{K}. 
A very interesting and rich historical account is given in \cite{BGKT}, where a long list of references is provided as well.

The actual PDE models will be presented in Sections~\ref{par_eq_flu_non_new} and \ref{nanoporous}: In \S~\ref{par_eq_flu_non_new} we will focus on non-Newtonian, polytropic fluids, whereas in 
\S~\ref{nanoporous}, we will briefly discuss the case of nanoporous media, where particular features occur. Finally, in \S~\ref{top-opt} we will shortly introduce a very recent application of doubly non-linear diffusion equations to engineering optimization problems.
\section{Some characteristic quantities}\label{par_grand_caratt} 
The geometry of a general 
porous medium can be described by introducing a small number
of variables. Among these, the most important one is {\em porosity}.
Let us consider a volume $V$ of a porous medium and denote by $V_{p}$
the volume occupied in it by the pores. We define the {\em average porosity} for
the element $V$ as the ratio 
$$
\displaystyle \varphi=\frac{V_{p}}{V}.
$$

Once a point in the medium has been fixed, it is possible to define the {\em local porosity} as
the limit of the average porosity for volume elements that contain
the point and whose size progressively decreases. We assume 
that these dimensions remain sufficiently large
compared to those of pores and grains: in this way, it is guaranteed
the existence of the limit. In no way the global behavior of the fluid is affected by the presence of isolated pores.
Therefore, these will be neglected in the porosity calculation.

To further specify the geometry of the porous medium 
it is necessary to introduce a characteristic length $d$ that takes into account
the size of the pores (porosity does not provide any
information in this sense). In general, $d$ has the form of a
probability distribution of a random variable $X$, which
describes the size of the pores, built on a sample
which is representative of the porous medium. In some cases, it suffices to
consider $d$ as a parameter equal to the average value
of the random variable $X$. This approximation makes sense if
the variance of $X$ is small, that is, the sizes of the different
pores do not differ too much from the average value. Conversely, if the
pore sizes differ a lot from one another, the variance of $X$
becomes meaningful. It is commonly assumed that the distribution of
$X$ has a standard form, which can be completely
determined by a small number of parameters. For a normal distribution, for example, two parameters are sufficient: average and
variance.

Besides the {\em geometric} description of the porous medium, we now introduce
some {\em physical} quantities that characterize the fluid. 

Let us consider a
volume $W$ of fluid and denote by $m_{W}$ the mass of the fluid contained in
it. We pick a point $Q$ within $W$ and we gradually restrict the volume $W$ around
$Q$. If the limit
\[
\lim_{W\to0}\frac{m_{W}}{W}
\]
exists, we define {\em density} of the fluid in $Q$ via the quantity
\[
\rho_{Q}=\lim_{W\to0}\frac{m_{W}}{W}.
\]
If the density does not depend on the choice of the point $Q$, the fluid
is called {\em uniform}. For simplicity, in the following we drop the suffix $Q$
when writing $\rho$.

Regarding the motion of the fluid in the
porous matrix, the fundamental quantities are \emph{pressure} and
\emph{flow velocity}. To define the first one,
some considerations about the forces acting on a fluid are in order. On each
volume element we can have volume forces or surface forces: indeed, we cannot talk
of a force applied to a single point. Let $dS$ be
an infinitesimal surface that surrounds a point $R$ of the fluid.
We call {\em pressure} at $R$ the ratio between the module $dF$ of the
resultant of the surface forces acting on $dS$ and the area of
$dS$. The pressure has no directional characteristics, it is just a 
scalar function, and does not depend on
the orientation of the surface on which it is measured. The letter $p$ is generally used to denote the pressure, and we let
\[
p=\frac{dF}{dS}.
\]
If we have a fluid of density $\rho$, upon which gravity acts, we define the {\em hydrostatic pressure} $p_{I}=\rho g z$,
where $g$ is the gravity acceleration and $z$ is the height with respect to a reference level. The hydrostatic pressure can be seen
as the \emph{gravitational potential energy}. To be completely rigorous, in the energy balance we should also take into account the kinetic energy. However, since the average velocity of the fluid flow in a typical porous medium is very low (typically, of the order of a meter per day),
it can be discarded.
The {\em actual pressure} is the difference between the {\em total pressure} and the hydrostatic pressure, i.e. 
\[
P=p-\rho g z.
\]

In general, the density assumes different values as the
temperature and the pressure to which the fluid is subject might change; in such a case, the fluid is called 
{\em compressible}. If the density depends on the
temperature, but not on the pressure, the fluid is called {\em expandable}.
Finally, if the density is constant, the fluid is said {\em incompressible}.

The real velocity of the groundwater in a porous medium is fluctuating in a highly and unpredictable fashion due to the intrinsic structure
of the material under consideration, and therefore, it cannot be used 
in a practical, physical description. Hence, we need to refer to an \emph{average velocity} of the flow. In order to define it, we consider 
a point $T$ of the porous medium and a surface element $\Delta S$, which contains it. 
Let $\mathbf{n}$ be one of the two possible unit vectors which are orthogonal to $\Delta S$, 
and denote with $\Delta M$ the amount of mass which flows per unit time
through $\Delta S$ in direction $\mathbf{n}$. We call {\em (average) velocity} of the fluid
 in direction $\mathbf{n}$ the limit $\frac{\Delta M}{\rho \Delta S}$, as $\Delta S$ shrinks around $T$, namely
\[
u_{\mathbf{n}}=\lim_{\Delta S \to 0}\frac{\Delta M}{\rho\Delta S}.
\]
By $\Delta S$ we mean the total area, not just the part taken by the pores: the assumption
that porous medium and fluid are both continuous entities is fundamental in the approach we develop. The quantity
$\frac{\Delta M}{\rho}$ represents the volume which flows per unit time through $\Delta S$ in direction $\mathbf{n}$.

In the applications, a Cartesian reference system is fixed, and the directions of the three coordinate axes $x_1$, $x_2$, $x_3$ are considered, so that the (average) velocity is
$$
\mathbf{u}=(u_{1},u_{2},u_{3}).
$$

Finally, a significant quantity for moving fluids is \emph{viscosity}. 
It is closely related to the concept
of internal friction, which we describe next. When there is
a relative flow between two fluid elements,
a tangential force, called
``internal friction force'' appears, which has direction opposite to
that of the relative speed. Therefore, it is a force that opposes sliding; experimentally,
its amplitude is written as
\[
dF=\mu dS \frac{dv}{dn},
\]
where $dS$ represents the contact area and $\frac{dv}{dn}$ is the variation
of the velocity modulus in direction orthogonal to $dS$. The coefficient $\mu$ is the {\em viscosity} of the fluid,
and it depends on the kind of fluid and on the temperature.

For many years there have been attempts of building sufficiently complete mathematical descriptions of mechanistic phenomena, starting from a few fundamental principles. These considerations lead to a system of partial differential equations involving the above-mentioned physical quantities. Despite the elegance and depth which have been achieved in many cases, some of the phenomena still elude such an analytical framework.
\section{Conservation of mass}\label{par_cons_massa}
The first equation we need in order to describe the flow of a fluid through a porous medium
is deduced from the conservation of mass. If we take into account the notation of the previous section,
we have the \emph{equation of continuity}
\begin{equation}\label{continuita}
\pl_t(\varphi\rho)+\Div(\rho\,\mathbf{u})=0.
\end{equation}
Its derivation is well-known, and we do not discuss it here. For more details, we refer, for example, to Section~4.2 of \cite{Bear}. In particular, for incompressible fluids one has $\dvg\mathbf{u}=0$. Moreover, strictly speaking $\vp$ depends both on space and time, but as a first approximation, it can be taken as time-independent or even constant.
\section{Conservation of momentum, velocity and pressure}
\subsection{Navier-Stokes equations}
Generally speaking, besides the continuity equation introduced above in \eqref{continuita}, from the 
conservation of momentum we obtain a second equation,
which describes the relationship between the velocity field $\mathbf{u}$ of the
flow and the volume and surface forces acting on the same fluid.
In the Cauchy formalism, on any surface dividing a body, there are forces acting on one part due to the other part, which opposes its deformation. Such forces are represented by the so-called \emph{traction vector} $\mathbf{T}(x,\mathbf{n},t)=\mathbb{T}\cdot \mathbf{n}$ where $\mathbf{n}$ is the normal vector on the surface and the symmetric matrix $\mathbb{T}\equiv(\tau_{ij})_{3\times3}$ is the \emph{stress tensor}. The conservation of momentum then induces the \emph{equation of motion} 
\begin{equation}\label{motion}
    \rho \big(\pl_t \mathbf{u}+ (\mathbf{u}\cdot \nabla)\mathbf{u}\big)=\dvg \mathbb{T},
\end{equation}
where here and in the remainder of the Chapter we conform to the notation employed by physicists, and use $\nabla f$ to denote $Df$.
If the traction vector is always normal to the surface, i.e. $$\mathbf{T}(x,\mathbf{n},t)=\mathbb{T}\cdot \mathbf{n}=-p(x,t)\mathbf{n}$$ for some function $p(x,t)$ called \emph{pressure}, the fluid is \emph{ideal}. However, in many cases, shear forces due to mutual sliding of layers have to be considered, and the stress tensor takes the general form 
\begin{equation}\label{tau-ij}
    \tau_{ij}=-p\dl_{ij}+\sig_{ij},
\end{equation}
 where $\sig_{ij}$ reflects shear forces. Since shear forces are due to different velocities of layers, we can deem $\sig_{ij}$ as a function of $\nabla \mathbf{u}$ and $\sig_{ij}=0$ if $\nabla\mathbf{u}=0$. A fluid is \emph{Newtonian} if $\sig_{ij}$ depends linearly on $\nabla \mathbf{u}$ and \emph{non-Newtonian} otherwise.
In the case of Newtonian fluids it can be shown that there exist parameters $\mu$ and $\lm$, such that
\begin{equation}\label{sigma-ij}
    \sig_{ij}=2\mu\,\frac{\pl_{x_j}u_i+\pl_{x_i}u_j}2+\lm \dl_{ij} \dvg \mathbf{u},\quad i,j=1,2,3.
\end{equation}
Here $\mu$ is called the \emph{kinematic viscosity}, and $\lambda$ is a further scalar function of the thermodynamic state. As discussed in \cite[Section~61]{serrin}, $\mu\ge0$, and for a compressible fluid $3\lambda+2\mu\ge0$. Rather frequently, for compressible fluids it is directly assumed $\lambda=-\frac23\mu$; on the other hand, for incompressible fluids, $\lambda=0$.

A Newtonian fluid is called {\em ideal} if $\mu=0$, that is, when the surface forces have zero tangential components. 

Plugging \eqref{tau-ij} and \eqref{sigma-ij} back to the equation of motion \eqref{motion} and using the incompressibility condition $\dvg \mathbf{u}=0$, we obtain the Navier-Stokes equations for incompressible fluids
\begin{equation}\label{Eq:N-S}
\left\{
\begin{array}{cc}
      \rho \big(\pl_t \mathbf{u}+ (\mathbf{u}\cdot \nabla)\mathbf{u}\big)=-\nabla p + \mu\Dl\mathbf{u},\\ [5pt]
      \dvg\mathbf{u}=0.
\end{array}\right.
\end{equation}
The Navier-Stokes equations consist of  \eqref{Eq:N-S}$_1$ derived from the conservation of momentum and \eqref{Eq:N-S}$_2$ derived from the conservation of mass. The interested reader is refered to \cite{serrin} and \cite[Chapter~12]{DB-mech} for more details. We stress that the Navier-Stokes equations \eqref{Eq:N-S} describe \emph{microscale} behavior of fluids.

For non-Newtonian fluids, in general, one cannot have a stress tensor as in \eqref{sigma-ij},
and the classification is usually based on the relation that exists between their stress tensor and the one for a Newtonian fluid.


\subsection{Darcy's law}
The equation of motion \eqref{motion} (or its derived equation~\eqref{Eq:N-S}$_1$) cannot be used directly to describe the behavior of fluids in a porous medium. 
What replaces the equation of motion is the key in understanding the behavior of fluids in a porous medium.

In the porous media theory the link between
$\mathbf{u}$ and $p$ has a local feature, due to the
complex structure of the system at the microscopic level. In fact, we have
already observed that the sizes of pores and grains
are very small.
Furthermore, the flow in the solid matrix is equivalent to the motion of
a fluid confined in an intricate system of interconnected tubes, and it is
generated by the pressure variations at the ends of the ducts. What governs the flow is the
pressure gradient.

Let us fix a point $Q$ of the porous medium: in a small neighborhood of $Q$
we can assume that $\mathbf{u}$ is continuous and that the parameters
of the porous medium and of the fluid are constant. Obviously it is vacuous 
to assume $p$ as constant, since in that case, it would not exist
flow motion. The quantities that come into play in the relationship
between $\mathbf{u}$ and $\nabla p$ are $\varphi$, $d$, $\rho$ and $\mu$.

Fix a Cartesian reference system $\{x_1,x_2,x_3\}$ and consider for simplicity
the case of an isotropic porous medium, that is, a medium whose properties are invariant 
with respect to rotations or inversion of the reference axes. Our aim is to determine a law of the kind
\[
\nabla p=h(\varphi,d,\rho,\mu,\mathbf{u}).
\]
All the quantities listed above are scalar quantities, only $\mathbf{u}$ and
$\nabla p$ are vectorial quantities. Due to the isotropy of the medium, $\nabla p$
must have the same direction as $\mathbf{u}$, a fact which is not hard to prove.
In particular, this means that there exists a function $l$, which depends on
$\varphi$, $d$, $\rho$, $\mu$, and $|\mathbf{u}|$ such that
\begin{equation}\label{Darcy1}
\nabla p=-l(\varphi,d,\rho,\mu,|\mathbf{u}|)\mathbf{u}.
\end{equation}
The law (\ref{Darcy1}) was first obtained in 1856 by Darcy, as a result of experiments. 

If we let $[\cdot]$ be the dimension of the indicated physical quantity, and we denote with $[L]$, $[M]$, $[T]$ the dimensions of length, mass, and time respectively, by their sheer definition it is apparent that 
\[
[\nabla p]=\frac{[M]}{[L]^2[T]^2},\quad [\mathbf{u}]=\frac{[L]}{[T]},\quad [\mu]=\frac{[M]}{[L][T]},
\]
so that \eqref{Darcy1} implies  the dimension of $l$ should be
\[
[l]=\frac{[M]}{[L]^3[T]}=\frac{[M]}{[L][T]}\cdot\frac{1}{[L]^2}\quad\Rightarrow\quad l\sim\frac{\mu}{d^2}.
\]
Since $\varphi$ is dimensionless, more generally we can express $l$ in the form 
\[
l=\frac{\mu}{d^{2}}f(\varphi).
\]
Substituting $l$ in \eqref{Darcy1} yields
\begin{equation*}
\nabla p=-\frac{\mu}{d^{2}}f(\varphi)\mathbf{u}.
\end{equation*}
Such an equation is usually referred to as {\em Darcy's seepage flow
law} and the quantity $k:=\frac{d^{2}}{f(\varphi)}$ is called
{\em permeability} of the porous medium; $k$ is a purely geometric property, 
it has the same dimensions as an area (i.e. $[k]=[L]^2$), and it is independent of the properties of the specific fluid under consideration.
Once the permeability $k$ is introduced, the Darcy law is usually stated as 
\begin{equation}\label{Darcy2}
\nabla p=-\frac{\mu}{k}\mathbf{u}\quad\text{ or equivalently, }\quad\mathbf{u}=-\frac k\mu \nabla p.
\end{equation}
We will get back to the permeability when dealing with flows in nanoporous media.

Interestingly, \eqref{Darcy2} can also be deduced from a statistical integration of the Navier-Stokes equations. Under this point of view, Darcy's law can be regarded as an average (macroscopic) momentum equation. 
To shed some light, suppose the fluid moves slowly such that the left-hand side of \eqref{Eq:N-S}$_1$ vanishes, and consequently, $-\nabla p=-\mu\Dl\mathbf{u}$. The right-hand term results from the shear force that resists deformation of the fluid and balances the pressure. From the viewpoint of dimensionality, $\mu\Dl \mathbf{u}$ agrees with $\frac{\mu}{k} \mathbf{u}$ of \eqref{Darcy2}. However, the presence of pores breaks this microscale description of Navier-Stokes equations. Instead, one could introduce a resisting force~$\mathbf{R}$ that arises from the drag force of the fluid particles acting on the grains, and assume it to be proportional to a certain averaged velocity. See, for example, \cite{Neuman} for more details and also for a historical account of this topic.

Not all fluids
satisfy a Darcy-like relation. A number of different laws have been obtained 
that describe the relationship between fluid velocity and
pressure for those flows, which we will call {\em non-Darcian}. 

It is worth pointing out that these laws are usually inferred from one-dimensional experiments. Since real groundwater flow
is three dimensional, there is the need to extrapolate from $1$-D to $3$-D. The extension is rather straightforward, if one takes into account that the properties of a porous material are the same in all directions (see the discussion in Section~2.5 of \cite{BGK}). 

An interesting example of these non-Darcian laws is the following one,
derived by Khristianovich in 1940
for isotropic porous media (see \cite{Kr}):
\begin{equation}\label{Khristianovic}
\left\{
\begin{aligned}
\nabla
&p=-\Pi\Phi\left(\frac{|\mathbf{u}|}{\lambda}\right)\frac{\mathbf{u}}{|\mathbf{u}|};\\
&\Phi(0)\ge 0,\quad\Phi'(z)\ge 0,\quad 0\le
z<\infty,\quad z=\frac{|\mathbf{u}|}{\lambda}.
\end{aligned}
\right.
\end{equation}
$\Pi$ and $\lambda$ are respectively the {\em characteristic} pressure and velocity. If $\Phi(0)=0$, then the fluid flows
even for very small pressure gradients. On the other hand, if $\Phi(0)>0$ there exists a threshold modulus $\Pi\Phi(0)$ for the pressure gradient,
such that the fluid is at rest if $|\nabla p|<\Pi\Phi(0)$. 

Here $\Phi$ is very general; in Section~\ref{par_eq_flu_non_new} we will give some explicit examples.
\section{Equations of state}
Up to this point we have obtained two equations, a scalar and
a vectorial one, which are respectively connected to the conservation of 
mass and momentum. They represent a so-called {\em non-closed system},
due to the large number of unknowns they contain.
They are generally accompanied by the conservation-of-energy equation, 
i.e. the first law of thermodynamics, which requires that the rate of change of the internal and kinetic energies be balanced by the rate of mechanical work and heat.
In particular, such an equation describes 
the link between the
characteristic quantities of the problem and the variation of
temperature. However, in most applications,
 the flow can be
considered isothermal and the conservation of
energy consequently loses importance. Instead, 
the so-called {\em equations of state} retain a significant role. They are
also called {\em constitutive equations}, since they describe in what
way the pressure acts on the properties that characterize the
fluid. Of particular importance is the relationship between pressure and
density. For each fluid it is possible to determine a modulus of
compressibility, which establishes a threshold pressure variation,
below which the fluid can be considered as
incompressible with good approximation. 

For homogeneous liquids, such as  
water and oil, the equation of state is the
linear law
\begin{equation}\label{Eq:oil}
\rho(p)=\rho_{o}\left[1+\frac{p-p_{o}}{K}\right],
\end{equation}
where $\rho_{o}$, $p_{o}$ and $K$ are constants. In particular, $1/K$ is the \emph{fluid compressibility}. When a fluid,
as for example oil, is \emph{weakly compressible}, we can assume $ \frac{p-p_{o}}{K}\ll1$.

Things are different when dealing with gases. 
If they can be considered as {\em ideal} from a thermodynamical point of view, 
then the equation of state that characterizes them is
\begin{equation}\label{gasperfetti}
\rho=\frac{p}{NRT},
\end{equation}
where the gas constant $R$ and the amount of gas measured in moles $N$ are constants. 

However, the value of the pressure inside the porous medium is frequently very high
and the gas cannot be considered as ideal. An important extension of (\ref{gasperfetti}) is given by
\begin{equation}\label{gaspolitropici}
p=\frac{\bar{p}\rho^{n}}{\bar{\rho}^{n}},
\end{equation}
where $\bar{p}$ and $\bar{\rho}$ are two reference parameters for
pressure and density, and $n>1$ is a dimensionless constant. The gases which satisfy 
(\ref{gaspolitropici}) are called {\em polytropic}.

We should now deal with viscosity. As for liquids,
dealing with pressure variations contained in the usual range of a few
megapascal, the viscosity can be assumed to be constant. This
hypothesis is obviously not acceptable for mixtures of liquids and
gases. In fact, let us think of a fluid consisting of a gas dissolved in
oil. A pressure drop causes a decrease in the
proportion of gas in the fluid, causing an impressive change of the
viscosity of the latter. 

Finally, according to the kinetic theory, the
viscosity of a gas should not depend on pressure. As a matter of fact, in nature, for a gas contained in a porous material, a
pressure variations of a few megapascal can result in variations of the order of
ten percent of the viscosity. We refrain from further developing this issue.
\section{Motion of a non-Newtonian fluid in a porous medium}\label{par_eq_flu_non_new} 
Let us go back to non-Darcian, non-Newtonian fluids, which flow in an isotropic porous medium. As we saw in Section~\ref{Khristianovic}, a suitable law, which describes the relationship between velocity and pressure, is  given by
\eqref{Khristianovic}. 

Moreover, let us suppose 
that the fluid is polytropic, namely, that its equation of state for density is given by \eqref{gaspolitropici}. 

We can rewrite
\eqref{Khristianovic} as
\begin{equation*}
\mathbf{u}=-\lambda\Psi\left(\frac{|\nabla
p|}{\Pi}\right)\frac{\nabla p}{|\nabla p|},
\end{equation*}
where $\Psi$ is the inverse function of $\Phi$. 

It is apparent that when
\[
\Psi(r)=C_1 r,
\]
we are back to the Darcy law \eqref{Darcy2}.

For a wide class of non-Newtonian fluids, such as polymers or mixtures of polymers, 
the law that properly describes the flow in the medium is 
\begin{equation}\label{Eq:power}
\Psi(r)=C_2 r^{\alpha},
\end{equation}
where $C_2$ and $\alpha$ are suitable positive, constant parameters. In the so-called Smreker-Izbash-Missbach law
we have $\alpha\in(\frac12,1)$, but a wide range of values of $\alpha$ has been considered in the literature (see, for example, Section~2.5 of 
\cite{BGK}). 

When \eqref{Eq:power} holds, we end up with an expression for velocity of the kind
\begin{equation}\label{psipotenza}
\mathbf{u}=-D\,|\nabla p|^{\alpha-1}\nabla p,
\end{equation}
with $D>0$ again a proper constant. Substituting \eqref{psipotenza} in \eqref{continuita} yields
\begin{equation*}
\partial_t(\varphi\rho) -D \Div(\rho\, |\nabla p|^{\alpha-1}\nabla p)=0.
\end{equation*}
Since the fluid is polytropic, we can take into account the expression $\displaystyle p=A\rho^{n}$, 
where $A:=\bar{p}/\bar{\rho}^{n}$ is a positive, constant parameter.
Hence, we obtain
\begin{equation*}
\partial_t(\varphi\rho) -D\Div(\rho\, |\nabla(A\rho^{n})|^{\alpha-1}\nabla(A\rho^{n}))=0,
\end{equation*}
which we can rewrite as
\begin{equation}\label{eq:fond_fluido}
\partial_t \rho -K \Div(\rho^{1+(n-1)\alpha}\,|\nabla\rho|^{\alpha-1}\nabla\rho)=0,
\end{equation}
since we may assume as a first approximation that $\varphi$ does not depend either on $x$ or on $t$, and we collect in $K$ all the constant parameters. A proper function $f=f(x_1,x_2,x_3,t)$ can also be added to the right-hand side of \eqref{eq:fond_fluido}, if sources or absorption effects are to be taken into account.

A further possible relation is represented by the so-called inverse Forchheimer law
\begin{equation}\label{Forch}
\Psi(r)=\frac{2r}{\sqrt{a^2 +4 b r}+a},\quad r>0,
\end{equation}
which is the inverse function of $\Phi(z)=a z +b z^2$. 

Whether using \eqref{Darcy2} or the non-linear models represented by \eqref{Eq:power} or \eqref{Forch} in the characterization of the motion 
of a fluid in a porous medium, is a delicate issue, which very much hinges on the value of the Reynolds number {\it Re}. It is well-known that low Reynolds number are characteristic of laminar flows, whereas turbulent flows show higher values of {\it Re}.

According to Section~4 of \cite{BGK}, there are at least three ranges of Reynolds number with three different laws, namely
\begin{itemize}
\item Non-Darcian law \eqref{Eq:power} with $\alpha>1$ for very low values of the Reynolds number;
\item Darcy law \eqref{Darcy2} for moderate values of the Reynolds number;
\item Non-Darcian law \eqref{Eq:power} with $\alpha\in(\frac12,1)$ or \eqref{Forch} for high values of the Reynolds number.
\end{itemize}
\section{Nanoporous filtration of oil and gas}\label{nanoporous}
In \cite{MRB} a mathematical model has been developed, which describes oil and gas flows in shales; in this context, the rocks consist of a porous or fissurized \emph{matrix}, approximately having normal porosity and permeability, and \emph{inclusions}, composed of kerogen, having normal porosity, but a very low permeability ($\sim10^{-21}$ m$^2$), due to the nanoscale size of pores, tubes, channels. This particular feature is the reason why in the technical literature these media are called \emph{nanoporous}.

The main difference with regular porous media, is that at the boundary layers of the inclusions, under a strong pressure gradient, there is a sharp increase in the permeability, which gives rise to an important flow of oil and gas.

Coming to the actual mathematical model of nanoporous rocks studied in \cite{MRB}, as in the previous section the starting point is represented by the conservation of mass
described by \eqref{continuita}. The connection between $\mathbf{u}$ and $\nabla p$ is given by a relation, which at first glance looks exactly as
\eqref{Darcy2}, namely
\begin{equation*}
\mathbf{u}=-\frac{k}\mu\nabla p,
\end{equation*} 
where $k$ is the permeability of the porous medium. The novelty with respect to \eqref{Darcy2} is that in nanoporous medium, $k$ is not
defined explicitly as a material property of the rock.

As discussed in \cite{MRB} and fully proven in 
\cite[Section~2]{BRM2}, for a $1$-dimensional framework $k$ is a power-like function of the pressure gradient, 
that is, 
\begin{equation}\label{perm}
k=A\left|\partial_x p\right|^m,
\end{equation}
where $A$ and $m$ are assumed to be positive constants, and for simplicity we write $x$ instead of $x_1$. Obtaining such a dependence is based on the following idea: if one requires scale invariance, and posits that only the pressure gradient is important, then there has to be a power law relationship. In \cite{MRB} it is also noted that, in a more general case the dependence of $A$ on $p$ can be taken into account without substantial complication.

Independently of the derivation of \eqref{perm}, combining it with \eqref{continuita}, and \eqref{Darcy2}, 
we obtain the following $1$-dimensional model
\[
\partial_t(\varphi\rho)- \partial_x\big(\tfrac A\mu \rho |\partial_x p|^m  \partial_x p
\big)=0.
\]
If we deal with an \emph{isothermal gas flow}, that is when \eqref{gasperfetti} holds and $T$ is considered constant, we obtain the doubly non-linear
equation
\[
\pl_t p -c_1 \pl_x\big(p |\pl_x p|^m \pl_x p
\big)=0,\qquad c_1=\tfrac A{\mu \varphi},
\]
where again we have taken $\varphi$ as constant, which is a good assumption in this context, and we have further supposed $p\ge0$.

The experimental values reported for $m$ in \cite{MRB} range from $0.23$ to $1.3$; this has probably to do with the different geological features of the deposits.

On the other hand, when working with \emph{weakly compressible fluids} such as oil, one needs to refer to \eqref{Eq:oil}. Assuming $\frac{p-p_o}K \ll1$ as mentioned before, and writing the equation in terms of $p$, yields
\[
\pl_t p -c_2\pl_x\big( |\pl_x p|^m \pl_x p
\big)=0,\qquad c_2=\tfrac{AK}{\mu\varphi},
\]
which is just the parabolic $p$-laplacian, the different notation notwithstanding. 

However, the equation can also be written in terms of the density $\rho$, and thus, once again, one ends up with a doubly non-linear parabolic equation, namely
\[
\pl_t \rho -c_3 \pl_x\big(\rho |\pl_x \rho|^m \pl_x \rho
\big)=0,\qquad c_3=\tfrac{A K^{m+1}}{\mu\varphi\rho_o^{m+1}}.
\]
\section{An application to topology optimization}\label{top-opt}
Topology optimization is a method developed in order to determine geometries that maximize the physical properties of a given material.

Different techniques have been so far introduced, and we refrain from going into details here. What is interesting for us is that in the framework of the so-called \emph{level set method}, the use of doubly non-linear diffusion equations, just like the ones we study here, has been recently proposed in \cite[Section~3]{yamada} as a more practical and versatile version of an analogous method based on reaction-diffusion equations. In particular, the singularity in the diffusion provides a  convergence to optimal
configurations, which is faster than the one given by the method using reaction–diffusion, as long as boundary structures do not
oscillate.

Different numerical simulations are discussed in \cite{yamada}, and in (6.5) of its Section~6, a partial application is presented for $N=2$, $p=6$, $q=2.5$, which, however, are values out of the range we consider here (see, for example, Theorem~\ref{THM:HARNACK:0} in this work).

\chapter[Technical preliminaries ]{Technical preliminaries}

\section{Notation}\label{sec:preliminaries}
We will write $\mathbb{N}_0$ to denote $\mathbb{N}\cup\{0\}$. 
For a point $z_o\in \R^N\times \R$, $N\in \N$,  we shall always write $z_o=(x_o,t_o)$. By $K_\rho(x_o)$
we denote the cube in $\R^N$ with center $x_o\in\R^N$ and side length $2\rho>0$, whose faces are parallel with the coordinate planes in $\rn$. When $x_o=0$ we simply write $K_\rho$, omitting the reference to the center. 
We shall use the symbol
$$
    (x_o,t_o)+Q_{R,S}:= K_R (x_o)\times (t_o-S,t_o]
$$
to denote a {\bf general backward parabolic cylinder} with the indicated parameters. 
For $\theta >0$ we define {\bf backward intrinsic parabolic cylinders}
$$
    (x_o,t_o)+Q_\rho(\theta):= K_\rho (x_o)\times (t_o-\theta\rho^p,t_o].
$$
When $\theta =1$ we write $(x_o,t_o)+Q_\rho$. 
In Chapters~\ref{sec:grad-bound} -- \ref{sec:Schauder} we are dealing with gradient estimates and therefore a different type of cylinders is suitable. For $\lambda>0$ we define a second type of {\bf backward intrinsic parabolic cylinders}
\begin{equation}\label{def:cyl-lambda}
    (x_o,t_o)+Q_\rho^{(\lambda)} 
    :=
    K_\rho (x_o)\times (t_o-\lambda^{2-p}\rho^2, t_o] 
\end{equation}
and the respective time interval
\begin{equation*}
    t_o+\Lambda_\rho^{(\lambda)} 
    :=
    (t_o-\lambda^{2-p}\rho^2, t_o] .
\end{equation*}

Define $E_T:= E\times (0,T]$ with $E$  a bounded open set of $\R^N$ and $T>0$.
The {\bf parabolic boundary} of $E_T$ is given by the union of its lateral and initial boundary
$$
    \pl_\mathrm{par} E_T
    :=
    \big(\overline E\times\{0\}\big) \cup \big(\partial E\times(0,T]\big).
$$
For two points $z_1=(x_1,t_1), z_2=(x_2,t_2)\in\R^{N+1}$ their {\bf parabolic distance} is defined as 
$$
    d_\mathrm{par}(z_1,z_2)
    :=
    |x_1-x_2| + \sqrt{|t_1-t_2|}.
$$
The associated distance of a subset $Q\subset E_T$ to the parabolic boundary $\partial_\mathrm{par} E_T$ is
$$
    \dist_\mathrm{par}(Q,\partial_\mathrm{par} E_T)
    :=
    \inf_{\substack{z_1\in Q\\ z_2\in\partial_\mathrm{par} E_T}} d_\mathrm{par}(z_1,z_2).
$$
For $\lambda>0$ we also define the following {\bf intrinsic} version of the {\bf parabolic distance} \begin{equation}\label{intrinsic-distance}
  d_\mathrm{par}^{(\lambda)}(z_1,z_2):=|x_1-x_2|+\sqrt{\lambda^{p-2}|t_1-t_2|}.
\end{equation}
Although $d_\mathrm{par}^{(\lambda)}$
also depends on $p$, since $p$ is always fixed we suppress this dependence.

Given a $p$-summable function $v\colon A\to\R^k$, with a measurable set $A$ and $k\in\mathbb{N}$, whenever we want to explicitly point out the set where the function takes values, we will write $v\in L^p(A;\R^k)$. In the same way, for example, $v\in C^\infty_0(E;[0,1])$ denotes a function $v\in C^\infty_0(E)$ defined on an open set $E$ which takes values in $[0,1]$.

For a function $v\in L^1(E_T;\R^k)$,  and a measurable subset $A\subset E$ with positive $\mathcal L^N$-measure we define the slice-wise mean $(v)_A\colon(0,T)\to\R^k$ of $v$ on $A$ by 
\begin{equation}\label{def:slice-wise-mean}
    (v)_A(t)
    :=
    \bint_A v(\cdot,t)\,\dx,
    \quad\mbox{for a.e. $t\in(0,T)$.}
\end{equation}
Note that the slicewise mean is defined for any $t\in(0,T)$ if $v\in C(0,T;L^1(E))$. 
If the set $A$ is a cube $K_\rho(x_o)$, then we write $(v)_{x_o;\rho}$ for short. Similarly, for a measurable set $Q\subset E_T$ of positive $\mathcal L^{N+1}$-measure we define the mean value $(v)_Q$ of $v$ on $Q$ by 
$$
    (v)_Q
    :=
    \biint_Q v \,\dx\dt.
$$
If the set $Q$ is a parabolic cylinder $z_o+Q_\rho^{(\lambda)}$ of the type \eqref{def:cyl-lambda}, we write $(v)_{z_o;\rho}^{(\lambda)}$ for short. 
When it is clear from the context, which vertex $x_o$ or $(x_o,t_o)$ is meant, we will omit the vertex from the above symbols for simplicity.

In the following we will work with {\bf Carath\'eodory functions}
\begin{equation*}
    E_T\times\R \times \R^{N}\ni (x,t, u,\xi)\mapsto \bl{A} (x,t, u,\xi)\in \R^N.
\end{equation*} 
This means that
\begin{equation*}
    \left\{
    \begin{array}{c}
        \mbox{$\R\times\R^N \ni(u,\xi)  \mapsto \bl{A} (x,t, u,\xi)$ is continuous for almost every $(x,t)\in E_T$, and } \\[6pt]
        \mbox{$E_T\ni (x,t)\mapsto \bl{A} (x,t, u,\xi)$ is measurable for every $(u,\xi)\in\R\times\R^N$. } 
    \end{array}
    \right.
\end{equation*}

\section{Parabolic function spaces}\label{sec:parabolic-Function-spaces}

Let $X$ be a Banach space with its natural norm $\|\cdot\|_X$, and $T>0$. 
For $1\le p<\infty$ we denote by $L^p(0,T;X)$ the (parabolic) Bochner space of all (strongly) measurable functions
$v\colon [0,T]\to X$ such that the Bochner norm is finite, i.e.
\begin{equation*}
    \| v\|_{L^p(0,T;X)}:=\bigg[\int_0^T\|v(t)\|_X^p\,\dt\bigg]^\frac1p <\infty.
\end{equation*}
Similarly, for $p=\infty$ the Bochner space $L^\infty(0,T;X)$ consists of all (strongly) measurable functions $v\colon [0,T]\to X$, such that $t\to\|v(t)\|_X$ is essentially bounded, which means that
\begin{equation*}
    \| v\|_{L^\infty (0,T;X)}:=\essup_{t\in [0,T]}\|v(t)\|_X<\infty.
\end{equation*}
These spaces are Banach spaces.  We could have defined these spaces on $(0,T)$ or $(0,T]$, which makes no difference at all, because in the end we get the same spaces. Finally, the Bochner space $C([0,T];X)$ denotes the space of all continuous functions $v\colon[0,T ]\to X$. 
The Bochner norm on this space is given by
\begin{equation*}
    \| v\|_{C([0,T];X)}:=\max_{t\in [0,T]}\|v(t)\|_X.
\end{equation*}
Next, the space $C((0,T);X)$ consists  of all continuous functions $v\colon (0,T)\to X$. 
Alternatively, $C([0,T];X)$  could be defined as the space of all continuous functions 
$v\colon(0,T)\to X$ that have a continuous extension (still denoted by $v$) to $[0,T]$. Analogous definitions hold for $C((0,T];X)$ and $C([0,T);X)$. 

Given a bounded, connected, open set $E\subset\R^N$ and the corresponding cylinder $E_T$, for $\alpha\in(0,1)$ the parabolic H\"older space $C^{\alpha,\alpha/2}(\bar E_T;\R^N)$ consists of all functions $u\colon E_T\to\R^N$, which admit continuous extension to $\bar E_T$, and such that  
\begin{equation*}
    \sup_{t\in [0,T]}\langle u(\cdot ,t)\rangle^{(\alpha)}_{x,E_T}+
    \sup_{x\in\bar E}\langle u(x,\cdot)\rangle^{(\alpha/2)}_{t,E_T}<\infty.
\end{equation*}
Here, we used the short hand abbreviation
\begin{align*}
    \langle u(\cdot ,t)\rangle^{(\alpha)}_{x,E_T}
    &=
    \sup_{x,x'\in\bar E}\frac{|u(x,t)-u(x',t)|}{|x-x'|^\alpha},\\
    \langle u(x,\cdot)\rangle^{(\alpha)}_{t,E_T}
    &=
    \sup_{t,t'\in [0,T]}\frac{|u(x,t)-u(x,t')|}{|t-t'|^\alpha}.
\end{align*}
The space $C^{\alpha,\alpha/2}(\bar E_T;\R^N)$ is a Banach space. Local variants are defined as usual by letting $C^{\alpha,\alpha/2}_{\mathrm{loc}}(E_T;\R^N)$ the space of functions $v\in C^{\alpha,\alpha/2}(\bar Q_T)$ for any compactly contained sub-cylinder $\bar Q_T\Subset E_T$.


\color{black}

\section{Auxiliary lemmas}

The following algebraic lemma can be found in \cite[Lemma~2.1]{BDL-21}.
\begin{lemma}\label{Lm:algebra}
For any $\al>0$, there exists a constant $c=c(\al)$ such that, for all $a,\,b\in\rr$, the following inequality holds true:
\begin{align*}
	\tfrac1{c}\big||b|^{\al-1}b - |a|^{\al-1}a\big|
	\le
	(|a| + |b|)^{\al-1}|b-a|
	\le
	c \big||b|^{\al-1}b - |a|^{\al-1}a\big|.
\end{align*}
\end{lemma}
%
%
Next, for $w,k\in\rr$ and $q>0$ we define the  $\mathfrak g_\pm $-functions according to
\begin{equation}\label{Eq:gpm}
	\mathfrak g_\pm (w,k):=\pm q\int_{k}^{w}|s|^{q-1}(s-k)_\pm\,\ds,
\end{equation}
where  the truncations are defined by
\[
(s-k)_+\equiv \max\{s-k,0\},\qquad (s-k)_-\equiv \max\{-(s-k),0\}.
\]
Note that $\mathfrak g_\pm$ are both non-negative. Moreover, we define
\begin{equation}\label{Eq:b}
	\mathfrak g  (w,k):=  q\int_{k}^{w}|s|^{q-1}(s-k) \,\ds.
\end{equation}
Computing the integral the $\mathfrak g$-function can be rewritten as
\begin{align*}
    \mathfrak g(w,k)
    &=\tfrac{q}{q+1}\big(|w|^{q+1} - |k|^{q+1} \big) -k (|w|^{q-1}w - |k|^{q-1}k)\\
    &\equiv\tfrac{1}{q+1}\big(|k|^{q+1} - |w|^{q+1} \big) -|w|^{q-1}w (k-w).
\end{align*}
The next lemma can be retrieved from \cite[Lemmas~2.2]{BDL-21}.
\begin{lemma}\label{Lm:g}
Let $q>0$ and $\mathfrak g$ be defined in \eqref{Eq:b}. There exists a constant $\boldsymbol\gm =\boldsymbol\gm (q)$ such that, for all $a,\,b\in\rr$, the following inequalities holds true:
\begin{align*}
	\tfrac1{\boldsymbol\gm} \big(|a| + |b|\big)^{q-1} |a-b|^2
	\le
	\mathfrak g (a,b)
	\le
	\boldsymbol\gm \big(|a| + |b|\big)^{q-1}|a-b|^2.
\end{align*}
\end{lemma}
and
\begin{align*}
	\tfrac1{\boldsymbol\gm} \big(|a| + |b|\big)^{q-1} (a-b)_\pm^2
	\le
	\mathfrak g_\pm (a,b)
	\le
	\boldsymbol\gm \big(|a| + |b|\big)^{q-1}(a-b)_\pm^2.
\end{align*}

For the Moser iteration we will rely on the following elementary lemma from  \cite[Lemma 2.3]{BDLS-boundary}.

\begin{lemma}\label{lem:A}
Let $A>1$, $\kappa>1$, $\gamma>0$ and $i\in\N$. Then, we have 
\begin{equation*}
	\prod_{j=1}^i
	A^{\frac{\kappa^{i-j+1}}{\gamma(\kappa^i-1)}}
	=
	A^\frac{\kappa}{\gamma(\kappa-1)}
\end{equation*}	
and 
\begin{equation*}
	\prod_{j=1}^i
	A^{\frac{j\kappa^{i-j+1}}{\gamma(\kappa^i-1)}}
	\le
	A^{\frac{\kappa^2}{\gamma(\kappa-1)^2}} .
\end{equation*}	
\end{lemma}

For the next lemma on geometric convergence we refer to \cite[Chapter~I, Lemma~4.1]{DB}.
\begin{lemma}[Fast geometric convergence]\label{lem:fast-geom-conv}
Let $(\boldsymbol Y_n)_{n\in \N_0}$ be a sequence of positive real 
numbers satisfying the recursive inequalities
\[
    \boldsymbol Y_{n+1}\le C\boldsymbol b^n \boldsymbol Y_n^{1+\alpha},
\]
where $C,\boldsymbol b>1$ and $\alpha>1$ are  given numbers. If 
\[
    \boldsymbol Y_o\le C^{-1/\alpha} \boldsymbol b^{-1/\alpha^2}
\]
then $\boldsymbol Y_n\to 0$ as $n\to\infty$.
\end{lemma}

For the next lemma we refer to \cite{Acerbi-Fusco} when $1<p<2$ and \cite{GiaquintaModica:1986-a} when $p\ge 2$.
\begin{lemma}\label{Monotonicity}
Let $\mu\in (0,1]$ and $p>1$. For any  $\xi,\tilde\xi\in\R^{Nn}$ we have 
\begin{align*}
	\Big|\big( \mu^2 +|\xi|^2\big)^\frac{p-2}{2}\xi 
    -\big( \mu^2 +|\tilde\xi|^2\big)^\frac{p-2}{2}\tilde \xi\Big|
	\le
	\bg \big(\mu^2 + |\xi|^2 + |\tilde\xi|^2\big)^{\frac{p-2}{2}}
	|\xi-\tilde\xi|
\end{align*}
and 
\begin{align*}
	\Big(\big( \mu^2 +|\xi|^2\big)^\frac{p-2}{2} \xi
    - \big( \mu^2 +|\tilde \xi|^2\big)^\frac{p-2}{2}\tilde\xi \Big)\cdot 
	\big(\xi-\tilde\xi\big)
	&\ge
	\tfrac{1}{\bg}
	\big(\mu^2 + |\xi|^2 + |\tilde\xi|^2\big)^{\frac{p-2}{2}}
	|\xi-\tilde\xi|^2,
\end{align*}
with a positive constant $\bg =\bg (p)$. 
\end{lemma}

For the proof of the next lemma we refer  to \cite[Lemma 6.1, p.191]{Giusti}.
  \begin{lemma}\label{lem:Giaq}
    For $r<\rho$, consider a bounded function
    $f:[r,\rho]\to[0,\infty)$ with
    \begin{equation*}
      f(R_1)\le\vartheta f(R_2)
      +\frac A{(R_2-R_1)^\alpha}+\frac B{(R_2-R_1)^\beta}+C
       \qquad\mbox{for all }r<R_1<R_2<\rho,
    \end{equation*}
    where $A,B,C$, and $\alpha>\beta$ denote non-negative constants
    and $\vartheta\in(0,1)$. Then we have
    \begin{equation*}
      f(r)\le c(\alpha,\vartheta)
     \Big(\frac A{(\rho-r)^\alpha}+\frac B{(\rho-r)^\beta}   +C\Big).
    \end{equation*}
\end{lemma}

The following results are extracted from \cite[Chapter~I, Propositions~2.1 \& 3.1]{DB}.
\begin{lemma}[Poincar\'e type inequality]\label{lem:Poincare}
Let $\Omega\subset\rr^N$ be a bounded convex set and let $\varphi\in C\big( \overline{\Omega}\big)$ with $0\le\varphi\le 1$. Moreover, assume that the super-level sets
$\big\{ \varphi >k\big\}$ are convex for all  $k\in (0,1)$. Let $v\in W^{1,p}(\Omega)$ with $p>1$ and assume  that the set
\[
    \mathcal E:=\{ v=0\}\cap \{ \varphi =1\}
\]
has positive measure. Then
\[
    \bigg[ \int_\Omega \varphi |v|^p\,\dx\bigg]^\frac{1}{p}
    \le 
    \boldsymbol\gm_*\frac{({\rm diam\,} \Omega)^N}{|\mathcal E|^\frac{N-1}{N}}
     \bigg[ \int_\Omega \varphi |Dv|^p\,\dx\bigg]^\frac{1}{p}
\]
with a positive constant $\boldsymbol\gm_*$ depending on $N$ and $p$, but independent of $v$ and $\varphi$.
\end{lemma}

\begin{lemma}[Parabolic Sobolev embedding]\label{lem:Sobolev}
Let $E$ be a bounded domain in $\rn$ and $E_T:=E\times(0,T]$.
    Suppose that $u\in L^{\infty}\big(0,T;L^m(E)\big)\cap L^{p}\big(0,T;W_0^{1,p}(E)\big)$ for $m,p\ge1$. Then setting $q=p\frac{N+m}{N}$, we have
    \[
    \iint_{E_T}|u|^q\,\dx\dt\le \boldsymbol\gm\iint_{E_T}|Du|^p\,\dx\dt\bigg(\essup_{t\in[0,T]}\int_{E\times\{t\}} |u|^m\,\dx\bigg)^{\frac{p}{N}}
    \]
   with a positive constant $\boldsymbol\gm$ depending on $N$, $p$ and $q$, but independent of $u$.
\end{lemma}

\section{Time mollification}\label{sec:moly}

Weak solutions to parabolic equations in general do not possess a time derivative in the Sobolev sense. On the other hand, it is desirable to use weak solutions in testing functions. In order to overcome this problem, proper mollifications in the time variable are often employed. A simple version, termed the Steklov average, is sufficient for many occasions. 
Indeed, given a function $v \in L^1(E_T)$ and $0<h<T$, we define its {\it forward Steklov-average} $[v]_h$ by 
\begin{equation}\label{def-stek}
	[v]_h(x,t) 
	:=
	\left\{
	\begin{array}{cl}
		\displaystyle{\frac{1}{h} \int_t^{t+h} v(x,s) \,\ds ,}
		& t\in (0,T-h) , \\[9pt]
		0 ,
		& t\in (T-h,T) 
\,.
	\end{array}
	\right.
\end{equation}
Steklov avereges  are differentiable with respect to $t$. More precisely we have
\begin{equation*}
    \partial_t [v]_h(x,t)=\tfrac1h \big(v(x,t+h)-v(x,t)\big)
\end{equation*}
for almost every $(x,t) \in E\times (0,T-h)$.
The relevant properties of Steklov averages are summarized in the following lemma. 
\begin{lemma}\label{lm:Stek}
    For any $r\ge 1$ we have
    \begin{itemize}
        \item[(i)] If $v\in L^r(E_T)$, then $[v]_h\in L^r(E_{T})$.  Moreover,  
        $
          \| [v]_h\|_{L^r(E_{T})}
            \le 
            \| v\|_{L^r(E_T)}
        $
        and
           $[v]_h\to v$ in $L^r(E_{T})$ and almost everywhere on $E_T$
           as $h\downarrow 0$.
        \item[(ii)] If $Dv\in L^r(E_T)$, then $D[v]_h=[Dv]_h\to Du$ in $L^r(E_T)$ and almost everywhere on $E_T$ as $h\downarrow 0$. 
        \item[(iii)] If  $v\in C\big(0,T;L^r(E)\big)$, then $[v]_h(\cdot ,t)\to v(\cdot, t)$ in $L^r(E)$
        as $h\downarrow 0$ for every $t\in (0,T-\eps)$ for any $\eps\in (0,T)$.
    \end{itemize}
\end{lemma}
In case of a nonlinearity with respect to $u$, as for instance present in the porous medium equation or doubly non-linear equations, Steklov averages do not have all the  properties needed. It turns out that a different kind of mollification is better suited. 
Indeed, for any $v\in L^1(E_T)$ and $h>0$, we introduce mollified functions
\begin{equation}\label{def:mol}
	\llbracket v \rrbracket_h(x,t)
	:= 
	\tfrac 1h \int_0^t \mathrm e^{\frac{\tau-t}h} v(x,\tau) \, \d \tau,\quad
	\llbracket v \rrbracket_{\bar{h}}(x,t)
	:= 
	\tfrac 1h \int_t^T \mathrm e^{\frac{t-\tau}h} v(x,\tau) \, \d \tau.
\end{equation}
 Some relevant properties are collected in the following; see~\cite[Appendix~B]{BDM-13} for more.

\begin{lemma}\label{Lm:mol} 
For any  $r\ge1$ we have
\begin{itemize}
    \item[(i)] If $v\in L^r(E_T)$, then $\llbracket v \rrbracket_h \in L^r(E_T)$. Moreover, 
    $\|\llbracket v \rrbracket_h\|_{L^r(E_T)}\le \|v\|_{L^r(E_T)}$ and $ \llbracket v \rrbracket_h \to v$
    in $L^r(E_T)$ and almost everywhere on $E_T$ as $h\downarrow 0$. 
    \item[(ii)] $\llbracket v \rrbracket_{h}\in C\big([0,T]; L^r(E)\big)$.
    \item[(iii)] Almost everywhere on $E_T$ holds
    \begin{equation*}
	 \partial_t \llbracket v \rrbracket_h
	 =
	 \tfrac{1}{h} \big(v-\llbracket v \rrbracket_h\big), \quad
	  \partial_t \llbracket v \rrbracket_{\bar{h}}
	 =
	 \tfrac{1}{h} \big(\llbracket v \rrbracket_{\bar{h}}- v\big).
\end{equation*}
    \item[(iv)] If $Du\in L^r(E_T)$, then $D\llbracket v \rrbracket_{h} =\llbracket Dv \rrbracket_{h}\to Dv$
    in $L^r(E_T)$ and almost everywhere in $E_T$ as $h\downarrow 0$. 
    \item[(v)] If $v\in C\big([0,T]; L^r(E)\big) $, then $\llbracket v \rrbracket_{h}(\cdot ,t)\to v(\cdot,t)$
    in $L^r(E)$ and almost everywhere on $E$  for every $t\in [0,T]$ as $h\downarrow 0$. 
\end{itemize}
Similar conclusions as in  {\rm (i)}, {\rm(ii)}, {\rm(iv)}, and {\rm (v)} also apply for $\llbracket v \rrbracket_{\bar{h}}$.
\end{lemma}

\chapter[Notion of solution and comparison principles]{Notion of solution and comparison principles}\label{sec:sol-compar}

We start with a general notion of parabolic partial differential equation
\begin{equation}  \label{Eq:1:1}
	\partial_t\big(|u|^{q-1}u\big)-\dvg\bl{A}(x,t,u, Du) = 0\quad \mbox{in $ E_T$,}
\end{equation}
where the function $\bl{A}(x,t,u,\xi)\colon E_T\times\rr^{N+1}\to\rn$ is Carath\'eodory 
and subject to the growth assumption \begin{equation}\label{Eq:1:2-}
	|\bl{A}(x,t,u,\xi)|\le C_1\big(1+|\xi|^{p-1}\big)
	\quad \mbox{for a.e.~$(x,t)\in E_T$, $\forall\,u\in\rr$, $\forall\,\xi\in\rn$,}
\end{equation}
for a given positive constant $C_1$. For the moment, we take $p>1$ and $q>0$.

The sole structure condition \eqref{Eq:1:2-} is insufficient to establish desired properties of weak solutions. Thereby various further requirements on $\bl{A}(x,t,u,\xi)$ will be introduced as they are called upon.
 
 In this chapter, some basic properties of solutions to \eqref{Eq:1:1}, such as continuity in the time variable and comparison principles will be addressed.  The results are scattered in the literature in various forms. We give detailed proofs here in accordance with our particular need.

\section{Notion of solution}\label{S:notion-sol}

\begin{definition}\label{Def:notion-sol}
A function $u$
is termed a {\bf weak super(sub)-solution} to the parabolic equation \eqref{Eq:1:1} with \eqref{Eq:1:2-} in $E_T$, if
\begin{equation*}  
	u\in C\big( [0,T]; L^{q+1}(E)\big) \cap L^p \big(0,T; W^{1,p} (E)\big)
\end{equation*}
and if the integral identity 
\begin{equation} \label{Eq:weak-form}
	\iint_{E_T} \Big[-|u|^{q-1}u\pl_t\z+\bl{A}(x,t,u,Du)\cdot D\z\Big]\dx\dt
	\ge (\le) 0
\end{equation}
is satisfied for all non-negative test functions 
\begin{equation}\label{Eq:test-func}
\z\in  W_0^{1,q+1} \big(0,T; L^{q+1}(E)\big)\cap L^p \big(0,T; W^{1,p}_{0}(E)\big).
\end{equation}
A function that is both a weak super- and sub-solution is called a {\bf weak solution}. 
\end{definition}

\begin{remark}
    The structure condition \eqref{Eq:1:2-} of $\bl{A}(x,t,u,\xi)$ and the function space \eqref{Eq:test-func} of $\z$ guarantee the convergence of integral in \eqref{Eq:weak-form}.
\end{remark}

\begin{remark}
A notion of {\bf local solution} is commonly used in the literature, cf.~\cite{DB, DBGV-mono}. 
We stress that this notion makes no essential difference from Definition~\ref{Def:notion-sol} modulo a localization. 
\end{remark}

\section{Existence of weak solutions}\label{S:existence}

Let us consider 
the following Cauchy-Dirichlet Problem:
\begin{equation}\label{Eq:CDP}
\left\{
\begin{array}{cl}
    \partial_t \big(|u|^{q-1}u\big) - \Div \bl{A}(x,t,u,Du)=0  & \quad \mbox{in  $E_T$,}\\[6pt]
    u =g &\quad \mbox{on $\partial E\times(0,T]$,}\\[6pt]
    u(\cdot,0)= u_o &\quad\mbox{in $E$,}
\end{array}
\right.
\end{equation}
where $p>1$, $q>0$, $u_o\in L^{q+1}(E)$ and $g\in  L^p(0,T;W^{1,p}(E))$ with $\partial_t g\in L^{p'}(0,T;W^{-1,p'}(E))$.
We first define what we mean by a weak solution to the Cauchy-Dirichlet Problem \eqref{Eq:CDP}.

\begin{definition}\label{def:weak}
Let \eqref{Eq:1:2-} be in force. 
A function $u$
is termed a {\bf weak solution to the Cauchy-Dirichlet Problem \eqref{Eq:CDP}}, if
\begin{equation*}  
	u\in C\big( [0,T]; L^{q+1}(E)\big) \cap g + L^p \big(0,T; W_0^{1,p} (E)\big)
\end{equation*}
and if the integral identity
\begin{equation}\label{Eq:weak-form-1}
\iint_{E_T} \Big[ \big(|u_o|^{q-1}u_o - |u|^{q-1}u\big) \pl_t\z+\bl{A}(x,t,u,Du)\cdot D\z\Big]\dx\dt=0
\end{equation}
is satisfied for any test function
$
\z\in  W^{1,q+1}(0,T; L^{q+1}(E))\cap L^p (0,T; W^{1,p}_{0}(E))
$
such that $\z(\cdot, T)=0$. 
\end{definition}

Note that in the setting of Definition~\ref{def:weak}, the integral identity \eqref{Eq:weak-form-1} implies that $u(\cdot,0)=u_o$; see Proposition~\ref{Prop:B:3}.

The issue of the existence of weak solutions is extensively studied in the literature. A natural assumption  typically used is the monotonicity of the vector field $\bl{A}(x,t,u,\xi)$ with respect to the $\xi$-variable, in the sense that
\begin{equation}\label{Eq:CP:mono}
\big(\bl{A}(x,t,u,\xi_1)-\bl{A}(x,t,u,\xi_2)\big)
\cdot(\xi_1-\xi_2)\ge0
\end{equation}
for all variables in the indicated domains. In this paper we need existence results in two different situations. The first one concerns non-negative weak solutions to the model equation \eqref{doubly-nonlinear-prototype} with zero lateral boundary data. The second one concerns signed solutions in the case $q=1$.

\begin{remark}\label{Rmk:CP1}\upshape
Let $p>1$, $q>0$ and consider the Cauchy-Dirichlet Problem \eqref{Eq:CDP} for the model equation \eqref{doubly-nonlinear-prototype} and with $g\equiv 0$. 
Then there exists a non-negative weak solution to the Cauchy-Dirichlet Problem~\eqref{Eq:CDP} in the sense of Definition~\ref{def:weak}. 
This can be retrieved from \cite[Theorem~1.7]{AL} or \cite[Theorem~1.3]{BDMS-18} and \cite[Theorem~5.1]{BDMS-survey}. Some adjustments are required for the application of \cite[Theorem~1.7]{AL}. First, in the case $1<p<2$ a stronger monotonicity assumption is required. Secondly, the domain $E$ must have a Lipschitz boundary. However, also  under the weaker conditions the existence proof of \cite{AL} still works.
Note that the existence for time-dependent boundary data with a weak time derivative was obtained in \cite{Schatzler-1, Schatzler-2}. 
%
%
\end{remark}

\begin{remark}\label{Rmk:CP2}\upshape
Let $p>1$ and $q=1$ and consider the Cauchy-Dirichlet Problem \eqref{Eq:CDP} for the parabolic equation 
$$
    \partial_tu - \Div\big(a(x,t)\big(\mu^2+|Du|^2\big)^{\frac{p-2}{2}}Du\big) = 0
$$
with $\mu\in[0,1]$ and $(x,t)\mapsto a(x,t)$ measurable satisfying $0<C_o\le a\le C_1$. 
Then there exists a weak solution to the Cauchy-Dirichlet Problem~\eqref{Eq:CDP} in the sense of Definition~\ref{def:weak}. This result can for instance be deduced from \cite[Chapter III, Proposition 4.1 and Example 4.A]{Showalter}.
\end{remark}

\section{Notion of parabolicity}
A well-known property of harmonic function theory asserts that if $u$ is sub-harmonic or super-harmonic, then so is $\max\{u,k\}$ or $\min\{u,k\}$ for $k\in\rr$ respectively. The next result shows that a similar property holds in the parabolic setting. When dealing with signed solutions, it is often convenient to use the short-hand notation $\boldsymbol{a}^q:=|a|^{q-1}a$ for $a\in\R$. 
\begin{proposition}\label{Prop:parab}
Assume that $p>1$ and $q>0$.
 In addition to \eqref{Eq:1:2-}, we assume $\bl{A}(x,t,z,\xi)\cdot\xi\ge0$ for a.e.~$(x,t)\in E_T$, for any $z\in\rr$ and for any $\xi\in\rn$. 
 If $u$ is a weak sub-solution to \eqref{Eq:1:1}, then the truncation $u_k=k+(u-k)_+$ with $k\in\rr$ satisfies
\begin{equation*}
	-
	\iint_{E_T} \partial_t\zeta \power{u_k}{q}\,\dx\dt
	+
 \iint_{E_T}
	\mathbf A(x,t,u_k,Du_k)\cdot D\z\boldsymbol\chi_{\{u>k\}}\, \dx\dt
 \le
 0,
\end{equation*}
for all non-negative test functions 
$
\z\in  W_0^{1,q+1} \big(0,T; L^{q+1}(E)\big)\cap L^p \big(0,T; W^{1,p}_{0}(E)\big)$.
 Similarly, if $u$ is a weak super-solution to \eqref{Eq:1:1}, then the truncation $u_k=k-(u-k)_-$ with $k\in\rr$ satisfies
\begin{equation*}
	-
	\iint_{E_T} \partial_t\zeta \power{u_k}{q}\,\dx\dt
	+
 \iint_{E_T}
	\mathbf A(x,t,u_k,Du_k)\cdot D\z\boldsymbol\chi_{\{u<k\}}\, \dx\dt
 \ge
 0,
\end{equation*}
for $\z$ the same type of test function.
\end{proposition}

\begin{proof}
We only deal with the case of sub-solution, as the case of super-solution can be treated analogously. In the integral formulation \eqref{Eq:weak-form} we choose the test function
\begin{equation*}
	E_T\ni(x,t)\mapsto
	\vp_h 
	:= 
	 \frac{\z \big(\llbracket u \rrbracket_{\bar{h}}-k\big)_+}{\big(\llbracket u \rrbracket_{\bar{h}}-k\big)_++\sig},
\end{equation*}
where $\z$ is as in \eqref{Eq:test-func} and $\sig>0$.
Let us first focus on the term with time derivative. Indeed, we rewrite it as
\begin{equation}\label{parab-time-part}
    -\iint_{E_T}\boldsymbol u^q\pl_t\vp_h\,\dx\dt=\iint_{E_T}\big(\boldsymbol{\llbracket u \rrbracket_{\bar{h}}}^{q}-\boldsymbol u^q\big)\pl_t\vp_h\,\dx\dt-\iint_{E_T}\boldsymbol{\llbracket u \rrbracket_{\bar{h}}}^{q}\pl_t\vp_h\,\dx\dt.
\end{equation}
We continue to estimate the first term on  the right-hand side of \eqref{parab-time-part} by
\begin{align*}
    \iint_{E_T}&\big(\boldsymbol{\llbracket u \rrbracket_{\bar{h}}}^{q}-\boldsymbol u^q\big)\pl_t\vp_h\,\dx\dt\\
    &=\iint_{E_T}\z\big(\boldsymbol{\llbracket u \rrbracket_{\bar{h}}}^{q}-\boldsymbol u^q\big)\pl_t\bigg[\frac{\big(\llbracket u \rrbracket_{\bar{h}}-k\big)_+}{\big(\llbracket u \rrbracket_{\bar{h}}-k\big)_++\sig}\bigg]\,\dx\dt\\
    &\quad + \iint_{E_T}\pl_t\z\big(\boldsymbol{\llbracket u \rrbracket_{\bar{h}}}^{q}-\boldsymbol u^q\big)\frac{\big(\llbracket u \rrbracket_{\bar{h}}-k\big)_+}{\big(\llbracket u \rrbracket_{\bar{h}}-k\big)_++\sig}\,\dx\dt\\
    &=\iint_{E_T}\z\big(\boldsymbol{\llbracket u \rrbracket_{\bar{h}}}^{q}-\boldsymbol u^q\big)\tfrac1h\big(\llbracket u \rrbracket_{\bar{h}}-u\big) \frac{\sig \chi_{\{\llbracket u \rrbracket_{\bar{h}}>k\}}}{\big[\big(\llbracket u \rrbracket_{\bar{h}}-k\big)_++\sig\big]^2} \,\dx\dt\\
    &\quad + \iint_{E_T}\pl_t\z\big(\boldsymbol{\llbracket u \rrbracket_{\bar{h}}}^{q}-\boldsymbol u^q\big)\frac{\big(\llbracket u \rrbracket_{\bar{h}}-k\big)_+}{\big(\llbracket u \rrbracket_{\bar{h}}-k\big)_++\sig}\,\dx\dt.
\end{align*}
Here, in calculating $\pl_t\llbracket u \rrbracket_{\bar{h}}$ we used Lemma~\ref{Lm:mol}~({\rm iii}). As a result, the first integral on the right-hand side is non-negative thanks to the monotonicity of the map $u\mapsto\boldsymbol u^q$, and hence can be discarded. Whereas the second integral tends to zero as $h\downarrow0$ due to Lemma~\ref{Lm:mol}~({\rm i}).

Next we deal with the second term on the right-hand side of \eqref{parab-time-part}. Indeed,  integration by parts yields that
\begin{align*}
	-\iint_{E_T}\boldsymbol{\llbracket u \rrbracket_{\bar{h}}}^{q}\pl_t\vp_h\,\dx\dt
 &=\iint_{E_T} 
	\partial_t \power{u_{\bar{h}}}{q} \varphi_h \,\dx\dt \\
	&= 
	\iint_{E_T}
	\zeta    \partial_t \power{\llbracket u \rrbracket_{\bar{h}}}{q} \frac{\big(\llbracket u \rrbracket_{\bar{h}}-k\big)_+}{\big(\llbracket u \rrbracket_{\bar{h}}-k\big)_++\sig}
	\,\dx\dt \\
	&= 
	\iint_{E_T}
	\zeta  
	\partial_t \mathfrak h_+ (\llbracket u \rrbracket_{\bar{h}},\sig,k) \, \dx\dt \\
	& = 
	- \iint_{E_T} 
	   \partial_t\zeta \mathfrak h_+ (\llbracket u \rrbracket_{\bar{h}},\sig,k) 
	\,\dx\dt,
\end{align*}
where we have defined
\[
	 \mathfrak h_+ (\llbracket u \rrbracket_{\bar{h}},\sig,k)
	 :=
	 q
	 \int_k^{\llbracket u \rrbracket_{\bar{h}}}\frac{|s|^{q-1}(s-k)_+}{(s-k)_++\sig}\,\d s .
\]
 With these estimates, as well as Lemma~\ref{Lm:mol}~({\rm i}), we can send $h\downarrow0$ in \eqref{parab-time-part} and obtain
 \begin{align*}
 \lim_{h\downarrow0}\bigg(-\iint_{E_T}\boldsymbol u^q\pl_t\vp_h\,\dx\dt \bigg)
	\ge
	-\iint_{E_T} \partial_t\zeta\mathfrak h_+(u,\sig, k)\,\dx\dt.
\end{align*}

Next, we consider the diffusion term.
To this end, we again send $h\downarrow0$  with the aid of Lemma~\ref{Lm:mol}~({\rm i}), and use in addition $\bl{A}(x,t,u,\xi)\cdot\xi\ge0$ to estimate
\begin{align*}
	& \lim_{h\downarrow 0}
	\iint_{E_T} 
	\mathbf A(x,t,u,Du)\cdot D \varphi_h\, \dx\dt\\
	&=\iint_{E_T}
	\mathbf A(x,t,u,Du)\cdot \bigg[D\z\frac{(u-k)_+}{(u-k)_++\sig}
	+\z\frac{\sigma D(u-k)_+}{\big[(u-k)_++\sig\big]^2}\bigg]\, \dx\dt \\
	&\ge\iint_{E_T}
	\mathbf A(x,t,u,Du)\cdot D\z\frac{(u-k)_+}{(u-k)_++\sig}\, \dx\dt.
\end{align*}
Combining all above estimates gives
\begin{align*}
	-
	\iint_{E_T} \partial_t\zeta\mathfrak h_+(u,\sig, k)\,\dx\dt
	+\iint_{E_T}
	\mathbf A(x,t,u,Du)\cdot D\z\frac{(u-k)_+}{(u-k)_++\sig}\, \dx\dt\le0.
\end{align*}
Finally, noting that $\lim_{\sig\downarrow0} \mathfrak h_+(u,\sig,k)=\power{[k+(u-k)_+]}{q}=\power{u_k}{q}$ a.e.~$E_T$, we send $\sig\downarrow0$ in the above display to conclude
that
\begin{align*}
	-
	\iint_{E_T} \partial_t\zeta \power{u_k}{q}\,\dx\dt
	+
 \iint_{E_T}
	\mathbf A(x,t,u,Du)\cdot D\z\boldsymbol\chi_{\{u>k\}}\, \dx\dt
 \le
 0.
\end{align*}
The claim follows.
\end{proof}

\begin{remark}
If we additionally assume that $\bl{A}(x,t,z,0)=0$ for a.e.~$(x,t)\in E_T$ and for any $z\in\rr$,
then  for any $k\in\rr$, then the truncation $u_k=k+(u-k)_+$ is a sub-solution  to the original
equation, i.e.~it
satisfies
\begin{equation*}
	-
	\iint_{E_T} \partial_t\zeta \power{u_k}{q}\,\dx\dt
	+
 \iint_{E_T}
	\mathbf A(x,t,u_k,Du_k)\cdot D\z\, \dx\dt
 \le
 0,
\end{equation*}
for all non-negative test functions 
$
\z\in  W_0^{1,q+1} \big(0,T; L^{q+1}(E)\big)\cap L^p \big(0,T; W^{1,p}_{0}(E)\big)$.  An analogous statement holds for super-solutions.
\end{remark}

\section{Continuity in the time variable}

In this section we discuss some continuity properties of weak solutions in the time variable, either locally or globally. The results apply to both signed and non-negative solutions and to all $q>0$ and $p>1$. Note that in our definition of weak solution we assume that $u$ belongs to the space $C([0,T]; L^{q+1} (E))$ and not $L^\infty(0,T; L^{q+1}(E))$. The next proposition ensures that this property is already included in the integral identity of weak solutions and therefore it is not restrictive.

\begin{proposition}\label{Prop:B:3}
Assume that $p>1$, $q>0$ and hypothesis \eqref{Eq:1:2-} holds.  
For some $u_o\in L^{q+1}(E)$, let 
$u\in L^{\infty}(0,T; L^{q+1}(E))\cap L^p(0,T; W_0^{1,p}(E))$
satisfy the integral identity \eqref{Eq:weak-form-1}. 
Then there exists a representative $u \in C([0,T]; L^{q+1} (E))$, 
which verifies
\[
\lim_{t\downarrow 0}\int_{E} |u(\cdot, t) -u_o|^{q+1} \,\dx=0.
\]
\end{proposition}

To prove Proposition~\ref{Prop:B:3}, we follow the strategy in \cite[\S~3.3]{Vespri-Vestberg} and first show the following result. For the exact formulation, we recall 
the definition of $\mathfrak{g}$ in \eqref{Eq:b}.

\begin{lemma}\label{Lm:B:3}
Suppose $u$ satisfies the hypothesis of Proposition \ref{Prop:B:3}. Then, for any function 
$v\in L^{q+1}(E_T)\cap L^p(0,T; W_0^{1,p}(E))$ satisfying $\pl_t v\in  L^{q+1}(E_T)$
and any  function $\psi\in C^1_0(0,T)$ we have 
\begin{align*}
    \iint_{E_T} \mathfrak{g}(u,v)\pl_t\psi \,\dx\dt 
    &=
    \iint_{E_T}\pl_t v \big(|u|^{q-1}u - |v|^{q-1}v\big)\psi  \,\dx\dt\\
    &\phantom{=\,}+
    \iint_{E_T}\bl{A}(x,t,u,Du)\cdot D(v-u)\psi \,\dx\dt.
\end{align*}
\end{lemma}

\begin{proof}
Throughout the proof we use the short-hand notation $\boldsymbol{u}^q:=|u|^{q-1}u$.  
Plug the test function $(v-\llbracket u \rrbracket_{h})\psi$ in the integral identity \eqref{Eq:weak-form-1}, where the mollification $\llbracket u \rrbracket_{h}$ of $u$ is defined in \eqref{def:mol}. By Lemma~\ref{Lm:mol} (iv), the diffusion part yields the limit
\[
    \iint_{E_T} \psi \bl{A}(x,t,u,Du)\cdot D(v-\llbracket u \rrbracket_{h})  \,\dx\dt
    \to \iint_{E_T} \psi \bl{A}(x,t,u,Du)\cdot D(v-u)\,\dx\dt
\] 
as $h\downarrow 0$.
Concerning the time part we note that the term containing $u_o$ vanishes, since $\psi(0)=0=\psi(T)$. The remaining term is estimated by
\begin{align*}
\iint_{E_T} & -\boldsymbol{u}^q\pl_t[(v-\llbracket u \rrbracket_{h})\psi]\,\dx\dt\\
&= - \iint_{E_T}\boldsymbol{u}^q(v-\llbracket u \rrbracket_{h})\pl_t\psi\,\dx\dt  - \iint_{E_T} \boldsymbol{u}^q\pl_t v \psi\,\dx\dt\\
&\quad+\iint_{E_T} (\boldsymbol{u}^q - \boldsymbol{\llbracket u \rrbracket_{h}}^q + \boldsymbol{\llbracket u \rrbracket_{h}}^q)\pl_t \llbracket u \rrbracket_{h}\psi\,\dx\dt\\
&\ge  - \iint_{E_T}\boldsymbol{u}^q(v-\llbracket u \rrbracket_{h})\pl_t\psi\,\dx\dt  - \iint_{E_T} \boldsymbol{u}^q\pl_t v \psi\,\dx\dt +
\iint_{E_T}  \boldsymbol{\llbracket u \rrbracket_{h}}^q \pl_t \llbracket u \rrbracket_{h}\psi\,\dx\dt\\
&= - \iint_{E_T}\boldsymbol{u}^q(v-\llbracket u \rrbracket_{h})\pl_t\psi\,\dx\dt  - \iint_{E_T} \boldsymbol{u}^q\pl_t v \psi\,\dx\dt\\
&\quad-\iint_{E_T} \tfrac1{q+1} | \llbracket u \rrbracket_{h}|^{q+1}\pl_t\psi\,\dx\dt\\
&\to - \iint_{E_T}\boldsymbol{u}^q(v-u)\pl_t\psi\,\dx\dt  - \iint_{E_T} \boldsymbol{u}^q\pl_t v \psi\,\dx\dt -
\iint_{E_T}  \tfrac1{q+1} |u|^{q+1}\pl_t\psi\,\dx\dt\\
&=\iint_{E_T} \mathfrak{g}(u,v)\pl_t\psi\,\dx\dt+\iint_{E_T}\pl_t v (\boldsymbol{v}^q- \boldsymbol{u}^q) \psi\,\dx\dt
\end{align*}
as $h\downarrow 0$. In the above inequality, we used Lemma~\ref{Lm:mol} (iii) and thus 
\[
(\boldsymbol{u}^q - \boldsymbol{\llbracket u \rrbracket_{h}}^q) \pl_t \llbracket u \rrbracket_{h}=(\boldsymbol{u}^q - \boldsymbol{\llbracket u \rrbracket_{h}}^q)\tfrac1h (u -  \llbracket u \rrbracket_{h})\ge0.
\]
In addition, to obtain the convergence, we employed Lemma~\ref{Lm:mol} (i).
This shows the desired identity with ``$\le$" instead of ``$=$". For the reverse inequality, we use the test function $(v-\llbracket u \rrbracket_{\bar{h}})\psi$ and run similar calculations.
\end{proof}

\begin{proof}[{\rm \textbf{Proof of Proposition~\ref{Prop:B:3}}}]
Let $\tau\in[0,\tfrac12 T]$ and $\varep\in(0,\tfrac14 T)$. Define $\psi_\varep$ to be $1$ on $[\tau+\varep, \tfrac34 T]$, vanish on $[0,\tau]\cup \{T\}$ and be linearly interpolated otherwise. 
Now we use $\psi=\psi_\varep$ and take $v\equiv v_h=\llbracket u \rrbracket_{\bar{h}}$ in Lemma~\ref{Lm:B:3}, and obtain
\begin{align*}
\frac1{\varep}\iint_{E\times(\tau,\tau+\varep)} \mathfrak{g}(u,v_h)\,\dx\dt &\le \frac4{T}\iint_{E_T} \mathfrak{g}(u,v_h)\,\dx\dt\\
&\quad+\iint_{E_T} \bl{A}(x,t,u,Du)\cdot D(v_h-u)\psi \,\dx\dt.
\end{align*}
Note that we have dropped the term containing $\pl_t v (|u|^{q-1}u- |v|^{q-1}v)$ due to its non-positive sign. 

Observe that the two terms on the right-hand side converge to zero as $h\to0$. Let us begin dealing with the first term. In fact, when $0<q\le 1$, using Lemma~\ref{Lm:g} and the triangle inequality, we have
\begin{align*}
    \iint_{E_T} \mathfrak{g}(u,v_h)\,\dx\dt 
    &\le 
    \boldsymbol\gm \iint_{E_T} (|u|+|v_h|)^{q-1} |u-v_h|^2\,\dx\dt\\
    &\le 
    \boldsymbol\gm\iint_{E_T} | u  -  v_h |^{q+1}\,\dx\dt;
\end{align*}
 when $q>1$, in addition to Lemma~\ref{Lm:g}, we also apply H\"older's inequality to estimate
\begin{align*}
    \iint_{E_T}& \mathfrak{g}(u,v_h)\,\dx\dt \\
    &\le 
    \boldsymbol\gm \iint_{E_T} (|u|+|v_h|)^{q-1} |u-v_h|^2\,\dx\dt\\ 
    &\le 
    \boldsymbol\gm 
    \bigg[\iint_{E_T} |u-v_h|^{q+1}\,\dx\dt\bigg]^{\frac2{q+1}} \bigg[ \iint_{E_T} (|u|+|v_h|)^{q+1}\,\dx\dt \bigg]^{1-\frac2{q+1}}.
\end{align*}
In either case, the claimed convergence of the first term follows. Whereas the convergence of the second term follows easily from Lemma~\ref{Lm:mol} (iv).

Next, observe that fixing $h>0$ and sending $\varep\downarrow 0$ in the above estimate, we have
\begin{align*}
     \int_{E\times\{\tau\}} \mathfrak{g}(u,v_h)\,\dx 
     &\le 
     \frac4{T}\iint_{E_T} \mathfrak{g}(u,v_h)\,\dx\dt\\
    &\phantom{\le\,}
    +\iint_{E_T} \bl{A}(x,t,u,Du)\cdot D(v_h-u)\psi \,\dx\dt
\end{align*}
for all $\tau\in[0,\tfrac12 T]\setminus N_h$, where $N_h$ is a null set.
Let us concentrate on the left-hand side in the last display. If $q\ge1$, then by Lemma~\ref{Lm:g}, we have
\begin{align*}
    \int_{E\times\{\tau\}} \mathfrak{g}(u,v_h)\,\dx 
    &\ge  
    \frac1{\boldsymbol\gm}\int_{E\times\{\tau\}} (|u|+|v_h|)^{q-1} |u-v_h|^2\,\dx \\
    &\ge  
    \frac1{\boldsymbol\gm} \int_{E\times\{\tau\}}  |u-v_h|^{q+1}\,\dx.
\end{align*}
If $0<q<1$, we instead use Lemmas~\ref{Lm:algebra} -- \ref{Lm:g} and H\"older's inequality to estimate
\begin{align*}
    \int_{E\times\{\tau\}} & |u-v_h|^{q+1}\,\dx \\
    &\le 
    \boldsymbol\gm \int_{E\times\{\tau\}} \big|\boldsymbol{u}^{\frac{q+1}2}-\boldsymbol{v_h}^{\frac{q+1}2}\big|^{q+1}  (|u|+|v_h|)^{\frac{(1-q)(q+1)}{2}}\,\dx\\
    &\le 
    \boldsymbol\gm \bigg[\int_{E\times\{\tau\}} \big|\boldsymbol{u}^{\frac{q+1}2}-\boldsymbol{v_h}^{\frac{q+1}2}\big|^2\,\dx\bigg]^{\frac{q+1}2} 
    \bigg[\int_{E\times\{\tau\}}(|u|+|v_h|)^{q+1}\,\dx\bigg]^{\frac{1-q}2}\\
     &\le 
     \boldsymbol\gm
     \bigg[ \int_{E\times\{\tau\}} \mathfrak{g}(u,v_h)\,\dx  \bigg]^{\frac{q+1}2} \cdot \bigg[\sup_{t\in[0,T]}\int_{E\times\{t\}} |u|^{q+1}\,\dx\bigg]^{\frac{1-q}2},
\end{align*}
where as before, we used the short-hand notation $\boldsymbol{u}^{\frac{q+1}2}:=|u|^{\frac{q-1}2}u$.

To proceed, choose $h\equiv h_j = 1/j$ with $j\in\mathbb{N}$ and set $N:=\cup_{j=1}^{\infty} N_{h_j}$. Then the above observations yield that for any $q>0$,
\[
\lim_{j\to\infty}\sup_{\tau\in[0,\frac12 T]\setminus N} \int_{E\times\{\tau\}}|u-v_{h_j}|^{q+1}\,\dx\dt=0. 
\]
According to this uniform convergence and the continuity of the map $[0,T]\ni t\mapsto v_{h_j}(\cdot, t)$ in $L^{q+1}(E)$ by Lemma~\ref{Lm:mol} (ii), we may conclude that there is a representative $u$ satisfying 
$$
u\in C \big([0,\tfrac12T]; L^{q+1} (E)\big).
$$
The continuous representative in $[\tfrac12 T, T]$ is sought in  a similar way. Patching the two representatives together  shows a representative
$$u\in C\big([0,T]; L^{q+1} (E)\big).$$

Now we only need to identify $u(\cdot, 0)=u_o$ a.e. in $E$. To this end, we choose $\z(x,t)=\vp(x)\psi_\varep(t)$ in \eqref{Eq:weak-form-1}, where $\vp\in C^1_0(E)$ and $\psi_\varep=1$ in $[0,\tau]$, $\psi_\varep=0$ in $[\tau+\varep, T]$, while linearly interpolated otherwise. A standard calculation yields that 
\begin{align*}
	\int_{E }	
	\boldsymbol{u_o}^{q}\vp \,\dx
	-
    \int_{E \times\{\tau\}}	
	\boldsymbol{u}^{q}\vp \,\dx 
	=
	\iint_{E\times(0,\tau)}\bl{A}(x,t,u,Du)\cdot D\vp\,\dx\dt. 
\end{align*}
From this and a density argument, one derives
\[
\lim_{\tau\to0}\int_{E \times\{\tau\}}	
	\boldsymbol{u}^{q}\vp \,\dx =
	\int_{E}	
	\boldsymbol{u_o}^{q}\vp \,\dx\quad\text{for any}\>\vp\in L^{q+1}(E).
\]
Consequently, we have $u(\cdot, 0)=u_o$ a.e. in $E$.
\end{proof}

\begin{remark}\upshape
Instead of the notion of super(sub)-solution defined in Section~\ref{S:notion-sol}, one can start with a somewhat weaker notion.
Indeed, we could call a function $u$
to be a weak  super-solution to \eqref{Eq:1:1} under assumption \eqref{Eq:1:2-}, if
\begin{equation}  \label{Eq:func-space-A}
	u\in L^{q+1} (E_T)\cap L^p \big(0,T; W^{1,p} (E)\big)
\end{equation}
and if 
\begin{equation} \label{Eq:weak-form-A}
	\iint_{E_T} \big[-\boldsymbol{u}^{q}\pl_t\z+\bl{A}(x,t,u,Du)\cdot D\z\big]\dx\dt
	\ge 0
\end{equation}
for all non-negative test functions 
\begin{equation}\label{Eq:test-func-A}
\z\in  W_0^{1,q+1} \big(0,T; L^{q+1}(E)\big)\cap L^p \big(0,T; W^{1,p}_{0}(E)\big).
\end{equation}
The notion of weak sub-solution is defined by requiring $-u$ to be a weak super-solution.
A function that is both a weak super-solution and a weak sub-solution is called a  weak solution.
\end{remark}

Adapting the  arguments from above, we can show the following. 

\begin{proposition}
Suppose $u$ is a weak solution to \eqref{Eq:1:1} under the assumption \eqref{Eq:1:2-} in the sense of \eqref{Eq:func-space-A}, \eqref{Eq:weak-form-A} and \eqref{Eq:test-func-A}. Then $u\in C([0,T]; L_{\loc}^{q+1}(E))$.
\end{proposition}

Therefore, the two notions of solution are effectively equivalent. However, it is noteworthy that generally, weak super(sub)-solutions in the sense of \eqref{Eq:func-space-A} -- \eqref{Eq:test-func-A} possess mere $L^\infty$ regularity in the time variable.

Finally, we discuss the way in which the initial value is assumed. From Proposition~\ref{Prop:B:3} we know that it is taken in the $L^{q+1}$-sense. The next lemma is a consequence of this fact.

\begin{lemma}\label{Lm:CP1}
Suppose $u$ satisfies the hypothesis of Proposition \ref{Prop:B:3}. Then we have
\begin{equation*}
    \lim_{t\downarrow 0}\int_{E\times\{t\}}\big||u|^{q-1}u-|u_o|^{q-1}u_o\big|\,\dx=0.
\end{equation*}
\end{lemma}
\begin{proof}
According to Proposition~\ref{Prop:B:3}, we have
\[
    \lim_{t\downarrow 0}\int_{E\times\{t\}}|u-u_o|^{q+1}\,\dx=0.
\]
When $0<q\le1$, the desired conclusion follows from the estimate
\begin{align*}
    \int_{E\times\{t\}}\big||u|^{q-1}u-|u_o|^{q-1}u_o\big|\,\dx
    &\le \boldsymbol\gm \int_{E\times\{t\}}|u - u_o |^q\,\dx\\
    &\le \boldsymbol\gm |E|^{\frac1{q+1}}
    \bigg[\int_{E\times\{t\}}|u - u_o|^{q+1}\,\dx\bigg]^{\frac{q}{q+1}}.
\end{align*}
When $q>1$, we use Lemma~\ref{Lm:algebra} and H\"older's inequality to estimate
\begin{align*}
    \int_{E\times\{t\}} &\big||u|^{q-1}u-|u_o|^{q-1}u_o\big|\,\dx\\
    &\le 
    \boldsymbol\gm \int_{E\times\{t\}}|u+u_o|^{q-1} |u - u_o |\,\dx\\
    &\le 
    \boldsymbol\gm
    \bigg[\int_{E\times\{t\}}|u - u_o|^{q+1}\,\dx\bigg]^{\frac{1}{q+1}}
    \bigg[\int_{E\times\{t\}}|u+u_o|^{\frac{(q+1)(q-1)}{q}}\,\dx\bigg]^{\frac{q}{q+1}}.
\end{align*}
Note that the second integral on the right-hand side is bounded uniformly in $t$, since $u\in L^\infty(0,T;L^{q+1}(E))$ and by H\"older's inequality,
\[
    \bigg[\int_{E\times\{t\}}|u+u_o|^{\frac{(q+1)(q-1)}{q}}\,\dx\bigg]^{\frac{q}{q+1}}
    \le 
    |E|^{\frac1{q+1}} \bigg[\int_{E\times\{t\}}|u+u_o|^{q+1}\,\dx\bigg]^{\frac{q-1}{q+1}}.
\]
As a result, the desired conclusion also follows in this case.
\end{proof}


\section{Comparison principles}\label{S:CP}

In this section we collect some comparison principles that will be used throughout the paper. On the one hand, the comparison principle  in the case $q=1$ is well-known to experts; it does not impose any sign restriction on solutions and allows for time-dependent boundary data. On the other hand, in the case $q\not =1$ the comparison principle is less understood. For our purpose, we present a version that deals with two non-negative solutions when one of them vanishes on the lateral boundary. 

In addition to the ellipticity \eqref{Eq:1:2}$_1$, the upper bound \eqref{Eq:1:2-} and the monotonicity \eqref{Eq:CP:mono}, the vector field $\bl{A}(x,t,u,\xi)$ is assumed to be Lipschitz continuous 
with respect to $u$, i.e.
\begin{equation}\label{Eq:CP:growth}
\big|\bl{A}(x,t,u_1,\xi)-\bl{A}(x,t,u_2,\xi)\big|\le \Lm|u_1-u_2| \big(1+|\xi|^{p-1}\big)
\end{equation}
for some given $\Lm>0$, and for the variables in the indicated domains.

In what follows, we define the Lipschitz function $\H_{\dl}(s)$ to be $1$ for $s\ge \dl$, to vanish for $s\le0$ and to be linearly interpolated otherwise, i.e.
\begin{equation}\label{def:Hd}
    \mathcal H_\delta(s)
    :=
    \left\{\begin{array}{cl}
         1, &  \mbox{for $s\ge\delta$,}\\[5pt]
         \frac{1}{\delta}s, & \mbox{for $0<s<\delta$,}\\[5pt]
         0, & \mbox{for $s\le 0$.} 
    \end{array}\right.
\end{equation}

\subsection{Comparison principle~I}
The first comparison principle is as follows.
\begin{proposition}\label{prop:comparison-plapl}
Let $p>1$, $q=1$ and the assumptions \eqref{Eq:1:2-}, \eqref{Eq:CP:mono} and \eqref{Eq:CP:growth} be in force and let 
$$
    v,w\in C\big([0,T];L^2(E)\big)\cap L^p\big(0,T;W^{1,p}(E)\big)
$$
be weak solutions to \eqref{Eq:CDP}$_1$ in the sense of Definition~\ref{def:weak}. Suppose that $v\le w$ on $\partial_{\rm par} E_T$ in the sense that $(v-w)_+\in L^p(0,T;W_0^{1,p}(E))$ and $v(\cdot,0)\le w(\cdot,0)$ a.e. in $E$. Then $v\le w$ a.e. in $E_T$. 
\end{proposition}

 \begin{proof} 
Even if the arguments are standard, we give the details for convenience of the reader. 
We first recall the definition of Steklov-averages from \eqref{def-stek}
and subtract the Steklov-formulations of the differential equations for $v$ and $w$, i.e.
\begin{equation*}
    \int_{E\times\{t\}} \partial_t[v]_h \zeta\,\dx
    +\int_{E\times\{t\}}
    \big[ \mathbf A(x,t,v,Dv) \big]_h\cdot D\zeta 
    \,\dx
    =
    0 
\end{equation*}
for any $t\in(0,T)$; see, for example, \cite[Chapter~II, (1.5)]{DB}. Integrating with respect to $t$ over $(0,\tau)$ for some $\tau\in(0,T)$, we obtain
\begin{align*}
    \iint_{E_\tau}& \partial_t[v-w]_h \zeta\,\dx\dt
    +\iint_{E_\tau}
    \big[ \mathbf A(x,t,v,Dv) -  \mathbf A(x,t,w,Dw)\big]_h\cdot D\zeta 
    \,\dx\dt
    =
    0 
\end{align*}
for $0<h\le\frac12(T-\tau)$ and for any $\zeta\in L^p(0,T;W_0^{1,p}(E))$. As test-function we choose $\zeta=\H_{\dl}([v-w]_h)$, which is admissible since $(v-w)_+\in L^p(0,T;W_0^{1,p}(E))$.
In this way, we obtain
\begin{align*}
    \int_{E\times\{\tau \} } & \mathcal G_\delta\big([v-w]_h\big)  \,\dx - \int_{E\times\{0\}} \mathcal G_\delta\big([v-w]_h\big) \,\dx \\
    &=
    -\iint_{E_\tau} \big[  \mathbf A(x,t,v,Dv) -  \mathbf A(x,t,w,Dw)\big]_h\cdot  D\H_{\dl}\big([v-w]_h\big) \,\dx\dt,
\end{align*}
where $\mathcal G_\delta$ is the primitive of $\mathcal H_\delta$, i.e.
$$
    \mathcal G_\delta(s)
    :=
    \int_0^s \H_{\dl}(\sigma)\,\d\sigma
    =
    \left\{\begin{array}{cl}
         s-\frac{1}{2}\delta, &  \mbox{for $s\ge\delta$,}\\[5pt]
         \frac{1}{2\delta}s^2, & \mbox{for $0<s<\delta$,}\\[5pt]
         0, & \mbox{for $s\le 0$.} 
    \end{array}\right.
$$
Note that $\mathcal G_\delta(s)\to \max\{ s,0\}$ as $\delta\downarrow 0$.
We now pass to the limit $h\downarrow 0$. Since $v,w\in C([0,T];L^2(E))$ the first integral on the left converges to $\int_{E\times\{\tau\}} \mathcal G_\delta(v-w) \,\dx$, while the second converges to $\int_{E\times\{0\}} \mathcal G_\delta(v-w) \,\dx=0$. Passing to the limit $h\downarrow 0$ also in the diffusion term, we obtain
\begin{align*}
    \int_{E\times\{\tau\}} & \mathcal G_\delta(v-w) \,\dx  \\
    &=
    -\iint_{E_\tau} 
    \H_\delta^{\prime}(v-w) 
    \big[\mathbf  A(x,t,v,Dv) - \mathbf A(x,t,w,Dw)\big]\cdot(Dv-Dw) \,\dx\dt \\
    &=
   -\iint_{E_\tau} 
   \H_\delta^{\prime}(v-w)
   \big[\mathbf A(x,t,v,Dv) - \mathbf A(x,t,v,Dw)\big]\cdot
    (Dv-Dw) \,\dx\dt \\
   &\phantom{=\,}
   -\iint_{E_\tau} \H_\delta^{\prime}(v-w)
   \big[ \mathbf  A(x,t,v,Dw) - \mathbf A(x,t,w,Dw)\big]\cdot
    (Dv-Dw) \,\dx\dt .
\end{align*}
The first term on the right-hand side is non-positive due to the monotonicity assumption \eqref{Eq:CP:mono} and henceforth can be discarded, 
whereas the second term tends to zero. Indeed, with  \eqref{Eq:CP:growth} we have
\begin{align*}
      \iint_{E_\tau} &\H_\delta^{\prime}(v-w)
   \big[ \mathbf  A(x,t,v,Dw) - \mathbf A(x,t,w,Dw)\big]\cdot
    (Dv-Dw) \,\dx\dt \\
   &=
   \frac{1}{\delta}\iint_{E_\tau\cap \{0<v-w<\delta\}} \big[ \mathbf  A(x,t,v,Dw) -  
    \mathbf A(x,t,w,Dw)\big]\cdot
    (Dv-Dw) \,\dx\dt \\
   & \le 
   \frac{\Lm}{\dl}\iint_{E_\tau\cap \{0<v-w<\delta\}}(v-w)\big(1+|Dw|^{p-1}\big)|D(v-w)|\,\dx\dt\\
   & \le \Lm\iint_{E_\tau\cap \{0<v-w<\delta\}}\big(1+|Dv|^{p-1}\big)|D(v-w)|\,\dx\dt   \to 0,
\end{align*}
as $\dl\downarrow0$. 
Therefore, we have
\begin{equation*}
    \int_{E\times\{\tau\}} (v-w)_+\,\dx
    =
    \lim_{\dl\downarrow0}  \int_{E\times\{\tau\}} \mathcal G_\delta(v-w) \,\dx
    \le 0.
\end{equation*}
Integrating this with respect to $\tau\in (0,T)$ we have
\begin{equation*}
    \iint_{E_T} (v-w)_+\,\dx\dt
    \le 0,
\end{equation*}
proving that $v\le w$ a.e.~in $E_T$.  This finishes the proof of the comparison principle.
\end{proof}

Joining Proposition~\ref{Prop:B:3} and Proposition~\ref{prop:comparison-plapl}, we obtain the uniqueness of weak solutions in the case $q=1$.
\begin{corollary}
Under the assumptions of Proposition~\ref{prop:comparison-plapl}, the weak solution to the Cauchy-Dirichlet Problem \eqref{Eq:CDP} with $q=1$ is unique.
\end{corollary}

\subsection{Comparison principle~II}
Notwithstanding the lack of time derivative of weak solutions, this difficulty could be readily fixed by a proper time mollification when $q=1$. 
This is, however, is not the case for $q\not=1$:  more delicate analysis is required in order to achieve a comparison result, because of the absence of time derivative. The next comparison principle applies for all $p>1$ and $q>0$ provided that one of the two non-negative solutions vanishes on the lateral boundary. 
\begin{proposition}\label{Prop:CP}
Let $p>1$, $q>0$ and consider a vector field $\bl{A}(x,u,\xi)$ that is independent of the $t$-variable and satisfies the assumptions \eqref{Eq:1:2-}, \eqref{Eq:CP:mono} and \eqref{Eq:CP:growth}. Suppose that $w$ is a non-negative weak solution to \eqref{Eq:CDP}$_1$ and $v$ a non-negative weak solution to the Cauchy-Dirichlet Problem \eqref{Eq:CDP} with $g\equiv 0$ and $v_o\ge 0$, both in the sense of Definition~\ref{def:weak}. 
If $v_o\le w(\cdot,0)$ a.e. in $E$, then we have $v\le w$ a.e. in $E_T$. 
\end{proposition}

The following uniqueness result is a direct consequence of Proposition~\ref{Prop:CP}. 

\begin{corollary}\label{Cor:Uniqueness}
Under the assumptions of Proposition~\ref{Prop:CP}, the non-negative weak solution to the Cauchy-Dirichlet Problem \eqref{Eq:CDP} with $g\equiv 0$ is unique.
\end{corollary}

The rest of this subsection is devoted to the proof of Proposition~\ref{Prop:CP}. We recall the definition of the Lipschitz function $\H_{\dl}(z)$ in \eqref{def:Hd}. Accordingly, we introduce the quantity
\begin{equation}\label{Eq:def-q}
\mathfrak{h}_\dl (z,z_o):=\int_{z_o}^z \H_{\dl}(s-z_o) q s^{q-1}\,\ds\quad\text{for}\>z,\,z_o\in\rr_{\ge 0}.
\end{equation}
The odd reflection of $\H_\dl$ is defined by $\widehat{\H}_\delta(s):=-\H_{\dl}(-s)$ and the associated quantity is
\begin{equation}\label{Eq:def-q-hat}
\widehat{\mathfrak{h}}_\dl (z,z_o):=\int_{z_o}^z \widehat{\H}_{\dl}(s-z_o) q s^{q-1}\,\ds\quad\text{for}\>z,\,z_o\in\rr_{\ge 0}.
\end{equation}

The first lemma for the function $v$ is as follows.
\begin{lemma}\label{Lm:CP3}
Let $v$ be as in Proposition~\ref{Prop:CP} and suppose $\widetilde v\in L^{q+1}(E)\cap W^{1,p}_0(E)$ is non-negative.  
Then we have
\[
    \iint_{E_T}\Big[-\mathfrak{h}_\dl(v,\widetilde v)\partial_t\psi+\bl{A}(x,v,Dv)\cdot D[\H_{\dl}(v-\widetilde v)\psi]\Big]\,\dx\dt=0
\]
for all  $\psi\in C^\infty_0(\R^N\times(0,T))$. 
\end{lemma}

\begin{proof}
Since $v=0$ on $\partial E\times (0,T)$ and $\widetilde v=0$ on $\partial E$, we have
\[
\H_\dl \big(v(\cdot, t)-\widetilde v(\cdot)\big)=0\quad \mbox{in the sense of traces on $\pl E$,}
\]
for a.e. $ t\in(0,T)$.
This allows to choose the test function $\z=\H_\dl(\llbracket v \rrbracket_{h}-\widetilde{v})\psi$ in the weak formulation \eqref{Eq:weak-form-1}. Since $\zeta (\cdot,0)=0$ on $E$, the term containing $v_o$ vanishes.
By Lemma~\ref{Lm:mol} (iv), 
we have $D[\H_\dl(\llbracket v \rrbracket_{h}-\widetilde{v})\psi]\to D[\H_\dl(v -\widetilde{v})\psi]$ as $h\downarrow 0$ in $L^p(E_T)$. 
Therefore we obtain for the diffusion part 
\begin{align*}
    \lim_{h\downarrow 0}\iint_{E_T} &\bl{A}(x,v,Dv)\cdot D\z \,\dx\dt 
    =
    \iint_{E_T} \bl{A}(x,v,Dv)\cdot D[\H_{\dl}(v-\widetilde v)\psi] \,\dx\dt.
\end{align*}
Whereas, in view of Lemma~\ref{Lm:mol} (iii), the time part is estimated by
\begin{align*}
    \iint_{E_T}
    & 
    -v^q\pl_t \z\,\dx\dt=\iint_{E_T} \big(-\llbracket v \rrbracket_{h}^q+\llbracket v \rrbracket_{h}^q-v^q\big) \pl_t \z\,\dx\dt\\
    &=-
    \iint_{E_T} \llbracket v \rrbracket_{h}^q \pl_t \z\,\dx\dt
    +  
    \iint_{E_T} \big( \llbracket v \rrbracket_{h}^q-v^q\big) \H_\dl (\llbracket v \rrbracket_{h}-\widetilde{v}) \pl_t \psi\,\dx\dt\\
    &\quad+ 
    \iint_{E_T} \big( \llbracket v \rrbracket_{h}^q-v^q\big) 
    \H_\dl^{\prime}(\llbracket v \rrbracket_{h}-\widetilde{v})
    \tfrac1h (v-\llbracket v \rrbracket_{h}) \psi\,\dx\dt\\
    &\le  
    \iint_{E_T} \pl_t \llbracket v \rrbracket_{h}^q  \z\,\dx\dt
    +  
    \iint_{E_T} \big( \llbracket v \rrbracket_{h}^q-v^q\big) 
    \H_\dl (\llbracket v \rrbracket_{h}-\widetilde{v}) \pl_t \psi\,\dx\dt\\
    &=-
    \iint_{E_T} \mathfrak{h}_\dl \big( \llbracket v \rrbracket_{h},\widetilde{v}\big) \pl_t\psi\,\dx\dt
    +  
    \iint_{E_T} \big( \llbracket v \rrbracket_{h}^q-v^q\big) 
    \H_\dl (\llbracket v \rrbracket_{h}-\widetilde{v}) \pl_t \psi\,\dx\dt.
\end{align*}
By Lemma~\ref{Lm:mol} (i) (note that $v\in L^{q+1}(E_T)$) the second term
on the right vanishes in the limit $h\downarrow 0$. With the help of
\begin{align*}
    \big| \mathfrak{h}_\dl \big( \llbracket v \rrbracket_{h},\widetilde{v}\big)
     - 
    \mathfrak{h}_\dl ( v ,\widetilde{v})\big|
    &=
    \bigg| \int_v^{\llbracket v \rrbracket_{h}} \H_{\dl}(s-\widetilde v) q s^{q-1}\,\ds\bigg|
    \le 
    \big| \llbracket v \rrbracket_{h}^q-v^q\big|
\end{align*}
and Lemma~\ref{Lm:mol} (i) it also results
\begin{align*}
    -\lim_{h\downarrow 0}  
    \iint_{E_T} \mathfrak{h}_\dl \big( \llbracket v \rrbracket_{h},\widetilde{v}\big) \pl_t\psi\,\dx\dt
    &=
    - \iint_{E_T} \mathfrak{h}_\dl ( v ,\widetilde{v}) \pl_t\psi\,\dx\dt.
\end{align*}
Combining the above estimates yields the desired identity with ``$\le$". The reverse inequality is obtained by using the test function $\z=\H_\dl(\llbracket v \rrbracket_{\bar{h}}-\widetilde{v})\psi$ and running similar calculations.
\end{proof}
By an analogous argument, we also have the following version for $w$.
\begin{lemma}\label{Lm:CP4}
Let $w$ be as in Proposition~\ref{Prop:CP} and suppose $\widetilde w\in L^{q+1}(E)\cap W_0^{1,p}(E)$ is non-negative. Then we have
\[
    \iint_{E_T}\Big[-\widehat{\mathfrak{h}}_\dl(w,\widetilde w)\partial_t\psi
    +
    \bl{A}(x,w,Dw)\cdot D[\widehat{\H}_\dl (w - \widetilde w)\psi]\Big]\dx\dt=0
\]
for all $\psi\in C^\infty_0(\R^N\times(0,T))$. 
\end{lemma}

\begin{proof}
Note that $\widetilde w=0$ on $\partial E$ and $w$  is non-negative. Therefore, for a.e.~$t\in (0,T)$, we have
\[
\widehat\H_\dl \big(w(\cdot, t)-\widetilde w(\cdot)\big)=0\quad \mbox{in the sense of traces on $\pl E$.}
\]
This allows us to choose the test function $\z=\widehat\H_\dl(\llbracket w \rrbracket_{h}-\widetilde{w})\psi$ in the weak formulation \eqref{Eq:weak-form-1}.
The remainder of the proof is analogous to the one for Lemma~\ref{Lm:CP3}.
\end{proof}


With all the preparation at hand, the proof of Proposition~\ref{Prop:CP} is based on the approach developed in \cite{Otto}, and relies on the doubling of the time variable. 
\begin{proof}[\textrm{\bf Proof of Proposition~\ref{Prop:CP}}]
Let us consider
\[
    (x,t_1,t_2)\in\widetilde{Q}:=E\times(0,T)^2,
\]
select a non-negative function $\psi\in C^{\infty}_0((0,T)^2)$,  and extend $v$ and $w$ to $\widetilde Q$ by setting
\[
    v(x,t_1,t_2)=v(x,t_1),\qquad w(x,t_1,t_2)=w(x,t_2).
\]
It is apparent that $\H_\delta$, the choices $\widetilde v(x)=w(x,t_2)=:w_{t_2}(x)$, and $(0,T)\ni t_1\mapsto \psi_{t_2}(t_1):= \psi(t_1,t_2)$ are admissible in Lemma~\ref{Lm:CP3} for fixed $\delta>0$ and a.e. $t_2\in(0,T)$. Hence, we have
\begin{align}\label{Eq:CP3}
    \iint_{E_T}\Big[ & -\mathfrak{h}_\delta(v,w_{t_2})\partial_{t_1}\psi_{t_2}
    +
    \bl{A}(x,v,Dv)\cdot D\big[\H_\dl \big(v-w_{t_2}\big)\psi_{t_2}\big]\Big]\,\dx\dt_1=0,
\end{align}
where $\mathfrak{h}_\delta$ is linked to $\H_\delta$ by \eqref{Eq:def-q}. 
Here $v$ is to be understood as a function of $(x,t_1)$, while $w_{t_2}$ depends on $x$,
and $\psi_{t_2}$ on $t_1$. Analogously, $\widehat{\H}_\delta$, the choices $\widetilde w(x)=v(x,t_1)=: v_{t_1}(x)$, and $(0,t)\ni t_2\mapsto \psi_{t_1}(t_2):= \psi(t_1,t_2)$ are admissible in Lemma~\ref{Lm:CP4} for fixed $\delta>0$ and a.e. $t_1\in(0,T)$. Hence
\begin{align}\label{Eq:CP4}
    \iint_{E_T}\Big[& -\widehat{\mathfrak{h}}_\delta(w,v_{t_1})\partial_{t_2}\psi_{t_1}
    + 
    \bl{A}(x,w,Dw)\cdot D\big[\widehat{\H}_\dl 
    \big(w-v_{t_1}\big)\psi_{t_1}\big]\Big]\,\dx\dt_2=0,
\end{align}
where $\widehat{\mathfrak{h}}_\delta$ is linked to $\widehat{\H}_\delta$ by \eqref{Eq:def-q-hat}. In obtaining \eqref{Eq:CP3} and \eqref{Eq:CP4}, we exploited the fact that $v=0$ on $\pl E\times(0,T]$.

We now integrate \eqref{Eq:CP3} over $t_2\in(0,T)$, \eqref{Eq:CP4} over $t_1\in(0,T)$, add both resulting identities, use the fact
\[
    \H_\delta(z)=-\widehat{\H}_\delta(-z),
\]
and conclude that 
\begin{align}\label{Eq:CP5}
    \iiint_{\widetilde Q}\Big[ &- \big[ \mathfrak{h}_\delta(v,w)\partial_{t_1}\psi + \widehat{\mathfrak{h}}_\delta(w,v)\partial_{t_2}\psi\big] \\\nonumber
    & + \big(\bl{A}(x,v,Dv)-\bl{A}(x,w,Dw)\big)\cdot  D[ \H_\delta(v-w)]\psi\Big]\,\dx\dt_1\dt_2=0.
\end{align}
In \eqref{Eq:CP5}, we will eventually pass to the limit with respect to $\delta$. Before doing it, we need to get rid of the terms which do not have a regular limit. Indeed, we have
\begin{align*}
    \big(\bl{A}(x,v,Dv)-&\bl{A}(x,w,Dw)\big)\cdot D[\H_\delta(v-w)] \\
    &= \H_\delta^{\prime}(v-w) \big(\bl{A}(x,v,Dv)-\bl{A}(x,v,Dw)\big)\cdot (D v-D w) \\
    &\quad +
    \H_\delta^{\prime}(v-w) \big(\bl{A}(x,v,Dw)-\bl{A}(x,w,Dw)\big)\cdot (D v-D w).
\end{align*}
The first term on the right-hand side is non-negative due to the monotonicity \eqref{Eq:CP:mono} and henceforth can be discarded. Whereas the second term tends to zero under the integral over $\widetilde{Q}$ because of \eqref{Eq:CP:growth}. Indeed, we estimate with the aid of \eqref{Eq:CP:growth} that
\begin{align*}
   \iiint_{\widetilde{Q}} & \H_\delta^{\prime}(v-w) \big(\bl{A}(x,v,Dw) - \bl{A}(x,w,Dw)\big)\cdot (Dv-Dw)\psi\,\dx\dt_1\dt_2\\
   & \le \frac{\Lm}{\dl}\iiint_{\widetilde{Q}\cap\{0<v-w<\dl\}}(v-w)\big(1+|Dw|^{p-1}\big)|D(v-w)|\psi\,\dx\dt_1\dt_2\\
   & \le \Lm\iiint_{\widetilde{Q}\cap\{0<v-w<\dl\}}\big(1+|Dw|^{p-1}\big)|D(v-w)|\psi\,\dx\dt_1\dt_2   \to 0,
\end{align*}
as $\dl\downarrow0$. 
Hence, we obtain from \eqref{Eq:CP5} that
\begin{equation}\label{Eq:CP7}
    \limsup_{\delta\downarrow0}\iiint_{\widetilde Q}-\big[ \mathfrak{h}_\delta(v,w)\partial_{t_1}\psi + \widehat{\mathfrak{h}}_\delta(w,v)\partial_{t_2}\psi\big]\,\dx\dt_1\dt_2\le0.
\end{equation}
In order to identify the limit as $\delta\downarrow 0$ in \eqref{Eq:CP7}, observe that for any non-negative $z\ge z_o$,
\begin{align*}
    \lim_{\delta\downarrow 0}\mathfrak{h}_\delta(z,z_o)
    &=
    \lim_{\delta\downarrow 0}\int_{z_o}^z \mathcal H_\delta (s-z_o)qs^{q-1}\,\ds\\
    &=
    \int_{z_o}^z \boldsymbol\chi_{(z_o,\infty)} (s)qs^{q-1}\,\ds\\
    &=\int_{z_o}^z qs^{q-1}\,\ds =z^q-z_o^q,
\end{align*}
and 
\begin{align*}
   0\le \mathfrak{h}_\delta(z,z_o)
    &\le  
    \int_{z_o}^z qs^{q-1}\,\ds =z^q-z_o^q.
\end{align*}
For non-negative $z<z_o$ we have $\mathfrak{h}_\delta(z,z_o)=0$. In summary, for $z,z_o\in \rr_{\ge0}$ we have
\begin{equation*}
    0\le \mathfrak{h}_\delta(z,z_o)\le(z^q-z_o^q)_+
    \quad\mbox{and}\quad 
    \lim_{\delta\downarrow 0}\mathfrak{h}_\delta(z,z_o)=(z^q-z_{o}^q)_+,
\end{equation*}
and similarly
\begin{equation*}
    0\le \widehat{\mathfrak{h}}_\delta(z,z_o)\le(z_o^q-z^q)_+
    \quad\mbox{and}\quad 
    \lim_{\delta\downarrow 0}\widehat{\mathfrak{h}}_\delta(z,z_o)=(z_o^q-z^q)_+
    .
\end{equation*}
Applying the Dominated Convergence Theorem as $\delta\downarrow 0$ in \eqref{Eq:CP7}, and taking into account that both $v$ and $w$ are non-negative yield
\begin{equation}\label{Eq:CP8}
    \iiint_{\widetilde Q} 
    -(v^q-w^q)_+(\partial_{t_1}\psi + \partial_{t_2}\psi) \,\dx\dt_1\dt_2\le0
\end{equation}
for all non-negative $\psi\in C^\infty_0((0,T)^2)$.

Let the non-negative $\phi\in C^\infty_0(0,T)$ be given. Fix a non-negative $\varphi\in C^\infty_0(\R)$ with unit $L^1$ norm and define
\[
    \psi_\epsilon(t_1,t_2)
    :=
    \frac1\epsilon\varphi\Big(\frac{t_1-t_2}\epsilon\Big)\phi\Big(\frac{t_1+t_2}2\Big).
\]
In order for $\psi_\epsilon$ to be admissible in \eqref{Eq:CP8}, we need
\begin{equation*}
    \psi_\epsilon\ge0
    \quad\mbox{and}\quad
    \spt\psi_\epsilon\Subset (0,T)^2.
\end{equation*}
The former condition is trivially satisfied, whereas the latter one holds, provided $\epsilon$ is chosen sufficiently small. By direct calculations, we have
\[
    \partial_{t_1}\psi_\epsilon+\partial_{t_2}\psi_\epsilon
    =
    \frac1{\epsilon}\varphi\Big(\frac{t_1-t_2}\epsilon\Big)
    \phi'\Big(\frac{t_1+t_2}2\Big).
\]
Making a change of variable by letting $\displaystyle t_1-t_2=\tau$, whence we obtain
\begin{align}\label{Eq:CP9}
    \int_{\R} \frac1\epsilon\varphi\Big(\frac{\tau}\epsilon\Big)\,
    \iint_{E_T}
    -\big(v^q(t)-w^q(t-\tau)\big)_+\phi'\Big(t-\frac\tau2\Big) \,\dx\dt\d\tau\le0.
\end{align}
It is apparent that
\begin{align}\label{Eq:CP10}
    \lim_{\tau\to 0}\iint_{E_T} 
    \big(v^q(t)&-w^q(t-\tau)\big)_+\phi'\Big(t-\frac\tau2\Big)\,\dx\dt\\\nonumber
    &=
    \iint_{E_T} \big(v^q(t)-w^q(t)\big)_+\phi'(t) \,\dx\dt.
\end{align}
Taking account of \eqref{Eq:CP10} in \eqref{Eq:CP9}, and passing to the limit as $\tau\to0$ and  $\epsilon\downarrow 0$ in \eqref{Eq:CP9} yields
\begin{equation*}
    \iint_{E_T} -(v^q-w^q)_+\phi'\,\dx\dt\le0
\end{equation*}
for any non-negative $\phi\in C^\infty_0(0,T)$.

Now we remove the assumption that $\phi(0)=0$. To obtain this, we choose $\phi$ in the form $\alpha\eta_\delta$ with $\alpha\in C^\infty_0((-\infty,T))$ and $\eta_\delta$ equal to
0 on $(-\infty, 0]$, equal to $1$ on $[\delta ,\infty)$ with $\delta>0$ and linearly interpolated on $(0,\delta)$. With the abbreviation $w_o=w(\cdot,0)$ we obtain
\begin{align*}
    \iint_{E_T} &-(v^q-w^q)_+\alpha'\eta_\delta\,\dx\dt\\
    &\le 
    \tfrac1{\delta}
    \iint_{E\times (0,\delta)} (v^q-w^q)_+\alpha\,\dx\dt\\
    &\le     
    \tfrac1{\delta}
    \iint_{E\times (0,\delta)}
    \Big[
    (v^q-v_o^q)_+ + (v_o^q-w_o^q)_+ +(w_o^q-w^q)_+\Big]\al\,\dx\dt\\
    &=
    \tfrac1{\delta}
    \iint_{E\times (0,\delta)}
    \Big[
    (v^q-v_o^q)_+ + (w_o^q-w^q)_+\Big]\al\,\dx\dt + \tfrac1{\delta}\int_0^\delta \alpha\, \dt
    \int_E  (v_o^q-w_o^q)_+\dx.
\end{align*}
As a consequence of Proposition \ref{Prop:B:3}, the fact that $w\in C ([0,T]; L^{q+1}(E))$, and 
Lemma~\ref{Lm:CP1} the first two terms on the right-hand side vanish in the limit $\delta\downarrow 0$, while the third integral converges to $\alpha (0)\int_E  (v_o^q-w_o^q)_+\,\dx$. Finally, for  the left-hand side we have
\begin{equation*}
    \lim_{\delta\downarrow 0}  \iint_{E_T} -(v^q-w^q)_+\alpha'\eta_\delta\,\dx\dt
    =
    \iint_{E_T} -(v^q-w^q)_+\alpha'\,\dx\dt.
\end{equation*}
Taking advantage of these convergences we get
\[
    -\int_0^T\alpha'(t)\bigg[\int_{E\times\{t\}} (v^q -w^q )_+\,\dx\bigg]\,\dt
    \le
    \alpha(0)\int_E (v_o^q - w_o^q)_+\,\dx.
\]
By assumption we have $(v_o^q-w_o^q)_+=0$ a.e.~in $E$, so that the right-hand side integral of the last display vanishes. Moreover, by density, we can consider $\alpha$, such that $\alpha(T)=0$, $\alpha(0)=1$, and  $\alpha'(t)=-\frac1T$ for all $ t\in(0,T)$. Hence, we obtain
\[
\iint_{E_T} (v^q -w^q )_+\,\dx\dt=0\quad\implies\quad v \le w \quad\text{a.e. in}\> E_T.
\]
This finishes the proof.
\end{proof}

\chapter{Boundedness of weak solutions}\label{sec:boundedness}

We first discuss the significance of the range of $p$ and $q$ in Corollary~\ref{Cor:bdd}, comparing them with previous results. Subsequently, we provide the full proof of Theorem~\ref{THM:BD:1}.

\section{Remarks on the literature}\label{S:bd-opt}
Corollary~\ref{Cor:bdd} was first stated in \cite[Theorem~1.1]{Ivanov-1995-2} for the prototype equation \eqref{doubly-nonlinear-prototype}, but written 
in a different form, namely
\begin{equation}\label{DNL-FSV}
u_t-\Div(|u|^{m-1}|Du|^{p-2}Du)=0,
\end{equation}
for $p>1$, $2<p+m<3$. Comparing \eqref{doubly-nonlinear-prototype} with \eqref{DNL-FSV}, it is apparent that they are (at least formally) equivalent, once we set
\[
q=\frac1{1+\frac{m-1}{p-1}},
\]
whence conditions $p>1$, $2<p+m<3$ translate into conditions $0<p-1<q$. In the sequel, 
for simplicity we treat the two formulations as if they were fully synonymous.

Local boundedness is given in \cite{Ivanov-1995-2} assuming
\[
p<N,\qquad\frac{\frac{m-1}{p-1}+2}{\frac{m-1}{p-1}+1}<\frac{Np}{N-p},
\]
conditions which are equivalent to the assumptions on $p$ and $q$ of Corollary~\ref{Cor:bdd}, provided $p<N$. Even though the statement is given only for \eqref{DNL-FSV}, the author briefly discusses how the results can be extended to a more general class of quasi-linear equations, which correspond to \eqref{Eq:1:1f}--\eqref{Eq:1:2} here. Moreover, it is important to remark that the statement of \cite{Ivanov-1995-2} is purely qualitative, and there is no proof, whereas a detailed demonstration is provided for the corresponding {\em global result}. On the other hand, in \cite[Section~4]{Ivanov-1995-2}, Ivanov shows the sharpness of his result. Indeed, for $r\in(0,1)$, $h>1$, $s>0$, he considers the unbounded function
\[
w\colon \overline{B_r\times(0,1)}\to\R,\quad w(x,t)=(1-ht)_+ v(x),\quad v(x)=\left[\frac{(r^2-|x|^2)^2}{|x|^{N/s} \ln^2|x|^2}\right]^{\frac1{1+\frac{m-1}{p-1}}},
\]
and proves that if 
\[
p<N,\quad s=\frac Np\left(\frac{\frac{m-1}{p-1}+2}{\frac{m-1}{p-1}+1}-p\right),\quad
\frac{\frac{m-1}{p-1}+2}{\frac{m-1}{p-1}+1}\ge\frac{Np}{N-p},
\]
there exist proper values of $r$ and $h$, such that $w$ is a sub-solution to the Cauchy-Dirichlet Problem
\begin{equation*}
\left\{
\begin{array}{cl}
    \partial_t w - \Div (|w|^{m-1} |Dw|^{p-2}Dw)=0  & \quad \mbox{in  $B_r\times(0,1]$,}\\[6pt]
    w =0 &\quad \mbox{on $\partial B_r\times(0,1]$,}\\[6pt]
    w(\cdot,0)= v &\quad\mbox{in $B_r$.}
\end{array}
\right.
\end{equation*}

A complete proof of Ivanov's local boundedness estimates is provided in \cite[Theorem~9.1]{Ivanov-1997}; the statement is again qualitative, but a careful check of the computations allows one to recover a quantitative result, much like the one given in Theorem~\ref{THM:BD:1}. It is interesting to point out that in \cite{Ivanov-1997} both the so-called {\em fast diffusion} (i.e. $2<m+p<3$) and {\em slow diffusion} (i.e. $m+p\ge3$) ranges are considered.

Statements and complete proofs for equations like \eqref{DNL-FSV} but with the full quasi-linear structure are first given in \cite[Theorem~3.1]{FSV-14}, under the extra assumption that the gradient $Du$ of a solution $u$ is well-defined, which is not in general the case. Later, in 
\cite[Theorems~7.1--7.2]{Vespri-Vestberg} Vespri \& Vestberg showed that such a hypothesis is not necessary, and the same result holds true, simply assuming that $u$ is a weak solution.

Concerning the range of parameters, Vespri \& Vestberg take $p\in(1,2)$, $m>1$, $2<m+p<3$, and show that local boundedness of solutions is a direct consequence of their sheer definition in the range $m+p>3-\frac{p}{N-\frac{N-p}p}$. Since they implicitly assume $N\ge2$, they always have $p<N$. Moreover, $m>1$ in \eqref{DNL-FSV} implies $0<q<1$ in \eqref{Eq:1:1f}. Keeping in mind these constraints on the values of the parameters, in \cite{Vespri-Vestberg} their optimal range amounts to assuming for \eqref{Eq:1:1f}
\[
q<\frac{N(p-1)+p}{N-p},
\]
and therefore, Corollary~\ref{Cor:bdd} represents a straightforward extension to a wider range.



\section{Optimality}
It is worth to point out that the borderline case distinguishing local boundedness is $q=\frac{N(p-1)+p}{(N-p)_+}$, and not $q=p-1$. Indeed, when $0<q\le p-1$, weak solutions are locally bounded, cf.~\cite{BHSS}.
The optimality of the range of $p$ and $q$ given in Corollary~\ref{Cor:bdd} is shown in the following counterexamples, 
built  on the analogous ones for the fast diffusion equation given in \cite{king, peletier} or for the parabolic $p$-Laplacian considered in \cite{bidaut}. 

A first counterexample is the function
\begin{equation}\label{Ex-bdd}
u(x,t)=C(N,p,q)\frac{(T-t)_+^{\frac1{q-(p-1)}}}{|x|^{\frac{p}{q-(p-1)}}},
\end{equation}
with
\[
C(N,p,q)=\left[\frac{N(q-(p-1))-pq}{q}\left(\frac p{q-(p-1)}\right)^{p-1}\right]^{\frac1{q-(p-1)}};
\]
$u$ is a non-negative, local, weak solution to the prototype equation \eqref{doubly-nonlinear-prototype} in ${\mathbb R}^N\times(-\infty,T)$ for $p<N$ and $0<p-1<\frac{N(p-1)+p}{(N-p)_+}<q$. Obviously, the solution is unbounded near $x=0$ for all $t<T$ and finite otherwise. 

A second counterexample in the same range of values for $p$ and $q$ can be built looking for solutions to \eqref{doubly-nonlinear-prototype} of the form
\begin{equation*}
u(x,t)=(T-t)_+^\alpha\, f\big(|x|(T-t)_+^\beta\big),
\end{equation*}
where $T\in\R$, $\alpha,\,\beta\in\R$ to be determined,
and $f(r)$ is a positive, 
monotone decreasing,
radial function also to be computed; for simplicity, we assume $q=1$, and we very briefly discuss later on what happens in a more general case.
Due to the conditions
\[
p<N,\quad 0<p-1<\frac{N(p-1)+p}{N-p}<q,
\]
the previous choice for $q$ implies $1<p<\frac{2N}{N+2}$.

If we assume $f$ to be regular and smooth, it is rather straightforward to verify that $f$ must satisfy the ordinary differential equation
\begin{equation*}
 -\big[\alpha f+\beta r f'\big]+\frac{N-1}r(-f')^{p-1}-(p-1)(-f')^{p-2}f''=0,
\end{equation*}
with $\displaystyle \alpha=\frac{1+\beta p}{2-p}$.
Choosing $\alpha=0$ yields $\beta=-\frac1p$, and $f$ solves
\[
    \frac{r}{p} f'+\frac{N-1}r (-f')^{p-1}-(p-1)(-f')^{p-2}f''=0,
\]
which we can rewrite as
\[
    \frac{r}{p} -\frac{N-1}r(-f')^{p-2}+(p-1)(-f')^{p-3}f''=0.
\]
The substitution $\displaystyle y:=(-f')^{p-2}$ leads to the first order linear ordinary differential equation
\begin{equation}\label{eq-for-y}
y'=\frac{(2-p)(N-1)}{p-1}\frac{y}r -\frac{2-p}{p(p-1)}r,
\end{equation}
which can be explicitly solved. Since $1<p<\frac{2N}{N+2}$, a straightforward integration gives
\[
y=r^2\left[C r^{\frac{|\lambda_1|}{p-1}}+\frac{2-p}{p|\lambda_1|}\right],
\]
where, according to \eqref{def:lambda-r}, $\lambda_1=N(p-2)+p<0$, and $C$ is an arbitrary constant. In order to streamline the presentation, we directly assume $C>0$. If we revert to $f'$, we conclude that
\[
f'(r)=-r^{-\frac 2{2-p}}\left[C r^{\frac{|\lambda_1|}{p-1}}+\frac{2-p}{p|\lambda_1|}\right]^{-\frac1{2-p}}.
\]
This result is also given in \cite[Section 1, (1.9)]{bidaut}, in the framework of a wider study of explicit and semi-explicit solutions to the parabolic $p$-Laplacian for $1<p<2$.

Once $f'$ is known, $f$ can be fully computed for specific values of $N$ and $p$, but we refrain from further going into details here.

Instead of proceeding as we did here, the same result can be obtained relying on the correspondence between radial solutions to the porous medium equation and to the parabolic $p$-Laplacian, studied in \cite{iagar}. In such a case, the starting point to recover the previous expression for $f'$, is the so-called \emph{dipole solution} to the porous medium equation
\[
u_t-\Delta u^m=0,
\]
for $0<m<\frac{N-2}{N+2}$, $N\ge3$, discussed in 
\cite[Section~3, (3.4)]{king1990} and in \cite[Proposition~7]{kosov}. 

Independently of the way used to obtain it, it is apparent that
\begin{equation*}
    f(r)
    \simeq 
    \left\{
    \begin{array}{cl}
    \displaystyle
    \bigg[\Big(\frac p{2-p}\Big)^{p-1}|\lambda_1|\bigg]^{\frac1{2-p}} 
    r^{-\frac p{2-p}} &\mbox{as $r\downarrow 0$},\\[9pt]
    \displaystyle
    C\frac{p-1}{N-p }r^{-\frac{N-p}{p-1}}&\mbox{as $r\to\infty$,}
    \end{array}
    \right.
\end{equation*}
which yields the asymptotic behaviour
\begin{equation*}
    u(x,t)
    \simeq 
    \left\{
    \begin{array}{cl}
    \displaystyle
    \bigg[\Big(\frac p{2-p}\Big)^{p-1}|\lambda_1|\bigg]^{\frac1{2-p}}\left[\frac{(T-t)_+}{|x|^p}\right]^{\frac1{2-p}}
    &\mbox{as $ \displaystyle \frac{|x|}{(T-t)_+^{1/p}}\downarrow 0$,}\\[10pt]
    \displaystyle
    C\frac{p-1}{N-p } \bigg[\frac{(T-t)_+^{1/p}}{|x|}\bigg]^{\frac{N-p}{p-1}}
    &
    \mbox{as $ \displaystyle\frac{|x|}{(T-t)_+^{1/p}}\to\infty$.}
     \end{array}
    \right.
\end{equation*}
Notice that the behavior at the origin is precisely the one given by \eqref{Ex-bdd}, once $q=1$ is assumed, whereas the decay at infinity is faster. Moreover, the condition $1<p<\frac{2N}{N+2}$ ensures that $u$ and $Du$ have the right integrability in a neighborhood of the origin, so that $u$ is indeed a weak solution.

Just for the sake of completeness, it is worth pointing out that for $p=\frac{2N}{N+1}$, \eqref{eq-for-y} has the solution
\[
f'(r)=-r^{-(N+1)}\left[C-\frac{N+1}{N(N-1)}\ln r\right]^{-\frac{N+1}2},\quad \mbox{$C\in\R$ arbitrary,}
\]
defined on the open interval $\big( 0,\exp\big[C\frac{N(N-1)}{N+1}\big]\big)$, which is unbounded both in a neighborhood of $r=0$ and as $r\uparrow\exp\big[C\frac{N(N-1)}{N+1}\big]$. However, it is a matter of straightforward computations, to verify that the corresponding solution $u=u(x,t)$ fails to be a weak solution.

Finally, we expect that radial, unbounded solutions similar to the ones we have just discussed,
exist also for $q\not=1$, but we refrain from going into details here.

\section{Energy estimates}\label{S:energy}

In this section we present certain energy estimates for weak sub(super)-solutions to \eqref{Eq:1:1f} under assumption \eqref{Eq:1:2}.
They are analogs of the energy estimates  derived in \cite[Proposition~3.1]{BDL-21}, which will be referred to for details.
Moreover, it is noteworthy that they actually hold true for signed solutions and for all $p>1$ and $q>0$. In deriving the following energy estimates, the mollification \eqref{def:mol} in the time variable is called for. 

\begin{proposition}\label{Prop:2:1}
    Let $p>1$ and $q>0$. There exists a constant  $\boldsymbol \gm (C_o,C_1,p)>0$, such that whenever
	$u$ is a non-negative weak sub(super)-solution to \eqref{Eq:1:1f} with \eqref{Eq:1:2} in $E_T$, 
    $Q_{\rho,s}=K_{\rho}(x_o)\times (t_o-s,t_o]\Subset E_T$ is a parabolic cylinder,
 	 $k>0$, and $\z$ any non-negative, piecewise smooth cutoff function
 	vanishing on $\pl K_{\rho}(x_o)\times (t_o-s,t_o)$,  then we have
\begin{align*}
	\max \bigg\{ \essup_{t_o-s<t<t_o}&\int_{K_\rho(x_o)\times\{t\}}	
	\z^p\mathfrak g_\pm (u,k)\,\dx\, ,\
	\iint_{Q_{\rho,s}}\z^p|D(u-k)_\pm|^p\,\dx\dt \bigg\}\\
	&\le
	\boldsymbol \gm\iint_{Q_{\rho,s}}
		\Big[
		(u-k)^{p}_\pm|D\z|^p + \mathfrak g_\pm (u,k)|\partial_t\z|
		\Big]
		\,\dx\dt\\
	&\quad +
	\int_{K_\rho(x_o)\times \{t_o-s\}} \z^p \mathfrak g_\pm (u,k)\,\dx,
\end{align*}
where $\mathfrak{g}_\pm$ is defined in \eqref{Eq:gpm}.
\end{proposition}
Here, we used a convention about $\pm$: for sub-solutions the energy estimate reads with $+$, whereas for super-solutions the energy estimate reads with $-$.

\section{Proof of Theorem~\ref{THM:BD:1}}

The first step is to derive a general iterative inequality in \eqref{pre-recursiv}, departing from which we distinguish three cases. 
Throughout the proof we assume $(x_o,t_o)=(0,0)$ for simplicity.

\subsection{A general iterative inequality for $r\ge q+1$}

In order to use the energy estimate for sub-solutions in Proposition~\ref{Prop:2:1}, we set for $\sig\in(0,1)$, some $k>0$ to be chosen and $n\in\mathbb{N}_0$,
\begin{align*}
	\left\{
	\begin{array}{c}
	\displaystyle k_n=k - \frac{k}{2^{n}},\quad \widetilde{k}_n=\frac{k_n+k_{n+1}}2,\\[5pt]
	\displaystyle \varrho_n=\sig\varrho+\frac{(1-\sig)\varrho}{2^{n}},
	\quad s_n=\sig s+\frac{(1-\sig)s}{2^{n}},\\[5pt]
	\displaystyle \widetilde{\varrho}_n=\frac{\varrho_n+\varrho_{n+1}}2,
	\quad\widetilde{s}_n=\frac{s_n+s_{n+1}}2,\\[5pt]
	\displaystyle K_n=K_{\varrho_n},\quad \widetilde{K}_n=K_{\widetilde{\varrho}_n},\\[5pt] 
	\displaystyle Q_n=K_n\times(-s_n,0],\quad
	\widetilde{Q}_n=\widetilde{K}_n\times(-\widetilde{s}_n,0].
	\end{array}
	\right.
\end{align*}
Introduce the cutoff function $0\le\z\le 1$ vanishing on the parabolic boundary of $Q_{n}$ and
equal to the identity in $\widetilde{Q}_{n}$, such that
\begin{equation*}
	|D\z|\le\frac{  2^{n+2}}{(1-\sig)\varrho}
	\quad\text{and}\quad 
	|\pl_t\z|\le\frac{  2^{n+2}}{(1-\sig)s}.
\end{equation*}
The above energy estimate in this setting gives that
\begin{align}\label{Eq:E-est}
    \essup_{-\widetilde{s}_n<t<0}\,  
    & 
    \int_{\widetilde{K}_n\times\{t\}}  \mathfrak g_+(u, \widetilde{k}_n)\,\dx 
    + 
    \iint_{\widetilde{Q}_n} |D(u-\widetilde{k}_n)_+|^p\,\dx\dt\\\nonumber
    &\le  
    \frac{ \boldsymbol \gm 2^{n}}{(1-\sig)s}\iint_{Q_{n}}  \mathfrak g_+(u, \widetilde{k}_n) \, \dx\dt
    + 
    \frac{ \boldsymbol \gm 2^{pn}}{(1-\sig)^p\varrho^p} \iint_{Q_{n}} (u-\widetilde{k}_n)_+^p \, \dx\dt,
\end{align}
with a constant $\boldsymbol\gm$ that depends only on $C_o,C_1$, and $p$. 
Here we also used that $\z_n$ vanishes on $K_n\times \{-s_n\}$, resulting in the boundary term on the right-hand side of the energy estimate  being zero.
To treat the first term on the left, we proceed to estimate $\mathfrak g_+(u, \widetilde{k}_n)$ from below. First observe that on the set $\{u>k_{n+1}\}$, there holds
\begin{equation*}
1\le \frac{u+\widetilde{k}_n}{u-\widetilde{k}_n}\le
\frac{2u}{u-\widetilde{k}_n}\le\frac{2k_{n+1}}{k_{n+1} - \widetilde{k}_n}\le 2^{n+3}.
\end{equation*}
Therefore, Lemma~\ref{Lm:g} yields the bound
\begin{align*}
    \mathfrak g_+ (u,\widetilde{k}_n)
    &\ge  
    \tfrac1{\boldsymbol\gamma} (u+\widetilde{k}_n)^{q-1}  (u-\widetilde{k}_n)_+^2
    \ge  
    \tfrac1{\boldsymbol\gamma} (u+\widetilde{k}_n)^{q-1}  (u-k_{n+1})_+^2 \\
    &\ge  
    \tfrac1{\boldsymbol\gamma}2^{-(n+3)(1-q)_+}(u-\widetilde{k}_n)^{q-1}_+(u-k_{n+1})_+^2\\
    &\ge  
    \tfrac1{\boldsymbol\gamma}2^{-(n+3)(1-q)_+}(u-k_{n+1})^{q+1}_+,
\end{align*} 
with a constant $\boldsymbol\gamma=\boldsymbol\gamma(q)$.
Next, we deal with the first term on the right-hand side of the energy estimate.
Observe that on the set $\{u>\widetilde{k}_n\}$, there holds
\begin{equation*}
1\le \frac{u+\widetilde{k}_n}{u-k_n}\le
\frac{2u}{u-k_n}\le\frac{2\widetilde{k}_n}{\widetilde{k}_n - k_n}\le 2^{n+3}.
\end{equation*}
In view of Lemma~\ref{Lm:g} we obtain
\begin{align*}
    \mathfrak g_+ (u,\widetilde{k}_n)
    &\le  
    \boldsymbol\gamma (u+\widetilde{k}_n)^{q-1}  (u-\widetilde{k}_n)_+^2\\\nonumber
    &\le  
    \boldsymbol\gamma 2^{(n+3)(q-1)_+}(u-k_n)^{q-1}_+(u-\widetilde{k}_n)_+^2 \\\nonumber
    &\le  
    \boldsymbol\gamma 2^{(n+3)(q-1)_+}(u-k_{n})^{q+1}_+ \chi_{\{u>\widetilde k_n\}},
\end{align*}
where $\boldsymbol\gamma=\boldsymbol\gamma(q)$. 
Therefore, the energy estimate~\eqref{Eq:E-est} gives  
\begin{align*}
    &2^{-n(1-q)_+}\essup_{-\widetilde{s}_n<t<0}\,  
    \int_{\widetilde{K}_n\times\{t\}} (u-k_{n+1})_+^{q+1}\,\dx 
    + 
    \iint_{\widetilde{Q}_n} |D(u-\widetilde{k}_n)_+|^p\,\dx\dt \\
    &\qquad\le  
    \frac{ \boldsymbol \gm 2^{n(1+(q-1)_+)}}{(1-\sig)s}\iint_{\widetilde A_{n}}  (u-k_n)_+^{q+1} \, \dx\dt
    + 
    \frac{ \boldsymbol \gm 2^{pn}}{(1-\sig)^p\varrho^p} \iint_{Q_{n}} (u-\widetilde{k}_n)_+^p \, \dx\dt,
\end{align*}
where we used the abbreviation $\widetilde{A}_n:=\{u>\widetilde{k}_n\}\cap Q_n$, and where $\boldsymbol\gamma$ depends only on the data $\{C_o,C_1,p,q\}$. Next, we use the measure estimate
\begin{equation}\label{Eq:Meas-int}
    \big|\widetilde{A}_n\big|
    = 
    \big|\big\{u>\widetilde{k}_n\big\}\cap Q_n\big|
    \le 
    \frac{2^{(n+2)r}}{k^r} \iint_{Q_{n}} (u-k_n)_+^{r} \, \dx\dt.
\end{equation}
Since $r\ge q+1\ge p$, H\"older's inequality and \eqref{Eq:Meas-int} imply
\begin{align*}
    \iint_{\widetilde{A}_n} (u-k_n)_+^{q+1} \, \dx\dt
    &\le 
    \bigg[ \iint_{Q_{n}} (u-k_n)_+^{r} \, \dx\dt\bigg]^{\frac{q+1}{r}} 
    \big|\widetilde{A}_n\big|^{1-\frac{q+1}{r}}\\
    &\le
    \boldsymbol \gm (q)\frac{2^{(r-q-1)n}}{k^{r-q-1}}\iint_{Q_{n}} (u-k_n)_+^{r} \, \dx\dt,
\end{align*}
and similarly,
\begin{align}\label{Meas-int-p}
    \iint_{Q_{n}} (u-\widetilde{k}_n)_+^p \, \dx\dt
    &\le 
    \bigg[ \iint_{Q_{n}} (u-\widetilde{k}_n)_+^{r} \, \dx\dt\bigg]^{\frac{p}{r}} 
    \big|\widetilde{A}_n\big|^{1-\frac{p}{r}} \nonumber\\
    &\le
    \boldsymbol \gm (q)\frac{2^{(r-p)n}}{k^{r-p}}\iint_{Q_{n}} (u-k_n)_+^{r} \, \dx\dt.
\end{align}
Inserting these inequalities above, we arrive at
\begin{align*}
    2^{-n(1-q)_+}\essup_{-\tilde{s}_n<t<0}\,  
    & 
    \int_{\widetilde{K}_n\times\{t\}}  (u-k_{n+1})^{q+1}_+\,\dx + \iint_{\widetilde{Q}_n} |D(u-k_{n+1})_+|^p\,\dx\dt\\
    &\le 
    \frac{ \boldsymbol \gm 2^{ nr}}{(1-\sig)^p s }\frac{1}{k^{r-q-1}} \Big[ 1+ \frac{s}{\rho^p}\cdot\frac1{k^{q+1-p}} \Big]
    \iint_{Q_{n}} (u-k_n)_+^{r} \, \dx\dt,
\end{align*}
with a constant $\boldsymbol\gm$ that depends only on the data $\{C_o,C_1,p,q\}$. 
At this stage we may stipulate to take $k$ to satisfy
\begin{equation}\label{eq:k:1}
k\ge\Big(\frac{s}{\rho^p}\Big)^{\frac1{q+1-p}},
\end{equation}
such that the quantity in the right-hand side bracket is bounded by $2$. Hence
\begin{align}\nonumber
    2^{-n(1-q)_+}\essup_{-\tilde{s}_n<t<0}
    \int_{\widetilde{K}_n\times\{t\}}  &(u-k_{n+1})^{q+1}_+\,\dx + \iint_{\widetilde{Q}_n} |D(u-k_{n+1})_+|^p\,\dx\dt\\
    &\le 
    \frac{ \boldsymbol \gm 2^{ nr}}{(1-\sig)^p s }\frac{1}{k^{r-q-1}} 
    \iint_{Q_{n}} (u-k_n)_+^{r} \, \dx\dt.
    \label{Energy-new}
\end{align}

Now introduce $0\le\phi\le1$ to be a cutoff function which vanishes on the parabolic boundary of $\widetilde{Q}_n$
and equals the identity in $Q_{n+1}$. We recall the abbreviation $m=p\frac{N+q+1}{N}$ for the parabolic Sobolev  exponent. An application of the Sobolev embedding Lemma~\ref{lem:Sobolev} gives that
\begin{align*}
    & \iint_{Q_{n+1}} (u-k_{n+1})_+^{m}\,\dx\dt
    \le 
    \iint_{\widetilde{Q}_{n}}[(u-k_{n+1})_+\phi]^{m}\,\dx\dt\\
    &\qquad\le 
    \boldsymbol \gm 
    \iint_{\widetilde{Q}_n} |D[(u-k_{n+1})_+\phi]|^p\,\dx\dt \bigg[\essup_{-\tilde{s}_n<t<0}\int_{\widetilde{K}_n\times\{t\}}  [(u-k_{n+1})_+\phi]^{q+1}\,\dx
    \bigg]^{\frac pN} \\
    &\qquad=:
    \mathbf{I}\cdot\mathbf{II}^{\frac pN}.
\end{align*}
The integrals on the right-hand side of the last display are bounded by the energy estimate \eqref{Energy-new}. 
To deal with the  integral $\mathbf I$ on the right-hand side, we use the energy estimate \eqref{Energy-new} and inequality \eqref{Meas-int-p}. In this way we obtain
\begin{align*}
     \mathbf{I}
     &\le
    2^{p-1}\bigg[
     \iint_{\widetilde{Q}_n} |D(u-k_{n+1})_+|^p\phi^p\,\dx\dt 
     +
     \iint_{\widetilde{Q}_n} (u-k_{n+1})_+^p|D\phi|^p\,\dx\dt 
     \bigg]\\
     &\le
     \boldsymbol\gm \bigg[ 
    \frac{2^{ nr}}{(1-\sig)^p s }\frac{1}{k^{r-q-1}}
    \iint_{Q_{n}} (u-k_n)_+^{r} \, \dx\dt
    +
    \frac{2^{(n+2)p}}{(1-\sigma)^p\rho^p} \iint_{\widetilde Q_{n}} (u-\widetilde k_{n})_+^{p} \, \dx\dt\bigg]\\
     &\le
     \boldsymbol\gm \frac{2^{ nr}}{(1-\sig)^ps}\frac{1}{k^{r-q-1}} \Big[ 1+ \frac{s}{\rho^p}\frac{1}{k^{q+1-p}}\Big]
    \iint_{Q_{n}} (u-k_n)_+^{r} \, \dx\dt\\
    &\le 
    \boldsymbol\gm \frac{2^{ nr}}{(1-\sig)^ps}\frac{1}{k^{r-q-1}}
    \iint_{Q_{n}} (u-k_n)_+^{r} \, \dx\dt.
\end{align*}
The last estimate follows from \eqref{eq:k:1}.
Using again energy estimate \eqref{Energy-new} we have
\begin{align*}
    \mathbf{II}&\le\essup_{-\tilde{s}_n<t<0}
    \int_{\widetilde{K}_n\times\{t\}}  (u-k_{n+1})^{q+1}_+\,\dx\\
    &\le \boldsymbol \gm
    \frac{  2^{n(r+(1-q)_+)}}{(1-\sig)^p s }\frac{1}{k^{r-q-1}} 
    \iint_{Q_{n}} (u-k_n)_+^{r} \, \dx\dt.
\end{align*}
Inserting the above results, recalling the definition of $\lambda_{q+1}$, and finally taking means on both sides of the resulting inequality yields after simple
manipulations of the appearing terms, the inequality
\begin{align}\label{pre-recursiv}
        \biint_{Q_{n+1}} & (u-k_{n+1})_+^{m}\,\dx\dt\\\nonumber
        &\le\boldsymbol\gm
        \frac{ 2^{n[r\frac{N+p}{N}+\frac{p}{N}(1-q)_+]}}{(1-\sig)^\frac{p(N+p)}{N} k^{(r-q-1)\frac{N+p}{N}}}
        \frac{\rho^p}{s}
        \bigg[ \biint_{Q_{n}} (u-k_n)_+^{r} \, \dx\dt\bigg]^{\frac{N+p}{N}}.
\end{align}
The constant $\boldsymbol\gm$ depends only on $\{C_o,C_1,p,q\}$. Recall that this inequality holds whenever $r\ge q+1$. 
As mentioned earlier we now distinguish three cases: 
\[
    \left\{
    \begin{array}{c}
    q+1\le r\le m,\\[5pt]
    \mbox{$r\le m$ and $ r<q+1$,}\\[5pt]
    r>m.
    \end{array}
    \right.
\]
Note that the first two cases cover all exponents $r\le m$, while in the third case we additionally assume $u\in L^\infty_{\mathrm{loc}}(E_T)$ {\it a priori}. This is because when $r\le m$ the right-hand side of the energy estimate \eqref{pre-recursiv} is ensured to be finite, according the embedding in Lemma~\ref{lem:Sobolev}, while it is generally not true when $r>m$.  In what follows, we will deal with these cases one by one.

%

\subsection{The case $q+1\le r\le m$. }In this case, we estimate the left-hand side of~\eqref{pre-recursiv} from below by H\"older's inequality. 
Letting 
\begin{equation}\label{def:Yn}
   \boldsymbol{Y}_n:=\biint_{Q_n}(u-k_n)_+^{r}\,\dx\dt,
\end{equation}
we obtain
\begin{align*}
   \boldsymbol{Y}_{n+1}
   &\le
   \frac{1}{|Q_{n+1}|}\bigg[\iint_{Q_{n+1}}(u-k_{n+1})_+^m\,\dx\dt\bigg]^{\frac{r}{m}}|\widetilde A_n|^{1-\frac{r}{m}}\\
   &\le
   \boldsymbol\gm\bigg[\biint_{Q_{n+1}}(u-k_{n+1})_+^m\,\dx\dt\bigg]^{\frac{r}{m}}\bigg[\frac{2^{(n+2)r}}{k^r}\biint_{Q_{n}}(u-k_n)_+^r\bigg]^{1-\frac{r}{m}}
\end{align*}
In the last step we used \eqref{Eq:Meas-int}. 
We estimate the left-hand side by means of~\eqref{pre-recursiv}. Recalling the definition of $\lambda_r$ in~\eqref{def:lambda-r}, we obtain
\begin{equation*}
    \boldsymbol{Y}_{n+1}
    \le
    \frac{\boldsymbol\gm \boldsymbol{b}^n}{(1-\sig)^{\frac{pr(N+p)}{Nm}}}
    \Big(\frac{\rho^p}{s}\Big)^{\frac{r}{m}}
    k^{-\frac{r\lm_r}{Nm}}
    \boldsymbol{Y}_n^{1+\frac{rp}{Nm}},
\end{equation*}
where we abbreviated  
$\boldsymbol b:= 2^{\frac{pr}{Nm}(r+(1-q)_+)+r}$.
Note that $\boldsymbol{b}$ depends only on the data $\{q,r,N\}$, whereas $\gm$ depends on $\{C_o,C_1,p,q,r,N\}$. Hence, by the fast geometric convergence 
Lemma \ref{lem:fast-geom-conv},
there exists a positive constant $\boldsymbol \gm$ depending only on the data, such that
$\boldsymbol  Y_n\to0$ if we require that 
\[
    \boldsymbol{Y}_o=\biint_{Q_o}u^{r}\,\dx\dt\le \boldsymbol \gm^{-1} (1-\sig)^{p+N}  \Big(\frac{s}{\rho^p}\Big)^{\frac{N}{p}} k^{\frac{\lm_{r}}{p}}.
\]
This amounts to requiring
\begin{equation}\label{Eq:k:2}
    k\ge \frac{\boldsymbol \gm}{(1-\sig)^{\frac{p(p+N)}{\lm_{r}}}} \Big(\frac{\rho^p}{s}\Big)^{\frac{N}{\lm_{r}}}  \bigg[\biint_{Q_o}u^{r}\,\dx\dt\bigg]^{\frac{p}{\lm_{r}}}.
\end{equation}
Taking both conditions \eqref{eq:k:1} and  \eqref{Eq:k:2} on $k$ into account, the convergence $\boldsymbol  Y_n\to0$ gives us that
\begin{align}\label{sup-est-sigma}
    \essup_{Q_{\sig\rho,\sig s}}u
    \le \frac{\boldsymbol \gm}{(1-\sig)^{\frac{p(p+N)}{\lm_{r}}}} 
    \Big(\frac{\rho^p}{s}\Big)^{\frac{N}{\lm_{r}}}  
    \bigg[\biint_{Q_{\rho, s}}u^{r}\,\dx\dt\bigg]^{\frac{p}{\lm_{r}}} +\Big(\frac{s}{\rho^p}\Big)^{\frac1{q+1-p}},
\end{align}
where $\boldsymbol\gm$ depends only on the data $\{C_o,C_1,p,q,N\}$.
This yields the asserted estimate in the case $q+1\le  r\le m$ if we choose $\sigma=\frac12$. 

\subsection{The case $r\le m$ and $r<q+1$. }
Note that in the following we will only make use of $r<q+1$. In this case, we have $\lm_{q+1}>\lm_r>0$, which implies $q+1<m$.
Therefore, we are able to apply the results from the previous case for the exponent $q+1$ in place of $r$. This guarantees $u\in L^\infty_{\mathrm{loc}}(E_T)$, and~\eqref{sup-est-sigma} implies the local estimate
\begin{align*}
    \essup_{Q_{\sig\rho,\sig s}}u
    \le \frac{\boldsymbol \gm}{(1-\sig)^{\frac{p(p+N)}{\lm_{q+1}}}} 
    \Big(\frac{\rho^p}{s}\Big)^{\frac{N}{\lm_{q+1}}}  
    \bigg[\biint_{Q_{\rho, s}}u^{q+1}\,\dx\dt\bigg]^{\frac{p}{\lm_{q+1}}} +\Big(\frac{s}{\rho^p}\Big)^{\frac1{q+1-p}}
\end{align*}
for any $\sigma\in(0,1)$, provided $Q_{\rho,s}\Subset E_T$. 
We define
\[
    {\bf M}_\sig = \essup_{Q_{\sig\rho,\sig s}} u\quad\text{and}\quad {\bf M}_1 = \essup_{Q_{\rho, s}} u.
\]
From the above estimate we get
\[
{\bf M}_\sig \le {\bf M}_1^{1-\frac{\lm_r}{\lm_{q+1}}}\frac{\boldsymbol \gm }{(1-\sig)^{\frac{p(p+N)}{\lm_{q+1}}}} \Big(\frac{\rho^p}{s} \Big)^{\frac{N}{\lm_{q+1}}}  \bigg[\biint_{Q_{\rho, s}}u^{r}\,\dx\dt\bigg]^{\frac{p}{\lm_{q+1}}} +\Big(\frac{s}{\rho^p}\Big)^{\frac1{q+1-p}},
\]
where we used $r<q+1$ and the definition of $\lambda_r$ and $\lambda_{q+1}$ in~\eqref{def:lambda-r}. 
Departing from this, a standard interpolation argument  yields the desired estimate.
Indeed, by Young's inequality we have
\[
    {\bf M}_\sig 
    \le 
    \tfrac12 {\bf M}_1
    +
    \frac{\boldsymbol \gm }{(1-\sig)^{\frac{p(p+N)}{\lm_{r}}}} \Big(\frac{\rho^p}{s} \Big)^{\frac{N}{\lm_{r}}}  \bigg[\biint_{Q_{\rho, s}}u^{r}\,\dx\dt\bigg]^{\frac{p}{\lm_{r}}} +\Big(\frac{s}{\rho^p}\Big)^{\frac1{q+1-p}}.
\]
Apply this to cylinders $Q_{\sig_2\rho,\sig_2s}$ and $Q_{\sig_1\rho,\sig_1s}$, with $\frac12 \le\sig_1<\sig_2\le 1$, to get
\begin{align*}
    {\bf M}_{\sig_1} 
    &\le   
    \tfrac12 {\bf M}_{\sig_2}
    +
    \frac{\boldsymbol \gm }{(\sig_2-\sig_1)^{\frac{p(p+N)}{\lm_{r}}}} \Big(\frac{\rho^p}{s} \Big)^{\frac{N}{\lm_{r}}}  \bigg[\biint_{Q_{\sig_1\rho, \sig_1s}}u^{r}\,\dx\dt\bigg]^{\frac{p}{\lm_{r}}} +\boldsymbol \gm \Big(\frac{s}{\rho^p}\Big)^{\frac1{q+1-p}}.
\end{align*}
Arriving at this point, the iteration lemma \ref{lem:Giaq} can be applied to $[\frac12, 1]\ni\sig\mapsto {\bf M}_\sig$, yielding the desired sup-bound in the case $1\le r< q+1$.

\subsection{The case $r>m$.}
Note that $r>m$ implies that $r>q+1$ due to our assumption that $\lm_r>0$. For otherwise, we would have obtained $0< \lambda_r\le \lambda_{q+1} = N(m-q-1)$ which in turn would imply $q+1<m$ and thus $r>q+1$.
We stress that in the present case $u\in L^\infty_{\mathrm{loc}}(E_T)$ is assumed {\it a priori}.
The starting point is again the integral inequality~\eqref{pre-recursiv}, and we continue to use the abbreviation $\boldsymbol{Y}_n$ introduced in~\eqref{def:Yn}. Since $r>m$, estimate~\eqref{pre-recursiv} implies
\begin{align*}
    \boldsymbol{Y}_{n+1}&=\biint_{Q_{n+1}}(u-k_{n+1})_+^{r}\,\dx\dt
    \le
    \|u\|^{r-m}_{\infty,Q_o}\biint_{Q_{n+1}}(u-k_{n+1})_+^{m}\,\dx\dt\\
    &\le
    \boldsymbol\gm \|u\|_{\infty,Q_o}^{r-m} \frac{\rho^p}s
    \frac{\boldsymbol{b}^n}{(1-\sig)^{p\frac{N+p}{N}} k^{(r-q-1)\frac{N+p}N}} \,\boldsymbol{Y}_n^{1+\frac pN},
\end{align*}
where $\boldsymbol{b}=2^{r\frac{N+p}N+\frac pN(1-q)_+}$. By the fast geometric convergence 
Lemma \ref{lem:fast-geom-conv}, 
there exists
a positive constant $\boldsymbol \gm$ depending only on the data, such that
$\boldsymbol  Y_n\to0$ if we require that 
\[
    \boldsymbol{Y}_o
    =
    \biint_{Q_o}u^{r}\,\dx\dt
    \le 
    \boldsymbol \gm^{-1} \|u\|_{\infty,Q_o}^{-(r-m)\frac Np} (1-\sig)^{N+p}  \Big(\frac{s}{\rho^p}\Big)^{\frac{N}{p}} k^{\frac{(r-q-1)(N+p)}{p}}.
\]
Since $r>q+1$, this amounts to asking for
\begin{equation}\label{eq:k:3}
    k
    \ge 
    \frac{\boldsymbol \gm \|u\|_{\infty,Q_o}^{\frac{N(r-m)}{(N+p)(r-q-1)}} }{(1-\sig)^{\frac{p}{r-q-1}}} \Big(\frac{\rho^p}{s}\Big)^{\frac{N}{(N+p)(r-q-1)}} \bigg[\biint_{Q_o}u^{r}\,\dx\dt\bigg]^{\frac{p}{(N+p)(r-q-1)}}.
\end{equation}
Taking both conditions \eqref{eq:k:1} and  \eqref{eq:k:3} into account, the convergence $\boldsymbol  Y_n\to0$ gives us 
\begin{align*}
    &\essup_{Q_{\sig\rho,\sig s}}u\\
    & \quad\le
    \boldsymbol \gm 
    \frac{\|u\|_{\infty,Q_o}^{\frac{N(r-m)}{(N+p)(r-q-1)}}}{(1-\sig)^{\frac{p}{r-q-1}}} 
    \Big(\frac{\rho^p}{s}\Big)^{\frac{N}{(N+p)(r-q-1)}}  
    \bigg[\biint_{Q_{\rho, s}}\!\! u^{r}\,\dx\dt\bigg]^{\frac{p}{(N+p)(r-q-1)}}
    +
    \Big(\frac{s}{\rho^p}\Big)^{\frac1{q+1-p}},
\end{align*}
with a constant $\boldsymbol\gm$ depending on the data $\{C_o,C_1,p,q,r,N\}$. 
Noticing that
\[
    \frac{N(r-m)}{(N+p)(r-q-1)}=1-\frac{\lm_r}{(N+p)(r-q-1)}\in(0,1), 
\]
as a consequence of our assumptions $\lambda_r>0$, $r>m$, and $r>q+1$,
a standard interpolation argument using the iteration Lemma~\ref{lem:Giaq}  allows to conclude the claim.  This is similar to the argument from Case $q+1\le r\le m$.
Consequently, we have completed the proof of the theorem. 


\chapter[Expansion of positivity]{Expansion of positivity}\label{sec:Expansion}

The main components of the proof of Proposition~\ref{PROP:EXPANSION} will be given in the following Sections~\ref{S:DG} -- \ref{S:shrink}. Afterwards, we present the proof of Proposition~\ref{PROP:EXPANSION} in Section~\ref{S:proof-expansion}.

\section{A De Giorgi type lemma}\label{S:DG}
Starting from the energy estimate in Proposition~\ref{Prop:2:1} we can perform the De Giorgi iteration and obtain the following result, which actually holds for all $q>0$ and $p>1$.
\begin{lemma}\label{Lm:DG:1}
 Let $u$ be a non-negative weak super-solution to \eqref{Eq:1:1f} in $E_T$, under assumption \eqref{Eq:1:2}.
 Set $\theta=\dl M^{q+1-p}$ for some $\dl\in(0,1)$ and $M>0$, and assume $(x_o,t_o) + Q_\varrho(\theta) \subset E_T$.
There exists a constant $\nu\in(0,1)$ depending only on 
 the data and $\delta$, such that if
\begin{equation*}
	\Big|\big\{u\le M\big\}\cap (x_o,t_o)+Q_{\varrho}(\theta)\Big|
	\le
	\nu|Q_{\varrho}(\theta)|,
\end{equation*}
then 
\begin{equation*}
	u\ge\tfrac{1}2 M
	\quad
	\mbox{a.e.~in $(x_o,t_o)+Q_{\frac{1}2\varrho}(\theta)$.}
\end{equation*}
Moreover, we have $\nu=\nu_o \dl^{\frac{N}{p}}$ for some $\nu_o$ depending only on the data.
\end{lemma}

\begin{proof}
Assume $(x_o,t_o)=(0,0)$ for simplicity. 
By the technical Lemma~\ref{Lm:g},
 the energy estimate in Proposition~\ref{Prop:2:1} yields  
 \begin{align}
	&\essup_{-\theta\varrho^p<t<0}
	\int_{K_\varrho}\z^p(u+k)^{q-1}(u-k)_-^2\,\dx
	+
	\iint_{Q_\varrho(\theta)}\z^p|D(u-k)_-|^p\,\dx\dt\nonumber\\
	&\quad\le
	\bg\iint_{Q_\varrho(\theta)}(u-k)^{p}_-|D\z|^p\,\dx\dt
	+
	\bg \iint_{Q_\varrho(\theta)}(u+k)^{q-1} (u-k)_-^2|\pl_t\z|\,\dx\dt\nonumber\\ \label{Eq:energy}
	&\quad=:\mathbf I_1+\mathbf I_2,
\end{align}
for any $k>0$ and
for any non-negative piecewise smooth cut-off function $\zeta$ vanishing on the parabolic boundary of $Q_\varrho(\theta)$.
In order to use this energy estimate \eqref{Eq:energy}, we set
\begin{align*}
	\left\{
	\begin{array}{c}
\dsty k_n=\frac{M}2+\frac{M}{2^{n+1}}, \\[5pt] 
\dsty \rho_n=\frac{\rho}2+\frac{\rho}{2^{n+1}},\quad\widetilde{\rho}_n=\frac{\rho_n+\rho_{n+1}}2,\\[5pt]
\dsty K_n=K_{\rho_n}, \quad \widetilde{K}_n=K_{\widetilde{\rho}_n},\\[5pt]
\dsty Q_n=K_n\times(-\theta\rho_n^p,0],\quad\widetilde{Q}_n=\widetilde{K}_n\times(-\theta\widetilde{\rho}_n^p,0].
\end{array}
	\right.
\end{align*}
Introduce the test function $\z$ vanishing on the parabolic boundary of $Q_{n}$ and
equal to the identity in $\widetilde{Q}_{n}$, such that
\begin{equation*}
|D\z|\le \frac{2^{n+4}}{\rho}\quad\text{ and }\quad |\pl_t\z|\le \frac{2^{p(n+4)}}{\theta\rho^p}.
\end{equation*}
Hence the energy estimate \eqref{Eq:energy} is now read with $k$, $K_\rho$ and $Q_\rho(\theta)$ replaced by $k_n$, $K_n$ and $Q_n$.
In this setting, we treat the two terms on the right-hand side of the energy estimate \eqref{Eq:energy} as follows.
For the first term, we estimate
\[
    \mathbf I_1\le\bg \frac{2^{pn}}{\varrho^p}M^{p}|A_n|,
\]
where
\begin{equation*}
	A_n=\big\{u<k_n\big\}\cap Q_n.
\end{equation*}
For the second term, note that when $u<  k_n$ we have 
\begin{equation}\label{Eq:obs:1}
  \tfrac12 M\le k_n\le u+k_n\le 2M,
\end{equation}
and then estimate
\begin{align*}
    \mathbf I_2
    &\le
    \bg \frac{2^{pn}}{\theta\rho^p}
	\iint_{Q_n}(u+k_n)^{q-1} (u-k_n)^2_- \,\dx\dt \\
	&\le
    \bg \frac{2^{pn}}{\theta\rho^p} M^{q+1}|A_n|.
\end{align*}
 Whereas the first term on the left-hand side of \eqref{Eq:energy} is estimated analogously by \eqref{Eq:obs:1}.
Inserting these estimates back in the energy estimate \eqref{Eq:energy}, we find that 
\begin{align*}
	\frac{M^{q-1}}{2^{|q-1|}} \essup_{-\theta\widetilde{\varrho}_n^p<t<0}&
	\int_{\widetilde{K}_n} (u-k_n)_-^2\,\dx
	+
	\iint_{\widetilde{Q}_n}|D(u-k_n)_-|^p \,\dx\dt\\
	&\le
	\bg \frac{2^{pn}}{\varrho^p}M^{p}\left(1+\frac{M^{q+1-p}}{\theta}\right)|A_n|.
\end{align*}
Now setting $0\le\phi\le1$ to be a cutoff function which vanishes on the parabolic boundary of $\widetilde{Q}_n$, 
equals the identity in $Q_{n+1}$ and satisfies $|D\phi|\le 2^{n+4}/\rho$, an application of the H\"older inequality  and the Sobolev imbedding
Lemma~\ref{lem:Sobolev} gives that
\begin{align*}
	\frac{M}{2^{n+4}}
	|A_{n+1}|
	&\le 
	\iint_{\widetilde{Q}_n}\big(u-k_n\big)_-\phi\,\dx\dt\\
	&\le
	\bigg[\iint_{\widetilde{Q}_n}\big[\big(u-k_n\big)_-\phi\big]^{p\frac{N+2}{N}}
	\,\dx\dt\bigg]^{\frac{N}{p(N+2)}}|A_n|^{1-\frac{N}{p(N+2)}}\\
	&\le\bg
	\bigg[\iint_{\widetilde{Q}_n}\big|D\big[(u-k_n)_-\phi\big]\big|^p\,
	\dx\dt\bigg]^{\frac{N}{p(N+2)}}\\
	&\quad\ 
	\times\bigg[\essup_{-\theta\tilde{\varrho}_n^p<t<0}
	\int_{\widetilde{K}_n}\big(u-k_n\big)^{2}_-\,\dx\bigg]^{\frac{1}{N+2}}
	 |A_n|^{1-\frac{N}{p(N+2)}}\\
	&\le 
	\bg 
	\bigg[\frac{2^{pn}}{\varrho^p}M^p\left(1+\frac{M^{q+1-p}}{\theta}\right)\bigg]^{\frac{N+p}{p(N+2)}}
	\bigg(\frac{2^{|q-1|}}{M^{q-1}}\bigg)^{\frac{1}{N+2}}
	|A_n|^{1+\frac{1}{N+2}}.
\end{align*}
In the last line we used the above energy estimate.
In terms of $ \boldsymbol{Y}_n=|A_n|/|Q_n|$, this can be rewritten as
\begin{equation*}
\begin{aligned}
	 \boldsymbol{Y}_{n+1}
	\le
	&\bg \boldsymbol{b}^n  \left(1+\frac{M^{q+1-p}}{\theta}\right)^{\frac{N+p}{p(N+2)}}\left(\frac{\theta}{M^{q+1-p}}\right)^{\frac1{N+2}} \boldsymbol{Y}_n^{1+\frac{1}{N+2}},
\end{aligned}
\end{equation*}
for positive constants $\boldsymbol{b}=2^{\frac{2N+p+2}{N+2}}$ and $\bg$ depending only on the data. 
Hence, by the fast geometric convergence Lemma~\ref{lem:fast-geom-conv} and recalling $\theta=\dl M^{q+1-p}$, 
there exists
a positive constant $\nu$ of the form $\nu_o \dl^{\frac{N}{p}}$ for some $\nu_o$ depending only on the data, 
 such that
$Y_n\to0$ if we require that $Y_o\le \nu$.
\end{proof}


\section{Propagation of positivity in measure}\label{S:time-propag} 
The following result forwards the positivity in measure of non-negative super-solutions in the time direction; it holds for all $q>0$ and $p>1$.

\begin{lemma}\label{Lm:3:1}
 Let $u$ be a non-negative weak super-solution to \eqref{Eq:1:1f} in $E_T$, under assumption \eqref{Eq:1:2}.
Suppose for some constants $M>0$ and $\al\in(0,1)$, there holds
	\begin{equation*}
	\Big|\big\{
		u(\cdot, t_o) \ge M
		\big\}\cap K_{\varrho}(x_o)\Big|
		\ge\al |K_{\varrho}|;
	\end{equation*}
	then there exist constants $\dl$ and $\eps$ in $(0,1)$, depending only on the data and $\al$, such that 
	\begin{equation*}
	\Big|\big\{
	u(\cdot, t) \ge \eps M\big\} \cap K_{\varrho}(x_o)\Big|
	\ge\tfrac{1}2 \al|K_\varrho|
	\quad\mbox{ for all $t\in\big(t_o,t_o+\dl M^{q+1-p}\varrho^p\big]$,}
\end{equation*}
provided this cylinder is included in $E_T$. Moreover, $\varep=\varep_o\al^{1/q}$ and $\dl=\dl_o\al^{p+1}$ for some $\varep_o$ and $\dl_o$ depending only on the data.
\end{lemma}


\begin{proof} 
Assume $(x_o,t_o)=(0,0)$ for simplicity. 
We plan to use the energy estimate in Proposition~\ref{Prop:2:1}
in the cylinder $Q=K_{\varrho}\times(0,\dl M^{q+1-p} \varrho^p]$, with
$k=M$. To this end, we choose 
$\z(x,t)\equiv \z(x)$ to be  a standard non-negative cutoff function that is independent of time,  equals $1$ on $K_{(1-\sig)\varrho}$ for $\sigma\in(0,1)$ to be selected 
and vanishes on $\pl K_{\varrho}$ satisfying
$|D\z|\le(\sig\varrho)^{-1}$.
As a result, for all $0<t\le\dl M^{q+1-p} \varrho^p$, the energy estimate yields that
\begin{align}\label{start:en}
	\int_{K_\varrho\times\{t\}}\int_{u}^M &s^{q-1}(s-M)_-\,\ds \z^p\,\dx\nonumber\\
	&\le
	\int_{K_\varrho\times\{0\}}\int_{u}^M s^{q-1}(s-M)_-\,\ds\z^p\dx
	+
	\boldsymbol \gm\iint_{Q}(u-M)^{p}_-|D\z|^p\,\dx\dt.
\end{align}
Let us first treat the right-hand side of \eqref{start:en}. Indeed,
employing the measure theoretical information
at the initial time $t=0$, the first integral  is estimated by 
\begin{align*}
	\int_{K_\varrho\times\{0\}}\int_{u}^M s^{q-1}(s-M)_-\,\ds\z^p\,\dx
	\le 
	(1-\al)|K_\varrho|\int_{0}^{M} s^{q-1}(s-M)_-\,\ds.
\end{align*}
Whereas it is standard to estimate  by
\begin{equation*}
	\iint_{Q}(u-M)^{p}_-|D\z|^p\,\dx\dt
	\le
	\frac{M^p}{(\sigma\varrho )^p}|Q| 
	= \frac{\dl M^{q+1}}{\sig^p} |K_\varrho |.
\end{equation*}
As for the left-hand side of the energy estimate \eqref{start:en}, it can be estimated from below by extending the integral over a smaller set. Indeed, we have
\begin{align*}
	\int_{K_\varrho\times\{t\}}\int_{u}^M s^{q-1}(s-M)_-\,\ds \z^p\,\dx
	\ge 
	\big|A_{\eps M,(1-\sig)\varrho}(t)\big| \int_{\eps M}^M s^{q-1}(s-M)_-\,\ds,
\end{align*}
where we have defined
\begin{equation*}
	A_{\eps M,(1-\sig)\varrho}(t)
	=
	\big\{ u(\cdot,t)\le \eps M\big\} \cap K_{(1-\sig)\varrho}
\end{equation*}
with $\epsilon\in(0,\frac12)$ to be chosen later. 
By the mean value theorem, we can further estimate from below
\begin{align}\label{prop_1}
	\int_{\eps M}^{M} s^{q-1}(s-M)_-\,\ds
	&\ge
        \int_{\frac12 M}^{M} s^{q-1}(s-M)_-\,\ds\nonumber\\
	&\ge
        2^{-(q-1)_+}M^{q-1} \int_{\frac12 M}^{M} (s-M)_-\,\ds\nonumber\\
	&=  2^{-(q-1)_+-3} M^{q+1}.
\end{align}
Observe that
\begin{align*}
	\big|A_{\eps M,\varrho}(t)\big|
	&=
	\big|A_{\eps M,(1-\sig)\varrho}(t)\cup (A_{\eps M,\varrho}(t)\setminus A_{\eps M,(1-\sig)\varrho}(t))\big|\\
	&\le 
	\big|A_{\eps M,(1-\sig)\varrho}(t)\big|+|K_\varrho\setminus K_{(1-\sig)\varrho}|\\
	&\le
	\big |A_{\eps M,(1-\sig)\varrho}(t)\big|+N\sig |K_\varrho|.
\end{align*}
Collecting all the above estimates in \eqref{start:en} yields that
\begin{align*}
	|A_{\eps M,\varrho}(t)|
	&\le 
	\frac{\dsty\int_{0}^M s^{q-1}(s-M)_-\,\ds }{\dsty\int_{\eps M}^M s^{q-1}(s-M)_-\,\ds }
	(1-\al)|K_\varrho|
	+
	\frac{\boldsymbol\gm\dl}{\sig^p}|K_\varrho| +N\sig|K_\varrho|,
\end{align*}
for a constant $\boldsymbol\gamma =\boldsymbol\gamma (p,C_o,C_1)$.
Let us rewrite the fractional number in the preceding  inequality in the form
\[
	1+ \mathbf {I}_\eps \quad\text{ where }\quad
	\mathbf {I}_\eps =
	\frac{\dsty\int_{0}^{\eps M} s^{q-1}(s-M)_-\,\ds}{\dsty\int_{\eps M}^{M}  s^{q-1}(s-M)_-\,\ds}.
\]
To control $I_\varepsilon$, we first easily estimate
\[
	\int_{0}^{\eps M} s^{q-1}(s-M)_-\,\ds
	\le 
	M\int_{0}^{\eps M}s^{q-1}\,\ds
	=
	\tfrac1{q} M^{q+1}\eps^q.
\]
This, together with \eqref{prop_1}, gives that
\begin{align*}
	\mathbf {I}_\eps
	\le
	\boldsymbol\gm (q) \eps^q.
\end{align*}

At this stage,  the various parameters are ready to be chosen quantitatively. In fact, 
 $\eps\in (0,\frac12)$ is chosen to verify
\begin{equation*}
	(1-\al)(1+\boldsymbol \gm(q)\eps^q)\le 1-\tfrac34\al.
\end{equation*}
For example, this holds true if we choose $\eps\le(\frac{\alpha}{4\boldsymbol\gm(q)})^{1/q}$. 
Whereas $\sig$ is defined by $\sig:=\frac{\al}{8N}$.
Finally,  $\delta\in (0,1)$ is chosen to satisfy
$$
	\frac{\boldsymbol \gm\dl}{\sig^p}\le\tfrac1{8}{\al}.
$$
Note that this specifies the choice of $\dl=\dl_o\al^{p+1}$ for a constant $\dl_o$ depending on the data. With these choices we have  $|A_{\eps M,\varrho}(t)|\le
(1-\tfrac{\al}2)|K_\varrho|$ for $0<t\le\delta M^{q+1-p}\varrho^p$. This proves the asserted propagation of positivity.
\end{proof}

\section{A measure shrinking lemma}\label{S:shrink}
The following is the key lemma in the proof of Proposition~\ref{PROP:EXPANSION}. An important feature is that  a smallness estimate in measure is given for each slice of time.
\begin{lemma}\label{Lm:shrink}
Assume that $0<p-1\le q$. Let $\dl\in (0,1)$, $M>0$, and $\al\in (0,1)$. Then for any 
$\nu\in (0,1)$ there exists $\xi\in(0,1)$ depending only on the data, $\dl$, $\nu$ and $\al$, such that the following statement holds: Whenever $u$ is a non-negative, local, weak super-solution to \eqref{Eq:1:1f} with \eqref{Eq:1:2} in $E_T$ satisfying on some cylinder $K_{2\varrho}(x_o)\times \big(t_o, t_o+\dl M^{q+1-p}\varrho^p\big]\subset E_T$ the slice-wise measure condition
\begin{equation*}
	\Big|\big\{
	u(\cdot, t)\ge  M\big\} \cap K_{\varrho}(x_o)\Big|
	\ge\al |K_\varrho|
	\quad\mbox{ for all $t\in\big(t_o, t_o+\dl M^{q+1-p}\varrho^p\big]$,}
\end{equation*}
then we have
\[
\Big|\big\{
	u(\cdot, t) \le \xi M \big\} \cap K_{\varrho}(x_o)
\Big|
	\le\nu |K_\varrho|,
\]
for any time
\[
    t\in  \big(t_o+\tfrac12\dl M^{q+1-p}\varrho^p, t_o+\dl M^{q+1-p}\varrho^p\big].
\]

\end{lemma}


\subsection{An auxiliary lemma}

Some preparations for the proof of Lemma~\ref{Lm:shrink} are in order.
For simplicity, we assume $(x_o,t_o)=(0,0)$,
and for $\lm\in(0,1)$, we denote
\begin{equation}\label{Eq:Q-I}
\left\{
\begin{array}{c}
    I= \big(0,\dl M^{q+1-p}\varrho^p\big],\quad \lm I= \big((1-\lm)\dl M^{q+1-p}\varrho^p,\dl M^{q+1-p}\varrho^p\big],\\[5pt]
    Q=K_{2\rho}\times I,\quad \lm Q= K_{\lm 2\rho}\times \lm I,
\end{array}\right.
\end{equation} 
and for some $c\in(0,1)$ to be determined, introduce the sequence
\begin{equation}\label{Eq:c-k_n}
    k_n:= c^nM,\quad n\in\mathbb{N}_0,
\end{equation}
and the quantity
\begin{equation}\label{Eq:Y_n}
    \boldsymbol Y_n:=\sup_{t\in I}\bint_{K_{2\rho}\times\{t\} }\z^p\chi_{ \{ u<k_n \}}\,\dx.
\end{equation}
Here $\z(x,t)=\z_1(x)\z_2(t)$ is a piecewise smooth cutoff function in $Q$, satisfying
\begin{equation}\label{Eq:z}
\left\{
    \begin{array}{c}
    \mbox{$0\le\z\le 1$ in $Q$,}\\[5pt]
    \mbox{$\z=1$ in $\tfrac12 Q$ and $\z=0$ in $Q\setminus\tfrac34 Q$,}\\[5pt]
    \mbox{$|D\z_1|\le \frac{2}{\rho}$ and $0\le\pl_t\z_2\le \frac{4}{\dl M^{q+1-p}\varrho^p}$, and that}\\[5pt]
    \mbox{the sets $\big\{x\in K_{2\rho}\colon\> \z_1(x)>a \big\}$ are convex for all $a\in(0,1)$.}
\end{array}\right.
\end{equation}
The proof of Lemma~\ref{Lm:shrink} hinges upon the following result. 
\begin{lemma}\label{Lm:Aux}
Let the hypotheses of Lemma~\ref{Lm:shrink} hold.
For any $\nu\in(0,1)$, there exist $\sig,\, c\in(0,1)$ depending on the data, $\dl$,  $\al$ and $\nu$, such that
for any $n\in\mathbb{N}_0$, 
either
\[
    \boldsymbol Y_n\le \nu
\]
or
\[
    \boldsymbol Y_{n+1}\le\max\big\{\nu,\, \sig \boldsymbol Y_n\big\}.
\]
\end{lemma}

\noindent
The \textbf{Proof of Lemma \ref{Lm:Aux}} consists of several steps. Once Lemma \ref{Lm:Aux} is proven, we will use it to present the proof of Lemma~\ref{Lm:shrink} in the end, cf.~\S~\ref{S:Lm:shrink-proof}.


\subsubsection{Step 1: Integral inequalities}
We will first proceed with an {\bf additional regularity assumption} that
\begin{equation}\label{Eq:C1-time}
\pl_t u^q\in C\big(I; L^1(K_{2\rho})\big).
\end{equation}
Recall also the notation $I$ and $Q$ defined in \eqref{Eq:Q-I}.
The purpose of introducing the temporary assumption \eqref{Eq:C1-time} is only expository, and it will be removed later in Section~\ref{S:remove}.
According to Proposition~\ref{Prop:parab}, the function $u_k:=k-(k-u)_+$ with $k\in(0,M)$ is a non-negative, local, weak super-solution to \eqref{Eq:1:1f} -- \eqref{Eq:1:2} in $Q$, that is,
\begin{equation}  \label{Eq:parabolicity}
	\partial_t u_k^q -\dvg\bl{A}(x,t,u_k, Du_k) \ge 0\quad \mbox{ weakly in $Q$.}
\end{equation}
In the weak formulation of \eqref{Eq:parabolicity}, let us use the testing function
\begin{equation}\label{Eq:test-func-1}
\frac{\z^p}{[k-(k-u)_+ +c k]^{p-1}},
\end{equation}
where $\z$ is defined in \eqref{Eq:z} and the positive constants $c\in(0,1)$ and $k\in(0,M)$ are to be specified.
Using the structure conditions \eqref{Eq:1:2}, a standard calculation yields for a.e. $t\in I$,
\begin{align}\label{Eq:int-inequ}
    \pl_t&
    \int_{K_{2\rho}\times\{t\}} \z^p \Phi_k(u) \,\dx 
    + 
    (p-1)C_o\int_{K_{2\rho}\times\{t\}} \z^p |D\Psi_k(u)|^p\,\dx\nonumber\\
    &\le p C_1 \int_{K_{2\rho}\times\{t\}} |D\Psi_k(u) |^{p-1} \z^{p-1} |D\z|\,\dx 
    + 
    \int_{K_{2\rho}\times\{t\}} \Phi_k(u) \pl_t\z^p\,\dx,
\end{align}
 where we have set
\begin{align}
    \Phi_k(u)
    &:= \int_0^{(k-u)_+}\frac{q(k-s)^{q-1}}{(k-s+ck)^{p-1}}\,\d s, \label{Eq:Phi}\\
    \Psi_k(u)
    &:= \ln \bigg[\frac{k(1+c)}{k(1+c)-(k-u)_+}\bigg].\label{Eq:Psi}
\end{align}
Note that the term involving the time derivative exists due to the additional assumption \eqref{Eq:C1-time}. 
The first term on the right-hand side of \eqref{Eq:int-inequ} is treated by Young's inequality. We have
\begin{align*}
    p C_1 \int_{K_{2\rho}\times\{t\}} |D\Psi_k(u) |^{p-1} \z^{p-1} |D\z|\,\dx
    &\le
    \tfrac12(p-1)C_o\int_{K_{2\rho}\times\{t\}} |D\Psi_k(u)|^p\z^p\,\dx\\
    &\phantom{\le\,} 
    +\widetilde{\boldsymbol\gm}(p,C_o, C_1)\int_{K_{2\rho}}|D\z|^p\,\dx,
\end{align*}
where $\widetilde{\boldsymbol\gm}=(pC_1)^p[\tfrac12 (p-1) C_o]^{1-p}$. For the second term on the right-hand side of \eqref{Eq:int-inequ}, we first observe that
\[
    \Phi_k(u)
    \le 
    \int_0^{k}\frac{q(k-s)^{q-1}}{(k-s+ck)^{p-1}}\,\d s 
    = 
    k^{q+1-p}\int_0^1\frac{q s^{q-1}}{(s+c)^{p-1}}\,\ds
    \le \overline{\boldsymbol\gm} k^{q+1-p}\ln\Big(\frac{1+c}{c}\Big).
\]
Here, the last estimate holds for any $0< p-1\le q$; the constant $\overline{\boldsymbol\gm}$ depends only on $p$ and $q$, and is stable as $q+1-p\downarrow0$. Indeed, let us  compute
\begin{align}\label{Eq:aux-inequ}
\int_0^1\frac{q s^{q-1}}{(s+c)^{p-1}}\,\ds 
&= \int_0^c\frac{q s^{q-1}}{(s+c)^{p-1}}\,\ds + \int_c^1\frac{q s^{q-1}}{(s+c)^{p-1}}\,\ds\\\nonumber
&\le  c^{q+1-p}  + q 2^{(1-q)_+} \int_c^1 (s+c)^{q-p}\,\ds\\\nonumber
&\le c^{q+1-p}  + q 2^{q+1-p+(1-q)_+} \int_c^1 \frac{1}{s+c} \,\ds \\\nonumber
&= c^{q+1-p}  + q 2^{q+1-p+(1-q)_+} \ln\Big(\frac{1+c}{2c}\Big) \\\nonumber
&\le \overline{\boldsymbol\gm}  \ln\Big(\frac{1+c}{c}\Big)
\end{align}
with $\overline{\boldsymbol\gm}=(q+1)2^{q+1-p+(1-q)_+}$.
Note that in estimating the integral over $(c,1)$ in the first line, we used the fact that $s+c>s>\tfrac12(s+c)$.  
Using \eqref{Eq:aux-inequ}, $k< M$ as well as the properties of $\z$ from \eqref{Eq:z}, we obtain
\begin{align*}
    \int_{K_{2\rho}\times\{t\}} \Phi_k(u) \pl_t\z^p\,\dx
    &\le 
    4p \overline{\boldsymbol\gm} \frac{k^{q+1-p}}{\dl M^{q+1-p}\rho^p} \ln\Big(\frac{1+c}{c}\Big)|K_{2\rho}|
    \le \frac{ 4p \overline{\boldsymbol\gm}}{\dl\rho^p}\ln\Big(\frac{1+c}{c}\Big)|K_{2\rho}|.
\end{align*}
Collecting all these estimates in \eqref{Eq:int-inequ} and using $\delta\in (0,1)$ as well as  $\ln\big(\frac{1+c}{c}\big)\ge\ln 2$, we have 
\begin{align}\label{Eq:int-inequ-1}
    \pl_t\int_{K_{2\rho}\times\{t\}} \z^p \Phi_k(u)\,\dx 
    + 
    \int_{K_{2\rho}\times\{t\}} \z^p |D\Psi_k(u)|^p\,\dx
    \le 
    \frac{\boldsymbol\gm}{\dl \rho^p}\ln\Big(\frac{1+c}{c}\Big) |K_{2\rho}|.
\end{align}
for a.e. $t\in I$.
Here $\boldsymbol\gm$ takes into account the constants $\{\widetilde{\boldsymbol\gm}, \overline{\boldsymbol\gm} \}$ which appeared previously, and depends only on $\{p,q,C_o,C_1\}$.

From the measure theoretical condition assumed in Lemma~\ref{Lm:shrink} and $k<M$ infer that
\[
    \Big| \big\{\Psi_k(u) = 0\big\}\cap K_{\rho} \Big| \ge \al2^{-N} |K_{2\rho}|\quad\text{for all}\>t\in I.
\]
Consequently, by the Poincar\'e type inequality in Lemma~\ref{lem:Poincare} 
we deduce that
\[
    \int_{K_{2\rho}\times\{t\}} \z^p\Psi^p_k(u)\,\dx
    \le 
    \frac{\boldsymbol\gm_\ast \rho^p}{\al^p} \int_{K_{2\rho}\times\{t\}} \z^p|D\Psi_k(u)|^p \,\dx \quad\mbox{for a.e. $t\in I$,}
\]
where $\boldsymbol\gm_\ast>0$ stands for the Sobolev constant and depends only on $p$ and $N$. 
This joint with \eqref{Eq:int-inequ-1} yields that  for $t\in I$,
\begin{align}\label{Eq:int-inequ-2}
    \pl_t\int_{K_{2\rho}\times\{t\}}\z^p \Phi_k(u) \,\dx 
    +
    \frac{\al^p}{\boldsymbol\gm_\ast\rho^p} \int_{K_{2\rho}\times\{t\}} \z^p\Psi^p_k(u)\,\dx
    \le 
    \frac{\boldsymbol\gm}{\dl \rho^p}\ln\Big(\frac{1+c}{c}\Big) |K_{2\rho}|.
\end{align}
The arguments to be unfolded rest upon this integral inequality.
\subsubsection{Step 2: Two cases}
In \eqref{Eq:int-inequ-2} we take $k\equiv k_n = c^n M$ with $n\in\mathbb{N}_0$ as in \eqref{Eq:c-k_n}, for some $c\in(0,1)$ to be determined.
From the definition of $\boldsymbol Y_n$ in \eqref{Eq:Y_n}, it follows that for every $n\in\N\cup\{0\}$ and every $\varep\in(0,1)$ there is some $t_\eps\in I$, such that
\begin{equation}\label{Eq:Y-eps}
    \bint_{K_{2\rho}\times\{t_\eps\} }\z^p\chi_{ \{ u<k_{n+1} \}}\,\dx\ge \boldsymbol Y_{n+1} -\varep.
\end{equation}
Assuming $n$ is fixed, at the time level $t_\eps$ let us consider two cases:
\begin{equation}\label{Eq:alt}
\left\{
\begin{array}{c}
    \dsty \pl_t \int_{K_{2\rho}\times\{t_\eps\}} \z^p \Phi_{k_n}(u) \, \dx \ge 0,\\[6pt]
    \dsty \pl_t \int_{K_{2\rho}\times\{t_\eps\}} \z^p \Phi_{k_n}(u) \, \dx < 0.
\end{array}\right.
\end{equation}
In either case, we may always assume $\boldsymbol Y_n>\nu$, otherwise there is nothing to prove.
The proof of Lemma~\ref{Lm:Aux}, under the additional assumption \eqref{Eq:C1-time}, hinges upon showing that
$\boldsymbol Y_{n+1}\le\max\{\nu,\, \sig \boldsymbol Y_n\}$ in either case of \eqref{Eq:alt}.
\subsubsection{Step 3: The first case}
Let us consider the first case of \eqref{Eq:alt}. It follows from the integral inequality \eqref{Eq:int-inequ-2} at $t=t_\eps$ that
\begin{equation}\label{Eq:log-est}
    \frac{\al^p}{\boldsymbol\gm_\ast \rho^p} \int_{K_{2\rho}\times\{t_\eps\}} \z^p \Psi^p_{k_n}(u)\,\dx
    \le 
    \frac{\boldsymbol\gm}{\dl \rho^p}\ln\Big(\frac{1+c}{c}\Big) |K_{2\rho}|.
\end{equation}
The left-hand side integral is estimated over the smaller set $\{u(\cdot, t_\eps)<k_{n+1}\}\cap K_{2\rho}$. Indeed, by the definition \eqref{Eq:Psi} of $\Psi_{k_n}$ we estimate
\[
    \int_{K_{2\rho}\times\{t_\eps\}} \z^p\Psi^p_{k_n}(u)\,\dx
    \ge \Big[\ln\Big(\frac{1+c}{2c}\Big)\Big]^p\int_{K_{2\rho}\times\{t_\eps\} }\z^p\chi_{ \{ u<k_{n+1} \}}\,\dx.
\]
Putting this in \eqref{Eq:log-est} and using \eqref{Eq:Y-eps}, we obtain
\[
    \boldsymbol Y_{n+1}\le\varep+ \frac{\boldsymbol\gm}{\al^p\dl} \Big[\ln\Big(\frac{1+c}{2c}\Big)\Big]^{1-p},
\]
provided we choose $c\le\frac13$, which implies $\ln(\frac{1+c}{c})\le 2 \ln(\frac{1+c}{2c})$. 
If we restrict $\varep\in(0,\tfrac12\nu)$ and choose $c$ small  enough to satisfy
\begin{equation}\label{Eq:choice-c}
    \frac{\boldsymbol\gm}{\al^p\dl} \Big[\ln\Big(\frac{1+c}{2c}\Big)\Big]^{1-p}  \le
    \tfrac12\nu,\quad\text{i.e.}\quad c\le \tfrac12\exp\Big\{-\Big(\frac{2\boldsymbol\gm}{\al^p\dl\nu}\Big)^{\frac1{p-1}}\Big\},
\end{equation}
it follows that $\boldsymbol Y_{n+1}\le\nu$.
We note that this choice clearly implies the previous assumption $c\le\frac13$ if we enlarge the constant $\boldsymbol\gm$ if necessary. 

\subsubsection{Step 4: The second case}
Let us deal with the second case of \eqref{Eq:alt}. This case is much more involved.
To this end, we introduce
\[
t_*:=\sup\bigg\{t\in(0, t_\eps ):\>   \pl_t \int_{K_{2\rho}\times\{t\}} \z^p \Phi_{k_n}(u) \, \dx \ge 0\bigg\}.
\]
Such a set is non-empty and $t_*$ is well-defined, since by the definition \eqref{Eq:z} of $\z$, we have that $\z(\cdot, t)=0$ for $t\in I\setminus \tfrac34 I$.
By the definition of $t_*$ and the condition \eqref{Eq:alt}$_2$ on $t_\varep$, we infer that $t_*<t_\varep$ and 
\begin{equation}\label{Eq:t-t}
    \bint_{K_{2\rho}\times\{t_\eps\}} \z^p \Phi_{k_n}(u) \, \dx 
    \le 
    \bint_{K_{2\rho}\times\{t_*\}} \z^p \Phi_{k_n}(u) \, \dx.
\end{equation}

In what follows, we manage to estimate the two sides of \eqref{Eq:t-t}. Let us first treat the right-hand side. Some preparatory estimates are in order. Indeed, by the definition of $t_*$, just like in the first case, the integral inequality \eqref{Eq:int-inequ-2} at $t=t_*$ yields that
\begin{equation*}
    \frac{\al^p}{\boldsymbol\gm_\ast \rho^p} \int_{K_{2\rho}\times\{t_*\}} \z^p\Psi^p_{k_n}(u)\,\dx
\le \frac{\boldsymbol\gm}{\dl \rho^p} \ln\Big(\frac{1+c}{c}\Big) |K_{2\rho}|.
\end{equation*}
From this we recall the definition \eqref{Eq:Psi} of $\Psi_{k_n}$ and estimate the left side integral on the smaller set $\{(k_n-u)_+>s k_n\}$ for $s\in(0,1]$, which gives 
\begin{align*}
    \int_{K_{2\rho}\times\{t_*\}} \z^p\chi_{\{(k_n-u)_+>s k_n\}}\,\dx
    &\le 
    \frac{\boldsymbol\gm_o}{\al^p\dl }    \ln \Big(\frac{1+c}{c}\Big)  \Big[\ln\Big(\frac{1+c}{1+c-s}\Big)\Big]^{-p}  |K_{2\rho}|, 
\end{align*}
where $\boldsymbol\gm_o:=\boldsymbol\gm_*\boldsymbol\gm$ with the constant  $\boldsymbol \gm$ from the last inequality.

By the definition of $\boldsymbol Y_n$, we have for all $s\in(0,1]$ that
\begin{align}\label{Eq:Y-s}
    \bint_{K_{2\rho}\times\{t_*\}}& \z^p\chi_{\{(k_n-u)_+>s k_n\}}\,\dx \nonumber\\
    &\le 
    \min\bigg\{\boldsymbol Y_n; \frac{\boldsymbol\gm_o}{\al^p\dl }   \ln \Big(\frac{1+c}{c}\Big)  \Big[\ln\Big(\frac{1+c}{1+c-s}\Big)\Big]^{-p} \bigg\}\nonumber\\
    &=\left\{
    \begin{array}{cl}
    \boldsymbol Y_n, & \text{for $s\in(0,s_*]$,}\\[7pt]
    \dsty \frac{\boldsymbol\gm_o}{\al^p\dl }    \ln \Big(\frac{1+c}{c}\Big) \Big[\ln\Big(\frac{1+c}{1+c-s}\Big)\Big]^{-p}, & \text{for $ s\in[s_*,1]$,}
\end{array}
\right.
\end{align}
where $s_*$ is the root of the algebraic equation
\[
    \boldsymbol Y_n=\frac{\boldsymbol\gm_o}{\al^p\dl }   \ln \Big(\frac{1+c}{c}\Big)  \Big[\ln\Big(\frac{1+c}{1+c-s}\Big)\Big]^{-p}.
\]
There is always a root, since we assumed $\boldsymbol Y_n>\nu$ below \eqref{Eq:alt} and impose a smallness condition on $c$, that is,
\[
    \nu>\frac{\boldsymbol\gm_o}{\al^p\dl }  
    \Big[\ln\Big(\frac{1+c}{c}\Big)\Big]^{1-p},\quad\text{or  }\quad
    c<\exp\Big\{-\Big(\frac{\boldsymbol\gm_o}{\al^p\dl\nu}\Big)^{\frac1{p-1}}\Big\}, 
\]
which is similar to the  requirement in \eqref{Eq:choice-c}.
From the above algebraic equation we conclude that
\begin{equation}\label{Eq:s-star}
    s_*
    <
    \sig_1(1+c)\quad\text{where}\>\sig_1=1-\exp\Big\{-\Big[\frac{\boldsymbol\gm_o}{\al^p\dl\nu} 
    \ln\Big(\frac{1+c}{c}\Big) \Big]^{\frac1{p}}\Big\}.
\end{equation}
In fact, $\sig_1(1+c)$ is the solution of the algebraic equation above where $\boldsymbol Y_n$ is replaced by $\nu$.
This has the consequence that 
\[
    \nu\ge \frac{\boldsymbol\gm_o}{\al^p\dl }  
    \ln \Big(\frac{1+c}{c}\Big)  \Big[\ln\Big(\frac{1+c}{1+c-s}\Big)\Big]^{-p}
    \qquad \mbox{for $s\ge\sig_1(1+c)$}.
\]

Equipped with these preparatory estimates, we are now able to estimate the right-hand side of \eqref{Eq:t-t}. In fact, by Fubini's theorem,
\begin{align*}
\int_{K_{2\rho}\times\{t_*\}}& \z^p \Phi_{k_n}(u) \, \dx\\
    & =
    \int_{K_{2\rho}\times\{t_*\}} \z^p \bigg[\int_0^{k_n}\frac{q(k_n-s)^{q-1}\chi_{\{s<(k_n-u)_+\}}}{[k_n(1+c)-s]^{p-1}}\,\d s\bigg]\dx\\
    & =
    \int_0^{k_n}\frac{q(k_n-s)^{q-1}}{[k_n(1+c)-s]^{p-1}}\bigg[\int_{K_{2\rho}\times\{t_*\}} \z^p\chi_{\{s<(k_n-u)_+\}}\,\dx \bigg]\ds\\
    & =
    k_n^{q+1-p}\int_0^1\frac{q(1-s)^{q-1}}{(1+c-s)^{p-1}}\bigg[\int_{K_{2\rho}\times\{t_*\}} \z^p\chi_{\{sk_n<(k_n-u)_+\}}\,\dx \bigg]\ds\\
    & =
    k_n^{q+1-p}\int_0^{s_*}\frac{q(1-s)^{q-1}}{(1+c-s)^{p-1}}\bigg[\int_{K_{2\rho}\times\{t_*\}} \z^p\chi_{\{sk_n<(k_n-u)_+\}}\,\dx \bigg]\ds\\
    & \phantom{=\,}
    +k_n^{q+1-p}\int_{s_*}^1\frac{q(1-s)^{q-1}}{(1+c-s)^{p-1}}\bigg[\int_{K_{2\rho}\times\{t_*\}} \z^p\chi_{\{sk_n<(k_n-u)_+\}}\,\dx\bigg]\ds.
\end{align*}
Let us continue to estimate the last two terms by \eqref{Eq:Y-s}. As a result, we have
\begin{align*}\label{Eq:t-t-right}
    &\bint_{K_{2\rho}\times\{t_*\}} \z^p \Phi_{k_n}(u) \, \dx \nonumber\\
    &\qquad\le 
    k_n^{q+1-p} \int_0^{s_*}\frac{q(1-s)^{q-1}}{(1+c-s)^{p-1}} \boldsymbol Y_n\,\ds \nonumber\\
    &\qquad\phantom{\le\,}
    +  
    k_n^{q+1-p}\frac{\boldsymbol\gm_o}{\al^p\dl }  \ln \Big(\frac{1+c}{c}\Big)  
    \int_{s_*}^1  \frac{q(1-s)^{q-1}}{(1+c-s)^{p-1}} \Big[\ln\Big(\frac{1+c}{1+c-s}\Big)\Big]^{-p}\,\ds \nonumber\\
    &\qquad\le 
    k_n^{q+1-p} \int_0^{1}\frac{q(1-s)^{q-1}}{(1+c-s)^{p-1}} \boldsymbol Y_n\,\ds \nonumber\\
    &\qquad\phantom{\le\,} -  
    k_n^{q+1-p} \int_{s_*}^1  \frac{q(1-s)^{q-1}}{(1+c-s)^{p-1}} 
    \bigg[
    \boldsymbol Y_n- \frac{\boldsymbol\gm_o}{\al^p\dl } 
    \ln \Big(\frac{1+c}{c}\Big) \Big[\ln\Big(\frac{1+c}{1+c-s}\Big)\Big]^{-p}\bigg]\, \ds.\nonumber
 \end{align*}   
Setting
\begin{align*}
    \boldsymbol{F}(\boldsymbol Y_n,c) 
    &=
    \int_0^{1}\frac{q(1-s)^{q-1}}{(1+c-s)^{p-1}}\,\ds\\
    &\phantom{=\,}
    -  
    \int_{s_*}^1  \frac{q(1-s)^{q-1}}{(1+c-s)^{p-1}} 
     \bigg[1- \frac{\boldsymbol\gm_o}{\al^p\dl\boldsymbol Y_n } 
    \ln \Big(\frac{1+c}{c}\Big) \Big[\ln\Big(\frac{1+c}{1+c-s}\Big)\Big]^{-p}\bigg]\,\ds
\end{align*} 
we get
\begin{equation}\label{Eq:t-t-right}
    \bint_{K_{2\rho}\times\{t_*\}} \z^p \Phi_{k_n}(u) \, \dx
    \le
    k_n^{q+1-p} \boldsymbol Y_n\,  \boldsymbol{F}(\boldsymbol Y_n,c).
\end{equation}
In order to estimate $ \boldsymbol{F}(\boldsymbol Y_n,c) $ from above, we evoke \eqref{Eq:s-star} and $\boldsymbol Y_n>\nu$, and obtain
\begin{align*}
    \boldsymbol{F}(\boldsymbol Y_n,c) 
    &\le 
    \int_0^{1}\frac{q(1-s)^{q-1}}{(1+c-s)^{p-1}}\,\ds \nonumber\\
    &\phantom{\le\,}
    -  
    \int_{\sig_1(1+c)}^1  \frac{q(1-s)^{q-1}}{(1+c-s)^{p-1}} \bigg[ 1 - \frac{\boldsymbol\gm_o}{\al^p\dl \nu } 
    \ln \Big(\frac{1+c}{c}\Big) 
    \Big[\ln\Big(\frac{1+c}{1+c-s}\Big)\Big]^{-p}\bigg]\,\ds \nonumber\\
    &\le 
    \int_0^{1}\frac{q(1-s)^{q-1}}{(1+c-s)^{p-1}}\,\ds \nonumber\\
    &\phantom{\le\,}
    -  
    \int_{\sig_o(1+c)}^1  \frac{q(1-s)^{q-1}}{(1+c-s)^{p-1}} \bigg[ 1 - \frac{\boldsymbol\gm_o}{\al^p\dl \nu } 
    \ln \Big(\frac{1+c}{c}\Big) 
    \Big[\ln\Big(\frac{1+c}{1+c-s}\Big)\Big]^{-p}\bigg]\,\ds ,
\end{align*}
where 
\begin{equation}\label{Eq:sig_o}
    \sig_o:=1-\exp\Big\{-\Big[\frac{2\boldsymbol\gm_o}{\al^p\dl\nu}\ln  \Big(\frac{1+c}{c}\Big) \Big]^{\frac1{p}}\Big\}
    >
    \sigma_1.
\end{equation}
Note that $(1+c)\sig_o$ is the solution of the algebraic equation
\[
    \tfrac12 \nu=
    \frac{\boldsymbol\gm_o}{\al^p\dl }  
    \ln \Big(\frac{1+c}{c}\Big)  \Big[\ln\Big(\frac{1+c}{1+c-s}\Big)\Big]^{-p}.
\]
Since the right-hand side for $s> (1+c)\sig_o$ is less than $\tfrac12\nu$, we conclude
\begin{align}\label{Eq:F_n}
    \boldsymbol{F}(\boldsymbol Y_n,c) 
    \le 
    \int_0^{1}\frac{q(1-s)^{q-1}}{(1+c-s)^{p-1}}\,\ds 
    -  
    \frac12 \int_{\sig_o(1+c)}^1  \frac{q(1-s)^{q-1}}{(1+c-s)^{p-1}} .
\end{align}
On the other hand, we estimate the left-hand side of \eqref{Eq:t-t} from below by shrinking the domain of integration to the smaller set $\{u<k_{n+1}\}$:
\begin{align}\label{Eq:t-t-left}
\bint_{K_{2\rho}\times\{t_\varep\}} \z^p \Phi_{k_n}(u) \, \dx &\ge k_n^{q+1-p} \bint_{K_{2\rho}\times\{t_\varep\}} \z^p\chi_{\{u<k_{n+1}\}}\,\dx \int_0^{1-c}\frac{q(1-s)^{q-1}}{(1+c-s)^{p-1}}\,\ds \nonumber\\
&\ge k_n^{q+1-p}(\boldsymbol Y_{n+1} -\varep ) \int_0^{1-c}\frac{q(1-s)^{q-1}}{(1+c-s)^{p-1}}\,\ds.
\end{align}
Here we also used the particular choice of $t_\eps$ from 
\eqref{Eq:Y-eps}.
Combining \eqref{Eq:t-t-right} and \eqref{Eq:t-t-left} in \eqref{Eq:t-t} and using \eqref{Eq:F_n}, we obtain
\[
    \boldsymbol Y_{n+1}-\varep \le \boldsymbol Y_n \big(1-f(c)\big)
\]
where $f(c)\in(0,1)$ is defined by
\begin{align}\label{Eq:f}
    f(c)
    &=
    \frac{\displaystyle\frac12 \int_{\sig_o(1+c)}^1  \frac{q(1-s)^{q-1}}{(1+c-s)^{p-1}} \,\ds -
    \int_{1-c}^1\frac{q(1-s)^{q-1}}{(1+c-s)^{p-1}}\,\ds }{\displaystyle\int_0^{1-c}
    \frac{q(1-s)^{q-1}}{(1+c-s)^{p-1}}\, \ds}\nonumber\\
    &=
    \frac12
    \frac{\displaystyle\int_{\sig_o(1+c)}^{1-c}  \frac{q(1-s)^{q-1}}{(1+c-s)^{p-1}} \,\ds -
    \int_{1-c}^1\frac{q(1-s)^{q-1}}{(1+c-s)^{p-1}}\,\ds}{\displaystyle\int_0^{1-c}
    \frac{q(1-s)^{q-1}}{(1+c-s)^{p-1}}\, \ds}\,.
\end{align}
The desired estimate follows by choosing $\sig=1-f(c)$ where $c$ has been chosen in \eqref{Eq:choice-c}. However, in order to trace the dependence, we further estimate $f(c)$ from below. In particular, as we shall see in a moment, $f(c)\ge\boldsymbol{\gm} c^{q+1-p}$ for some $\boldsymbol{\gm}>0$ that depends only on  $\{p,q\}$ and is stable as $q+1-p\downarrow0$. 

Consequently, stipulating $c<\frac12(1-\sig_o)$, and taking into account the lower bound of the interval for $s$, the integral is estimated from below by
\begin{align*}
    \tfrac12\int_{\sig_o(1+c)}^1  \frac{q(1-s)^{q-1}}{(1+c-s)^{p-1}}  \,\ds 
    &\ge 
    \tfrac12\int_{\sig_o(1+c)}^{1-c}  \frac{q(1-s)^{q-1}}{(1+c-s)^{p-1}}  \,\ds\\
    &\ge \frac{q }{2^{1+(q-1)_+}} \int_{\sig_o(1+c)}^{1-c} (1+c-s)^{q-p}  \,\ds\\
    &\ge
    \frac{q (2c)^{q-1-p}}{2^{1+(q-1)_+}}
    \int_{\sig_o(1+c)}^{1-c} \frac{1}{1+c-s}\,\ds\\
    &=
    \widehat{\boldsymbol\gm} c^{q+1-p} 
    \ln\Big[\frac{(1-\sig_o)(1+c)}{2c}\Big]
\end{align*}
for
$ \widehat{\boldsymbol\gm}=q2^{q-p-(q-1)_+}$.
Here, to obtain the first inequality we needed $1-c>\sig_o(1+c)$ and accordingly used the requirement $c<\frac12(1-\sig_o)$, which will be guaranteed by restricting $c$ shortly; to get the second inequality we used $1+c-s\ge 1-s\ge\frac12(1+c-s)$. 

Next, by similar calculations as in \eqref{Eq:aux-inequ}, 
\begin{align*}
    \int_{1-c}^1\frac{q(1-s)^{q-1}}{(1+c-s)^{p-1}}\,\ds 
    &= 
    \int_0^c\frac{q s^{q-1}}{(s+c)^{p-1}}\,\ds \le c^{q+1-p},
    \\
    \int_0^{1-c}\frac{q(1-s)^{q-1}}{(1+c-s)^{p-1}}\, \ds
    &= 
    \int_c^1\frac{q s^{q-1}}{(s+c)^{p-1}}\,\ds 
    \le \overline{\boldsymbol\gm} \ln\Big(\frac{1+c}{c} \Big), 
\end{align*}
where $\overline{\boldsymbol\gm}=(q+1)2^{q+1-p+(1-q)_+}$. 

Combining these estimates in \eqref{Eq:f}, we obtain
\begin{equation}\label{Eq:f(c)}
    \frac{f(c)}{c^{q+1-p}}\ge 
    \frac{\widehat{\boldsymbol\gm}}{\overline{\boldsymbol\gm}} 
    \bigg[ 
    1 - \frac{ \ln\big(\frac2{1-\sig_o} \big)}{ \ln\big(\frac{1+c}{c} \big)} 
    \bigg]
    - 
    \frac1{\overline{\boldsymbol\gm }\ln\big(\frac{1+c}{c} \big)}.
\end{equation}
From the above display we further require $c$ to satisfy
\[
    \frac{ \ln\big(\frac2{1-\sig_o} \big)}{ \ln\big(\frac{1+c}{c} \big)} \le\tfrac12
    \quad\text{and}\quad  
    \frac1{ \ln\big(\frac{1+c}{c} \big)}\le\tfrac{1}{4}\widehat{\boldsymbol\gm},
\]
which is implied if we require 
\[
    c
    \le 
    \min\Big\{\exp\Big\{-\frac{4}{\widehat{\boldsymbol\gm}} \Big\}, \, 
    \Big[\frac{1-\sig_o}{2}\Big]                   ^2\Big\}.
\]
Taking also into account the previous requirement  \eqref{Eq:choice-c}  of $c$ and the value of $\sig_o$ in \eqref{Eq:sig_o}, the final choice of $c$ is made by
\[
c:=\exp\Big\{-\Big(\frac{\boldsymbol\gm}{\al^p\dl\nu}\Big)^{\frac1{p-1}}\Big\}
\]
for some proper $\boldsymbol\gm$ depending only on the data. Indeed, by \eqref{Eq:sig_o} the requirement $c\le \tfrac14 (1-\sig_o)^2$ is equivalent to
\[
    4c\le  \exp\Big\{-2\Big[\frac{2\boldsymbol\gm_o}{\al^p\dl\nu}\ln  \Big(\frac{1+c}{c}\Big) \Big]^{\frac1{p}}\Big\}
    \quad
    \Longleftrightarrow
    \quad
    \frac{\big[\ln\big( \frac1{4c}\big)\big]^p}{ \ln\big(\frac{1+c}{c}\big)}\ge \frac{2^{p+1}\boldsymbol\gm_o}{\al^p\dl\nu}. 
\]
This, however,  can be achieved by an appropriate choice of $\boldsymbol\gamma$ in the definition of $c$ as above. With this choice of $c$, we can estimate from \eqref{Eq:f(c)} that
\[
f(c)\ge \frac{\widehat{\boldsymbol\gm}}{4\overline{\boldsymbol\gm}} c^{q+1-p}.
\]
Hence the choice of $\sig$ is made to be
\[
\sig=1- \frac{\widehat{\boldsymbol\gm}}{4\overline{\boldsymbol\gm}} c^{q+1-p}.
\]
This proves Lemma~\ref{Lm:Aux} under the additional assumption \eqref{Eq:C1-time}.
\subsubsection{Step 5: Removing the regularity assumption \eqref{Eq:C1-time}}\label{S:remove}
When the technical regularity assumption \eqref{Eq:C1-time} is not present, the first difference appears in the use of the testing function \eqref{Eq:test-func-1}. Indeed, in contrast to \eqref{Eq:test-func-1}, now an admissible one is
\[
\vp_h=\frac{\z^p \psi_{\varep}}{[k-(k-\llbracket u \rrbracket_{\bar{h}})_+ +c k]^{p-1}}.
\]
Here, the mollified function $\llbracket u \rrbracket_{\bar{h}}$ is defined in \eqref{def:mol}, while $c$, $k$ and $\z$ are just like in \eqref{Eq:test-func-1}. Moreover, for $(t_1,t_2)\subset I$ and for $2\varep<t_2-t_1$ the function $t\mapsto\psi_\varep(t)$ is a Lipschitz approximation of $\chi_{(t_1,t_2)}$, i.e., it takes $1$ in $(t_1+\varep, t_2-\varep)$, vanishes outside $(t_1,t_2)$ and is linearly interpolated otherwise.
Plug this testing function in the weak formulation of \eqref{Eq:parabolicity} to obtain
\begin{equation}\label{Eq:weak-form:0}
\iint_{E_T} \big[ -u_k^q\pl_t \vp_h + \bl{A}(x,t,u_k, Du_k) D\vp_h\big]\dx\dt\ge0.
\end{equation}

Let us first deal with the time part in \eqref{Eq:weak-form:0}. Indeed, setting $\llbracket u \rrbracket_{\bar{h}, k}:=k-(k-\llbracket u \rrbracket_{\bar{h}})_+$, we observe that
\begin{align*}
    -\iint_{E_T} u_k^q\pl_t \vp_h\,\dx\dt 
    &= 
    \iint_{E_T} (-u_k^q + \llbracket u \rrbracket^q_{\bar{h}, k} - \llbracket u \rrbracket^q_{\bar{h},k} )\pl_t \vp_h\,\dx\dt\\
    &= 
    \iint_{E_T}   \frac{(p-1)(u_k^q - \llbracket u \rrbracket^q_{\bar{h},k} )\pl_t \llbracket u \rrbracket_{\bar{h}, k}}{(\llbracket u \rrbracket_{\bar{h}, k} +c k)^{p}}\,\dx\dt \\
    &\phantom{=\,}
    -\iint_{E_T}   \frac{(u_k^q - \llbracket u \rrbracket^q_{\bar{h},k} )\pl_t (\z^p\psi_{\varep}) }{(\llbracket u \rrbracket_{\bar{h}, k} +c k)^{p-1}}\,\dx\dt\\
    &\phantom{=\,} 
    - \iint_{E_T} \llbracket u \rrbracket^q_{\bar{h},k} \pl_t \vp_h\,\dx\dt.
\end{align*}
Let us treat the terms on the right-hand side.
The first term  is non-positive due to Lemma~\ref{Lm:mol} (iii) and the monotonicity of the map $u\mapsto k-(k-u)_+$. Indeed, we have
\begin{equation*}
    \pl_t \llbracket u \rrbracket_{\bar{h}, k}= \chi_{\{\llbracket u \rrbracket_{\bar{h}}<k\}}\partial_t \llbracket u \rrbracket_{\bar{h}} = \tfrac1h \chi_{\{\llbracket u \rrbracket_{\bar{h}}<k\}} \big(\llbracket u \rrbracket_{\bar{h}}-u\big),
\end{equation*}
and therefore
\begin{align*}
    (u_k^q &- \llbracket u \rrbracket^q_{\bar{h},k} )
    \pl_t \llbracket u \rrbracket_{\bar{h}, k} \\
    &= 
    \tfrac1h \chi_{\{\llbracket u \rrbracket_{\bar{h}}<k\}} 
    \Big(\big(k-(k-u)_+\big)^q - \big(k-(k-\llbracket u \rrbracket_{\bar{h}})_+\big)^q\Big) \big(\llbracket u \rrbracket_{\bar{h}}-u\big)\le 0.
\end{align*}
The second term vanishes in the limit $h\downarrow 0$ due to Lemma~\ref{Lm:mol} (i). The last term is estimated by 
\begin{align*}
    - \iint_{E_T} \llbracket u \rrbracket^q_{\bar{h},k} \pl_t \vp_h\,\dx\dt 
    &=  
    \iint_{E_T}  \pl_t  \llbracket u \rrbracket^q_{\bar{h},k}\vp_h\,\dx\dt\\
    & =
    \iint_{E_T} \z^p\psi_{\varep} \pl_t\bigg(\int_k^{\llbracket u \rrbracket_{\bar{h},k}}\frac{q s^{q-1}}{(s+ck)^{p-1}}\,\d s\bigg)\dx\dt\\
    & =
    -\iint_{E_T} \z^p \pl_t \psi_{\varep}\bigg(\int_k^{\llbracket u \rrbracket_{\bar{h},k}}\frac{q s^{q-1}}{(s+ck)^{p-1}}\,\d s\bigg)\dx\dt\\
    & \phantom{=\ }
    -\iint_{E_T}  \pl_t\z^p\psi_{\varep}  \bigg(\int_k^{\llbracket u \rrbracket_{\bar{h},k}}\frac{q s^{q-1}}{(s+ck)^{p-1}}\,\d s\bigg)\dx\dt\\
    &\to 
    -\int_{K_{4\rho}\times\{t\}} \z^p \Phi_k(u) \,\dx\bigg|_{t_1}^{t_2} +\iint_{K_{4\rho}\times(t_1,t_2)} \!\! \pl_t\z^p\Phi_k(u) \,\dx\dt
\end{align*}
where, to obtain the convergence, we first let $h\downarrow 0$ in view of  Lemma~\ref{Lm:mol} (i) and then  $\varep\downarrow 0$. Recall also the definition of $\Phi_k(u)$ in \eqref{Eq:Phi}.

Next, we deal with the diffusion part in \eqref{Eq:weak-form:0}. With the aid of  Lemma~\ref{Lm:mol} (iv) we let $h\downarrow 0$ and then $\varep\downarrow 0$ to obtain
\begin{align*}
    \lim_{\varep\downarrow 0}\lim_{h\downarrow 0}
    \iint_{E_T} &  \bl{A}(x,t,u_k, Du_k) D\vp_h\,\dx\dt\\
    &=-(p-1)
    \iint_{K_{4\rho}\times(t_1,t_2)} \z^p  \frac{ \bl{A}(x,t,u_k, Du_k) \cdot Du_k }{(u_k +c k)^{p}}\, \dx\dt\\
    &\phantom{=\,}
    +
\iint_{K_{4\rho}\times(t_1,t_2)}    \frac{ \bl{A}(x,t,u_k, Du_k) \cdot D\z^p }{(u_k +c k)^{p-1}}\, \dx\dt \\
&\le -(p-1) C_o\iint_{K_{4\rho}\times(t_1,t_2)} \z^p |D\Psi_k(u)|^p\,\dx\dt \\
&\quad + p C_1 \iint_{K_{4\rho}\times(t_1,t_2)} |D\Psi_k(u) |^{p-1} \z^{p-1} |D\z|\,\dx\dt,
\end{align*}
where we invoked the structure conditions \eqref{Eq:1:2} to estimate. Recall also the definition of $\Psi_k(u)$ in \eqref{Eq:Psi}. Combining the above estimates in \eqref{Eq:weak-form:0}, we arrive at
\begin{align*}
    \int_{K_{4\rho}\times\{t\}} &\z^p \Phi_k(u) \,\dx\bigg|_{t_1}^{t_2}  
    + 
    (p-1)C_o\iint_{K_{4\rho}\times(t_1,t_2)} \z^p |D\Psi_k(u)|^p\,\dx\dt\nonumber\\
    &\le 
    p C_1 \iint_{K_{4\rho}\times(t_1,t_2)} |D\Psi_k(u) |^{p-1} \z^{p-1} |D\z|\,\dx\dt  \\
    &\phantom{\le\,}
    +\iint_{K_{4\rho}\times(t_1,t_2)}  \pl_t\z^p\Phi_k(u) \,\dx\dt,
\end{align*}
which can be viewed as an integration of \eqref{Eq:int-inequ} in $(t_1,t_2)$. 
Therefore, we can perform calculations similar to those that started with \eqref{Eq:int-inequ}  and led to 
\eqref{Eq:int-inequ-1} and obtain the following estimate
\begin{align}\label{Eq:int-inequ-3}
    \int_{K_{4\rho}\times\{t\}} \z^p \Phi_k(u)\,\dx\bigg|_{t_1}^{t_2} 
    + 
    \iint_{K_{4\rho}\times(t_1,t_2)}& \z^p |D\Psi_k(u)|^p\,\dx\dt \nonumber\\
    &\le 
    \frac{\boldsymbol\gm}{\dl \rho^p}\ln\Big(\frac{1+c}{c}\Big) |K_{4\rho}|(t_2-t_1), 
\end{align}
which  can be viewed as an integration of \eqref{Eq:int-inequ-1} in $(t_1,t_2)$. Like before
an application of the Poincar\'e type inequality in Lemma~\ref{lem:Poincare} 
on each time slices $K_{4\rho}\times\{t\}$ yields
\[
    \iint_{K_{4\rho}\times ( t_1,t_2)} \z^p\Psi^p_k(u)\,\dx\dt
    \le 
    \frac{\boldsymbol\gm_\ast \rho^p}{\al^p} \iint_{K_{4\rho}\times(t_1,t_2)} \z^p|D\Psi_k(u)|^p \,\dx\dt.
\]
Plugging this into \eqref{Eq:int-inequ-3} we deduce
\begin{align*}
    \int_{K_{4\rho}\times\{t\}} \z^p \Phi_k(u)\,\dx\bigg|_{t_1}^{t_2} 
    + 
    \frac{\al^p}{\boldsymbol\gm_\ast\rho^p}\iint_{K_{4\rho}\times(t_1,t_2)}& \z^p \Psi_k^p(u)\,\dx\dt \nonumber\\
    &\le 
    \frac{\boldsymbol\gm}{\dl \rho^p}\ln\Big(\frac{1+c}{c}\Big) |K_{4\rho}|(t_2-t_1).
\end{align*}
Now we divide both sides by $t_2-t_1$ and then let $t_1\uparrow t_2$ to obtain an analog of \eqref{Eq:int-inequ-2}:
\begin{align*}
    \pl_t^{-}\int_{K_{4\rho}\times\{t\}}\z^p \Phi_k(u) \,\dx 
    &+   
    \frac{\al^p}{\boldsymbol\gm_\ast\rho^p}\int_{K_{4\rho}\times\{t\}}   \z^p\Psi^p_k(u)\,\dx\nonumber\\
    &\qquad\qquad\le 
    \frac{\boldsymbol\gm}{\dl \rho^p}\ln\Big(\frac{1+c}{c}\Big) |K_{4\rho}|\quad\text{for any}\> t\in I.
\end{align*}
Here, we have used the notion
\begin{align*}
\pl_t^{-}&\int_{K_{4\rho}\times\{t\}}\z^p \Phi_k(u) \,\dx\\
&\equiv \limsup_{h\downarrow0}\frac1{h} \bigg[\int_{K_{4\rho}\times\{t\}} \z^p \Phi_k(u)\,\dx-\int_{K_{4\rho}\times\{t-h\}} \z^p \Phi_k(u)\,\dx\bigg].
\end{align*}
To proceed, we define the set
\[
\mathcal{S}:=\bigg\{t\in I: \pl_t^{-}\int_{K_{4\rho}\times\{t\}}\z^p \Phi_k(u) \,\dx\ge 0\bigg\}
\]
and let $t_\eps$ be defined as in \eqref{Eq:Y-eps}. 
If $t_\eps\in\mathcal{S}$, then we have \eqref{Eq:log-est}. If $t_\eps\notin\mathcal{S}$ and $\sup\{t<t_\eps: t\in \mathcal{S}\}=t_\eps$,  by working with a sequence $t_n\in \mathcal{S}$ and $t_n\to t_\eps$ as $n\to\infty$, the estimate \eqref{Eq:log-est} still holds. If $t_\eps\notin\mathcal{S}$ but $t_*\equiv\sup\{t<t_\eps: t\in \mathcal{S}\}<t_\eps$, then we obtain \eqref{Eq:t-t}. This, although quite standard, requires an argument; we refer to \cite[p.~226, Problem~2.3]{DB-RA} for details.
Therefore, in any case we could reason just like before.

\subsection{Proof of Lemma~\ref{Lm:shrink}} \label{S:Lm:shrink-proof}
For $n_o\in\mathbb{N}$, let us 
iterate Lemma~\ref{Lm:Aux} to obtain 
\[
    \boldsymbol Y_{n_o}\le\max\big\{\nu,\, \sig^{n_o}\boldsymbol Y_o\big\}.
\]
Since $\boldsymbol Y_o\le 1$, we may choose $n_o$ such that $ \sig^{n_o}\le \nu$, in order to obtain $\boldsymbol Y_{n_o}\le\nu$.
In this way, the above estimate, according to the definition \eqref{Eq:Y_n} of $Y_{n_o}$ and the properties \eqref{Eq:z} of $\z$, yields that
\[
    \frac1{|K_{2\rho}|}
    \Big|\big\{u(\cdot, t)\le  c^{n_o}M\big\}\cap K_{\rho}\Big|\le 
    \boldsymbol Y_{n_o}\le\nu\quad\mbox{for all  $t\in\tfrac12 I$.}
\]
 To get the estimate, one could replace $\nu$ by $2^{-N}\nu$ and adjust relevant constants. The proof is concluded with the choice $\xi=c^{n_o}$.

\section{Proof of Proposition~\ref{PROP:EXPANSION}}\label{S:proof-expansion}
The  measure theoretical information at $t_o$ given in Proposition~\ref{PROP:EXPANSION} implies that
\begin{equation*}
		\Big|\big\{ u(\cdot, t_o) \ge M \big\}\cap K_{4\varrho}(x_o)\Big|
		\ge
        4^{-N}
		\al  |K_{4\varrho} |.
\end{equation*}
Starting from this we first apply Lemma~\ref{Lm:3:1} to determine
$\dl$ and $\eps$ in $(0,1)$, depending only on the data and $\al$, such that 
\begin{equation*}
	\Big|\big\{
	u(\cdot, t) \ge \eps M\big\} \cap K_{4\varrho}(x_o)\Big|
	\ge\tfrac12 4^{-N}\al |K_{4\varrho}|
\end{equation*}
for all 
$$
    t\in\big(t_o,t_o+\dl M^{q+1-p}(4\varrho)^p\big].
$$
This measure theoretical information allows  to apply Lemma~\ref{Lm:shrink}.

For given $\nu\in (0,1)$ we obtain  $\xi\in(0,1)$ depending only on the data, $\nu$, $\dl$ and $\al$,  such that
\[
\Big|\Big\{
	u(\cdot, t) \le \xi \varep M \Big\} \cap K_{4\varrho}(x_o)
 \Big|
	\le\nu |K_{4\varrho}|
\]
for each slice of time 
\[
t\in \big(t_o+\tfrac12\dl (\varep M)^{q+1-p}(4\varrho)^p, t_o+\dl (\varep M)^{q+1-p}(4\varrho)^p\big].
\]
Observe, for any
\[
    \bar{t}\in 
    \big(t_o+\tfrac34\dl (\varep M)^{q+1-p}(4\varrho)^p, 
    t_o+\dl (\varep M)^{q+1-p}(4\varrho)^p\big]
\]
the (backward) cylinder
$$
    (x_o, \bar{t})+Q_{4\rho}(\theta)\quad 
    \text{where $\theta=\tfrac14\dl(\xi \varep M)^{q-1+p}$,}
$$
is contained in 
\[
    K_{4\rho} (x_o)\times \big(t_o+\tfrac12\dl (\varep M)^{q+1-p}(4\varrho)^p, 
    t_o+\dl (\varep M)^{q+1-p}(4\varrho)^p\big].
\]
Hence, for any $\bar{t}$ as above we have
\[
    \Big| \big\{ u\le \xi\varep M \big\}\cap (x_o, \bar{t})+Q_{4\rho}(\theta)\Big|
    \le 
    \nu |Q_{4\rho}(\theta)|.
\]
Finally, let $\nu$ be determined in Lemma~\ref{Lm:DG:1} (with $M$ replaced by $\xi\varep M$ and $(x_o,t_o)+Q_\rho (\theta)$ by $(x_o,\bar{t})+Q_{4\rho} (\theta)$) in terms of the data and $\dl$.
Note that the  above smallness estimate in measure allows us to apply Lemma~\ref{Lm:DG:1} in the cylinder  
$
(x_o, \bar{t})+Q_{4\rho}(\theta)
$
to  obtain that
\[
u\ge\tfrac12\xi\varep M \quad\text{a.e. in}\>(x_o, \bar{t})+Q_{2\rho}(\theta).
\]
The arbitrariness of $\bar{t}$ gives us the desired estimate upon redefining relevant constants.

\chapter[Integral Harnack inequality]{Integral Harnack inequality}\label{sec:int-harnack}
The main purpose of this chapter is to show Theorem~\ref{Thm:bd:2}, which will be employed repeatedly in the proof of Harnack's inequality.
Theorem~\ref{Thm:bd:2} is a straightforward consequence of the quantitative $L^\infty$-bound from Theorem~\ref{THM:BD:1} and the following result.


\begin{proposition}\label{PROP:L1:0} 
Assume that $0<p-1<q$. There exists a positive constant
$\boldsymbol\gm$  depending only on the data, such that whenever
$u$ is a non-negative
weak  solution to \eqref{Eq:1:1f} with \eqref{Eq:1:2} in $E_T$, then
for any cylinder $K_{2\rho}(y)\times [s,\tau]\Subset E_T$,  we have 
\begin{equation*}
    \sup_{t\in [s,\tau]}\int_{K_\rho(y)\times\{t\}} u^q\,\dx
    \le 
    \boldsymbol\gm \inf_{t\in [s,\tau]}\int_{K_{2\rho}(y)\times\{t\}}u^q\,\dx
    +
    \boldsymbol\gm\Big(\frac{\tau-s}{\rho^\lm}\Big)^{\frac{q}{q+1-p}},
\end{equation*}
where 
\[
    \lm:=\frac{\lm_q}{q} \equiv\frac{N}{q}(p-q-1)+p.
\]
\end{proposition}
\begin{remark}\upshape
If $u$ is merely a non-negative weak {\it super-solution}, then we will have
\begin{equation*}
\sup_{t\in [s,\tau]}\int_{K_\rho(y)\times\{t\}} u^q\,\dx\le \boldsymbol\gm
\int_{K_{2\rho}(y)\times\{\tau\}}u^q \,\dx
+\boldsymbol\gm\Big(\frac{\tau-s}{\rho^\lm}\Big)^{\frac{q}{q+1-p}}.
\end{equation*}
The proof requires minor modifications.
In fact, the only difference occurs in \eqref{L1-integral}, which only holds for times $t_1\le t_2$. Therefore, in the case of supersolutions, we have to choose $t_2=\tau$. The rest of the proof is then analogous. 
\end{remark}

\section{Anomalous energy estimates}

The proof of Proposition~\ref{PROP:L1:0} hinges upon the following lemmas. 

\begin{lemma}\label{Lm:L1:1}
Assume that $0<p-1<q$. There exists   $\boldsymbol\gm>0$ depending on the data such that whenever
 $u$ is a non-negative weak super-solution to \eqref{Eq:1:1f} with \eqref{Eq:1:2} in $E_T$, then
 for all $K_{\rho}(y)\times [s,\tau]\subset E_T$, and all $\sig\in(0,1)$, we have
\begin{align*}
    \iint_{K_{\sig\rho}(y)\times(s,\tau)} &(t-s)^{\frac1p}(u+\kappa)^{-\frac{1+q}{p}}|Du|^p\,\dx\dt\\
    &\le 
    \frac{\boldsymbol\gm\rho}{(1-\sig)^p}\Big(\frac{\tau-s}{\rho^{\lm}}\Big)^{\frac1p}
    \bigg[\sup_{t\in[s,\tau]}\int_{K_\rho(y)\times\{t\}} u^q \,\dx +\kappa^q\rho^N\bigg]^{\frac{(p-1)(q+1)}{pq}},
\end{align*}
where
\[
    \lm=\frac{N}{q}(p-q-1)+p,
    \quad\mbox{and}\quad
    \kappa=\Big(\frac{\tau-s}{\rho^{p}}\Big)^{\frac{1}{q+1-p}}.
\]
\end{lemma}
\begin{proof}
Let us assume  $(y,s)=(0,0)$. In the weak formulation of weak super-solutions to the equation \eqref{Eq:1:1}, we choose the testing function
\[
    \vp_h(x,t)=t^{\frac1p}\big(\llbracket u \rrbracket_{\bar{h}}+\kappa\big)^{-\frac{q+1-p}{p}}\z^p(x)\psi_{\varep}(t).
\]
The properties of the exponential time mollification $\llbracket u \rrbracket_{\bar{h}}$  are collected in Lemma \ref{Lm:mol}. We  obtain
\begin{equation}\label{weak-form-1}
\iint_{E_T} \big[ -u^q\pl_t \vp_h + \bl{A}(x,t,u, Du) D\vp_h\big]\dx\dt\ge0.
\end{equation}
Here, $\z\in C_0^1(K_\rho;[0,1])$ satisfies $\z=1$ in $K_{\sig\rho}$ and $|D\z|\le\frac2{(1-\sig)\rho}$,
whereas for $\varep>0$, the Lipschitz function $\psi_\eps$ satisfies  $\psi_\varep=1$ in $(\varep, \tau-\varep)$, $\psi_\varep=0$ outside $(0,\tau)$, while linearly interpolated otherwise.
Let us first treat the time part of \eqref{weak-form-1}. Namely, using  Lemma~\ref{Lm:mol} (iii), we have
\begin{align*}
    -\iint_{E_T} & u^q \pl_t \vp_h\,\dx\dt\\
    &=
    -\iint_{E_T} \llbracket u \rrbracket_{\bar{h}}^q \pl_t \vp_h\,\dx\dt
    +
    \iint_{E_T} \big(\llbracket u \rrbracket_{\bar{h}}^q - u^q\big) \pl_t \vp_h\,\dx\dt\\
    &=
    \iint_{E_T}  \pl_t \llbracket u \rrbracket_{\bar{h}}^q \vp_h\,\dx\dt \\
    &\phantom{=\,}
    +
    \iint_{E_T} \big(\llbracket u \rrbracket_{\bar{h}}^q - u^q\big) 
    \big(\llbracket u \rrbracket_{\bar{h}}+\kappa\big)^{-\frac{q+1-p}{p}}\z^p(x)\pl_{t}\big(t^{\frac1p}\psi_{\varep}(t)\big)\,\dx\dt\\
    &\phantom{=\,}
    -
    \tfrac{q+1-p}{p}\iint_{E_T} 
    \big(\llbracket u \rrbracket_{\bar{h}}+\kappa\big)^{-\frac{q+1}{p}}
    \big(\llbracket u \rrbracket_{\bar{h}}^q - u^q\big)\tfrac1{h}
    \big(\llbracket u \rrbracket_{\bar{h}} - u\big)
    t^{\frac1p} \z^p(x)\psi_{\varep}(t) \,\dx\dt.
\end{align*}
The various terms on the right-hand side of the above display are treated as follows.
The second term vanishes in the limit $h\downarrow 0$ due to Lemma~\ref{Lm:mol} (i), whereas the third term is negative, thanks to the monotonicity of the map 
$u \mapsto u^q$. When it comes to the first term, one estimates
\begin{align*}
    \iint_{E_T}  \pl_t \llbracket u \rrbracket_{\bar{h}}^q \vp_h\,\dx\dt
    &=
    q\iint_{E_T} \z^p(x)  t^{\frac1p} \psi_{\varep}(t)\pl_t
    \bigg[\int_0^{\llbracket u \rrbracket_{\bar{h}}} s^{q-1}(s+\kappa)^{-\frac{q+1-p}{p}}\,\ds\bigg]\dx\dt\\
    &=
    -q\iint_{E_T} \z^p(x)  t^{\frac1p} \pl_t\psi_{\varep}(t) 
    \int_0^{\llbracket u \rrbracket_{\bar{h}}} s^{q-1}(s+\kappa)^{-\frac{q+1-p}{p}}\,\ds \dx\dt\\
    &\phantom{=\,}
    -q\iint_{E_T} \z^p(x) \pl_t t^{\frac1p} \psi_{\varep}(t) 
    \int_0^{\llbracket u \rrbracket_{\bar{h}}} s^{q-1}(s+\kappa)^{-\frac{q+1-p}{p}}\,\ds \dx\dt\\
    &\to 
    q\int_{K_\rho} \z^p(x)  t^{\frac1p}   \int_0^{u} s^{q-1}(s+\kappa)^{-\frac{q+1-p}{p}}\,\ds\dx\Big|_{0}^{\tau}\\
    &\phantom{\to}
    -q\iint_{K_\rho\times(0,\tau)} \z^p(x) \pl_t t^{\frac1p}   
    \int_0^{u} s^{q-1}(s+\kappa)^{-\frac{q+1-p}{p}}\,\ds \dx\dt\\
    &=
    q\tau^{\frac1p}\int_{K_\rho\times\{\tau\}} \z^p(x)     \int_0^{u} s^{q-1}(s+\kappa)^{-\frac{q+1-p}{p}}\,\ds\dx\\
    &\phantom{=\,}
    -q\iint_{K_\rho\times(0,\tau)} \z^p(x) \pl_t t^{\frac1p}   
    \int_0^{u} s^{q-1}(s+\kappa)^{-\frac{q+1-p}{p}}\,\ds \dx\dt
\end{align*}
where, to obtain the convergence, we  first let $h\downarrow 0$ in view of  Lemma~\ref{Lm:mol} (i) and then  $\varep\downarrow 0$. Next, we treat the diffusion part in \eqref{weak-form-1}. Letting $h\downarrow 0$ with the aid of  Lemma~\ref{Lm:mol} (iv) and then $\varep\downarrow 0$, we obtain
\begin{align*}
-\tfrac{q+1-p}{p} &\iint_{K_\rho\times(0,\tau)}  \bl{A}(x,t,u, Du)\cdot Du (u+\kappa)^{-\frac{q+1}{p}}t^{\frac1p}\z^p(x) \, \dx\dt\\
&+ \iint_{K_\rho\times(0,\tau)}  \bl{A}(x,t,u, Du)\cdot D\z^p(x)   (u+\kappa)^{-\frac{q+1-p}{p}}t^{\frac1p} \, \dx\dt.
\end{align*}
Then we use the structure conditions \eqref{Eq:1:2} to estimate
\begin{align*}
    -\tfrac{q+1-p}{p}& \iint_{K_\rho\times(0,\tau)}  \bl{A}(x,t,u, Du)\cdot Du (u+\kappa)^{-\frac{q+1}{p}}t^{\frac1p}\z^p(x) \, \dx\dt\\
    &\le 
    -\tfrac{q+1-p}{p} C_o \iint_{K_\rho\times(0,\tau)} |Du|^p  (u+\kappa)^{-\frac{q+1}{p}}t^{\frac1p}\z^p(x) 
    \, \dx\dt
\end{align*}
and
\begin{align*}
    \iint_{K_\rho\times(0,\tau)} 
    &  
    \bl{A}(x,t,u, Du)\cdot D\z^p(x) (u+\kappa)^{-\frac{q+1-p}{p}}t^{\frac1p} \, \dx\dt\\
    & \le  
    p C_1  \iint_{K_\rho\times(0,\tau)}  |Du|^{p-1} \z^{p-1} |D\z|   (u+\kappa)^{-\frac{q+1-p}{p}}t^{\frac1p} \, \dx\dt\\
    &\le 
    \tfrac{q+1-p}{2p} C_o \iint_{K_\rho\times(0,\tau)} |Du|^p  (u+\kappa)^{-\frac{q+1}{p}}t^{\frac1p}\z^p(x) \, \dx\dt\\
    &\phantom{\le\,}
    +\boldsymbol\gm \iint_{K_\rho\times(0,\tau)} (u+\kappa)^{\frac{p^2-1-q}{p}}|D\z|^p t^{\frac1p} \,\dx\dt.
\end{align*}
Here, Young's inequality is appealed to in the last estimate. Combining the above estimates we arrive at
\begin{align*}
    \iint_{K_\rho\times(0,\tau)} & |Du|^p  (u+\kappa)^{-\frac{q+1}{p}}t^{\frac1p}\z^p(x) \, \dx\dt\\
    &\le 
    \boldsymbol\gm \tau^{\frac1p}  \int_{K_\rho\times\{\tau\}} \z^p(x)   
    \int_0^{u} s^{q-1}(s+\kappa)^{-\frac{q+1-p}{p}}\,\ds \dx\\
    &\phantom{\le\,}
    -
   \boldsymbol\gm\iint_{K_\rho\times(0,\tau)}   \z^p(x) \pl_t t^{\frac1p} 
    \int_0^{u} s^{q-1}(s+\kappa)^{-\frac{q+1-p}{p}}\,\ds \dx\dt\\
    &\phantom{\le\,}
    +
    \frac{\boldsymbol\gm}{(1-\sig)^p\rho^p} 
    \iint_{K_\rho\times(0,\tau)} t^{\frac1p} (u+\kappa)^{\frac{p^2-1-q}{p}}\,\dx\dt.
\end{align*} 
To proceed, observe that, according to the range $0<p-1<q$,
\[
 \tfrac{(p-1)(q+1)}{pq}\in(0,1) \quad\text{and}\quad \int_0^{u} s^{q-1}(s+\kappa)^{-\frac{q+1-p}{p}}\,\ds\le\boldsymbol\gm u^{\frac{(p-1)(q+1)}{p}}.
\]
Hence the first term on the right-hand side is estimated by H\"older's inequality:
\[
    \boldsymbol\gm \tau^{\frac1p} \int_{K_\rho\times\{\tau\}} u^{\frac{(p-1)(q+1)}{p}}\,\dx
    \le 
    \boldsymbol\gm\rho\Big(\frac{\tau}{\rho^\lm}\Big)^{\frac1p}
    \bigg[\sup_{t\in[0,\tau]}\int_{K_\rho\times\{t\}} u^q \,\dx\bigg]^{\frac{(p-1)(q+1)}{pq}}.
\]
Recalling the definition of $\kappa$, the last term is estimated by
\begin{align*}
    \frac{\boldsymbol\gm}{(1-\sig)^p\rho^p} 
    &
    \iint_{K_\rho\times(0,\tau)} t^{\frac1p} (u+\kappa)^{\frac{p^2-1-q}{p}}\,\dx\dt\\
    &= 
    \frac{\boldsymbol\gm}{(1-\sig)^p\rho^p} 
    \iint_{K_\rho\times(0,\tau)} t^{\frac1p} (u+\kappa)^{-(q+1-p)} (u+\kappa)^{\frac{(p-1)(q+1)}{p}}\,\dx\dt\\
    &\le 
    \frac{\boldsymbol\gm  \tau^{\frac1p}}{(1-\sig)^p}\frac{\tau}{\rho^p}  \kappa ^{-(q+1-p)} \sup_{t\in[0,\tau]}\int_{K_\rho\times\{t\}} (u+\kappa)^{\frac{(p-1)(q+1)}{p}}\,\dx\\
    &= 
    \frac{\boldsymbol\gm  \tau^{\frac1p}}{(1-\sig)^p}   
    \sup_{t\in[0,\tau]}\int_{K_\rho\times\{t\}} (u+\kappa)^{\frac{(p-1)(q+1)}{p}}\,\dx\\
    &\le
    \frac{\boldsymbol\gm  \rho}{(1-\sig)^p}\Big(\frac{\tau}{\rho^\lm}\Big)^{\frac1p}
    \bigg[\sup_{t\in[0,\tau]}\int_{K_\rho\times\{t\}} u^q\,\dx +\kappa^q\rho^N\bigg]^{\frac{(p-1)(q+1)}{pq}}.
\end{align*}
Collecting the last estimates, we conclude the proof.
\end{proof}

A direct consequence of Lemma~\ref{Lm:L1:1} is the following.
\begin{lemma}\label{Lm:L1:2}
Let the assumptions of Lemma~\ref{Lm:L1:1} hold.
There exists   $\boldsymbol\gm>0$ depending on the data, such that for all $\dl,\,\sig\in(0,1)$, we have
\begin{align*}
    \frac{1}{\rho}& \iint_{K_{\sig\rho}(y)\times(s,\tau)} |Du|^{p-1}\,\dx\d\tau\\
    &\quad\le 
    \dl \sup_{t\in[s,\tau]}\int_{K_\rho(y)\times\{t\}} u^q\,\dx
    +
    \frac{\boldsymbol\gm}{[\dl^{q+1}(1-\sig)^{pq}]^{\frac{p-1}{q+1-p}}}\Big(\frac{\tau-s}{\rho^\lm}\Big)^{\frac{q}{q+1-p}}.
\end{align*}
\end{lemma}
\begin{proof}
Let us fix $(y,s)$ at $(0,0)$. By H\"older's inequality, we have
\begin{align*}
    \iint_{K_{\sig\rho}\times(0,\tau)}|Du|^{p-1}\,\dx\dt
    &\le
    \bigg[\iint_{K_{\sig\rho}\times(0,\tau)} |Du|^{p} (u+\kappa)^{-\frac{q+1}{p}} t^{\frac1p}\,\dx\dt\bigg]^{\frac{p-1}{p}}\\
    &\qquad\quad\qquad\cdot
    \bigg[\iint_{K_{\sig\rho}\times(0,\tau)}(u+\kappa)^{\frac{(p-1)(q+1)}{p}} t^{\frac{1-p}{p}}\,\dx\dt\bigg]^{\frac1p}.
\end{align*}
The first integral on the right-hand side is estimated by Lemma~\ref{Lm:L1:1}, whereas the second integral is estimated by H\"older's inequality as
\begin{align*}
    \iint_{K_{\sig\rho}\times(0,\tau)}&(u+\kappa)^{\frac{(p-1)(q+1)}{p}} t^{\frac{1-p}{p}}\,\dx\dt\\
    &\le
    \int_0^{\tau} t^{\frac{1-p}{p}}\,\dt \cdot \sup_{t\in[0,\tau]}\int_{K_{\sig\rho}\times\{t\}}(u+\kappa)^{\frac{(p-1)(q+1)}{p}}\,\dx\\
    &\le
    \boldsymbol\gm \rho \Big(\frac{\tau}{\rho^\lm}\Big)^{\frac1p}
    \bigg[\sup_{t\in[0,\tau]}\int_{K_{\sig\rho}\times\{t\}} u^q\,\dx +\kappa^q\rho^N\bigg]^{\frac{(p-1)(q+1)}{pq}}.
\end{align*}
Combining them,  applying Young's inequality, and recalling the definition of $\kappa$ we obtain
\begin{align*}
    \iint_{K_{\sig\rho}\times(0,\tau)}&|Du|^{p-1}\,\dx\dt\\  
    &\le
    \frac{\boldsymbol\gm  \rho}{(1-\sig)^{p-1}}\Big(\frac{\tau}{\rho^\lm}\Big)^{\frac1p}
    \bigg[\sup_{t\in[0,\tau]}\int_{K_\rho\times\{t\}} u^q\,\dx +\kappa^q\rho^N\bigg]^{\frac{(p-1)(q+1)}{pq}}\\
    &\le 
    \dl\rho \sup_{t\in[0,\tau]}\int_{K_\rho\times\{t\}} u^q\,\dx + \frac{\boldsymbol\gm  \rho}{[\dl^{q+1}(1-\sig)^{pq}]^{\frac{p-1}{q+1-p}}}\Big(\frac{\tau}{\rho^\lm}\Big)^{\frac{q}{q+1-p}}.
\end{align*}
Dividing both sides by $\rho$ yields the conclusion.
\end{proof}
\section{Proof of Proposition~\ref{PROP:L1:0}}
Let us fix $(y,s)=(0,0)$. For $n\in\mathbb{N}_0$ introduce
\[
    \rho_n=\sum_{j=0}^{n}\frac{\rho}{2^j},\quad \widetilde\rho_n=\frac{\rho_n+\rho_{n+1}}{2},\quad K_n=K_{\rho_n},\quad\widetilde{K}_n=K_{ \widetilde\rho_n}.
\]
Let $\z\in C_0^1(\widetilde{K}_n;[0,1])$ satisfy $\z=1$ in $K_n$ and $|D\z|\le 2^{n+3}/\rho$.
Plugging this testing function in the weak formulation \eqref{Eq:weak-form} and using the structure condition \eqref{Eq:1:2} we get
\begin{equation}\label{L1-integral}
    \int_{\widetilde{K}_n\times\{t_1\}} u^q \z\,\dx
    \le
    \int_{\widetilde{K}_n\times\{t_2\}} u^q \z\,\dx +
    \frac{C_1 2^{n+3}}{\rho}\iint_{\widetilde{K}_n\times (t_1,t_2)}|Du|^{p-1}\,\dx\dt 
\end{equation}
for arbitrary $t_1,t_2\in[0,\tau]$. 
More precisely, in order to obtain this estimate for times $t_1>t_2$, we use \eqref{Eq:weak-form} with the test function $-\zeta$ instead of $\zeta$. 
Let us choose $t_2$ to satisfy
\[
\int_{K_{2\rho}\times\{t_2\}} u^q\,\dx=\inf_{t\in[0,\tau]}\int_{K_{2\rho}\times\{t\}} u^q\,\dx =: \bl{I}.
\]
In addition, we set
\[
\boldsymbol{S}_n:=\sup_{t\in[0,\tau]}\int_{K_n\times\{t\}} u^q\,\dx.
\]
Using this notation, by the arbitrariness of $t_1$ in \eqref{L1-integral}, we obtain from \eqref{L1-integral} that
\[
    \boldsymbol{S}_n\le \bl{I} + \frac{C_1 2^{n+2}}{\rho} \iint_{\widetilde{K}_n\times (0,\tau)}|Du|^{p-1}\,\dx\dt.
\]
To estimate the last integral, we appeal to Lemma~\ref{Lm:L1:2} with $\widetilde{K}_n\subset K_{n+1}$ and take into account
that $\rho\le \rho_n< 2\rho$ and $1-\sig\ge 2^{-(n+3)}$. This gives
\begin{align*}
    \frac{1}{2\rho} \iint_{\widetilde K_n\times(0,\tau)}&|Du|^{p-1}\,\dx\dt\le
   \frac{1}{\rho_{n+1}} \iint_{\widetilde K_n\times(0,\tau)}|Du|^{p-1}\,\dx\dt\\  
    &\le 
    \dl \sup_{t\in[0,\tau]}\int_{K_{n+1}\times\{t\}} u^q\,\dx + \frac{\boldsymbol\gm  }{[\dl^{q+1}(1-\sig)^{pq}]^{\frac{p-1}{q+1-p}}}\Big(\frac{\tau}{\rho^\lm_{n+1}}\Big)^{\frac{q}{q+1-p}}\\
    &\le    
    \dl \boldsymbol S_{n+1}
    +
    \frac{\boldsymbol\gm 2^{n\frac{pq(p-1)}{q+1-p}}  }{\dl^{\frac{(q+1)(p-1)}{q+1-p}}}\Big(\frac{\tau}{\rho^\lm}\Big)^{\frac{q}{q+1-p}}.
\end{align*}
Multiplying by $C_12^{n+3}$ we obtain
\begin{align*}
    \frac{C_12^{n+2}}{\rho} \iint_{\widetilde K_n\times(0,\tau)}&|Du|^{p-1}\,\dx\dt\\  
    &\le    
    C_1 2^{n+3}\dl \boldsymbol S_{n+1}
    +
    \frac{8C_1\boldsymbol\gm 2^{n(1+\frac{pq(p-1)}{q+1-p})}  }{\dl^{\frac{(q+1)(p-1)}{q+1-p}}}\Big(\frac{\tau}{\rho^\lm}\Big)^{\frac{q}{q+1-p}},
\end{align*}
and choosing $\dl= \varep/(C_1 2^{n+3})$ for some $\varep\in(0,1)$,  we have
\[
     \frac{C_1 2^{n+2}}{\rho} \iint_{\widetilde K_n\times (0,\tau)} |Du|^{p-1}\,\dx\dt
    \le \varep \boldsymbol{S}_{n+1} +\boldsymbol\gm(\varep) \boldsymbol{b}^n \Big(\frac{\tau}{\rho^\lm}\Big)^{\frac{q}{q+1-p}},
\]
for some  $\boldsymbol{b}=\boldsymbol{b}(p,q)>1$.
Plugging it back to the estimate for $\mathbf S_n$, we  get the recursive inequalities
\[
\boldsymbol{S}_n\le \varep \boldsymbol{S}_{n+1} +\boldsymbol\gm(\varep) \boldsymbol{b}^n \Big[\bl{I} +  \Big(\frac{\tau}{\rho^\lm}\Big)^{\frac{q}{q+1-p}}\Big]\quad\text{for}\quad n\in\mathbb{N}_0.
\]
We iterate these inequalities to obtain
\[
\boldsymbol{S}_0\le \varep^n \boldsymbol{S}_{n} +\boldsymbol\gm(\varep)   \Big[\bl{I} +  \Big(\frac{\tau}{\rho^\lm}\Big)^{\frac{q}{q+1-p}}\Big]
\sum_{i=0}^{n-1}(\varep \boldsymbol{b})^i.
\]
The number $\varep$ is chosen to be $1/ (2\boldsymbol{b})$ so that the summation in the last display is bounded by two, and meanwhile, $\varep^n \bl{S}_{n}\to0$ as $n\to\infty$.

\chapter[Proof of Harnack's inequality]{Proof of Harnack's inequality}

The plan of this chapter is to deal with Harnack's inequality presented in Theorems~\ref{THM:HARNACK:0}--\ref{THM:HARNACK}. We first discuss the optimality of our theorems in Section~\ref{S:Harnack-opt}.
Then, with the preparation from Chapters~\ref{sec:boundedness}--\ref{sec:int-harnack} at hand, the proof of Theorem~\ref{THM:HARNACK} is expanded on in the following. In particular, we focus on the proof of the left-hand side of the Harnack inequality in Sections~\ref{S:Harnack:1}--\ref{S:Harnack:3.5}; the right-hand side inequality is dealt with in Section~\ref{S:Harnack:4}. Whereas the proof of Theorem~\ref{THM:HARNACK:0} will be sketched in Section~\ref{S:Thm:0} and Corollary~\ref{Cor:Liouville} will be treated in Section~\ref{Proof:Liouville}.  


\section{Optimality}\label{S:Harnack-opt}

The purpose of this section is to show by explicit examples that the range of $p$ and $q$ in 
Theorems~\ref{THM:HARNACK:0}--\ref{THM:HARNACK} is optimal for such Harnack inequality to hold.

First of all, the case $q=p-1$ is generally not allowed. This can be seen from the following solution to the prototype equation \eqref{doubly-nonlinear-prototype} in $\rn\times\rr_+$ with $ N\ge1$ and $p>1$, i.e.
\[
u(x,t)=C t^{-\frac{N}{p(p-1)}}\exp\bigg\{-\frac{p-1}{p}\bigg(\frac{|x|^p}{pt}\bigg)^{\frac1{p-1}}\bigg\}
\]
for an arbitrary positive parameter $C$. Indeed, letting $e_1=(1,0,\dots,0)\in\rn$,  one only needs to check that the ratio $\frac{u(\ell e_1,2)}{u(\ell e_1+e_1,2)}$ tends to infinity as $\ell\to\infty$.
\color{black}

In what follows, we assume $0<p-1<q$. Let us consider the following counterexamples, which can be built as the ones discussed in \S~\ref{S:bd-opt} of Chapter~\ref{sec:boundedness}. The first one is constructed on ${\mathbb R}^N\times{\mathbb R}$. For parameters
\begin{equation*}
     N\ge2,
     \quad 
     N>p,
     \quad
     q=\frac{N(p-1)}{N-p},
     \quad
    b=\left(\frac{N}{q}\right)^{p}
\end{equation*}
the function
\begin{equation*}
    u(x,t)=\Big(|x|^{\frac{N(q+1)}{q(N-1)}}+e^{bt}\Big)^{-\frac{N-1}{q+1}}
\end{equation*}
is a non-negative solution to the prototype 
equation \eqref{doubly-nonlinear-prototype} in ${\mathbb R}^N\times{\mathbb R}$, and one verifies that it fails  to satisfy Theorem~\ref{THM:HARNACK}. The latter is easily seen by evaluating the quotient $\frac{u(x,t)}{u(x_o,t_o)}$ with $x=x_o=0$, i.e.~for $(0, t_o)$ and $(0,t)$, and then letting $t\to\pm\infty$.

The second counterexample is constructed on ${\mathbb R}^N\times(-\infty,T)$. For parameters
\begin{equation*}
     \max\left\{\frac{q+1}q,p\right\}<N,\qquad q={\frac{N(p-1)+p}{N-p}},
\end{equation*}
the two-parameter family of functions
\begin{equation*}
    u(x,t)=(T-t)_+^{\frac{N+q+1}{(q+1)^2}}\Big(a+
    b|x|^{\frac{N(q+1)}{Nq-q-1}}\Big)^{-\frac{N}{q+1}},
\end{equation*}
with arbitrary parameters $a>0$ and $T\in{\mathbb R}$, and 
\begin{equation*}
    b=b(N,q,a)={\frac{Nq-q-1}{N^2}}\bigg({\frac{q(N+q+1)}{(q+1)^2Na}}
\bigg)^{\frac{N+q+1}{Nq-q-1}}
\end{equation*}
are non-negative, locally bounded, weak solutions to the prototype equation \eqref{doubly-nonlinear-prototype} in ${\mathbb R}^N\times(-\infty,T)$ and they do not satisfy Theorem~\ref{THM:HARNACK}. For the calculation that they are indeed solutions,
one replaces $q$ in the formulas for $u$ and $b$ by its concrete value in terms of $N$ and $p$. The exchange of $q$ results in 
\begin{equation*}
    u(x,t)
    =
    (T-t)_+^{\frac{N-p}{p^2}}\Big(a+b|x|^{\frac{p}{p-1}}
    \Big)^{-\frac{N-p}{p}},
\end{equation*}
and 
\begin{equation*}
b=b(N,p,a)=\frac{p-1}{N-p}\Big({\frac{N(p-1)+p}{p^2Na}}
\Big)^{\frac{1}{p-1}}.
\end{equation*}
Analogous to the first counterexample, the evaluation of  $\frac{u(x,t)}{u(x_o,t_o)}$ with $x=x_o=0$ and  $t,t_o\in (-\infty, T)$ yields the contradiction to the Harnack inequality in the limit $t\uparrow T$. At this point we have exploited the fact that $u$ can be extended to $0$ on ${\mathbb R}^N\times [T,\infty)$ to a locally bounded, non-negative  weak solution of the prototype equation on all of ${\mathbb R}^N\times{\mathbb R}$.

We have a second two-parameter family of functions that serves as  a counterexample to the Harnack inequality. For
\begin{equation*}
    p<N,\quad N\ge2,\quad q>\frac{N(p-1)}{N-p}
\end{equation*}
we let
\begin{equation*}
    \lambda_{q}:=N(p-1-q)+qp.
\end{equation*}
Note that $\lambda_{q}<0$ for the indicated ranges of $p$ and $q$. Then the two-parameter family of functions
\begin{equation*}
    u(x,t)
    =(T-t)_+^{\frac{1}{q+1-p}}
    \Big(a|x|^{\frac p{p-1}}+C(T-t)_+^{\frac{pq}{(p-1)\lambda_{q}}}\Big)^{-\frac{p-1}{q+1-p}}
\end{equation*}
with arbitrary parameters $C>0$ and $T\in\rr$, and with 
\begin{equation*}
    a=a(N,p,q)=\left(\frac q{|\lambda_{q}|}\right)^{\frac1{p-1}}\frac{q+1-p}{p},
\end{equation*}
are once more non-negative, locally bounded, weak solutions 
to the prototype equation \eqref{doubly-nonlinear-prototype} in ${\mathbb R}^N\times{(-\infty,T)}$, and they do not satisfy Theorem~\ref{THM:HARNACK}. 
As in the first example, one first extends $u$ by $0$ to ${\mathbb R}^N\times[T,\infty)$. This again leads to a locally bounded, non-negative
weak solution of the prototype equation, violating the Harnack inequality. Incidentally, the above examples also show that non-constant, non-negative weak solutions exist in the whole $\rn\times\rr$ for $\lm_q<0$, confirming the optimality of the Liouville-type result in Corollary~\ref{Cor:Liouville}.

The counterexample \eqref{Ex-bdd} in \S~\ref{S:bd-opt} to boundedness can be formally obtained from this family of solutions, letting $C\downarrow 0$. Moreover, it is straightforward to check that one can also assume $C<0$. In this latter case, solutions are defined in the non-cylindrical domain
\begin{equation*}
    \Big\{ (x,t)\in \rr^N\times (-\infty ,T): |x|>R(t)\Big\}
\end{equation*}
with
\begin{equation*}
    R(t):=\left(\frac{|C|}{a}\right)^{\frac{p-1}p}(T-t)_+^{\frac{q}{\lambda_{q}}}.
\end{equation*}
Since $\lambda_{q}<0$, it is apparent that $R(t)\uparrow+\infty$ as 
$t\uparrow T$; moreover, $u$ blows up at the boundary 
$|x|=R(t)$, and the rate of divergence can be estimated using the same techniques employed in \cite{bidaut}.

\color{black}


\section{Scaling of Harnack's inequality}\label{S:Harnack:1}
The first step towards Harnack's inequality starts here. Throughout the proof we deem $u$ as its own unique {\it upper semi-continuous regularization}. However, the reader could ignore this technicality for first reading and proceed with a {\it continuous} representative.

Let us introduce a new function
\begin{equation}\label{Eq:u-to-v}
    v(x,t)
    :=
    \frac{u\big(x_o+\rho x, t_o+[u(x_o,t_o)]^{q+1-p}\rho^p t\big)}{u(x_o,t_o)}\quad\text{in}\> \widetilde{Q}_8=K_8\times(-8^p, 8^p).
\end{equation}
It satisfies $v(0,0)=1$ and 
\[
    \pl_t v^q - \dvg \big(|Dv|^{p-2} Dv\big)=0 \quad\text{weakly in}\> \widetilde{Q}_8. 
\]
Thus $v$ is also granted with the properties shown in previous sections  with $C_o=C_1=1$ in the structure conditions \eqref{Eq:1:2}.

In what follows, we first prove that there exist $\boldsymbol\gm>1$ and $\sig\in(0,1)$ depending only on the data, such that 
\[
    v\ge \boldsymbol\gm^{-1} \quad\text{in}\> K_1\times (- \sig, \sig).
\]
This will yield the desired left-hand side inequality in Theorem~\ref{THM:HARNACK} upon rescaling.
\section{Locating the supremum of \texorpdfstring{$v$}{v} in \texorpdfstring{$K_1$}{K1}}\label{S:Harnack:2}
Let $\tau\in(0,1)$ and introduce two quantities
\[
    M_\tau := \sup_{K_\tau} v(\cdot,0),
    \qquad 
    N_\tau := (1-\tau)^{-\be}\quad\text{where}\>\be=\tfrac{p}{q+1-p}.
\]
Note that $M_0=N_0=1$ by definition and $N_\tau\to\infty$ as $\tau\uparrow 1$, whereas $M_\tau$ remains finite.
Hence the equation $M_\tau=N_\tau$ has roots. Denote by $\tau_*\in[0,1)$ the largest one, such that
\[
    M_{\tau_*}=N_{\tau_*}\quad\text{and}\quad M_\tau\le N_\tau\quad\text{for all}\>\tau\ge\tau_*.
\]
By the continuity of $v$, the supremum $M_{\tau_*}$ is achieved at some $\bar{x}\in K_{\tau_*}$. Note that upper-semicontinuity suffices in order to achieve the supremum, cf.~Remark~\ref{Rmk:semicontin}.
Choose $\bar\tau\in(\tau_\ast,1)$ from
\[
    N_{\bar\tau}=(1-\bar\tau)^{-\be}=4 (1-\tau_*)^{-\be},\quad\text{i.e.,}\quad \bar\tau=1-4^{-\frac{1}{\be}}(1-\tau_*)
\]
and set
\[
    2r:=\bar\tau - \tau_*=(1-4^{-\frac{1}{\be}})(1-\tau_*).
\]
Therefore, we have $K_{2r}(\bar{x})\subset K_{\bar\tau}$, $M_{\bar\tau}\le N_{\bar\tau}$, and
\begin{align}\label{sup-v}
    \sup_{K_{\tau_*}} v(\cdot, 0) 
    &= 
    M_{\tau_*}= v(\bar{x},0)\le \sup_{K_{2r}(\bar{x})} v(\cdot, 0)\nonumber \\
    &\le  
    \sup_{K_{\bar\tau}} v(\cdot, 0)
    =M_{\bar\tau}\le N_{\bar\tau}=4 (1-\tau_*)^{-\be}.
\end{align}
\section{Expanding the positivity of \texorpdfstring{$v$}{v}}\label{S:Harnack:3}
Let us consider the cylinder centered at $(\bar{x},0)$, that is,
\[
\widetilde{Q}_{r}(\theta_*):=K_{r}(\bar{x})\times(-\theta_*r^p,\theta_*r^p)\quad\text{where}\quad \theta_*:=(1-\tau_*)^{-\be(q+1-p)}.
\]
Thanks to our choices of $\be=\tfrac{p}{q+1-p}$ and $r$, a simple calculation yields
\begin{equation*}
    \theta_\ast r^p= (1-\tau_\ast)^{-p}r^p =(1-\tau_\ast)^{-p} \big[\tfrac12(1-4^{-\frac{1}{\be}})(1-\tau_*)\big]^p=2^{-p}(1-4^{-\frac{1}{\be}})^p=:c
\end{equation*}
so that
\[
\widetilde{Q}_{2r}(\theta_*)=K_{2r}(\bar{x})\times(-2^p c,2^p c)\quad\text{where}\quad c=2^{-p}(1-4^{-\frac{1}{\be}})^p.
\]
We also note that 
$$
    r=c^{\frac{1}{p}}(1-\tau_\ast).
$$
The next lemma estimates the supremum of $v$ over the cylinder $\widetilde{Q}_{r}(\theta_*)$ by the same quantity as in \eqref{sup-v}.
\begin{lemma}\label{Lm:sup-v}
There exist $\boldsymbol\gm_1>1$ depending only on the data, such that
\[
\sup_{\widetilde{Q}_{r}(\theta_*)} v \le \boldsymbol\gm_1 (1-\tau_*)^{-\be}.
\]
\end{lemma}
\begin{proof}
Let us apply Theorem~\ref{Thm:bd:2} to $v$, over $\widetilde{Q}_{r}(\theta_*)\subset \widetilde{Q}_{2r}(\theta_*)$.
Using \eqref{sup-v} to estimate the integral on the right-hand side, and recalling the definition of $r$, $c$, $\be$ and $\lm$,
a simple calculation yields the claim. Indeed, 
we have
\begin{align*}
    \sup_{\widetilde{Q}_{r}(\theta_*)} v^q 
    &\le 
    \frac{\boldsymbol\gm}{c^{\frac{N}{\lm}}}\bigg[\int_{K_{2r}(\bar{x})} v^q(x,0)\,\dx \bigg]^{\frac{p}{\lm}}
    +
    \boldsymbol\gm\Big(\frac{c}{r^p}\Big)^{\frac{q}{q+1-p}}\\
    &\le 
    \frac{\boldsymbol\gm}{c^{\frac{N}{\lm}}}\Big[r^N (1-\tau_*)^{-q\be} \Big]^{\frac{p}{\lm}}
    +
    \boldsymbol\gm\Big[(1-\tau_\ast)^{-p}\Big]^{\frac{q}{q+1-p}}\\
    &=
    \frac{\boldsymbol\gm}{c^{\frac{N}{\lm}}}\Big[c^\frac{N}{p} (1-\tau_*)^{N-q\be} \Big]^{\frac{p}{\lm}}
    +
    \boldsymbol\gm\Big[(1-\tau_\ast)^{-p}\Big]^{\frac{q}{q+1-p}}\\
    &=
    \boldsymbol \gm (1-\tau_*)^{-q\be},
\end{align*}
with $\boldsymbol\gm$ depending on the data.
\end{proof}

With the supremum estimate of Lemma~\ref{Lm:sup-v} at hand, we estimate the measure of positivity set of $v$.
\begin{lemma}\label{Lm:measure-v}
There exist $\dl$, $\bar{c}$ and $\al$ in $(0,1)$ depending only on the data, such that 
\[
    \Big|\big\{ v(\cdot, t)\ge\bar{c}(1-\tau_*)^{-\be}\big\}\cap K_r(\bar{x})\Big|
    \ge
    \al |K_r|\quad\text{for all}\> t\in[-\dl\theta_*r^p,\dl\theta_*r^p].
\]
\end{lemma}
\begin{proof}
Let us apply Theorem~\ref{Thm:bd:2} to $v$, over $\widetilde{Q}_{\frac12 r}(\delta\theta_*)\subset \widetilde{Q}_{r}(\delta\theta_*)$, to obtain 
\begin{align*}
    (1-\tau_*)^{-\be q} 
    &= 
    v^q(\bar{x},0) \le \sup_{K_{\frac12r}(\bar{x})} v^q(\cdot, 0)\\
    &\le  
    \frac{\boldsymbol\gm}{(\dl\theta_* r^p)^{\frac{N}{\lm}}}\bigg[\int_{K_{r}(\bar{x})} v^q(x,t)\,\dx \bigg]^{\frac{p}{\lm}}+\boldsymbol\gm(\dl\theta_*)^{\frac{q}{q+1-p}}\\
    &=  
    \frac{\boldsymbol\gm}{(\dl\theta_* r^p)^{\frac{N}{\lm}}}\bigg[\int_{K_{r}(\bar{x})} v^q(x,t)\,\dx \bigg]^{\frac{p}{\lm}}
    +
    \boldsymbol\gm\dl^{\frac{q}{q+1-p}}(1-\tau_*)^{-\be q}
\end{align*}
for all $t\in[-\dl\theta_*r^p,\dl\theta_*r^p]$. In the last line we used the definition of $\theta_*$.
Hence, we may choose $\dl$ to satisfy $\boldsymbol\gm\dl^{\frac{q}{q+1-p}}\le\frac12$ and absorb the second term on the right into the left-hand side.
Consequently, recalling also $\theta_* r^p=c$ depends only on $p$ and $q$,  we have
\[
    (1-\tau_*)^{-\be q}\le\boldsymbol\gm_2\bigg[\int_{K_{r}(\bar{x})} v^q(x,t)\,\dx \bigg]^{\frac{p}{\lm}},
\]
for some generic $\boldsymbol\gm_2$ depending on the data. For $\bar{c}\in(0,1)$ to be determined, we continue
to estimate the above integral by
\begin{align*}
    \int_{K_{r}(\bar{x})}& v^q(x,t)\,\dx \\
    &=
    \int_{K_{r}(\bar{x})} v^q(x,t)\chi_{\{v<\bar{c}(1-\tau_*)^{-\be}\}}\,\dx
    + 
    \int_{K_{r}(\bar{x})} v^q(x,t)\chi_{\{v\ge\bar{c}(1-\tau_*)^{-\be}\}}\,\dx\\
    &\le
    \bar{c}^q(1-\tau_*)^{-\be q} (2r)^N
    +
    \boldsymbol\gm_1^q (1-\tau_*)^{-\be q} \Big|\big\{v\ge\bar{c}(1-\tau_*)^{-\be}\big\}\cap K_{r}(\bar{x})\Big|,
\end{align*}
where we used Lemma~\ref{Lm:sup-v} in the last line. We have $\frac{p}{\lm}= \frac{pq}{\lm_q}>1$; indeed since $\lm_q>0$ and $p-1<q$ we have $0<\lm_q= pq-N(q-(p-1))<pq$. Hence $s\mapsto s^\frac{p}{\lm}$ is convex. Substituting this in the last estimate, we have
\begin{align*}
    (1-\tau_*)^{-\be q}
    &\le 
    \boldsymbol\gm_2 2^{\frac{p}{\lm}-1}\big[\bar{c}^q(1-\tau_*)^{-\be q} (2r)^N\big]^{\frac{p}{\lm}}\\
    &\phantom{\le\,}
    + 
    \boldsymbol\gm_2 2^{\frac{p}{\lm}-1}\boldsymbol\gm_1^{q\frac{p}{\lm}} (1-\tau_*)^{-\be q\frac{p}{\lm}} 
    \Big|\big\{v\ge\bar{c}(1-\tau_*)^{-\be}\big\}\cap K_{r}(\bar{x})\Big|^{\frac{p}{\lm}}\\
    &= 
    \boldsymbol\gm_3 \bar{c}^{q\frac{p}{\lm}}(1-\tau_*)^{-\be q}\\
    &\phantom{\le\,}
    + 
    \boldsymbol\gm_2 2^{\frac{p}{\lm}-1}\boldsymbol\gm_1^{q\frac{p}{\lm}} (1-\tau_*)^{-\be q\frac{p}{\lm}} \Big|\big\{v\ge\bar{c}(1-\tau_*)^{-\be}\big\}\cap K_{r}(\bar{x})\Big|^{\frac{p}{\lm}},
\end{align*}
where $\boldsymbol\gm_3$ is another generic constant.
If we choose $\bar{c}$ to verify $\boldsymbol\gm_3 \bar{c}^{q\frac{p}{\lm}}=\frac12$, then we obtain
\[
    (1-\tau_*)^{-\be q}
    \le \overline{\boldsymbol\gm} (1-\tau_*)^{-\be q\frac{p}{\lm}} 
    \Big|\big\{v\ge\bar{c}(1-\tau_*)^{-\be}\big\}\cap K_{r}(\bar{x})\Big|^{\frac{p}{\lm}}
\]
for some generic $\overline{\boldsymbol\gm} $ depending only on the data. From this, recalling the definition of $r$, $\beta$, and $\lm$, we easily deduce the claimed estimate.
\end{proof}

The measure theoretical information from Lemma~\ref{Lm:measure-v} combined with the expansion of positivity in Proposition~\ref{PROP:EXPANSION} yields a pointwise estimate.
\begin{lemma}\label{Lm:expansion-1}
There exists $\eta,\delta\in(0,1)$ depending only on the data, such that
\[
v \ge\eta (1-\tau_*)^{-\be}\quad\text{in}\> K_{2r}(\bar{x})\times\big[-\tfrac12\dl\theta_*r^p,\dl\theta_*r^p\big].
\]
\end{lemma}

Up to now, all the arguments work for the general equation \eqref{Eq:1:1f} with structure conditions \eqref{Eq:1:2}.
To prove the Harnack estimate, we need to further expand the pointwise positivity of $v$ to $K_2(\bar{x})$ and, consequently, to $K_1$.
This is achieved by a comparison argument, which compels us to work with equations of special structure. 

\section{A comparison argument}\label{S:Harnack:3.5}
Recalling $\theta_*r^p=c=2^{-p}(1-4^{-\frac{1}{\be}})^p$
depends only on $p$ and $q$, let us introduce $\sig \in(0,\tfrac12\dl c)$, which will be chosen later. 
Consider the initial-boundary value problem
\begin{equation}\label{Eq:comp-func}
\left\{
\begin{array}{cl}
    \pl_t w^q -\dvg \big(|Dw|^{p-2} Dw\big)=0 
    &\mbox{weakly in $ K_4(\bar{x})\times(-\sig, 1]$,}\\[6pt]
    w = 0 
    &\mbox{on $\pl K_4(\bar{x})\times (-\sig, 1]$,}\\[6pt]
    \dsty w^q(\cdot, -\sig)=\eta(1-\tau_*)^{-N} \boldsymbol \chi_{ K_{2r}(\bar{x}) }(\cdot)
    &
    \mbox{in  $K_4(\bar{x})$.}
\end{array}
\right.
\end{equation}
The existence of this comparison function $w$ can be found in \cite{AL, BDMS-18}; see also Remark~\ref{Rmk:CP1} for more discussion.
The following comparison result for $v$ and $w$ holds true.
\begin{lemma}\label{Lm:exp-comp}
We have $w\le v$ a.e. in $K_4(\bar{x})\times[-\sig, 1]$.
\end{lemma}

\begin{proof}
Since $\be q>N$ thanks to our assumption on $p$ and $q$ in Theorem~\ref{THM:HARNACK}, it is easily seen that $w\le v$ on the parabolic boundary of the domain $K_4(\bar{x})\times(-\sig, 1]$, recalling the lower bound of $v$ in Lemma~\ref{Lm:expansion-1}. Therefore, an application of Proposition~\ref{Prop:CP} yields the claim.
\end{proof}

\begin{remark}\label{Rmk:Harnack-cp}
Although we only deal with the model equation here,  the techniques can be generalized to equations of the form $\pl_t u^q-\dvg\bl{A}(x,u,Du)=0$ under the assumptions \eqref{Eq:1:2-}, \eqref{Eq:CP:mono} and \eqref{Eq:CP:growth}. Indeed, the module needing some care is the above comparison argument. In order for that, a unique solution needs to be constructed for the Dirichlet problem~\eqref{Eq:comp-func} of such general equations, and the comparison Lemma~\ref{Lm:exp-comp} still holds in view of Proposition~\ref{Prop:CP}.
\end{remark}

At this point the desired pointwise positivity of $v$ reduces to a similar property of $w$, that is,
\begin{equation}\label{bound-w}
w\ge\boldsymbol\gm^{-1}\quad \text{in}\> K_2(\bar{x})\times [-\tfrac14\sig, \tfrac14\sig]
\end{equation}
for some $\boldsymbol\gm>1$ and $\sig\in(0,\tfrac12\dl c)$ depending only on the data. Once this is shown, the left-hand side of the Harnack estimate is established as explained in Section~\ref{S:Harnack:1}, after redefining $\frac14\sig$ to be $\sig$.

Indeed, an application of Proposition~\ref{PROP:L1:0} in $K_2(\bar{x})\times[-\sig, \sig]$ yields that
\begin{equation}\label{Eq:L1:w}
\int_{K_1(\bar{x})} w^q(x,-\sig)\,\dx\le\boldsymbol\gm\int_{K_2(\bar{x})} w^q(x,t)\,\dx +\boldsymbol\gm \sig^{\frac{q}{q+1-p}}\quad\text{for all}\> t\in[-\sig, \sig].
\end{equation}
 The left-hand side is estimated by the assigned initial datum in \eqref{Eq:comp-func}, that is,
\[
    \int_{K_1(\bar{x})} w^q(x,-\sig)\,\dx
    =
    \eta(1-\tau_*)^{-N}(4r)^N
    = 
    2^N \eta (1-4^{-\frac1{\be}})^N=:\eta c_o.
\]
Now we may fix $\sig$ by letting
\[
\boldsymbol\gm \sig^{\frac{q}{q+1-p}}=\tfrac12 \eta c_o.
\]
Substituting them back to \eqref{Eq:L1:w}, we have
\[
    \tfrac12 \eta c_o\le\boldsymbol\gm\int_{K_2(\bar{x})} w^q(x,t)\,\dx
    \quad\text{for all}\>t\in[-\sigma,\sigma].
\]
To proceed, we use Theorem~\ref{Thm:bd:2} in $K_2(\bar{x})\times[-\frac14\sig,\frac12\sig]\subset K_4(\bar{x})\times[-\sig,\frac12\sig]$ and the initial datum in \eqref{Eq:comp-func} to estimate
\[
    \sup_{K_{2}(\bar{x})\times[-\frac14\sig,\frac12\sig]} w^q 
    \le 
    \frac{\boldsymbol\gm}{\sig^{\frac{N}{\lm}}}( \eta c_o)^\frac{p}{\lm}
    + 
    \boldsymbol\gm \sig^{\frac{q}{q+1-p}}=\widetilde{\boldsymbol\gm} \eta c_o.
\]
Here, we have taken into account the choice of $\sig$ and the definition of $\lm$, whereas $ \widetilde{\boldsymbol\gm}$ is another generic constant depending on the data.
With this supremum estimate at hand, we observe that, for some $b>0$ to be chosen and for all $t\in[-\tfrac14\sig, \tfrac12\sig]$,
\begin{align*}
    \tfrac12 \eta c_o
    &\le
    \boldsymbol\gm \int_{K_2(\bar{x})} w^q(x,t)\,\dx\\
    &= 
    \boldsymbol\gm\int_{K_2(\bar{x})\cap\{w< b\}} w^q(x,t)\,\dx
    + 
    \boldsymbol\gm\int_{K_2(\bar{x})\cap\{w\ge b\}} w^q(x,t)\,\dx\\
    &\le 
    \boldsymbol\gm b^q |K_2| 
    + 
    \boldsymbol\gm\widetilde{\boldsymbol\gm} \eta c_o 
    \Big| \big\{w(\cdot, t)\ge b\big\}\cap K_2(\bar{x})\Big|.
\end{align*}
If  $b$ is chosen from
\[
\boldsymbol\gm b^q |K_2| =\tfrac14\eta c_o,
\]
we arrive at
\[
    \Big| \big\{w(\cdot, t)\ge b\big\}\cap K_2(\bar{x})\Big|
    \ge 
    \frac1{4\boldsymbol\gm\widetilde{\boldsymbol\gm}}\quad\text{for all}\>t\in[-\tfrac14\sig, \tfrac12\sig].
\]
As a result of this measure estimate and the expansion of positivity in Proposition~\ref{PROP:EXPANSION}, the desired pointwise positivity of $w$ is reached after properly defining $\boldsymbol\gm$ and $\sigma$ in~\eqref{bound-w}. This completes the proof of the first inequality stated in Theorem~\ref{THM:HARNACK}.

\section{Proof of Theorem~\ref{THM:HARNACK} concluded}\label{S:Harnack:4}
Fix $(x_o,t_o)\in E_T$ such that $u(x_o,t_o)>0$. For $\rho>0$, suppose that 
\[
K_{16\rho}(x_o)\times\big(t_o- [u(x_o,t_o)]^{q+1-p}(16\rho)^p,t_o+ [u(x_o,t_o)]^{q+1-p}(16\rho)^p\big]\subset E_T.
\]
The just proven left-hand side inequality in Theorem~\ref{THM:HARNACK} asserts that there exists  some $\boldsymbol\gm>1$ depending on the data, such that
\begin{equation}\label{Eq:Harnack-inf}
 u(x_o,t_o)\le\boldsymbol{\gm} \inf_{K_{2\rho}(x_o) } u (\cdot, t)   
\end{equation}
for all times
\[
t\in\big(t_o- \sig[u(x_o,t_o)]^{q+1-p}(2\rho)^p,t_o+ \sig[u(x_o,t_o)]^{q+1-p}(2\rho)^p\big].
\]
Next we claim that 
\begin{equation}\label{Eq:Harnack-sup}
\sup_{K_{\rho}(x_o)} u(\cdot, t)\le 2 \boldsymbol{\gm} u(x_o,t_o)
\end{equation}
for all times
\[
t\in\big(t_o- \sig[u(x_o,t_o)]^{q+1-p}\rho^p,t_o+ \sig[u(x_o,t_o)]^{q+1-p}\rho^p\big]
\]
and for the same $\boldsymbol\gm$ as in \eqref{Eq:Harnack-inf}. To obtain \eqref{Eq:Harnack-sup}, we first show that, {\it qualitatively} $u$ is actually continuous in the cylinder
\[
K_{2\rho}(x_o)\times\big(t_o- \sig[u(x_o,t_o)]^{q+1-p}(2\rho)^p,t_o+ \sig[u(x_o,t_o)]^{q+1-p}(2\rho)^p\big].
\]
This hinges on the lower bound \eqref{Eq:Harnack-inf} and the upper bound estimate of $u$ from Theorem~\ref{THM:BD:1} within this cylinder. Consequently, the classical continuity results of \cite[Chapters~III, IV]{DB} apply. 
Now we take on the proof of \eqref{Eq:Harnack-sup}. Indeed, let us suppose the contrary of \eqref{Eq:Harnack-sup} were to hold, then by the continuity of $u$, there would exist 
\begin{equation}\label{Eq:point-star}
(x_*,t_*)\in K_\rho(x_o)\times\big(t_o- \sig[u(x_o,t_o)]^{q+1-p}\rho^p,t_o+ \sig[u(x_o,t_o)]^{q+1-p}\rho^p\big]
\end{equation}
satisfying that 
\begin{equation}\label{Eq:contrary}
u(x_*, t_*)=2 \boldsymbol{\gm} u(x_o,t_o).
\end{equation}

For the moment, let us assume that
\begin{equation}\label{Eq:set-incl-star}
K_{16\rho}(x_*)\times\big(t_o-[u(x_*,t_*)]^{q+1-p}(16\rho)^p,t_o+[u(x_*,t_*)]^{q+1-p}(16\rho)^p\big]\subset E_T,
\end{equation}
then by \eqref{Eq:contrary} and applying the just proven left-hand side inequality at $(x_*,t_*)$,  we would have
\begin{equation}\label{Eq:contrary:1}
2 \boldsymbol{\gm} u(x_o,t_o)=u(x_*,t_*)\le \boldsymbol\gm\, \inf_{K_{2\rho}(x_*)} u(\cdot,t)
\end{equation}
for all times
\[
t\in\big(t_*-\sig[u(x_*,t_*)]^{q+1-p}(2\rho)^p,t_*+\sig[u(x_*,t_*)]^{q+1-p}(2\rho)^p\big].
\]
Employing \eqref{Eq:point-star} and \eqref{Eq:contrary}, we find that
\[
(x_o,t_o)\in K_{2\rho}(x_*)\times\big(t_*-\sig[u(x_*,t_*)]^{q+1-p}(2\rho)^p,t_*+\sig[u(x_*,t_*)]^{q+1-p}(2\rho)^p\big].
\]
This would yield a contradiction in \eqref{Eq:contrary:1} since
\[
2 \boldsymbol{\gm} u(x_o,t_o)\le  \boldsymbol{\gm} u(x_o,t_o).
\] 
Hence the claim \eqref{Eq:Harnack-sup} is proven under the assumption \eqref{Eq:set-incl-star}, and the set inclusion \eqref{Eq:set-incl-star} is actually verified if we require, using \eqref{Eq:contrary} and recalling $x_*\in K_\rho(x_o)$, that
\[
K_{32\rho}(x_o)\times\big(t_o- [2\boldsymbol\gm u(x_o,t_o)]^{q+1-p}(16\rho)^p,t_o+ [2\boldsymbol\gm u(x_o,t_o)]^{q+1-p}(16\rho)^p\big]\subset E_T.
\]
Therefore, the proof is 
concluded by suitably redefining the various constants and 
the radius $\rho$.




\section{Proof of Theorem~\ref{THM:HARNACK:0}}\label{S:Thm:0}
Since the argument is analogous to the previous sections, we only sketch it here.
First as in Section~\ref{S:Harnack:1}, we introduce the function $v$ as in \eqref{Eq:u-to-v}, which verifies the equations \eqref{Eq:1:1f} -- \eqref{Eq:1:2} in 
\begin{equation*}
\mathcal{Q}_{\mathcal{M}}\equiv K_8\times\bigg(-8^p \Big[\frac{\mathcal{M}}{u(x_o,t_o)}\Big]^{q+1-p},8^p \Big[\frac{\mathcal{M}}{u(x_o,t_o)}\Big]^{q+1-p} \bigg).
\end{equation*}
This cylinder is transformed from the cylinder in the qualitative assumption \eqref{Eq:set-incl} of Theorem~\ref{THM:HARNACK:0}.
 As before, the goal is to show  $v\ge \boldsymbol\gm^{-1}$ in $K_1\times (- \sig, \sig)$ for some $\boldsymbol\gm>1$ and $\sig\in(0,1)$. In this way, the desired left-hand side inequality in Theorem~\ref{THM:HARNACK:0} is proven upon rescaling.

To this end, Sections~\ref{S:Harnack:2} -- \ref{S:Harnack:3} remain the same except that  $\be>1$ is allowed to be chosen.
The largeness of $\be$, on the other hand, induces the largeness of the height of the cylinder $\widetilde{Q}_{2r}(\theta_*)$.
This is exactly where we need the qualitative assumption \eqref{Eq:set-incl}. 
Indeed, recalling the quantities from Sections~\ref{S:Harnack:2} -- \ref{S:Harnack:3}:
\begin{equation}\label{Eq:theta-r-recall}
\theta_*=(1-\tau_*)^{-\be(q+1-p)},\quad r=\tfrac12(1-4^{-1/\be})(1-\tau_*),
\end{equation}
we may estimate
\begin{align*}
\theta_*(2r)^p&=(1-\tau_*)^{-\be(q+1-p)} (1-4^{-1/\be})^p (1-\tau_*)^p\\
&\le(1-\tau_*)^{-\be(q+1-p)} =\Big[\frac{u(x_o+\rho\bar{x}, t_o)}{u(x_o,t_o)}\Big]^{q+1-p}\le \Big[\frac{\mathcal{M}}{u(x_o,t_o)}\Big]^{q+1-p}.
\end{align*}
Consequently, we have
$$
K_{2r}(\bar{x})\times\big(-\theta_*(2r)^p,\theta_*(2r)^p\big)\subset \mathcal{Q}_{\mathcal{M}},
$$ 
and hence the qualitative assumption \eqref{Eq:set-incl} legitimates the arguments to follow.

The proof departs from Lemma~\ref{Lm:expansion-1}.
Given the quantitative estimate in Lemma~\ref{Lm:expansion-1}, after $n$ applications of Proposition~\ref{PROP:EXPANSION} with $\al=1$, $n\in\mathbb{N}$, we arrive at
\begin{equation}\label{Eq:lower-bd-general}
v(\cdot, t)\ge M\bar{\eta}^{n} \equiv \eta\bar{\eta}^{n} (1-\tau_*)^{-\be}\quad\text{in}\> K_{2^{n+1}r}(\bar{x})
\end{equation}
for all times
\[
-\tfrac12\dl\theta_*r^p+\bar{\dl}\sum_{j=0}^{n-1} (\bar\eta^j M)^{q+1-p}(2^j r)^p<t<\dl\theta_*r^p,
\]
where $\bar\eta$ and $\bar\dl$ are determined in Proposition~\ref{PROP:EXPANSION} in terms of the data only. With no loss of generality, we may take $\bar\dl$ and $\bar\eta$ to be even smaller to ensure that 
\[
\bar{\dl}\sum_{j=0}^{n-1} (\bar\eta^j M)^{q+1-p}(2^j r)^p<\tfrac14\dl\theta_*r^p.
\]
In order for this, recalling the quantities $M$ in \eqref{Eq:lower-bd-general}, and $\theta_*$ in \eqref{Eq:theta-r-recall},
it suffices to require $\bar{\eta}^{q+1-p}2^p<\frac12$ and $\bar\dl<\frac14\dl$, 
provided $\eta^{q+1-p}\le \frac12$, which we can always assume. As a result, the lower bound \eqref{Eq:lower-bd-general} holds for $t\in[-\tfrac14\dl\theta_*r^p, \dl\theta_*r^p]$.

Next, we choose $n$ to satisfy 
\[
1\le 2^nr\le 2
\]
which indicates from \eqref{Eq:lower-bd-general}  that
\begin{equation}\label{Eq:lower-bd-general:1}
v(\cdot, t)\ge \eta(1-\tau_*)^{-\be} \Big(\frac2{r}\Big)^{\frac{\ln\bar\eta}{\ln2}}\quad\text{in}\>K_{2}(\bar{x})
\end{equation}
for all times
\begin{equation}\label{Eq:time-interval}
t\in[-\tfrac14\dl\theta_*r^p, \dl\theta_*r^p].
\end{equation}
Using the definition of $r$ and $\theta_*$ in \eqref{Eq:theta-r-recall}, we estimate the lower bound in \eqref{Eq:lower-bd-general:1} by
\begin{align*}
    \eta(1-\tau_*)^{-\be} \Big(\frac2{r}\Big)^{\frac{\ln\bar\eta}{\ln2}}
    &=
    \eta
    \bigg[\frac4{1-4^{-1/\be}}\bigg]^{\frac{\ln\bar\eta}{\ln2}}
    (1-\tau_*)^{-\be-\frac{\ln\bar\eta}{\ln2}}\\
    &\ge
    \eta\bigg[\frac4{1-4^{-1/\be}}\bigg]^{\frac{\ln\bar\eta}{\ln2}}=:\boldsymbol{\gm}^{-1}
\end{align*}
and the time in \eqref{Eq:time-interval} by
\[
\dl\theta_*r^p=\dl (1-\tau_*)^{-\be(q+1-p)}\big[\tfrac12(1-4^{-1/\be})(1-\tau_*)\big]^p\ge \dl\big[\tfrac12(1-4^{-1/\be}) \big]^p=:4\sig,
\]
provided we choose
\[
    \be
    \ge
    \max\Big\{\frac{p}{q+1-p}, -\frac{\ln\bar\eta}{\ln2} \Big\}.
\]
Thanks to this choice of $\be$, we conclude from \eqref{Eq:lower-bd-general:1} and \eqref{Eq:time-interval} that
\begin{equation*}
v(\cdot, t)\ge \boldsymbol{\gm}^{-1}\quad\text{in}\>K_{1}
\end{equation*}
for all times
\begin{equation*}
t\in[-\sig, \sig].
\end{equation*}
This finishes the proof of the left-hand side inequality in Theorem~\ref{THM:HARNACK:0} upon rescaling, whereas the proof of the right-hand side inequality is the same as in Section~\ref{S:Harnack:4}.

\section{Proof of Corollary~\ref{Cor:Liouville}}\label{Proof:Liouville}
Let us suppose for some $(x_o,t_o)$ we have  $u_o\equiv u(x_o,t_o)\neq0$. Then, according to Theorem~\ref{THM:HARNACK:0} there exists $\bg_{\rm h}>1$ depending on the data $\{p,q,N, C_o, C_1\}$, such that 
$$
\boldsymbol\gm^{-1}_{\rm h} u_o\le u(x,t)\le \boldsymbol\gm_{\rm h} u_o
$$ for all $(x,t)\in\rn\times\rr$. 
Define a new function
\begin{align*}
	\tilde u(x,t)
	:=
	\tfrac1{u_o}\, u\big(x,  u_o^{q+1-p}  t\big)
\end{align*}
which lies in $[ \bg_{\rm h}^{-1} ,\bg_{\rm h}]$ and satisfies a doubly non-linear parabolic equation of the form \eqref{Eq:1:1f} with the ellipticity and growth conditions \eqref{Eq:1:2} in any sub-domain of $\rn\times\rr$.
Substituting $v=\tilde u^q$, equation \eqref{Eq:1:1f} transforms into 
\begin{equation*}
	\partial_t v - 
	\dvg\widetilde{\bl{A}}(x,t,v, Dv)
	=
	0
\end{equation*}
in any sub-domain of $\rn\times\rr$
and $v$ lies in $\big[ \bg_{\rm h}^{-q},\bg_{\rm h}^q\big]$. Here, we have defined
\[
\widetilde{\bl{A}}(x,t,y, \z):=\bl{A}\Big(x,t,\widetilde{y}, \tfrac1q \widetilde{y}^{\frac{1-q}q}\z\Big),\quad \widetilde{y}:=\min\Big\{\max\big\{y,\tfrac12\bg_{\rm h}^{-q}\big\}, 2\bg_{\rm h}^q\Big\}.
\]
Moreover, using conditions \eqref{Eq:1:2} one verifies that
there exist some constants $\widetilde{C}_o$ and $\widetilde{C}_1$  depending on the data and $\bg_{\rm h}$, such that
\[
\widetilde{\bl{A}}(x,t,y, \z)\cdot\z\ge \widetilde{C}_o|\z|^p\quad\text{and}\quad |\widetilde{\bl{A}}(x,t,y, \z)|\le \widetilde{C}_1|\z|^{p-1}
\]
for a.e. $(x,t)\in\rn\times\rr$, any $y\in\rr$ and any $\z\in\rn$.

Therefore, we may apply Chapter III, \S\,1, Theorem~1.1, resp. Chapter IV, \S\,1, Theorem~1.1 in \cite{DB} and 
obtain the oscillation estimate
\[
\osc_{(y,s)+Q_r(\bg_{\rm h}^{q(2-p)})} v\le \bg  \bg_{\rm h}^q \Big(\frac{r}{\rho}\Big)^{\al}
\]
for some $\bg>1$ and $\al\in(0,1)$ depending on $\{p,q,N,C_o,C_1,\bg_{\rm h}\}$, whereas $(y,s)\in\rn\times\rr$ and $0<r<\rho$ are arbitrary. Fixing $r$ and letting $\rho\to\infty$, the desired conclusion follows by arbitrariness of $(y,s)$.

\chapter{Gradient bound for \texorpdfstring{$p$}{p}-Laplace equations with differentiable coefficients}\label{sec:grad-bound}

In this chapter, we provide a sup-bound for the spatial gradient of solutions to parabolic $p$-Laplace type equations with differentiable coefficients. More precisely, let us consider weak solutions to
\begin{equation}\label{non-deg-pde}
  \partial_tu-\mathrm{div}\,
  \big(b(x,t)(\mu^2+|Du|^2)^{\frac{p-2}{2}}Du\big)=0
    \qquad\mbox{in }E_T,
\end{equation}
where $p>1$ and $\mu\in(0,1]$. For the coefficient function $b$,
we assume that the map $E\ni x\mapsto b(x,t)$ is differentiable for
a.e.~$t\in(0,T)$, that $(0,T)\ni t\mapsto b(x,t)$ is measurable for every
$x\in E$, and finally that 
\begin{equation}\label{prop-a-diff}
\left\{
  \begin{array}{c}
    C_o\le b(x,t)\le C_1,\\[6pt]
    |D_xb(x,t)|\le C_2,
  \end{array}
  \right.
\end{equation}
for every $x\in E$ and a.e.~$t\in(0,T)$, with
constants $0<C_o\le C_1$ and $C_2\ge0$. 

For the statement of the gradient bound we distinguish the cases $1<p\le2$ and $p\ge 2$. In the former case we deal with {\it bounded} solutions in order to cover the whole range $1<p\le2$. Note that in the supercritical range $p>\frac{2N}{N+2}$ solutions are locally bounded; see~\cite[Chapter~5, Theorem~3.1]{DB}. In fact, the assumptions of \cite[Chapter~5, Theorem~3.1]{DB} are satisfied for the function $\mathbf{a}(x,t,\xi):=b(x,t)(\mu^2+|\xi|^2)^{\frac{p-2}{2}}\xi$, since \eqref{prop-a-diff}$_1$ implies 
\begin{align*}
   \mathbf{a}(x,t,\xi)\cdot\xi 
   &\ge  
   2^{-\frac{(2-p)_+}{2}}C_o(|\xi|^p-\mu^p),\\
   |\mathbf{a}(x,t,\xi)|
   &\le 
   2^{p-1}C_1(|\xi|^{p-1}+\mu^{p-1})
\end{align*} 
for a.e.$(x,t)\in  E_T$ and every $\xi\in\R^N$. For the derivation of the first inequality in the case $\frac{2N}{N+2}<p<2$, we distinguished between the cases $|\xi|\ge\mu$  and  $|\xi|<\mu$.
Therefore, the boundedness assumption is used only in the subcritical range $1<p\le \frac{2N}{N+2}$. 

Note that in the Chapters~\ref{sec:grad-bound} and~\ref{sec:Schauder} we will employ cylinders of the form $Q_\rho^{(\lambda)}$ defined in \eqref{def:cyl-lambda}.
This notation differs from the one in the previous chapters and the time scaling $\lambda^{p-2}\rho^2$ instead of $\theta\rho^p$ is used here. For estimates concerning a solution $u$, the scaling $\theta\rho^p$, as we have already seen, is the natural one, whereas for the spatial derivatives $Du$, the scaling $\lambda^{p-2}\rho^2$, as we shall see, is the one that leads to homogeneous estimates.

\begin{proposition}\label{PROP:LINFTY}
Let $\mu\in(0,1]$ and $1<p\le 2$. Then, for any bounded weak solution  $u$
to \eqref{non-deg-pde}, under assumptions~\eqref{prop-a-diff}, we have $|Du|\in
L^\infty_{\mathrm{loc}}(E_T)$.  
Moreover, there exist positive constants
$\boldsymbol\gm =\boldsymbol\gm(N,p,C_o,C_1,C_2)$ and $\theta=\theta(N,p)$, such that for every $\epsilon\in(0,1]$
and every cylinder $z_o+Q_{2\rho}^{(\lambda)}\Subset E_T $
we have
\begin{align*}
    &\essup_{z_o+Q_{\rho/2}^{(\lambda)}}|Du|\\
    &\quad \le
    \boldsymbol\gm \epsilon\lambda
    +
    \frac{\boldsymbol\gm\lambda^{\frac12}}{\epsilon^{\theta}}
    \bigg[\Big(\frac{\boldsymbol \omega}{\rho\lambda}\Big)^{\frac{2}{p}}
    +
    \frac{\boldsymbol\omega}{\rho\lambda}+\frac\mu\lambda\bigg]^{\frac{N(2-p)+2p}{4p}}
    \bigg[\biint_{z_o+Q_{2\rho}^{(\lambda)}}\big(\mu^2+|Du|^2\big)^{\frac{p}{2}}\,\dx\dt\bigg]^{\frac{1}{2p}}.
\end{align*}
Here we abbreviated $\boldsymbol \omega:=\osc_{z_o+Q_{2\rho}^{(\lambda)}}u$.
\end{proposition}

In the super-quadratic case $p\ge 2$ the local boundedness assumption for weak solutions can be omitted, because they are automatically locally bounded. This is also reflected in the gradient bound, which does not depend on the oscillation of the solution. The precise statement is
as follows. 

\begin{proposition}\label{PROP:LINFTY:P>2}
Let $\mu\in(0,1]$ and $p\ge 2$. Then, for any  weak solution $u$
to \eqref{non-deg-pde}, under assumptions~\eqref{prop-a-diff}, we have $|Du|\in
L^\infty_{\mathrm{loc}}(E_T)$. Moreover, there exist positive constants
$\boldsymbol\gm =\boldsymbol\gm(N,p,C_o,C_1,C_2)$ and $\theta=\theta(N,p)$, such that for every $\eps\in(0,1]$ and every cylinder $z_o+Q_{\rho}^{(\lambda)}\Subset E_T$ we have 
\begin{align*}
    \essup_{z_o+Q_{\rho/2}^{(\lambda)}}|Du|
    \le
      \boldsymbol\gm \eps\lambda
      +
      \frac{\boldsymbol\gm}{\eps^\theta}\bigg[\lambda^{2-p}
      \biint_{z_o+Q_{\rho}^{(\lambda)}}\big(\mu^2+|Du|^{2}\big)^{\frac{p}{2}}\,\dx\dt\bigg]^{\frac{1}{2}}.
\end{align*}
\end{proposition}

\begin{remark}
The estimates in both Propositions are stable as $p\to2$ and we obtain two versions of sup-estimates for $p=2$. Indeed, from Proposition~\ref{PROP:LINFTY} we get in the limit $p\uparrow 2$ that
\begin{align*}
    \essup_{z_o+Q_{\rho/2}}|Du|
    &\le
    \boldsymbol\gm \epsilon\lambda
    +
    \boldsymbol\gm
    \bigg[
    \frac{\boldsymbol\omega}{\rho}+\mu\bigg]^{\frac12}
    \bigg[\biint_{z_o+Q_{2\rho}}\big(\mu^2+|Du|^2\big)\,\dx\dt\bigg]^{\frac{1}{4}},
\end{align*}
while Proposition~\ref{PROP:LINFTY:P>2} leads in the limit $p\downarrow 2$ to
\begin{align*}
    \essup_{z_o+Q_{\rho/2}}|Du|
    \le
      \boldsymbol\gm \eps\lambda
      +
      \bg\bigg[
      \biint_{z_o+Q_{\rho}}\big(\mu^2+|Du|^{2}\big)\,\dx\dt\bigg]^{\frac{1}{2}}.
\end{align*}
Here, we used $\theta\downarrow 0$ in both cases.
From the viewpoint of dimensions the right-hand sides coincide, since $\big[\frac{\boldsymbol\omega}{\rho}\big]\sim [Du]$.
\end{remark}
 In the next chapter, Propositions~\ref{PROP:LINFTY} and~\ref{PROP:LINFTY:P>2} will be used at two stages as we take on Schauder estimates for an equation with a merely H\"older continuous coefficient $a(x,t)$. First, it is important that the quantitative estimates are employed independent of $C_2$. This will be implemented by freezing the coefficient $a(x,t)$ at some $x_o\in E$, 
 and therefore $C_2=0$. At this stage, all estimates hinge upon that we know $|Du|\in L^{\infty}_{\loc}(E_T)$ {\it a priori}. Second, in the final approximation argument the  coefficient $a(x,t)$ is mollified and a sequence of approximating solutions are constructed. At this stage, we need to know that the spatial gradients of the approximating solutions are locally bounded in order to apply the {\it a priori} estimates. Hence, at the second stage, we only use the propositions {\it qualitatively}. 

\section{Energy estimate for second order derivatives}

The starting point for gradient boundedness  is an energy inequality involving second order weak derivatives of solutions. The proof can be found in~\cite[Proposition~4.1]{BDLS-Tolksdorf}. 
\begin{proposition}[energy inequality for second order derivatives]
  Assume that $p>1$, $\mu\in(0,1]$, and that $u$
  is a weak
  solution to~\eqref{non-deg-pde}, under
  assumptions~\eqref{prop-a-diff}. Assume, in addition, that for some $\alpha\ge0$ the
  solution satisfies 
  $$
  |Du|\in
  L^{p+2\alpha}\big(z_o+Q_S^{(\lambda)}\big)
  \cap L^{2+2\alpha}\big(z_o+Q_S^{(\lambda)}\big)
  $$ 
  on 
  $z_o+Q_S^{(\lambda)}\Subset E_T$.  Then, for every $\frac12S\le R<S$,
  we have 
  \begin{align}\label{energy-est}
    \frac{\lambda^{p-2}}{(1+\alpha)R^2}\essup_{\tau\in t_o+\Lambda_R^{(\lambda)}
    }&
      \mint_{K_{R}(x_o)\times\{\tau\}}\big(\mu^2+|Du|^2\big)^{1+\alpha}\,\dx\\\nonumber
    &\hphantom{\le\,}+
      \biint_{z_o+Q_{R}^{(\lambda)}}
      \big(\mu^2+|Du|^2\big)^{\frac{p-2}2+\alpha}|D^2u|^2\,\dx\dt\\ \nonumber
    &\le
      \frac{\boldsymbol\gm (1+\alpha)}{(S-R)^2}
      \biint_{z_o+Q_{S}^{(\lambda)}}\big(\mu^2+|Du|^2\big)^{\frac{p}2+\alpha}
      \,\dx\dt\\\nonumber
    &\phantom{\le\,} +
      \frac{\boldsymbol\gm \lambda^{p-2}}{(S^2-R^2)}\biint_{z_o +Q_{S}^{(\lambda)}}
      \big(\mu^2 +|Du|^2\big)^{1+\alpha} \,\dx\dt
  \end{align}
  with a constant $\boldsymbol\gm = \boldsymbol\gm (N,p,C_o,C_1,C_2)$. 
\end{proposition}

\section{\texorpdfstring{$L^m$}{Lm}-gradient estimates in the subquadratic case}

Our goal is to implement a Moser iteration procedure to prove local boundedness of the gradient for any $p>1$. However, in the case $p\le\frac{2N}{N+2}$, this is only possible for solutions that satisfy the additional regularity property $|Du|\in L^m_{\mathrm{loc}}(E_T)$ for a sufficiently large power $m>p$. In the subcritical range $p\le\frac{2N}{N+2}$, the latter can be shown for bounded solutions by using ideas from \cite[Chapter VIII, Lemma 4.1]{DB}, see also \cite{DiBenedetto-Friedman3, Choe:1991}. In what follows, we present the proof for all $1<p\le2$ in a unified fashion.

\begin{lemma}\label{lem:Lq-est}
Let $\mu\in(0,1]$ and $1<p\le 2$. Then, for any bounded weak solution $u$
to \eqref{non-deg-pde}, under
assumptions~\eqref{prop-a-diff}, we have $|Du|\in L^m_{\mathrm{loc}}(E_T)$ for every $m>1$.
Moreover, for every $m> p+1$, we have the quantitative estimate
\begin{align}\label{Lq-est}
    \biint_{z_o+Q_\rho^{(\lambda)}}&\big(\mu^2+|Du|^2\big)^{\frac{m}{2}}\dx\dt\nonumber\\
    &\le
    \boldsymbol\gm \bigg[\lambda\Big(\frac{\boldsymbol \omega}{\rho\lambda}\Big)^{\frac{2}{p}}
      +
      \frac{\boldsymbol \omega}{\rho}+\mu\bigg]^{m-p}
      \biint_{z_o+ Q_{2\rho}^{(\lambda)}}\big(\mu^2+|Du|^2\big)^{\frac{p}{2}}\,\dx\dt,
\end{align}
for every cylinder $z_o+Q_{2\rho}^{(\lambda)}\Subset E_T$, where 
$\boldsymbol \omega:=\osc_{z_o+Q^{(\lambda)}_{2\rho}}u$,
and the dependencies of the constant are given by $\boldsymbol\gm = \boldsymbol\gm (N,p,m,C_o,C_1,C_2)$. 
\end{lemma}

\begin{proof}
  We omit $z_o$ from the notation for simplicity, 
   consider a nested pair of cylinders
  $Q_{r}^{(\lambda)}\subset Q_{S}^{(\lambda)}\subset E_T$ with
  radii $\rho\le r<S\le 2\rho$ and let
  $R:=\frac12(r+S)$. Next we choose a cutoff function
$\zeta\in C^\infty_0(K_R,[0,1])$ with
  $\zeta\equiv1$ in $K_{r}$ and $|D\zeta|\le\frac2{R-r}=\frac{4}{S-r}$ in
  $K_R$. 
  
  To start with, we need $|Du|\in
  L^2_{\mathrm{loc}}(E_T)$. This can be achieved by adapting
  the difference quotient technique as illustrated in the appendix
  of~\cite{Choe:1991}. We briefly sketch the 
  argument by giving the formal computations. First, we note that an
  integration by parts yields 
  \begin{align}\label{Lp->L2}
    \biint_{Q_r^{(\lambda)}}&|Du|^2\,\dx\dt
      \le
      \biint_{Q_R^{(\lambda)}}|Du|^2\zeta\,\dx\dt\\\nonumber
    &=
      -\biint_{Q_R^{(\lambda)}}[u-(u)_R(t)]\Delta u\,\zeta\,\dx\dt
      -\biint_{Q_R^{(\lambda)}}[u-(u)_R(t)]Du\cdot D\zeta\,\dx\dt\\\nonumber
    &\le  
    \boldsymbol\gm(N)\boldsymbol \omega\,
      \biint_{Q_R^{(\lambda)}}|D^2u|\,\dx\dt
    +
      \frac{4\boldsymbol \omega}{S-r}
      \biint_{Q_R^{(\lambda)}}|Du|\,\dx\dt.
  \end{align}  
 Here $(u)_R(t)$ is  the integral average of $u(\cdot,t)$ over $K_R(x_o)$ as defined in \eqref{def:slice-wise-mean}. 
  For the estimate of the first integral on the right-hand side, we
  apply H\"older's inequality and~\eqref{energy-est} with $\alpha=0$
  to obtain
  \begin{align*}
    \biint_{Q_R^{(\lambda)}}&|D^2u|\,\dx\dt\\
    &\le
      \bigg[\biint_{Q_R^{(\lambda)}}\big(\mu^2+|Du|^2\big)^{\frac{p-2}2}|D^2u|^2\dx\dt\bigg]^{\frac12}
      \bigg[\biint_{Q_R^{(\lambda)}}\big(\mu^2+|Du|^2\big)^{\frac{2-p}2}\dx\dt\bigg]^{\frac12}\\
    &\le
    \frac{\boldsymbol\gm}{S-R}\bigg[
    \biint_{Q_{S}^{(\lambda)}}
    \big(\mu^2+|Du|^2\big)^{\frac{p}2}\dx\dt +
      \lambda^{p-2}\biint_{Q_{S}^{(\lambda)}}
      \big(\mu^2 +|Du|^2\big)\dx\dt\bigg]^{\frac12}\\\nonumber
    &\qquad\qquad\qquad
      \cdot \bigg[\biint_{Q_R^{(\lambda)}}\big(\mu^2+|Du|^2\big)^{\frac{p}{2}}\dx\dt\bigg]^{\frac{2-p}{2p}},
  \end{align*}
  where we also used $S^2-R^2\ge(S-R)^2$ for the last
  step.
  We use this to bound the right-hand side of~\eqref{Lp->L2} and then
  apply Young's inequality to get
  \begin{align*}
    \biint_{Q_r^{(\lambda)}}&|Du|^2\,\dx\dt\\
    &\le
    \tfrac12\biint_{Q_S^{(\lambda)}}|Du|^2\,\dx\dt+\tfrac12\mu^2
    +
      \frac{\boldsymbol\gm \boldsymbol \omega^2\lambda^{p-2}}{(S-r)^2}\bigg[\biint_{Q_R^{(\lambda)}}\big(\mu^2+|Du|^2\big)^{\frac{p}{2}}\dx\dt\bigg]^{\frac{2-p}{p}}\\
    &\phantom{\le\,}
    +
     \frac{\boldsymbol\gm\,\boldsymbol \omega}{S-r} 
     \bigg[\biint_{Q_R^{(\lambda)}}\big(\mu^2+|Du|^2\big)^{\frac{p}{2}}\dx\dt\bigg]^{\frac{1}{p}}, 
  \end{align*}
  for all radii with $\rho\le r<S\le2\rho$. 
  In view of the iteration Lemma~\ref{lem:Giaq}, we can re-absorb the first term on the
  right-hand side and obtain
  \begin{align*}
    \biint_{Q_\rho^{(\lambda)}}|Du|^2\,\dx\dt
    &\le
      \frac{\boldsymbol\gm \boldsymbol\omega^2\lambda^{p-2} }{\rho^2}\bigg[\biint_{Q_R^{(\lambda)}}\big(\mu^2+|Du|^2\big)^{\frac{p}{2}}\,\dx\dt\bigg]^{\frac{2-p}{p}}\\
     &\phantom{\le\,}
     +
     \frac{\boldsymbol \gm\boldsymbol \omega}{\rho} 
     \bigg[\biint_{Q_R^{(\lambda)}}\big(\mu^2+|Du|^2\big)^{\frac{p}{2}}\,\dx\dt\bigg]^{\frac{1}{p}}, 
  \end{align*}
  where the right-hand side is finite by assumption. 
  Hence $|Du|\in L^2_{\mathrm{loc}}(E_T)$ modulo proper difference quotient.
  
  Next, we assume {\it a priori} that $|Du|\in
  L^{m-p+1}_{\mathrm{loc}}(E_T)$ for some $m\ge p+1$. 
  Similarly as in~\eqref{Lp->L2}, we compute via an integration by parts and estimate
  \begin{align*}
    \biint_{Q_R^{(\lambda)}}&(\mu^2+|Du|^2)^{\frac{m}{2}}\zeta\,\dx\dt\\
    &=
      \biint_{Q_R^{(\lambda)}}Du\cdot Du\big(\mu^2+|Du|^2\big)^{\frac{m-2}{2}}\zeta\,\dx\dt
      +
      \mu^2\biint_{Q_R^{(\lambda)}}\big(\mu^2+|Du|^2\big)^{\frac{m-2}{2}}\zeta\,\dx\dt\\
    &=
      -\biint_{Q_R^{(\lambda)}}[u-(u)_R(t)]\,\mathrm{div}\big[Du\big(\mu^2+|Du|^2\big)^{\frac{m-2}{2}}\big]\zeta\,\dx\dt\\
    &\phantom{=\,}  
    -
     \biint_{Q_R^{(\lambda)}}[u-(u)_R(t)] \big(\mu^2+|Du|^2\big)^{\frac{m-2}{2}}Du\cdot D\zeta\,\dx\dt\\
     &\phantom{=\,}
      +
      \mu^2\biint_{Q_R^{(\lambda)}}\big(\mu^2+|Du|^2\big)^{\frac{m-2}{2}}\zeta\,\dx\dt\\
    &\le
      \boldsymbol\gm m\boldsymbol \omega\,
      \biint_{Q_R^{(\lambda)}}\big(\mu^2+|Du|^2\big)^{\frac{m-2}{2}}|D^2u|\,\dx\dt\\
    &\phantom{=\,}+
      \frac{\boldsymbol\gm \boldsymbol\omega}{S-r}
      \biint_{Q_R^{(\lambda)}}\big(\mu^2+|Du|^2\big)^{\frac{m-2}{2}}|Du|\,\dx\dt\\
     &\phantom{=\,}+
      \mu^2\biint_{Q_R^{(\lambda)}}\big(\mu^2+|Du|^2\big)^{\frac{m-2}{2}}\dx\dt.
  \end{align*}
In particular,  we used the inequality
  \begin{align*}
    \Big|\mathrm{div}\Big[Du\big(\mu^2+|Du|^2\big)^{\frac{m-2}{2}}\Big]\Big|
    &\le
    \sqrt{N}(m-1) |D^2u| \big(\mu^2+|Du|^2\big)^{\frac{m-2}{2}}.
\end{align*}
In order to estimate the integral containing the second order derivatives
that occurs on the right-hand side,
we apply H\"older's inequality and then~\eqref{energy-est} with
$\alpha=\frac12(m-p-1)$, with the result
  \begin{align*}
    &\biint_{Q_R^{(\lambda)}}\big(\mu^2+|Du|^2\big)^{\frac{m-2}{2}}|D^2u|\,\dx\dt\\
    &\quad\le
      \bigg[
      \biint_{Q_R^{(\lambda)}}\big(\mu^2+|Du|^2\big)^{\frac{m-3}{2}}|D^2u|^2\,\dx\dt
      \bigg]^{\frac12}
      \bigg[
     \biint_{Q_R^{(\lambda)}}\big(\mu^2+|Du|^2\big)^{\frac{m-1}{2}}\dx\dt\bigg]^{\frac12}\\
   &\quad\le
     \frac{\boldsymbol\gm }{S-R}\bigg[
      \biint_{Q_{S}^{(\lambda)}}\big(\mu^2+|Du|^2\big)^{\frac{m-1}{2}}
      \dx\dt\bigg]^{\frac12}\\\nonumber
    & \quad\phantom{\le\,}
      \cdot\bigg[
      m
      \biint_{Q_{S}^{(\lambda)}}\big(\mu^2+|Du|^2\big)^{\frac{m-1}{2}}\,\dx\dt
      +
      \lambda^{p-2}\biint_{Q_{S}^{(\lambda)}}
      \big(\mu^2+|Du|^2\big)^{\frac{m-p+1}{2}} \,\dx\dt\bigg]^{\frac12}\\
    &\quad\le
      \frac{\boldsymbol\gm m^{\frac12}}{S-R}\biint_{Q_{S}^{(\lambda)}}\big(\mu^2+|Du|^2\big)^{\frac{m-1}{2}}\,\dx\dt\\\nonumber
    &\quad\phantom{\le\,} 
      +\frac{\boldsymbol \gm \lambda^{\frac{p-2}{2}}}{S-R}
      \bigg[\biint_{Q_{S}^{(\lambda)}}\!
      \big(\mu^2 +|Du|^2\big)^{\frac{m-p+1}{2}}
      \dx\dt\bigg]^{\frac12}
      \bigg[\biint_{Q_{S}^{(\lambda)}}\!\big(\mu^2+|Du|^2\big)^{\frac{m-1}{2}}
      \dx\dt\bigg]^{\frac12},
  \end{align*}
with a constant $\boldsymbol\gm =\boldsymbol\gm (N,p,C_o,C_1,C_2)$. 
Inserting this in the above inequality, we have
\begin{align*}
    \biint_{Q_r^{(\lambda)}}\big(\mu^2+|Du|^2\big)^{\frac{m}{2}}\dx\dt
    &
    \le
    \biint_{Q_R^{(\lambda)}}\big(\mu^2+|Du|^2\big)^{\frac{m}{2}}\zeta\,\dx\dt
    \le
    \mathbf{I}+\mathbf{II}
\end{align*}
with
\begin{align*}
    \mathbf{I}
    &:=
    \frac{\boldsymbol\gm m\lambda^{\frac{p-2}{2}}\boldsymbol \omega}{S-r}
      \bigg[\biint_{Q_{S}^{(\lambda)}}\!\!
      \big(\mu^2 +|Du|^2\big)^{\frac{m-p+1}{2}}
      \dx\dt\bigg]^{\frac12}
      \bigg[\biint_{Q_{S}^{(\lambda)}}\!\!\big(\mu^2+|Du|^2\big)^{\frac{m-1}{2}}
      \dx\dt\bigg]^{\frac12},
\end{align*}
and
\begin{align*}
    \mathbf{II}
    &:=
      \bigg[\frac{\boldsymbol\gm m^{\frac32}\boldsymbol \omega}{S-r} + \mu\bigg]
      \biint_{Q_{S}^{(\lambda)}}\big(\mu^2+|Du|^2\big)^{\frac{m-1}{2}}\,\dx\dt.
\end{align*}
Note that the right-hand side is finite for $m=p+1$ since
$|Du|\in L^p(E_T)\cap L^2_{\mathrm{loc}}(E_T)$. Therefore,
for this particular choice of $m$, the above estimate implies
$|Du|\in L^{p+1}_{\mathrm{loc}}(E_T)$. By
iteratively applying the preceding estimate with
$m_\ell=p+1+\ell(p-1)$ for $\ell=0,1,2,\ldots$ we obtain
$|Du|\in L^{m}_{\mathrm{loc}}(E_T)$ for every $m>1$. 

Therefore, it remains
to prove the asserted bound~\eqref{Lq-est}. To this end, we consider
the preceding inequality for an arbitrary $m>p+1$ and estimate
the two terms on the right-hand side 
separately. For the estimate of the first term, we first apply Young's
inequality with exponent $2$. In a second step,  after properly splitting the exponents, we apply Young's
inequality once with exponents $\frac{m-p}{p-1}$
and $\frac{m-p}{m-2p+1}$, and once with exponents $m-p$ and
$\frac{m-p}{m-p-1}$,
which is possible by our assumption $m>p+1$. In this
way, we obtain the bound
\begin{align*}
    \mathbf{I}
    &\le
    \Big(\frac{\boldsymbol\gm m\lambda^{\frac{p-2}{2}}\boldsymbol\omega}{S-r}\Big)^{\frac{2(p-1)}{p}}
    \biint_{Q_{S}^{(\lambda)}}
    \big(\mu^2+|Du|^2\big)^{\frac{(p-1)p}{2(m-p)}}
    \big(\mu^2 +|Du|^2\big)^{\frac{m(m-2p+1)}{2(m-p)}}\,\dx\dt\\
    &\phantom{\le\,}
    +
    \Big(\frac{\boldsymbol\gm m\lambda^{\frac{p-2}{2}}\boldsymbol\omega}{S-r}\Big)^{\frac{2}{p}}
    \biint_{Q_{S}^{(\lambda)}}
    \big(\mu^2+|Du|^2\big)^{\frac{p}{2(m-p)}}
    \big(\mu^2+|Du|^2\big)^{\frac{m(m-p-1)}{2(m-p)}}\,\dx\dt\\
    &\le
    \tfrac14
    \biint_{Q_{S}^{(\lambda)}}\big(\mu^2+|Du|^2\big)^{\frac{m}{2}}\,\dx\dt
    +
    \Big(\frac{\boldsymbol\gm m\lambda^{\frac{p-2}{2}}\boldsymbol \omega}{S-r}\Big)^{\frac{2(m-p)}{p}}
    \biint_{Q_{S}^{(\lambda)}}\big(\mu^2+|Du|^2\big)^{\frac{p}{2}}\,\dx\dt.
\end{align*}
For the estimate of $\mathbf{II}$, we apply once more Young's inequality, now with
exponents $m-p$ and $\frac{m-p}{m-p-1}$, and acquire the result
\begin{align*}
    \mathbf{II}
    &=
    \bigg[\frac{\boldsymbol\gm m^{\frac32}\boldsymbol\omega}{S-r}+\mu\bigg]
    \biint_{Q_{S}^{(\lambda)}}
    \big(\mu^2+|Du|^2\big)^{\frac{p}{2(m-p)}}
    \big(\mu^2+|Du|^2\big)^{\frac{m(m-p-1)}{2(m-p)}}\,\dx\dt\\
    &\le
    \tfrac14
     \biint_{Q_{S}^{(\lambda)}}\big(\mu^2+|Du|^2\big)^{\frac{m}{2}}\,\dx\dt
     +
     \bigg[\frac{\boldsymbol\gm m^{\frac32}\boldsymbol\omega}{S-r}+\mu\bigg]^{m-p}
     \biint_{Q_{S}^{(\lambda)}}\big(\mu^2+|Du|^2\big)^{\frac{p}{2}}\,\dx\dt. 
\end{align*}
Collecting the estimates, we arrive at 
\begin{align*}
    \biint_{Q_r^{(\lambda)}}&\big(\mu^2+|Du|^2\big)^{\frac{m}{2}}\dx\dt\\
    &\le
      \tfrac12\biint_{Q_{S}^{(\lambda)}}\big(\mu^2+|Du|^2\big)^{\frac{m}{2}}\dx\dt\\
    &\phantom{\le\,}+
      \boldsymbol\gm    
      \Bigg[\bigg(\frac{ m \lambda^{\frac{p-2}{2}}\boldsymbol\omega}{S-r}\bigg)^{\frac{2(m-p)}{p}}
      +
      \bigg(\frac{m^{\frac32}\boldsymbol \omega}{S-r}+\mu\bigg)^{m-p}\Bigg]
      \biint_{Q_{2\rho}^{(\lambda)}}\big(\mu^2+|Du|^2\big)^{\frac{p}{2}}\,\dx\dt,
\end{align*}
for all radii $r,S$ with $\rho\le r<S\le2\rho$, and a constant $\boldsymbol\gm =\boldsymbol\gm (N,p,C_o,C_1,C_2)$. 
Using the iteration Lemma~\ref{lem:Giaq}, we can discard the first term on the
right-hand side  and obtain
\begin{align*}
    \biint_{Q_\rho^{(\lambda)}}&\big(\mu^2+|Du|^2\big)^{\frac{m}{2}}\dx\dt\\
    &\le 
    \boldsymbol\gm   
    \Bigg[\bigg(\frac{m\lambda^{\frac{p-2}{2}}\boldsymbol\omega}{\rho}\bigg)^{\frac{2(m-p)}{p}}
      +
      \bigg(\frac{m^{\frac32}\boldsymbol\omega}{\rho}+\mu\bigg)^{m-p}\Bigg]
      \biint_{Q_{2\rho}^{(\lambda)}}\big(\mu^2+|Du|^2\big)^{\frac{p}{2}}\dx\dt\\
    &\le
      \boldsymbol\gm  
      \bigg[\lambda\bigg(\frac{m\boldsymbol\omega}{\lambda\rho}\bigg)^{\frac{2}{p}}
      +
      \frac{m^{\frac32}\boldsymbol\omega}{\rho}+\mu\bigg]^{m-p}
      \biint_{Q_{2\rho}^{(\lambda)}}\big(\mu^2+|Du|^2\big)^{\frac{p}{2}}\dx\dt,
\end{align*}
where $\boldsymbol\gm =\boldsymbol\gm (N,p,C_o,C_1,C_2)$. This yields
the asserted estimate~\eqref{Lq-est}, and completes
the proof of the lemma. 
\end{proof}

\section{Proof of Proposition~\ref{PROP:LINFTY}}

For the sake of simplicity, we write $Q_{2\rho}^{(\lambda)}$ instead of
$z_o+Q_{2\rho}^{(\lambda)}$ and consider $0<R<S\le\rho$. 
We plan to use the energy inequality~\eqref{energy-est} without the integral averages.
First observe that, since $p\le2$, the integrand of the first integral on the right-hand side admits by  Young's inequality that
\begin{align*}
    \big(\mu^2+|Du|^2\big)^{\frac p2+\alpha}
    &\le
     (\epsilon\lambda)^{p-2}\big(\mu^2+|Du|^2\big)^{1+\alpha}+(\epsilon\lambda)^{p+2\alpha}\\
    &=
      (\epsilon\lambda)^{p-2}\big[\big(\mu^2+|Du|^2\big)^{1+\alpha}+(\epsilon\lambda)^{2+2\alpha}\big]
\end{align*}
for every $\epsilon\in(0,1]$. 
As a result, the energy inequality ~\eqref{energy-est} gives
\begin{align*}
    \frac{1}{1+\alpha}&\essup_{\tau\in (-\lambda^{2-p}R^2,0]}
      \int_{K_{R}\times\{\tau\}}\big(\mu^2+|Du|^2\big)^{1+\alpha}\,\dx\\\nonumber
    &\phantom{\le\,}+
      \iint_{Q_{R}^{(\lambda)}}
      \big(\mu^2+|Du|^2\big)^{\frac{p-2}2+\alpha}|D^2u|^2\,\dx\dt\\ \nonumber
    &\le
      \boldsymbol \gm (1+\alpha)\frac{(\epsilon\lambda)^{p-2}}{(S-R)^2}\iint_{Q_{S}^{(\lambda)}}
      \Big[\big(\mu^2 +|Du|^2\big)^{1+\alpha} + (\epsilon\lambda)^{2+2\alpha}\Big]\,\dx\dt
\end{align*}
for every $\alpha\ge0$ and with a  constant $\boldsymbol\gm =\boldsymbol\gm (N,p,C_o,C_1,C_2)$.  With the function
$$
    w:=\big(\mu^2+|Du|^2\big)^{\frac{p+2\alpha}{4}},
$$
the above estimate gives
\begin{align}\label{energy-Moser-w}
    \frac{1}{1+\alpha}
    &\essup_{\tau\in(-\lambda^{2-p}R^2,0]}
        \int_{K_{R}\times\{\tau\}}w^{\frac{4+4\alpha}{p+2\alpha}}\,\dx     
        +
        \frac{1}{(p+2\alpha)^2}\iint_{Q_{R}^{(\lambda)}}|Dw|^2\,\dx\dt\\\nonumber
    &\le 
    \boldsymbol\gm (1+\alpha)\frac{(\epsilon\lambda)^{p-2}}{(S-R)^2}
    \iint_{Q_{S}^{(\lambda)}}\Big[w^{\frac{4+4\alpha}{p+2\alpha}}+(\epsilon\lambda)^{2+2\alpha}\Big]\,\dx\dt.
\end{align}
This is the starting point of Moser's iteration scheme.
  In order to iterate this energy estimate, we consider an exponent
  $p_o>\max\{\frac{N}{2}(2-p),p+1\}$, and a parameter
  $\sigma\in[\frac12,1)$. For any $i\in\N_0$, we introduce the abbreviations
  \begin{equation}\label{Eq:k_n}
    \left\{
      \begin{array}{c}
	\displaystyle\rho_i=\sigma\rho+\frac{(1-\sigma)\rho}{2^i},\quad 
	\displaystyle\tilde{\rho}_i=\frac{\rho_i+\rho_{i+1}}{2},\\[6pt]
	\displaystyle K_{i}=K_{\rho_i},\quad
	Q_{i}=Q_{\rho_i}^{(\lambda)},\\[6pt]
	\widetilde{K}_{i}=K_{\tilde{\rho}_i},\quad 
	\widetilde{Q}_i=Q_{\widetilde\rho_i}^{(\lambda)},\\[6pt]
  	\kappa=1+\frac{2}N,\quad
  	\alpha_o=\frac{p_o}{2}-1,\quad 
	\alpha_{i+1}=\kappa\alpha_i+\frac2N-\frac{2-p}{2},\\[6pt]
   	\displaystyle p_i=2+2\alpha_{i},\quad 
	q_i=\frac{4+4\alpha_i}{p+2\alpha_i}.
      \end{array}
    \right.
  \end{equation}
  Note that the choice of sequences in the last two lines implies
  $\alpha_i=\frac\nu4\kappa^i-1+\tfrac{N(2-p)}{4}$, where
  $\nu:=2p_o-N(2-p)>0$, and
  \begin{equation}\label{pi+1}
    p_{i+1}=(p+2\alpha_i)\frac{N+q_i}N
  \end{equation}
  for every $i\in\N_0$.   Now we choose a standard cutoff function
  $\zeta\in C^\infty(\widetilde Q_i,[0,1])$ that vanishes on
  $\partial_{p}\widetilde{Q}_i$ and equals one in $Q_{i+1}$, such that
  $|D\zeta|\le 2^{i+2}/((1-\sigma)\rho)$.  An application of the Sobolev
  embedding (cf. Lemma~\ref{lem:Sobolev} with $p=2$, $m=q_i$), together with
  the energy estimate~\eqref{energy-Moser-w} for the cylinders
  $\widetilde{Q}_i\subset Q_i$ and $\alpha_i$ in place of $\alpha$, yields that
  \begin{align*}
    \iint_{Q_{i+1}}&\big(\mu^2+|Du|^2\big)^{\frac{p_{i+1}}{2}}\,\dx\dt
                     \le
                     \iint_{\widetilde{Q}_i}(w\zeta)^{2\frac{N+q_i}N}\,\dx\dt\\\nonumber
                   &\le \boldsymbol\gm\,\iint_{\widetilde{Q}_i}|D(w\zeta)|^{2}\,\dx\dt\left[\essup_{\tau\in(-\lambda^{2-p}\tilde{\rho}^2_i,0]}
                     \int_{\widetilde{K}_i\times\{\tau\}}(w\zeta)^{q_i}\,\dx\right]^{\frac{2}N}\\\nonumber
                   &\le \boldsymbol\gm 
                    (p+2\alpha_i)^{3+\frac4N}
                     \left[\frac{2^{2i}(\epsilon\lambda)^{p-2}}{(1-\sigma)^2\rho^2}\right]^{\kappa}
                     \left[\iint_{Q_i}\big[w^{q_i}+(\epsilon\lambda)^{2+2\alpha_i}\big]\,\dx\dt\right]^{\kappa}\\\nonumber  
                   &\le \boldsymbol\gm \boldsymbol b^{i\kappa}\left[\frac{(\epsilon\lambda)^{p-2}}{(1-\sigma)^2\rho^2}\right]^{\kappa}\left[\iint_{Q_i}
                   \big[\big(\mu^2+|Du|^2\big)^{\frac{p_i}{2}}+(\epsilon\lambda)^{p_i}\big]\,\dx\dt\right]^{\kappa},
  \end{align*}
  for some $\boldsymbol b=\boldsymbol b(N)\ge1$ and $\boldsymbol \gm =\boldsymbol\gm (N,p,p_o,C_o,C_1,C_2)\ge1$. 
  In the transition from the second to the third line, i.e.~when applying the energy estimate \eqref{energy-Moser-w},
  the constant from the second line must be increased by the factor $\boldsymbol\gm^\kappa$, where $\boldsymbol\gm$ is the constant from \eqref{energy-Moser-w}.
  In the passage from the penultimate to the last line, we used that  $p+2\alpha_i\le \boldsymbol \gm (N,p,p_o)\kappa^i$ for every
  $i\in\N_0$. In the second display the Sobolev constant $\boldsymbol\gm$  depends on $N$ and $q_i$. Since
  $q_i= 2+\frac{2(2-p)}{p+2\al_i}\in [2,4]$ we can bound  this constant by $\boldsymbol\gm (N)$.
  Next, we divide the above iterative  inequality by
  $|Q_{i+1}|$. Noting that $|Q_i|\le 2^{N+2}|Q_{i+1}|$  for
  any $i\in\N_0$, we obtain
  \begin{align}\label{rev-hoelder}
    \biint_{Q_{i+1}}&\big(\mu^2+|Du|^2\big)^{\frac{p_{i+1}}{2}}\,\dx\dt\\\nonumber
    &\le
      \frac{\boldsymbol\gm \boldsymbol b^{i\kappa}\epsilon^{(p-2)\kappa}\lambda^{p-2}}{(1-\sigma)^{2\kappa}}\left[\biint_{Q_i}\big[
      \big(\mu^2+|Du|^2\big)^{\frac{p_i}{2}}+(\epsilon\lambda)^{p_i}
      \big]\,\dx\dt\right]^{\kappa}.
  \end{align}
  Moreover, the definition of $p_i$ and the facts $\epsilon\le 1$ and $p\le2$ imply
  \begin{align*}
    (\epsilon\lambda)^{p_{i+1}}
    =
    (\epsilon\lambda)^{p-2}(\epsilon\lambda)^{p_i\kappa}
    \le
    \epsilon^{(p-2)\kappa}\lambda^{p-2}(\epsilon\lambda)^{p_i\kappa}.
  \end{align*}
  Hence, with the abbreviation
  \begin{equation*}
    \boldsymbol Y_i=\left[\biint_{Q_i}\big[ (\mu^2+|Du|^2)^\frac{p_i}2+(\epsilon\lambda)^{p_i}\big]\,\dx\dt\right]^{\frac1{p_i}},
  \end{equation*}
  we infer from~\eqref{rev-hoelder} that
  \begin{equation}\label{iteration-1}
    \boldsymbol Y_{i+1}^{p_{i+1}}\le \big(A \boldsymbol b^{i}
    \boldsymbol Y_i^{p_i}\big)^{\kappa}
    \qquad\mbox{for every }i\in\N_0,
  \end{equation}
  where
  \[
    A:=\frac{\boldsymbol\gm \epsilon^{p-2}\lambda^{\frac{N(p-2)}{N+2}}}{(1-\sigma)^2},
  \]
  for a suitable constant $\boldsymbol\gm =\boldsymbol\gm (N,p,C_o,C_1,C_2)\ge1$.  Iterating
  \eqref{iteration-1}, we obtain 
  \begin{align}\label{iteration-2-1}
    \boldsymbol Y_{i}^{p_i}
    &\le
    \prod_{j=1}^{i} A^{\kappa^{i-j+1}}
    \prod_{j=1}^{i} \boldsymbol b^{j\kappa^{i-j+1}}
    \boldsymbol Y_o^{p_o\kappa^{i}}
  \end{align}
  for every $i\in\N$. In \eqref{iteration-2-1} we take the power
  $\frac{1}{p_i}$ on both sides and use the facts
  \begin{align*}
    \lim_{i\to\infty}\frac{\frac{\nu}{2}(\kappa^i-1)}{p_i}
    =
    \lim_{i\to\infty}\frac{\frac{\nu}{2}(\kappa^i-1)}{\frac\nu2\kappa^i+\tfrac{N(2-p)}{2}}
    =1,
  \end{align*}
  as well as
  \begin{align*}
    \lim_{i\to\infty}\frac{\kappa^i}{p_i}=\frac2\nu.
  \end{align*}
  Consequently, due to Lemma~\ref{lem:A} we obtain
  \begin{align*}
    \limsup_{i\to\infty}\boldsymbol Y_{i}
    &\le
      \limsup_{i\to\infty}\prod_{j=1}^{i} A^{\frac{\kappa^{i-j+1}}{\frac{\nu}{2}(\kappa^i-1)}}
      \prod_{j=1}^{i} \boldsymbol b^{\frac{j\kappa^{i-j+1}}{\frac{\nu}{2}(\kappa^i-1)}}
      Y_o^{\frac{2p_o}{\nu}}\\
    &\le
      A^{\frac{2\kappa}{\nu(\kappa-1)}}
      \boldsymbol b^{\frac{2\kappa^2}{\nu(\kappa-1)^2}}
      \boldsymbol Y_o^{\frac{2p_o}{\nu}}\\
    &=
      A^{\frac{N+2}{\nu}}\boldsymbol b^{\frac{(N+2)^2}{2\nu}} 
      \boldsymbol Y_o^{\frac{2p_o}{\nu}},
  \end{align*}
  from which we deduce
  \[
    \essup_{Q_{\sigma\rho}^{(\lambda)}}|Du| \le
    \left[\frac{\boldsymbol\gm \epsilon^{\frac{(N+2)(p-2)}{2}}\lambda^{\frac{N(p-2)}{2}}}{(1-\sigma)^{N+2}}\biint_{Q_{\rho}^{(\lambda)}}\big[
    (\mu^2+|Du|^2)^\frac{p_o}2+(\epsilon\lambda)^{p_o}
    \big]\dx\dt\right]^{\frac{2}{\nu}}
  \]
  with a constant $\boldsymbol\gm =\boldsymbol\gm (N,p,p_o,C_o,C_1,C_2)$.
  We choose $\sigma=\frac12$ and estimate the right-hand side by means
  of Lemma~\ref{lem:Lq-est} with $m=p_o$. Moreover, to simplify the notation, we define
  $$
    \boldsymbol \tau:= \Big(\frac{\boldsymbol \omega}{\rho\lambda}\Big)^{\frac{2}{p}}
      +
      \frac{\boldsymbol\omega}{\rho\lambda}+\frac{\mu}{\lambda}.
  $$
In this way we get
  \begin{align*}
    &\essup_{Q_{\rho/2}^{(\lambda)}}|Du|\\
    &\quad\le
    \boldsymbol\gm \epsilon^{\frac{(N+2)(p-2)}{\nu}}\lambda^{\frac{N(p-2)}{\nu}}
    \bigg[
      (\lambda\boldsymbol\tau)^{p_o-p}
      \biint_{Q_{2\rho}^{(\lambda)}}\big(\mu^2+|Du|^2\big)^{\frac{p}{2}}\,\dx\dt
      +
      (\epsilon\lambda)^{p_o}
      \bigg]^{\frac{2}{\nu}}\\
    &\quad\le
      \boldsymbol\gm\epsilon^{\frac{2(p-2)}{\nu}+1}\lambda
      +
      \boldsymbol\gm \epsilon^{\frac{(p-2)(N+2)}{\nu}}
      \lambda^{\frac{\nu-2p}{\nu}}\boldsymbol\tau^{\frac{2(p_o-p)}{\nu}}
      \bigg[\biint_{Q_{2\rho}^{(\lambda)}}\big(\mu^2+|Du|^2\big)^{\frac{p}{2}}\,\dx\dt\bigg]^{\frac{2}{\nu}}.
  \end{align*}
  At this stage, we fix 
  $p_o:=\frac{N}{2}(2-p)+2p$, so that $\nu=4p$.
  Note that this exponent satisfies the requirement 
  $p_o>\max\{\frac{N}{2}(2-p),p+1\}$. With this choice, the preceding
  inequality takes the form
  \begin{align*}
    &\essup_{Q_{\rho/2}^{(\lambda)}}|Du|\\
    &\quad\le 
      \boldsymbol\gm \epsilon^{\frac{3p-2}{2p}}\lambda
      +
      \boldsymbol\gm \epsilon^{-\frac{(2-p)(N+2)}{4p}}
       \lambda^{\frac12}\boldsymbol\tau^{\frac{N(2-p)+2p}{4p}}
      \bigg[\biint_{Q_{2\rho}^{(\lambda)}}\big(\mu^2+|Du|^2\big)^{\frac{p}{2}}\,\dx\dt\bigg]^{\frac{1}{2p}}
  \end{align*}
  with $\boldsymbol\gm =\boldsymbol\gm (N,p,C_o,C_1,C_2)$. After replacing $\epsilon$ by
  $\epsilon^{\frac{2p}{3p-2}}$ we infer the asserted estimate with the exponent $\theta=\frac{(2-p)(N+2)}{6p-4}$.

\section{Proof of Proposition~\ref{PROP:LINFTY:P>2} }

Similarly to the proof of Proposition~\ref{PROP:LINFTY} in the case $p\le2$
we implement a Moser iteration procedure. Since $p>2$, Young's inequality yields 
\begin{align*}
    \lambda^{p-2}(\mu^2+|Du|^2)^{1+\alpha}
    &\le
      \eps^{2-p}(\mu^2+|Du|^2)^{\frac{p}{2}+\alpha}+\eps^{2+2\alpha}\lambda^{p+2\alpha}\\
    &=
      \eps^{2-p}\big[(\mu^2+|Du|^2)^{\frac{p}{2}+\alpha}+(\eps\lambda)^{p+2\alpha}\big]
\end{align*}
for every $\epsilon\in(0,1]$. Therefore, in this case the energy estimate~\eqref{energy-est} implies
\begin{align*}
    \frac{1}{1+\alpha}
    &\essup_{\tau\in(-\lambda^{2-p}R^2,0]}
    \int_{K_{R}\times\{\tau\}}\big(\mu^2+|Du|^2\big)^{1+\alpha}\dx\\\nonumber
    &\phantom{\le\,}+
    \iint_{Q_{R}^{(\lambda)}}
    \big(\mu^2+|Du|^2\big)^{\frac{p-2}2+\alpha}|D^2u|^2\,\dx\dt\\ \nonumber
    &\le
      \frac{\boldsymbol\gm (1+\alpha)}{\eps^{p-2}(S-R)^2}
      \iint_{Q_{S}^{(\lambda)}}
      \Big[\big(\mu^2 +|Du|^2\big)^{\frac{p}{2}+\alpha} + (\eps\lambda)^{p+2\alpha}\Big]\dx\dt
\end{align*}
for every $\alpha\ge0$ and with a constant $\boldsymbol\gm =\boldsymbol\gm (N,p,C_o,C_1,C_2)$. We re-write this inequality in terms of  
\[ 
    w:=\big(\mu^2+|Du|^2\big)^{\frac{p+2\alpha}{4}}
\]
and obtain
\begin{align}\label{energy-Moser-w-2}
    \frac{1}{1+\alpha}
    &\essup_{\tau\in(-\lambda^{2-p}R^2,0]}
    \int_{K_{R}\times\{\tau\}}w^{\frac{4+4\alpha}{p+2\alpha}}\,\dx
    +
    \frac{1}{(p+2\alpha)^2}\iint_{Q_{R}^{(\lambda)}}|Dw|^2\,\dx\dt\\\nonumber
    &\le 
    \frac{\boldsymbol\gm (1+\alpha)}{\eps^{p-2}(S-R)^2}
    \iint_{Q_{S}^{(\lambda)}}\big(w^2+(\eps\lambda)^{p+2\alpha}\big)\dx\dt.
\end{align}
We continue to use the notation $\rho_i,$ $\tilde\rho_i$, $B_i$, $Q_i$,
$\widetilde B_i$, $\widetilde Q_i$, and $\kappa$ introduced in
\eqref{Eq:k_n}. Unlike in \eqref{Eq:k_n}, however, we now define 
\begin{equation*}
    \left\{
    \begin{array}{c}
  	 \alpha_o=\tfrac12(p_o-p),\quad 
	   \alpha_{i+1}=\kappa\alpha_i+\frac2N,\\[6pt]
   	    \displaystyle p_i=p+2\alpha_{i},\quad 
	   q_i=\frac{4+4\alpha_i}{p+2\alpha_i},
    \end{array}
    \right.
\end{equation*}
for an arbitrary $p_o\ge p$, to be later chosen. With these choices we have  \eqref{pi+1}
and $\alpha_i=\vartheta\kappa^i-1$ with $\vartheta:=\tfrac12(p_o-p)+1$
for every $i\in\N_0$. Next, we choose a  cutoff function
$\zeta\in C^\infty(\widetilde Q_i,[0,1])$ that vanishes on
$\partial_{p}\widetilde{Q}_i$ and satisfies $\zeta\equiv1$ in
$Q_{i+1}$ and $|D\zeta|\le 2^{i+2}/((1-\sigma)\rho)$.  The Sobolev
inequality in Lemma~\ref{lem:Sobolev} 
(with $m=q_i$ and $p=2$) together with the energy 
estimate~\eqref{energy-Moser-w-2} with $\alpha_i$ in place of $\alpha$ implies
\begin{align*}
    \iint_{Q_{i+1}}
    &\big(\mu^2+|Du|^2\big)^{\frac{p_{i+1}}{2}}\,\dx\dt
    \le
    \iint_{\widetilde{Q}_i}(w\zeta)^{2\frac{N+q_i}N}\,\dx\dt\\\nonumber
    &\le 
    \boldsymbol\gm (N,q_i)
    \,\iint_{\widetilde{Q}_i}|D(w\zeta)|^{2}\,\dx\dt
    \left[
    \essup_{\tau\in(t_o-\lambda^{2-p}\tilde{\rho}^2_i,t_o)}
    \int_{\widetilde{K}_i\times\{\tau\}} (w\zeta)^{q_i}\,\dx\right]^{\frac{2}N}\\\nonumber
    &\le \boldsymbol\gm 
    (p+2\alpha_i)^{3+\frac4N}
    \left[\frac{2^{2i}}{\eps^{p-2}(1-\sigma)^2\rho^2}\right]^{\kappa}
    \left[\iint_{Q_i}
    \big[w^2+(\eps\lambda)^{p+2\alpha_i}\big]
    \,\dx\dt\right]^{\kappa}\\\nonumber  
    &\le \boldsymbol\gm  
    \left[\frac{\boldsymbol b^i}{\eps^{p-2}(1-\sigma)^2\rho^2}\right]^{\kappa}
    \left[\iint_{Q_i}\big[(\mu^2+|Du|^2)^{\frac{p_i}{2}}+(\eps\lambda)^{p_i}\big]\,\dx\dt\right]^{\kappa},
\end{align*}
for some $\boldsymbol b=\boldsymbol b(N)\ge1$ and $\boldsymbol\gm = \boldsymbol\gm (N,p,p_o,C_1,C_2)\ge1$, where we relied on the fact $\alpha_i<\vartheta\kappa^i$ for every $i\in\N_0$. Moreover, we used that $q_i\in [\frac4p, 2]$, in order to get the constant $\boldsymbol\gm (N,q_i)$ uniformly bounded (with respect to $i$). Dividing both sides by $|Q_{i+1}|$, we get
\begin{align}\label{rev-hoelder-p>2}
    \biint_{Q_{i+1}}&\big(\mu^2+|Du|^2\big)^{\frac{p_{i+1}}{2}}\,\dx\dt\\\nonumber
    &\le
      \frac{\boldsymbol\gm \boldsymbol b^{i\kappa}\lambda^{(2-p)\frac{2}{N}}}{\eps^{(p-2)\kappa}(1-\sigma)^{2\kappa}}\left[\biint_{Q_i}\big[\big(\mu^2+|Du|^2\big)^{\frac{p_i}{2}}+(\eps\lambda)^{p_i}\big]\,\dx\dt\right]^{\kappa}.
\end{align}
Moreover, by definition of $p_i$ and since $0<\eps\le1$, $p>2$, and
$N\ge2$, we have 
\begin{align*}
    (\eps\lambda)^{p_{i+1}}
    =
    (\eps\lambda)^{(2-p)\frac2N}(\eps\lambda)^{p_i\kappa}
    \le
    \eps^{(2-p)\kappa}\lambda^{(2-p)\frac2N}(\eps\lambda)^{p_i\kappa}.
\end{align*}
Therefore, abbreviating 
\begin{equation*}
    \boldsymbol Y_i=\left[\biint_{Q_i}\big[(\mu^2+ |Du|^2)^\frac{p_i}2
    +(\eps\lambda)^{p_i}\big]\,\dx\dt\right]^{\frac1{p_i}},
\end{equation*}
we can rewrite~\eqref{rev-hoelder-p>2} as 
\begin{equation}\label{iteration-1-p>2}
     \boldsymbol Y_{i+1}^{p_{i+1}}
     \le \big(A  \boldsymbol b^{i}
     \boldsymbol Y_i^{p_i}\big)^{\kappa}
    \qquad\mbox{for every }i\in\N_0,
\end{equation}
where
\[
    A:=\frac{ \boldsymbol\gm \lambda^{\frac{2(2-p)}{N+2}}}{\eps^{p-2}(1-\sigma)^2},
\]
for a suitable constant $ \boldsymbol\gm = \boldsymbol\gm (N,p,p_o,C_1,C_2)\ge1$.  Iterating \eqref{iteration-1-p>2} yields 
\begin{align*}
     \boldsymbol Y_{i}^{p_i}
    &\le
    \prod_{j=1}^{i} A^{\kappa^{i-j+1}}
    \prod_{j=1}^{i}  \boldsymbol b^{j\kappa^{i-j+1}}
     \boldsymbol Y_o^{p_o\kappa^{i}}
\end{align*}
for every $i\in\N$. Here, we take the power
$\frac{1}{p_i}$ on both sides and use the fact $p_i\ge
2\alpha_i\ge 2\vartheta(\kappa^i-1)$ and then 
Lemma~\ref{lem:A} to obtain
\begin{align*}
    \limsup_{i\to\infty} \boldsymbol Y_{i}
    &\le
      \limsup_{i\to\infty}
      \prod_{j=1}^{i} A^{\frac{\kappa^{i-j+1}}{2\vartheta(\kappa^i-1)}}
      \prod_{j=1}^{i}  
      \boldsymbol b^{\frac{j\kappa^{i-j+1}}{2\vartheta(\kappa^i-1)}}
       \boldsymbol Y_o^{\frac{p_o\kappa^i}{2\vartheta(\kappa^i-1)}}\\
    &\le
      A^{\frac{N+2}{4\vartheta}} 
      \boldsymbol b^{\frac{(N+2)^2}{8\vartheta}}
       \boldsymbol Y_o^{\frac{p_o}{2\vartheta}}.
\end{align*}
This implies
\begin{align*}
    \essup_{Q_{\sigma\rho}^{(\lambda)}}&\big(|Du|^2+\mu^2\big)^\frac12\\
    &\le
    \bigg[\frac{ \boldsymbol \gm\lambda^{2-p}}{\eps^{(p-2)\frac{N+2}{2}}(1{-}\sigma)^{N+2}}\biint_{Q_{\rho}^{(\lambda)}}\!
    \big[(\mu^2 +|Du|^2)^\frac{p_o}2+(\eps\lambda)^{p_o}\big]\,\dx\dt\bigg]^{\frac{1}{2\vartheta}}
\end{align*}
for every $\eps\in(0,1]$, with a constant $ \boldsymbol\gm = \boldsymbol\gm (N,p,p_o,C_1,C_2)$. In particular, for the choice $p_o=p$, the right-hand side is finite, which implies $|Du|\in L^\infty_{\loc}(E_T)$. For the
derivation of the asserted quantitative estimate, we choose
$p_o:=(p-2)\frac{N+2}{2}+p$ and obtain  
\begin{align*}
    &\essup_{Q_{\sigma\rho}^{(\lambda)}}\big(|Du|^2+\mu^2\big)^\frac12\\
    &\ \ 
    \le
     \frac{ \boldsymbol\gm \lambda\eps^{\frac{p}{2\vartheta}}}{(1-\sigma)^{\frac{N+2}{2\vartheta}}}\\
     &\ \   \phantom{\le\,}
     +
     \bigg[\frac{ \boldsymbol\gm \lambda^{2-p}}{\eps^{(p-2)\frac{N+2}{2}}(1{-}\sigma)^{N+2}}
     \essup_{Q_{\rho}^{(\lambda)}} \big(|Du|^2+\mu^2\big)^\frac{p_o-p}2
     \biint_{Q_{\rho}^{(\lambda)}}\! \big(|Du|^{2}+\mu^2\big)^\frac{p}{2}
     \dx\dt\bigg]^{\frac{1}{2\vartheta}}\\
    &\ \ 
    \le
    \tfrac12\essup_{Q_{\rho}^{(\lambda)}}\big(|Du|^2+\mu^2\big)^\frac12
    +
      \frac{\boldsymbol\gm \eps^{\frac{p}{2\vartheta}}\lambda}{(1-\sigma)^{\frac{N+2}{2\vartheta}}}\\
    &\ \  \phantom{\le\,}
    +
      \bigg[\frac{\boldsymbol\gm \lambda^{2-p}}{\eps^{(p-2)\frac{N+2}{2}}(1-\sigma)^{N+2}}
      \biint_{Q_{\rho}^{(\lambda)}}\big(|Du|^{2}+\mu^2\big)^\frac{p}{2}\,\dx\dt\bigg]^{\frac{1}{2}},
\end{align*}
for every $\sigma\in[\tfrac12,1)$. 
In the last step, we used Young's inequality with exponents
$\frac{2\vartheta}{p_o-p}$ and $\vartheta$. In view of
the Iteration Lemma~\ref{lem:Giaq},
we can omit the supremum  on the right-hand side.
Therefore, we arrive at the bound 
\begin{align*}
    \essup_{Q_{\rho/2}^{(\lambda)}}|Du|
    \le
      \boldsymbol\gm\eps^{\frac{p}{2\vartheta}}\lambda
      +
      \bigg[\frac{\boldsymbol\gm \lambda^{2-p}}{\eps^{(p-2)\frac{N+2}{2}}}
      \biint_{Q_{\rho}^{(\lambda)}}\big(|Du|^{p}+\mu^p\big)\,\dx\dt\bigg]^{\frac{1}{2}},
\end{align*}
for every $\eps\in(0,1]$, with a constant $\boldsymbol\gm =\boldsymbol\gm (N,p,C_o,C_1,C_2)$.
    This yields the desired estimate after replacing $\eps$ by
    $\eps^{\frac{2\vartheta}{p}}$, with the parameter $\theta=\frac{2\vartheta(p-2)(N+2)}{4p}$.

\chapter{Schauder estimates for \texorpdfstring{$p$}{p}-Laplace type equations}\label{sec:Schauder}

In this chapter we present the proof of Theorem~\ref{theorem:schauder:intro}.  
In contrast to the previous chapter, we consider equations of $p$-Laplace type with coefficients that are merely H\"older continuous. More precisely, we study bounded weak solutions to the equation
\begin{equation}\label{p-laplace}
  \partial_tu-\Div\big(a(x,t)(\mu^2+|Du|^2)^{\frac{p-2}{2}}Du\big)=0
  \qquad\mbox{in $E_T$},
\end{equation}
for parameters $p>1$ and $\mu\in[0,1]$, with a function $a\in L^\infty(E_T)$ satisfying
\begin{equation}\label{prop-a}
\left\{
  \begin{array}{c}
    C_o\le a(x,t)\le C_1,\\[6pt]
    |a(x,t)-a(y,t)|\le C_1|x-y|^\alpha,
  \end{array}
  \right.
\end{equation}
for a.e.~$x,y\in E$ and $t\in(0,T)$, with a H\"older exponent
$\alpha\in(0,1)$ and structural constants $0<C_o\le C_1$.

\section{Energy estimate}

Weak solutions to \eqref{p-laplace} satisfy a basic energy estimate,  which follows by testing the equation  against $(u-\xi)\z^p$,
 with $\xi\in\R$ and a suitable cut-off function $\z$, and modulo a Steklov averaging procedure. The assumption~\eqref{prop-a}$_2$, i.e.~H\"older continuity with respect to the spatial variable, is not used in this context. 
\begin{proposition}

Assume that $p>1$, $\mu\in [0,1]$, and that $u$ is a weak
solution to~\eqref{p-laplace}, under assumption~\eqref{prop-a}$_1$.
Then, for any $\xi\in\R$, any cylinder $z_o+Q_S^{(\lambda)}\Subset E_T$, and any
$\frac12 S\le R<S$ we have
\begin{align}\label{energy-est-zero-order}
    \frac{\lambda^{p-2}}{R^2}\essup_{\tau\in t_o+\Lambda_R^{(\lambda)}}
    &
      \mint_{K_{R}(x_o)\times\{\tau\}}|u-\xi|^2\,\dx+
      \biint_{z_o+Q_{R}^{(\lambda)}}
      \big(\mu^2+|Du|^2\big)^{\frac{p}2}\,\dx\dt\\ \nonumber
    &\le
      \boldsymbol\gm
      \biint_{z_o+Q_{S}^{(\lambda)}}\bigg[
      \frac{|u-\xi|^p}{(S-R)^p} +\lambda^{p-2}\frac{|u-\xi|^2}{S^2-R^2} +\mu^p\bigg]
 \,\dx\dt
  \end{align}
  with a constant $\boldsymbol\gm = \boldsymbol\gm (N,p,C_o,C_1)$. 
\end{proposition}

\section{An a priori gradient H\"older estimate}

In this section we assume the gradient is locally bounded and derive its H\"older regularity by using the results from
\cite{BDLS-Tolksdorf} as {\it a priori} estimates.
As pointed out in \cite[Remark~1.4]{BDLS-Tolksdorf}, 
 the proof of \cite[Theorem~1.3]{BDLS-Tolksdorf} remains valid for any exponent $p>1$
provided the gradient is locally bounded. Moreover, if the coefficient function is independent of the spatial variable $x$, the smallness conditions on the radius of the cylinder appearing in \cite[Propositions~5.1 and 5.2]{BDLS-Tolksdorf} are unnecessary here; see also~\cite[Chapter~9, Propositions~1.1 and ~1.2]{DB}. 
More precisely, the arguments from \cite{BDLS-Tolksdorf}
yield the following Campanato-type estimate.

\begin{lemma}\label{lem:campanato}
Let $\mu\in(0,1]$, $p>1$, and suppose that
assumptions~\eqref{prop-a-diff}$_1$ are satisfied for coefficients $b(t)$
independent of the spatial variable $x$. 
There exist constants 
$\boldsymbol\gm >0$ and $\beta\in(0,1)$, depending
only on the data $N,p,C_o$, and $C_1$, such that whenever $u$
is a weak solution to equation~\eqref{non-deg-pde} and 
$z_o+Q_{2\rho}^{(\lambda)}\Subset E_T$ is a cylinder with
$\rho>0$ 
and $\lambda\ge \mu$, such that
\begin{equation}\label{intrinsic}
    \sup_{z_o+Q_{2\rho}^{(\lambda)}}\big(\mu^2+|Du|^2\big)\le\lambda^2,
\end{equation}
then, for all radii $r\in(0,\rho]$, we have  
\begin{align}\label{campanato}
	\biint_{z_o+Q_r^{(\lambda)}}\big|Du-(Du)_{z_o,r}^{(\lambda)}\big|^p \,\dx\dt
	\le
	\boldsymbol\gm \Big(\frac{r}{\rho}\Big)^{\beta p}\lambda^{p},
\end{align}
where $(Du)_{z_o,r}^{(\lambda)}$ stands for the integral average of $Du$ on
$z_o+Q_r^{(\lambda)}$.
\end{lemma}
\begin{proof}
The proof of \cite[Theorem~1.3]{BDLS-Tolksdorf} is presented in Section~5.2; the estimates are stated for cylinders of the form
\begin{equation*}
    (x_o,t_o)+ \widetilde Q_r^{(\lambda)}
    :=
    K_r(x_o)\times\big(t_o-(\mu^2+\lambda^2)^{\frac{2-p}{2}}r^2,t_o\big].
\end{equation*}
In particular, with $\lambda_\mu:=\sqrt{\lambda^2-\mu^2}$ we have 
\begin{equation}\label{two-cylinders}
    (x_o,t_o)+ \widetilde Q_r^{(\lambda_\mu)}
    =
    (x_o,t_o)+ Q_r^{(\lambda)}
\end{equation}
for every $r\in(0,2\rho]$,
and assumption \eqref{intrinsic} can be rewritten in the form
\begin{equation*}
     \sup_{z_o+ \widetilde Q_{2\rho}^{(\lambda_\mu)}}|Du|\le\lambda_\mu.
\end{equation*}
%
This corresponds to the sup-estimate in {\it Step~1} on p.~27 of \cite{BDLS-Tolksdorf}.
Then, we can repeat the
arguments in {\it Steps~2 -- 4} on pp.~27 -- 29 and arrive at \cite[Eqn. (5.16)]{BDLS-Tolksdorf} with $\lambda_\mu$
in place of $\lambda$. Since the coefficients $b(t)$ are independent of $x$, the upper bound $\rho<\rho_o$ on the radius can be avoided in the present situation. 
Therefore, with the Lebesgue representative $\Gamma_{z_o}$ of $Du(z_o)$, 
we obtain the Campanato type estimate 
\begin{equation*}
	\biint_{z_o+\widetilde Q_r^{(\lambda_\mu)}}|Du-\Gamma_{z_o}|^p \,\dx\dt
	\le
	\boldsymbol\gm  \Big(\frac{r}{\rho}\Big)^{\beta p} 
	\lambda_\mu^{p}
        \le
	\boldsymbol\gm \Big(\frac{r}{\rho}\Big)^{\beta p} 
	\lambda^{p}
\end{equation*}
for every $r\in(0,\rho]$, with constants
$\boldsymbol\gm >0$ and $\beta\in(0,1)$ depending on
$N,p,C_o,C_1$. In view of~\eqref{two-cylinders},
this implies 
\begin{align*}
  \biint_{z_o+Q_r^{(\lambda)}}\big|Du-(Du)_r^{(\lambda)}\big|^p \,\dx\dt
  \le
    \boldsymbol\gm (p)\biint_{z_o+Q_r^{(\lambda)}}|Du-\Gamma_{z_o}|^p \,\dx\dt
  \le
    \boldsymbol\gm \Big(\frac{r}{\rho}\Big)^{\beta p}\lambda^{p},
\end{align*}
where we write $(Du)_r^{(\lambda)}$ instead of $(Du)_{z_o,r}^{(\lambda)}$ for simplicity. This yields the claim~\eqref{campanato}.
\end{proof}

\section{A comparison estimate} 
Let $p>1$ and let $v$ be a weak solution to the parabolic equation
\begin{equation}\label{EqA}
    \left\{
    \begin{array}{c}
    v\in L^p\big(0,T;W^{1,p}(E)\big)\cap C^0\big([0,T];L^2(E)\big),\\[6pt]
    \partial_tv-\Div \mathbf A(x,t,Dv)=0\,\,  \mbox{weakly in $E_T$.}
     \end{array}
    \right.
\end{equation}
Next, let us consider a weak solution of the following Cauchy-Dirichlet problem
\begin{equation}\label{EqB}
    \left\{
    \begin{array}{c}
        w\in L^p\big(0,T;W^{1,p}(E)\big)\cap C^0\big([0,T];L^2(E)\big),\\[6pt]
         \partial_tw-\Div \mathbf B(x,t,Dw)=0 \, \, \mbox{weakly in $E_T$,}\\[6pt]
         \mbox{$w=v$ on $\partial_p E_T$,}
    \end{array}
    \right.
\end{equation}
where the boundary datum is taken in the sense of Definition~\ref{def:weak}.

The vector-fields $\mathbf A, \mathbf B\colon E_T\times\R^N\to\R^N$ are assumed to be Carath\'eodory functions satisfying the monotonicity and growth conditions
\begin{equation}\label{p-growth}
    \left\{
    \begin{array}{c}
      \big(\mathbf B(x,t,\xi)-\mathbf B(x,t,\eta) \big)\cdot(\xi-\eta)
      \ge
      C_o(\mu^2+|\xi|^2+|\eta|^2)^{\frac{p-2}{2}}|\xi-\eta|^2,\\[6pt]
      \big|\mathbf A(x,t,\xi)\big|,\ \big|\mathbf B(x,t,\xi)\big| 
      \le C_1(\mu^2+|\xi|^2)^{\frac{p-1}{2}},
    \end{array}
    \right.
\end{equation}
for a.e.~$(x,t)\in E_T$ and every $\xi,\eta\in\R^N$, with
constants $0<C_o\le C_1$. 

The following comparison estimate, true for all  $p > 1$, is the main result of
this section.

\begin{lemma}\label{lem:Comp2}
Let $v$ be a weak solution to the parabolic equation \eqref{EqA} and $w$ be a weak solution to \eqref{EqB}, with structure conditions 
\eqref{p-growth}. Then, in the case $1<p<2$ we have the comparison estimate
  \begin{align}\label{VglAbsch2}
  \iint_{E_T}&|Dv-Dw|^p\,\dx\dt\\\nonumber
  &\le
  \boldsymbol\gm  
  \big\|\mathbf A(\cdot,Dv)-\mathbf B(\cdot,Dv)\big\|_{L^{p'}(E_T)}^{p'}\\\nonumber
  &\phantom{\le\,}+
  \boldsymbol\gm 
  \|\mathbf A(\cdot,Dv)-\mathbf B(\cdot,Dv)\|_{L^{p'}(E_T)}^{p}
    \bigg[\iint_{E_T}\big(\mu^2+|Dv|^2\big)^{\frac{p}2}\dx\dt\bigg]^{2-p}.
\end{align}
In the case $p\ge2$, this holds without the last term, i.e. 
\begin{align}\label{comparison:p>2}
  \iint_{E_T}|Dv-Dw|^p\,\dx\dt
  \le
  \boldsymbol\gm \big\|\mathbf A(\cdot,Dv)-\mathbf B(\cdot,Dv)\big\|_{L^{p'}(E_T)}^{p'}.
\end{align}
Furthermore, for any $p>1$ we have the energy estimate
\begin{align}\label{EAbsch2}
    \iint_{E_T}&\big(\mu^2+|Dw|^2\big)^{\frac{p}{2}}\dx\dt \\\nonumber
    &\le
      \boldsymbol\gm 
      \iint_{E_T}\big(\mu^2+|Dv|^2\big)^{\frac{p}2}\dx\dt
     +
      \boldsymbol\gm \|\mathbf B(\cdot,Dv)-\mathbf A(\cdot,Dv)\|_{L^{p'}(E_T)}^{p'}.
\end{align}
In all estimates, the constant $\boldsymbol\gm $ depends only on $p$ and $C_o$.
\end{lemma}

\begin{proof} 
According to the notion of solution in Definition~\ref{def:weak}, writing down the integral identity \eqref{Eq:weak-form-1} for $v$ and $w$ in its Steklov-form, subtracting one from another, we arrive at
\begin{equation}\label{Eq:weak:v-w}
     \iint_{E_T}\Big[
  -[v-w]_h\pl_t\z
  +
  \big[\mathbf A(x,t,Dv)-\mathbf B(x,t,Dw)\big]_h\cdot 
  D\z
  \Big]\, \dx\dt=0
\end{equation}
for any 
$
\z\in  W^{1,2}(0,T; L^{2}(E))\cap L^p (0,T; W^{1,p}_{0}(E))
$
such that $\z(\cdot, 0)=0=\z(\cdot, T)$. Now we choose $\z=[
  v-w]_h\psi_\varep$, where $\psi_\varep$ is the Lipschitz function in $[0,T]$ that equals $1$ in $[t_1,t_2]\subset(0,T)$, vanishes outside $[t_1-\varep, t_2+\varep]$ for some $0<\varep<\min\{t_1, T-t_2\}$ and is linearly interpolated otherwise. Thereby, $[
  v-w]_h$ is the Steklov average of $v-w$ and $0<h<T-(t_2+\eps)$.
There are two terms to estimate in \eqref{Eq:weak:v-w}.  The term containing the time derivative is computed as
  \begin{align*}
  -\iint_{E_T}&
  [v-w]_h\pl_t([
  v-w]_{h}\psi_\varep)\,\dx\dt \\
  &=
  -\tfrac12\iint_{E_T}\pl_t[v-w]_{h}^2\psi_{\varep}\,\dx\dt -
  \iint_{E_T}[v-w]_{h}^2 \pl_t\psi_\varep\,\dx\dt\\
  &=
  -\tfrac12 \iint_{E_T}[v-w]_{h}^2 \pl_t\psi_\varep\,\dx\dt.
\end{align*}
We let $h\downarrow0$ with the aid of Lemma~\ref{lm:Stek}~(i) and then let $\varep\downarrow0$, to obtain that the last term converges to
\[
\tfrac12\int_{E\times\{t_2\}} (v-w)^2\,\dx -\tfrac12\int_{E\times\{t_1\}} (v-w)^2\,\dx. 
\]
Note that the first integral gives a non-negative contribution and can be discarded, while the second integral tends to $0$ as $t_1\downarrow0$ according to the initial condition in the problem~\eqref{EqB}.
Moreover, the remaining term in \eqref{Eq:weak:v-w} tends to
\[
\iint_{E\times(t_1,t_2)}
  \big[\mathbf A(x,t,Dv)-\mathbf B(x,t,Dw)\big]\cdot 
  D(v-w)\, \dx\dt
\]
if we first let $h\downarrow0$ with the aid of Lemma~\ref{lm:Stek}~(iv) and then let $\varep\downarrow0$. Consequently, after letting
$t_1\downarrow0$ and $t_2\uparrow T$, we have
\begin{align*}
  \iint_{E_T}&
   \big[\mathbf B(x,t,Dv)-\mathbf B(x,t,Dw)\big]\cdot(Dv-Dw)\,\dx\dt\\
  &\le
  \iint_{E_T}
   \big[\mathbf B(x,t,Dv)-\mathbf A(x,t,Dv)\big]\cdot(Dv-Dw)\,\dx\dt.
\end{align*}
Using the monotonicity assumption~\eqref{p-growth}$_1$ of $\mathbf B$ to estimate the
left-hand side from below, and H\"older's inequality to bound the
right-hand side, we deduce
\begin{align}\label{pre-comparison}
  C_o\iint_{E_T}
  &\big(\mu^2+|Dv|^2+|Dw|^2\big)^{\frac{p-2}{2}}|Dv-Dw|^2\dx\dt\\\nonumber
  &\le
  \|\mathbf B(\cdot,Dv)-\mathbf A(\cdot,Dv)\|_{L^{p'}(E_T)}
  \bigg[\iint_{E_T}|Dv-Dw|^p\dx\dt\bigg]^{\frac1p}.
\end{align}
In the case $p\ge2$, this readily implies
\begin{align*}
  \iint_{E_T}|Dv-Dw|^p\dx\dt
  \le
  \boldsymbol\gm \|\mathbf B(\cdot,Dv)-\mathbf A(\cdot,Dv)\|_{L^{p'}(E_T)}
  \bigg[\iint_{E_T}|Dv-Dw|^p\dx\dt\bigg]^{\frac1p}
\end{align*}
for a constant $\boldsymbol\gm =\boldsymbol\gm (p,C_o)$, which yields the asserted comparison estimate \eqref{comparison:p>2} in this case. For exponents $1<p<2$, however, we apply first H\"older's inequality and then estimate \eqref{pre-comparison} to obtain 
\begin{align*}
  \iint_{E_T}&|Dv-Dw|^p\, \dx\dt\\
  &\le
    \bigg[\iint_{E_T}\big(\mu^2+|Dv|^2+|Dw|^2\big)^{\frac{p-2}{2}}|Dv-Dw|^2\, \dx\dt\bigg]^{\frac
    p2}\\
  &\qquad\phantom{\le\,}
  \cdot
\bigg[\iint_{E_T}\big(\mu^2+|Dv|^2+|Dw|^2\big)^{\frac{p}2}\, \dx\dt\bigg]^{\frac{2-p}{2}}\\
  &\le
    \boldsymbol\gm \|\mathbf B(\cdot,Dv)-\mathbf A(\cdot,Dv)\|_{L^{p'}(E_T)}^{\frac{p}{2}}
    \bigg[\iint_{E_T}|Dv-Dw|^p\, \dx\dt\bigg]^{\frac12}\\
  &\qquad\phantom{\le\,}\cdot
    \bigg[\iint_{E_T}|Dv-Dw|^p\, \dx\dt
    +\iint_{E_T}\big(\mu^2+|Dv|^2\big)^{\frac{p}2}\, \dx\dt\bigg]^{\frac{2-p}{2}}\\
  &\le
    \tfrac12\iint_{E_T}|Dv-Dw|^p\, \dx\dt
    +\boldsymbol\gm \|\mathbf B(\cdot,Dv)-\mathbf A(\cdot,Dv)\|_{L^{p'}(E_T)}^{p'}\\
  &\phantom{\le\,}
    +
    \boldsymbol\gm \|\mathbf B(\cdot,Dv)-\mathbf A(\cdot,Dv)\|_{L^{p'}(E_T)}^{p}
    \bigg[\iint_{E_T}\big(\mu^2+|Dv|^2\big)^{\frac{p}2}\, \dx\dt\bigg]^{2-p}
\end{align*}
with a constant $\boldsymbol\gm =\boldsymbol\gm (p,C_o)$. Here we applied Young's inequality in the
last step. 
After re-absorbing the integral of $|Dv-Dw|^p$ in the left-hand side, we
obtain the asserted comparison estimate~\eqref{VglAbsch2}. The remaining
energy estimate~\eqref{EAbsch2} follows from
\begin{align*}
  \iint_{E_T}|Dw|^p\dx\dt
  \le
  2^{p-1}\iint_{E_T}|Dv|^p\dx\dt
  +
  2^{p-1}\iint_{E_T}|Dv-Dw|^p\dx\dt
\end{align*}
and estimating the last integral by~\eqref{VglAbsch2} and \eqref{comparison:p>2}, respectively. In the case $p<2$, we additionally apply Young's inequality with exponents $\frac{1}{p-1}$ and
$\frac{1}{2-p}$. This yields assertion~\eqref{EAbsch2} and completes
the proof of the lemma.  
\end{proof}

\color{black}

\section{Oscillation estimates for Lipschitz solutions}

In this section we collect two {\it a priori} oscillation estimates of weak solutions whose spatial gradient is assumed to be locally bounded.

\begin{lemma}\label{lem:osc-bound-0}
Let $p>1$ and $u$ be a weak solution to \eqref{p-laplace} with \eqref{prop-a}$_1$. Assume that $|Du|\in L^\infty_{\mathrm{loc}}(E_T)$. Then, for every cylinder $Q=K_\rho(x_o)\times(\tau_1,\tau_2] \subset E_T$ we have 
\begin{equation*}
    \osc_{Q}u
    \le
     4\sqrt{N}\rho\|Du\|_{L^\infty(Q)}+\boldsymbol\gm\frac{\tau_2-\tau_1}{\rho}\Big(\mu^2+\|Du\|_{L^\infty(Q)}^2\Big)^{\frac{p-1}{2}}
\end{equation*}
with $\boldsymbol\gm =\boldsymbol\gm (N,C_1)$. 
\end{lemma}
\begin{proof}
Throughout the proof we omit $z_o$ in our notation. 
Let $[t_1,t_2]\subset(\tau_1,\tau_2)$ and let
      $\varepsilon\in(0,\min\{t_1-\tau_1,\tau_2-t_2\})$. We define the Lipschitz function $\xi_\epsilon$ to be $1$ in $[t_1,t_2]$, be $0$ outside $[t_1-\varep, t_2+\varep]$, and be linearly interpolated otherwise. 
Moreover, we choose a non-negative function $\eta\in C^\infty_0(K_1)$ with
$\int_{K_1}\eta\,\dx=1$ and let $\eta_\rho(x):=\rho^{-N}\eta(\frac{x}{\rho})$.
Then we test the weak form of the equation~\eqref{p-laplace}
with $\varphi_{\epsilon}=\xi_\epsilon\eta_\rho$. Letting
$\varepsilon\downarrow 0$ and recalling the upper
bound for the coefficients from~\eqref{prop-a}$_1$, we obtain
\begin{align*}
  \bigg|\int_{K_{\rho}}& \big[u(x,t_2)-u(x,t_1)\big]\eta_\rho(x)\,
  \dx\bigg|\\\nonumber
  &= 
  \bigg|\iint_{K_\rho\times (t_1,t_2)}
    a(x,t)(\mu^2+|Du|^2)^{\frac{p-2}{2}}Du\cdot D\eta_\rho\, 
    \dx\d t\bigg|\\\nonumber
  &\le
    \frac{\boldsymbol\gm (N)C_1}{\rho^{N+1}}
    \iint_{K_\rho\times (t_1,t_2)}
    (\mu^2+|Du|^{2})^\frac{p-1}{2}\, \dx\dt\\\nonumber
  &\le
   \boldsymbol\gm(N,C_1)\frac{t_2-t_1}{\rho}\Big(\mu^2+\|Du\|_{L^\infty(Q)}^2\Big)^{\frac{p-1}{2}}.
\end{align*}
Consequently, for a.e. $(x_1,t_1),(x_2,t_2)\in Q$ we estimate 
\begin{align*}
    |u(x_1,t_1)&-u(x_2,t_2)|\\
    &\le
    \int_{K_\rho}|u(x_1,t_1)-u(x,t_1)|\eta_\rho(x)\, \dx
    +
    \int_{K_\rho}|u(x,t_2)-u(x_2,t_2)|\eta_\rho(x)\, \dx\\
    &\phantom{\le\,}
    +
    \bigg|\int_{K_\rho}\big[u(x,t_1)-u(x,t_2)\big]\eta_\rho(x)\,\dx\bigg|\\
  &\le
    4\sqrt{N}\rho\|Du\|_{L^\infty(Q)}\int_{K_\rho}\eta_\rho(x)\,\dx
    +
   \boldsymbol\gm\frac{t_2-t_1}{\rho}\Big(\mu^2+\|Du\|_{L^\infty(Q)}^2\Big)^{\frac{p-1}{2}}\\
   &\le 4\sqrt{N}\rho\|Du\|_{L^\infty(Q)}
    +
   \boldsymbol\gm\frac{\tau_2-\tau_1}{\rho}\Big(\mu^2+\|Du\|_{L^\infty(Q)}^2\Big)^{\frac{p-1}{2}}.
\end{align*}
This finishes the proof.
\end{proof}
In the sub-quadratic case $1<p\le 2$, a special form of the previous lemma is needed in the next section.
\begin{lemma}\label{lem:osc-bound}
Let $1<p\le2$ and $u$ be a
 weak solution to \eqref{p-laplace} with \eqref{prop-a}$_1$. Assume that $|Du|\in
L^\infty_{\mathrm{loc}}(E_T)$. Then, for every cylinder $z_o+Q_\rho^{(\lambda)}\Subset E_T$ we have 
\begin{equation*}
    \osc_{z_o+Q_\rho^{(\lambda)}}u
    \le
    \boldsymbol\gm \rho\Big(\|Du\|_{L^\infty(z_o+Q_\rho^{(\lambda)})}+\mu+\lambda\Big)
\end{equation*}
with $\boldsymbol\gm =\boldsymbol\gm (N,C_1)$. 
\end{lemma}
\begin{proof}
Assume $z_o=(0,0)$ for simplicity.  Let us apply Lemma~\ref{lem:osc-bound-0} with $Q= Q_\rho^{(\lambda)}$. Consequently,
for a.e. $(x_1,t_1),(x_2,t_2)\in
Q_\rho^{(\lambda)}$ we have 
\begin{align*}
    |u(x_1,t_1)&-u(x_2,t_2)|\\
    &\le 4\sqrt{N}\rho\|Du\|_{L^\infty(Q_\rho^{(\lambda)})} +\boldsymbol\gm \rho\lambda^{2-p}\Big(\mu^2+\|Du\|_{L^\infty(Q_\rho^{(\lambda)})}^2\Big)^\frac{p-1}{2}\\
  &\le
   \boldsymbol\gm\rho\Big(\|Du\|_{L^\infty(Q_\rho^{(\lambda)})}+\mu+\lambda\Big).
\end{align*} 
If $p=2$, the last step is obvious. Otherwise, 
we applied Young's inequality with exponents $\frac{1}{p-1}$ and
$\frac{1}{2-p}$.
\end{proof}

\section{Schauder estimates for Lipschitz solutions}\label{sec:Schauder-Lip}

In this section, we prove Theorem~\ref{theorem:schauder:intro}
in the case $\mu\in(0,1]$ under {\bf the additional assumption}: $|Du|\in
L^\infty_{\mathrm{loc}}(E_T)$.  
To this end, let us consider coefficients $a\in L^\infty(E_T)$ satisfying
\eqref{prop-a}, and a bounded weak solution
$
    u\in L^p\big(0,T;W^{1,p}(E)\big)\cap C\big([0,T];L^2(E)\big)
$
to the equation \eqref{p-laplace}.
The proof is divided into several steps. 

\subsection{Step~1: Freezing  coefficients}

Consider  nested  cylinders 
$\tilde z+Q_{R_1}\subset \tilde z+ Q_{R_2}\Subset E_T$ with
$0<R_1<R_2$. 
Consider parameters $\lambda\ge\mu$ that satisfy
\begin{equation}
  \label{def-lambda}
  \lambda\ge \boldsymbol\gm_o\Big(\mu^2+\|Du\|_{L^\infty(\tilde z+Q_{R_2})}^2\Big)^{\frac12},
\end{equation}
with a constant $\boldsymbol\gm_o\ge1$ to be specified later in dependence on
$N,p,C_o$ and $C_1$.

In Steps~1 -- 3, we will determine $\boldsymbol\gm_o\ge1$ in dependence on $N,p,C_o$ and $C_1$, such that whenever the parameter $\lambda$ satisfies  condition~\eqref{def-lambda}, 
we have the Campanato type estimates \eqref{Campanato-bound} and \eqref{Campanato-bound-2}. 
Step~2 consists in determining $\bg_o$ and applying the Campanato estimates in Lemma~\ref{lem:campanato} to a comparison function $w$, which is the solution to a more regular equation and meanwhile agrees with $u$ on the boundary, cf.~\eqref{comparison-problem}. Then in Step~3 such estimates can be ``transferred" to $u$, thanks to the H\"older continuous coefficient in the equation \eqref{p-laplace}.
They are {\it a priori} estimates for
$\mu\in(0,1]$ and under the assumption: $|Du|\in
L^\infty_{\mathrm{loc}}(E_T)$. 
In Steps~4 -- 5, they will be employed to establish the gradient bound~\eqref{gradient-sup-bound-intro} and the H\"older estimate~\eqref{gradient-holder-bound-intro}. Henceforth, the parameter $\lm$ will be subject to various selections as long as condition~\eqref{def-lambda} is verified.

Let us set up the basic notation for the next steps here. For the radius  
\begin{equation}\label{choice-R_o}
  R_o:=\tfrac12\min\big\{1,\lambda^{\frac{p-2}{2}}\big\}(R_2-R_1)
\end{equation}
we have $z_o+Q_R^{(\lambda)}\subset \tilde z +Q_{R_2}$
whenever $z_o\in \tilde z + Q_{R_1}$  and $R\le 2R_o$.
Now we consider a radius $r\in(0,\frac18 R_o)$
and let
\begin{equation}\label{choice-R}
  R:=\Big(\frac{8r}{R_o}\Big)^\kappa R_o
  \quad\implies\quad
  r=\tfrac18\Big(\frac{R}{R_o}\Big)^{\frac1\kappa}R_o,
\end{equation}
for a parameter $\kappa\in(0,1)$  to be fixed later. 
With these choices  we have $R<R_o$ and also
$r<\frac18 R$.  
For ease of notation, we write $Q_r^{(\lambda)}$, $Q_R^{(\lambda)}$
instead of $z_o+Q_r^{(\lambda)}$, $z_o+Q_R^{(\lambda)}$  in the
sequel.

Let 
$w\in L^p(\Lambda^{(\lambda)}_R;W^{1,p}(K_R))
\cap C(\Lambda^{(\lambda)}_R;L^2(K_R))$ be the unique weak solution to the 
Cauchy-Dirichlet problem
\begin{equation}
  \label{comparison-problem}
  \left\{
    \begin{array}{c}
    \partial_tw-\Div\big(a(x_o,t)(\mu^2+|Dw|^2)^{\frac{p-2}{2}}Dw\big)=0\,\,
    \quad\mbox{in $Q_R^{(\lambda)}$},\\[6pt]
      w=u \quad\mbox{on $\partial_\mathrm{par} Q_R^{(\lambda)}$,}
    \end{array}
  \right.
\end{equation}
where the boundary datum is taken in the sense of Definition~\ref{def:weak}, and the existence was discussed in Remark~~\ref{Rmk:CP2}.


The solution $w$ plays the role of a comparison function.
\subsection{Step 2: Estimates for the comparison function}
By the comparison principle  Proposition~\ref{prop:comparison-plapl} we have 
that $w$ is bounded, and moreover that, 
\begin{equation*}
  \osc_{Q_R^{(\lambda)}}w\le \osc_{Q_R^{(\lambda)}}u=:
  \boldsymbol \omega_R^{(\lm)}.
\end{equation*}
Since $\lm$ is fixed, we write $\boldsymbol \omega_R$ instead of $\boldsymbol \omega_R^{(\lm)}$ to keep the notation as simple as possible.
Next, we apply Lemma~\ref{lem:Comp2} on 
$Q_R^{(\lambda)}$ in place of $E_T$ to $u$ (instead of $v$)
and $w$ and with the corresponding vector-fields
\[
    \mathbf A(x,t,\xi):=a(x,t)\big(\mu^2+|\xi|^2\big)^{\frac{p-2}{2}}\xi
    \quad\mbox{and}\quad
    \mathbf B(t,\xi):=a(x_o,t)\big(\mu^2+|\xi|^2\big)^{\frac{p-2}{2}}\xi.
\]
The monotonicity condition  \eqref{p-growth}$_1$ for the vector-field $\mathbf B$, results with the help of Lemma \ref{Monotonicity}.
Keeping in mind the
assumptions~\eqref{prop-a} of the coefficients, we infer on the one hand the comparison estimate 
\begin{align}\label{comparison-est}
   \iint_{Q_R^{(\lambda)}}&|Du-Dw|^p\,\dx\dt\\\nonumber
  &\le
    \boldsymbol\gm \sup_{(x,t)\in Q_R^{(\lambda)}}\big|a(x_o,t)-a(x,t)\big|^{p'}
    \iint_{Q_R^{(\lambda)}}\big(\mu^2+|Du|^2\big)^{\frac{p}2}\, \dx\dt\\\nonumber
   &\phantom{\le\,}+
    \boldsymbol\gm
    \sup_{(x,t)\in Q_R^{(\lambda)}}\big|a(x_o,t)-a(x,t)\big|^{p}
     \iint_{Q_R^{(\lambda)}}\big(\mu^2+|Du|^2\big)^{\frac{p}2}\, \dx\dt\\\nonumber
   &\le
     \boldsymbol\gm R^{\alpha_\ast p}\iint_{Q_R^{(\lambda)}}\big(\mu^2+|Du|^2\big)^{\frac{p}2}\, \dx\dt,
\end{align}
with
\begin{equation}\label{def:alpha-ast}
    \alpha_\ast:=
    \left\{
    \begin{array}{cl}
        \alpha, &\mbox{if }1<p\le2,  \\[5pt]
        \tfrac{\alpha}{p-1}, &\mbox{if } p>2, 
    \end{array}
    \right.
\end{equation}
and, on the other hand, the energy estimate
\begin{align}\label{energy-bound-w}
    \iint_{Q_R^{(\lambda)}}&(\mu^2+|Dw|^2)^{\frac{p}{2}}\,\dx\dt \\\nonumber
    &\le
    \boldsymbol\gm\bigg[\sup_{(x,t)\in
      Q_R^{(\lambda)}}\big|a(x_o,t)-a(x,t)\big|^{p'}+1\bigg]
      \iint_{Q_R^{(\lambda)}}\big(\mu^2+|Du|^2\big)^{\frac{p}2}\,\dx\dt\\\nonumber
   &\le \boldsymbol\gm
    \iint_{Q_R^{(\lambda)}}\big(\mu^2+|Du|^2\big)^{\frac{p}2}\,\dx\dt,
\end{align}
with a  constant $\boldsymbol\gm = \boldsymbol\gm (p,C_o,C_1)$.

Propositions~\ref{PROP:LINFTY} and~\ref{PROP:LINFTY:P>2}  ensure that 
$Dw$ is locally bounded for $1<p<2$ and $p\ge2$, respectively.
We proceed to examine the quantitative gradient bounds. The goal is to select $\bg_o$ such that $w$ verifies the intrinsic relation~\eqref{intrinsic}.

First consider {\bf the sub-quadratic case $1<p<2$}.  In this case, Proposition~\ref{PROP:LINFTY}
provides the estimate 
\begin{align}\label{sup-bound-Dw}\nonumber
    \essup_{Q_{R/4}^{(\lambda)}}|Dw|&\le  \boldsymbol\gm \epsilon\lambda\\
    & \phantom{\le\,}
   +
    \frac{\boldsymbol\gm\lambda^{\frac12}}{\epsilon^{\theta}}
    \bigg[\Big(\frac{\boldsymbol\omega_R}{R\lambda}\Big)^{\frac{2}{p}}
      {+}
      \frac{\boldsymbol\omega_R}{R\lambda}
      {+}
      \frac\mu\lambda
      \bigg]^{\frac{N(2-p)+2p}{4p}}\!\!
      \bigg[\biint_{Q_{R}^{(\lambda)}}\big(\mu^2+|Dw|^2\big)^{\frac{p}{2}}\dx\dt\bigg]^{\frac{1}{2p}}
\end{align}
for every $\epsilon\in(0,1]$. Since the coefficients
in~\eqref{comparison-problem} do not depend on $x$, we can apply 
Proposition~\ref{PROP:LINFTY} with $C_2=0$, so that the dependencies
of the constants in \eqref{sup-bound-Dw} are given by $\boldsymbol\gm =\boldsymbol\gm (N,p,C_o,C_1)$ and
$\theta=\theta(N,p)>0$. In view of Lemma~\ref{lem:osc-bound} and
the requirement~\eqref{def-lambda} for  $\lambda$, we have 
\begin{equation*}
   \frac{\boldsymbol\omega_R}{R}
   \le
      \boldsymbol\gm (N,C_1)\Big(\|Du\|_{L^\infty(z_o+ Q_R^{(\lambda)})}+\mu+\lambda\Big)
   \le
      \boldsymbol\gm (N,C_1)\lambda.
\end{equation*}
Using this in combination with the definition of $\lambda$ from \eqref{def-lambda},
i.e.~the fact that $\frac{\mu}{\lambda}\le \frac{1}{\boldsymbol \gm_o}\le 1$ as long as $\boldsymbol\gm_o$ is chosen larger than 1,
the term in the squared bracket in~\eqref{sup-bound-Dw}
is bounded by a universal constant $\boldsymbol\gm (N,C_1)$. Moreover, we estimate the
integral in~\eqref{sup-bound-Dw} by means
of~\eqref{energy-bound-w}. Therefore, in the case $1<p<2$, we arrive at
\begin{align}\label{sup-bound-Dw-p<2}
    \essup_{Q_{R/4}^{(\lambda)}}|Dw|\le
    \boldsymbol\gm \epsilon\lambda
    +
    \frac{\boldsymbol\gm \lambda^{\frac12}}{\epsilon^{\theta}}
      \bigg[\biint_{Q_{R}^{(\lambda)}}\big(\mu^2+|Du|^2\big)^{\frac{p}{2}}\,\dx\dt\bigg]^{\frac{1}{2p}}
\end{align}
for every $\epsilon\in (0,1]$, where the constant $\boldsymbol\gm$ now depends on
$N,p,C_o$ and $C_1$.

Next, we establish the same type of estimate in {\bf the super-quadradic case} $p\ge2$. To this end, we apply in turn Proposition~\ref{PROP:LINFTY:P>2} with $C_2=0$, estimate~\eqref{energy-bound-w}, and finally \eqref{def-lambda}, and obtain the result
\begin{align}\label{sup-bound-Dw-2}
    \essup_{Q_{R/4}^{(\lambda)}}|Dw|
    &\le
    \boldsymbol\gm \eps\lambda
    +
    \frac{\boldsymbol\gm}{\eps^\theta}\bigg[\lambda^{2-p}
    \biint_{Q_{R}^{(\lambda)}}\big(\mu^2+|Dw|^{2}\big)^{\frac{p}{2}}\dx\dt\bigg]^{\frac{1}{2}}\\\nonumber
    &\le
    \boldsymbol\gm\eps\lambda
    +
    \frac{\boldsymbol\gm\lambda^{\frac{2-p}{2}}}{\eps^\theta}\bigg[   \biint_{Q_{R}^{(\lambda)}}\big(\mu^2+|Du|^{2}\big)^{\frac{p}{2}}\dx\dt\bigg]^{\frac{1}{2}}\\\nonumber
    &\le
    \boldsymbol\gm\eps\lambda
    +
    \frac{\boldsymbol\gm\lambda^{\frac12}}{\eps^\theta}\bigg[     \biint_{Q_{R}^{(\lambda)}}\big(\mu^2+|Du|^{2}\big)^{\frac{p}{2}}\dx\dt\bigg]^{\frac{1}{2p}}
\end{align}
for every $\eps\in(0,1]$, with constants $\boldsymbol\gm =\boldsymbol\gm (N,p,C_o,C_1)$  and
$\theta=\theta(N,p)\ge0$. In view of \eqref{sup-bound-Dw-p<2}, this estimate holds for arbitrary exponents $p>1$. 

In what follows, the estimate \eqref{sup-bound-Dw-2} will be used several times and each time requires a different selection of $\varep$. At this point, we first use it in combination with \eqref{def-lambda} and deduce that
\begin{equation*}
  \essup_{Q_{R/4}^{(\lambda)}}\big(\mu^2+|Dw|^2\big)^{\frac12}
  \le
    \boldsymbol\gm \epsilon\lambda
    +
    \frac{\boldsymbol\gm\lambda}{\boldsymbol\gm_o^{1/2}\epsilon^{\theta}}.
\end{equation*}
From the last display, we choose $\eps$ to satisfy $\boldsymbol\gm \eps\le\frac12$ and then $\boldsymbol\gm_o$ so large that $\boldsymbol\gm_o^{1/2}\epsilon^\theta\ge 2\boldsymbol\gm$. In this way, the requirements in Lemma~\ref{lem:campanato} are satisfied in the sense that
\begin{equation}
  \label{sub-bound-Dw-lambda}
  \essup_{Q_{R/4}^{(\lambda)}}\big(\mu^2+|Dw|^2\big)^{\frac12}\le\lambda.
\end{equation}
This fixes the parameter $\boldsymbol\gm_o$ in \eqref{def-lambda} in dependence
on $N,p,C_o,$ and $C_1$. 
Yet, the estimate \eqref{sup-bound-Dw-2} is still at our disposal
for any $\eps
\in (0,1]$.  Because of~\eqref{sub-bound-Dw-lambda},
the assumptions of
Lemma~\ref{lem:campanato} are satisfied for $w$ in place of $u$ and 
$\frac18R$ in place of $\rho$. Since
$0<r<\frac18R$, Lemma~\ref{lem:campanato} with $C_2=0$ implies
the Campanato type estimate
\begin{equation*}
  	\biint_{Q_r^{(\lambda)}}\big|Dw-(Dw)_{r}^{(\lambda)}\big|^p \,\dx\dt
	\le
	\boldsymbol\gm \Big(\frac{r}{R}\Big)^{\beta p}\lambda^{p},
\end{equation*}
for some exponent $\beta=\beta(N,p,C_o,C_1)\in(0,1)$.       
We combine this estimate
with~\eqref{sup-bound-Dw-2} to deduce 
\begin{align*}
  	\biint_{Q_r^{(\lambda)}}&\big|Dw-(Dw)_{r}^{(\lambda)}\big|^p \,\dx\dt\\\nonumber
	&\le
    \boldsymbol\gm\|Dw\|_{L^\infty(Q_{R/4}^{(\lambda)})}^{p-1}
    \bigg[\biint_{Q_r^{(\lambda)}}\big|Dw-(Dw)_{r}^{(\lambda)}\big|^p
    \,\dx\dt\bigg]^{\frac1p}\\\nonumber
    &\le
    \boldsymbol\gm \Bigg[\epsilon\lambda
          +
          \frac{\lambda^{1/2}}{\epsilon^{\theta}}
          \bigg[\biint_{Q_{R}^{(\lambda)}}\big(\mu^2+|Du|^2\big)^{\frac{p}{2}}\,\dx\dt\bigg]^{\frac{1}{2p}}\Bigg]^{p-1}
          \Big(\frac{r}{R}\Big)^{\beta}\lambda\\\nonumber
    &\le
    \boldsymbol\gm\Big(\frac{r}{R}\Big)^{\beta}
          \left[\epsilon^{p-1}\lambda^p
          +
          \frac{\lambda^{\frac{p+1}{2}}}{\epsilon^{\theta(p-1)}}
          \bigg[\biint_{Q_{R}^{(\lambda)}}\big(\mu^2+|Du|^2\big)^{\frac{p}{2}}\,\dx\dt\bigg]^{\frac{p-1}{2p}}\right].
\end{align*}
To the second term in the bracket we apply Young's inequality
with exponents $\frac{2p}{p+1}$ and $\frac{2p}{p-1}$ to get
\begin{align}\label{campanato-w}
  	\biint_{Q_r^{(\lambda)}}&\big|Dw-(Dw)_{r}^{(\lambda)}\big|^p \,\dx\dt\\\nonumber
        &\le
          \boldsymbol\gm \Big(\frac{r}{R}\Big)^{\beta}
          \left[\epsilon^{p-1}\lambda^p
          +
          \epsilon^{-\theta_2}
          \biint_{Q_{R}^{(\lambda)}}\big(\mu^2+|Du|^2\big)^{\frac{p}{2}}\,\dx\dt\right],
\end{align}
for an exponent $\theta_2=\theta_2(N,p)$. 

\subsection{Step 3: A Campanato type estimate}
Next, we use the quasi-minimality of the mean value, the comparison
estimate~\eqref{comparison-est} and~\eqref{campanato-w}
to estimate  
\begin{align}\label{campanato-bound-}
  \biint_{Q_r^{(\lambda)}}&\big|Du-(Du)_r^{(\lambda)}\big|^p\,\dx\dt
    \le \boldsymbol\gm \,\biint_{Q_r^{(\lambda)}}\big|Du-(Dw)_r^{(\lambda)}\big|^p\dx\dt\\\nonumber
  &\le
    \boldsymbol\gm \Big(\frac{R}{r}\Big)^{N+2}\biint_{Q_R^{(\lambda)}}|Du-Dw|^p\dx\dt\\
    &\qquad\qquad\qquad\qquad\nonumber
    +
    \boldsymbol\gm \,\biint_{Q_r^{(\lambda)}}\big|Dw-(Dw)_r^{(\lambda)}\big|^p\dx\dt\\\nonumber
  &\le
    \boldsymbol\gm\Big(\frac{R}{r}\Big)^{N+2}R^{\alpha_\ast
    p}\biint_{Q_R^{(\lambda)}}\big(\mu^2+|Du|^2\big)^{\frac{p}{2}}\, \dx\dt\\\nonumber
  &\phantom{\le\,}
    + 
    \boldsymbol\gm\Big(\frac{r}{R}\Big)^{\beta}\bigg[\epsilon^{p-1}\lambda^p
    +
    \epsilon^{-\theta_2}
    \biint_{Q_{R}^{(\lambda)}}\big(\mu^2+|Du|^2\big)^{\frac{p}{2}}\,\dx\dt\bigg]\\\nonumber
  &=:
    \mathbf{I}+\mathbf{II}.
\end{align}
For a parameter $\delta\in(0,1)$ to be chosen, we estimate
\begin{align}\label{delta-trick}
  \biint_{Q_R^{(\lambda)}}&\big(\mu^2+|Du|^2\big)^{\frac{p}{2}}\,\dx\dt\\\nonumber
  &=
    \bigg[\biint_{Q_R^{(\lambda)}}\big(\mu^2+|Du|^2\big)^{\frac{p}{2}}\,\dx\dt\bigg]^{1-\delta}
    \bigg[\biint_{Q_R^{(\lambda)}}\big(\mu^2+|Du|^2\big)^{\frac{p}{2}}\,\dx\dt\bigg]^{\delta}\\\nonumber
  &\le
    \boldsymbol\gm(N)\lambda^{p(1-\delta)}\Big(\frac{R_o}{R}\Big)^{\delta(N+2)}
    \bigg[\biint_{Q_{R_o}^{(\lambda)}}\big(\mu^2+|Du|^2\big)^{\frac{p}{2}}\,\dx\dt\bigg]^{\delta}.
\end{align}
In the last step we used $z_o+Q_R^{(\lambda)}\subset \tilde z+Q_{R_2}$ and 
\eqref{def-lambda}.
Next, we use \eqref{delta-trick} to estimate   the terms $\mathbf{I}$  and
$\mathbf{II}$. First, note that by the choice  of $R$ in~\eqref{choice-R}, we have
\begin{equation*}
  \Big(\frac{r}{R_o}\Big)^\kappa\le\frac{R}{R_o}
  \le
  8\Big(\frac{r}{R_o}\Big)^\kappa
  \quad\mbox{and}\quad
  \frac{R}{r}\le8\Big(\frac{R_o}{r}\Big)^{1-\kappa}.
\end{equation*}
This implies
\begin{align}\label{Bound-of-I}
  \mathbf{I}
  &\le
    \boldsymbol\gm 
    \Big(\frac{R}{r}\Big)^{N+2}\Big(\frac{R}{R_o}\Big)^{\alpha_\ast
    p-\delta(N+2)}R_o^{\alpha_\ast p}
    \lambda^{p(1-\delta)}
    \bigg[\biint_{Q_{R_o}^{(\lambda)}}\big(\mu^2+|Du|^2\big)^{\frac{p}{2}}\,\dx\dt\bigg]^{\delta}\\\nonumber
  &\le
    \boldsymbol\gm 
    \Big(\frac{r}{R_o}\Big)^{\alpha_\ast p\kappa-(1+\kappa\delta-\kappa)(N+2)}
    R_o^{\alpha_\ast p}\lambda^{p(1-\delta)}
    \bigg[\biint_{Q_{R_o}^{(\lambda)}}\big(\mu^2+|Du|^2\big)^{\frac{p}{2}}\, \dx\dt\bigg]^{\delta},
\end{align}
and 
\begin{align}\label{Bound-of-II}
    &\mathbf{II}
    \le
    \boldsymbol\gm \Big(\frac{r}{R}\Big)^{\beta}\\\nonumber
    &
    \qquad\qquad\cdot
    \Bigg[\epsilon^{p-1}\lambda^p
    +
    \frac{\lambda^{p(1-\delta)}}{\epsilon^{\theta_2}}\Big(\frac{R_o}{R}\Big)^{\delta(N+2)}\bigg[\biint_{Q_{R_o}^{(\lambda)}}
    \big(\mu^2+|Du|^2\big)^{\frac{p}{2}}\dx\dt\bigg]^{\delta}
    \Bigg]\\\nonumber
    &
    \phantom{\mathbf{II}\,}\le
    \boldsymbol\gm 
    \Big(\frac{r}{R_o}\Big)^{\beta}\Big(\frac{R}{R_o}\Big)^{-\beta-\delta(N+2)}\\\nonumber
    &\qquad\qquad\cdot\Bigg[\epsilon^{p-1}\lambda^p
    +
    \frac{\lambda^{p(1-\delta)}}{\epsilon^{\theta_2}}\bigg[\biint_{Q_{R_o}^{(\lambda)}}
    \big(\mu^2+|Du|^2\big)^{\frac{p}{2}}\dx\dt\bigg]^{\delta}\Bigg]\\\nonumber
    &\phantom{\mathbf{II}\,}\le
    \boldsymbol\gm \Big(\frac{r}{R_o}\Big)^{\beta-\kappa\beta-\kappa\delta(N+2)}\\\nonumber
     &\qquad\qquad\cdot
    \Bigg[\epsilon^{p-1}\lambda^p
    +
    \frac{\lambda^{p(1-\delta)}}{\epsilon^{\theta_2}}
    \bigg[\biint_{Q_{R_o}^{(\lambda)}}
    \big(\mu^2+|Du|^2\big)^{\frac{p}{2}}\dx\dt\bigg]^{\delta}\Bigg].
\end{align}
In the second step, we  also used that $R\le R_o$. 
At this stage we choose
\begin{align*}
  \kappa:=\frac{N+2+\beta}{N+2+\beta+\alpha_\ast p}\in(0,1)
  \quad\mbox{and}\quad
  \delta:=\frac{\alpha_\ast\beta p}{2(N+2)(N+2+\beta)}\in(0,1).
\end{align*}
In this way, the exponents of $r/R_o$ on the right-hand sides
of~\eqref{Bound-of-I} and~\eqref{Bound-of-II} both  coincide with the positive
exponent $\alpha_o p$, where 
\begin{equation}\label{def:alpha-0}
   \alpha_o:=\frac{\alpha_\ast\beta}{2(N+2+\beta+\alpha_\ast p)}\in(0,1).
\end{equation}
Combining~\eqref{Bound-of-I} and \eqref{Bound-of-II} in \eqref{campanato-bound-} yields
\begin{align*}
  \biint_{Q_r^{(\lambda)}}&\big|Du-(Du)_r^{(\lambda)}\big|^p\dx\dt\\
  &\le
    \boldsymbol\gm \Big(\frac{r}{R_o}\Big)^{\alpha_o p}
    \Bigg[\epsilon^{p-1}\lambda^p
    +
    \frac{\lambda^{p(1-\delta)}}{\epsilon^{\theta_2}}\bigg[\biint_{Q_{R_o}^{(\lambda)}}
    \big(\mu^2+|Du|^2\big)^{\frac{p}{2}}\, \dx\dt\bigg]^{\delta}\Bigg]
\end{align*}
for every $r\in(0,\frac18 R_o)$, where the
dependencies of the constant are given by $\boldsymbol\gm =\boldsymbol\gm (N,p,C_o,C_1)$. Here, we also used
\eqref{choice-R_o}, i.e.~the fact that $R_o\le\frac12 R_2\le \frac12\rho_o\le \frac12$.
We continue to apply
Young's inequality with exponents $\frac{1}{1-\delta}$ and
$\frac{1}{\delta}$ to estimate the right-hand side and obtain for every $r\in(0,\frac18 R_o)$,
\begin{align*}
  \biint_{Q_r^{(\lambda)}}&\big|Du-(Du)_r^{(\lambda)}\big|^p\,\dx\dt\\
  &\le
    \boldsymbol\gm\Big(\frac{r}{R_o}\Big)^{\alpha_o p}
    \bigg[\epsilon^{p-1}\lambda^p
    +
    \epsilon^{-\theta_3 p}
    \biint_{Q_{R_o}^{(\lambda)}}
    \big(\mu^2+|Du|^2\big)^{\frac{p}{2}}\,\dx\dt\bigg]
\end{align*}
 with a positive constant
$\theta_3=\theta_3(N,p,\alpha,\beta)=\theta_3(N,p,C_o,C_1,\alpha)$. 
Whereas for radii $r\in[\frac18 R_o,R_o]$, we have
\begin{align*}
    \biint_{Q_r^{(\lambda)}}\big|Du-(Du)_r^{(\lambda)}\big|^p\, \dx\dt
    &\le
    \boldsymbol\gm\,\biint_{Q_r^{(\lambda)}}|Du|^p\,\dx\dt\\
    &\le
    \boldsymbol\gm\Big(\frac{r}{R_o}\Big)^{\alpha_o p} \Big(\frac{R_o}{r}\Big)^{N+2+\alpha_o p} 
    \biint_{Q_{R_0}^{(\lambda)}}|Du|^p\,\dx\dt \\
    &\le 
    \boldsymbol\gm\Big(\frac{r}{R_o}\Big)^{\alpha_o p}
    \biint_{Q_{R_o}^{(\lambda)}}|Du|^p\,\dx\dt.
\end{align*}
In the last step, we used the assumption $r\ge\tfrac18 R_o$.
In view of the two preceding estimates, we have shown that for any
$r\in(0,R_o)$ there holds 
\begin{align}\label{Campanato-bound}
  \biint_{Q_r^{(\lambda)}}&\big|Du-(Du)_r^{(\lambda)}\big|^p\,\dx\dt\\\nonumber
  &\le
    \boldsymbol\gm\Big(\frac{r}{R_o}\Big)^{\alpha_o p}
    \bigg[\epsilon^{p-1}\lambda^p
    +
    \epsilon^{-\theta_3 p}
    \biint_{Q_{R_o}^{(\lambda)}}
    \big(\mu^2+|Du|^2\big)^{\frac{p}{2}}\,\dx\dt\bigg]
\end{align}
for every $\epsilon\in(0,1]$, where $\boldsymbol\gm$ depends on $N,p,C_o,C_1,$ and
$\alpha_o\in(0,1)$ and $\theta_3>0$ depends additionally on $\alpha$.
In particular, using \eqref{Campanato-bound} with
$\epsilon=1$ and recalling the property~\eqref{def-lambda} of
$\lambda$, we deduce
\begin{align}\label{Campanato-bound-2}
  \biint_{Q_r^{(\lambda)}}\big|Du-(Du)_r^{(\lambda)}\big|^p\dx\dt
  &\le\boldsymbol\gm
    \Big(\frac{r}{R_o}\Big)^{\alpha_o p}\lambda^p
    \qquad\forall \, r\in (0,R_o).
\end{align}


\subsection{Step 4: Upper bound for the gradient}

The goal of this section is to prove \eqref{gradient-sup-bound-intro}. The arguments will be based on the Campanato type estimate \eqref{Campanato-bound}.
Indeed, consider a
cylinder $\tilde z+Q_{2\rho}\Subset E_T$ and radii $\rho\le
R_1<R_2\le 2\rho$, so that we are in the setting of Steps\,1 -- 3 and we will use the same notation as there. In particular, we consider $z_o\in \tilde z +Q_{R_1}$ and write $Q_r^{(\lambda)}$ instead of $z_o+Q_r^{(\lambda)}$. Recall also that $\boldsymbol\gm_o=\boldsymbol\gm _o(N,p,C_o,C_1)$ has been chosen in the course of deriving \eqref{sub-bound-Dw-lambda}.
Then let us fix the parameter $\lambda$ by letting
\begin{equation}\label{lambda-gradient-bound}
\lambda:=\boldsymbol\gm_o\Big(\mu^2+\|Du\|_{L^\infty(\tilde z+ Q_{R_2} )}^2\Big)^{\frac12},
\end{equation}
so that~\eqref{def-lambda} is satisfied and hence the Campanato type estimate \eqref{Campanato-bound} is at our disposal.
To proceed, consider the dyadic sequence of radii $r_i:=2^{-i}R_o$, with
$i\in\N_0$. For indices $k<\ell$ in $\N_0$, we use the triangle
inequality and~\eqref{Campanato-bound} to estimate
\begin{align}\label{Cauchy-sequence}
  \big|(Du)_{r_{\ell}}^ {(\lambda)}&-(Du)_{r_k}^{(\lambda)}\big|
  \le
  \sum_{j=k}^{\ell-1}\big|(Du)_{r_{j+1}}^
    {(\lambda)}-(Du)_{r_j}^{(\lambda)}\big|\\\nonumber
  &\le
    \boldsymbol\gm\sum_{j=k}^{\ell-1}\bigg[\biint_{Q_{r_j}^{(\lambda)}}\big|Du-(Du)_{r_j}^{(\lambda)}\big|^p\,\dx\dt\bigg]^{\frac1p}\\\nonumber
  &\le
    \boldsymbol\gm\sum_{j=k}^{\ell-1}
    \Big(\frac{r_j}{R_o}\Big)^{\alpha_o}
    \Bigg[\epsilon^{\frac1{p'}}\lambda
    +
    \epsilon^{-\theta_3}
    \bigg[\biint_{Q_{R_o}^{(\lambda)}}\big(\mu^2+|Du|^2\big)^{\frac{p}{2}}\,\dx\dt\bigg]^{\frac
    1p}\Bigg]\\\nonumber
  &\le
    \boldsymbol\gm\Big(\frac{r_k}{R_o}\Big)^{\alpha_o}
    \Bigg[\epsilon^{\frac1{p'}}\lambda
    +
    \epsilon^{-\theta_3}
    \bigg[\biint_{Q_{R_o}^{(\lambda)}}\big(\mu^2+|Du|^2\big)^{\frac{p}{2}}\, \dx\dt\bigg]^{\frac 1p}\Bigg],
\end{align}
with a constant $\boldsymbol\gm =\boldsymbol\gm (N,p,C_o,C_1,\alpha)$. 
Since the right-hand side vanishes as $k\to\infty$, we have
that $\big( (Du)_{r_{i}}^ {(\lambda)}\big)_{i\in\N_0}$ is a Cauchy sequence. 
Let us write its limit as
\begin{align*}
  \Gamma_{z_o}:=\lim_{i\to\infty}(Du)_{r_{i}}^ {(\lambda)}.
\end{align*}
 For an arbitrary $r\in(0,R_o]$, we choose
$k\in\N_0$ with $r_{k+1}<r\le r_k$ and estimate
\begin{align}\label{convergence-means}
  \big|\Gamma_{z_o}-(Du)_r^{(\lambda)}\big|
  &\le
    \big|\Gamma_{z_o}-(Du)_{r_k}^{(\lambda)}\big|
    +
    \big|(Du)_{r_k}^{(\lambda)}-(Du)_r^{(\lambda)}\big|\\\nonumber
  &\le
    \big|\Gamma_{z_o}-(Du)_{r_k}^{(\lambda)}\big|
    +
    \boldsymbol\gm 
    \bigg[\biint_{Q_{r_k}^{(\lambda)}}\big|Du-(Du)_{r_k}^{(\lambda)}\big|^p\, \dx\dt\bigg]^{\frac1p}\\\nonumber
  &\le
    \boldsymbol\gm\Big(\frac{r}{R_o}\Big)^{\alpha_o}
    \Bigg[\epsilon^{\frac1{p'}}\lambda
    +
    \epsilon^{-\theta_3}
    \bigg[\biint_{Q_{R_o}^{(\lambda)}}\big(\mu^2+|Du|^2\big)^{\frac{p}{2}}\, \dx\dt\bigg]^{\frac 1p}\Bigg].
\end{align}
The last estimate follows by letting $\ell\to\infty$
in~\eqref{Cauchy-sequence}, using~\eqref{Campanato-bound} with
$r_k$ in place of $r$, and recalling the fact $r_k< 2r$.
In particular, this implies that $\Gamma_{z_o}$ is the Lebesgue
representative of $Du$ at $z_o$. 
Choosing $r=R_o$ in the preceding inequality, we obtain 
\begin{align}\label{Gamma-bound}
  |\Gamma_{z_o}|
  &\le
  \big|(Du)_{R_o}^ {(\lambda)}\big|
    +\big|\Gamma_{z_o}-(Du)_{R_o}^{(\lambda)}\big|\\\nonumber
  &\le
    \biint_{Q_{R_o}^{(\lambda)}}|Du|\,\dx\dt
    +
    \boldsymbol\gm\Bigg[\epsilon^{\frac1{p'}}\lambda
    +
    \epsilon^{-\theta_3}
    \bigg[\biint_{Q_{R_o}^{(\lambda)}}\big(\mu^2+|Du|^2\big)^{\frac{p}{2}}\, \dx\dt\bigg]^{\frac 1p}\Bigg]\\\nonumber
  &\le
    \boldsymbol\gm
    \Bigg[\epsilon^{\frac1{p'}}\lambda
    +
    \epsilon^{-\theta_3}
    \bigg[\biint_{Q_{R_o}^{(\lambda)}}\big(\mu^2+|Du|^2\big)^{\frac{p}{2}}\dx\dt\bigg]^{\frac 1p}\Bigg].
\end{align}
In order to estimate the last integral in terms of the oscillation, we apply the energy estimate \eqref{energy-est-zero-order} on $Q_{R_o}^{(\lambda)}\subset Q_{2R_o}^{(\lambda)}$
with $\xi= (u)_{2R_o}^{(\lambda)}$
to get
\begin{align}\label{energy-bound-Ro}
  \biint_{Q_{R_o}^{(\lambda)}}&\big(\mu^2+|Du|^2\big)^{\frac{p}{2}}\,\dx\dt\\\nonumber
  &\le
    \boldsymbol\gm\,\biint_{Q_{2R_o}^{(\lambda)}}\bigg[\frac{\big|u-(u)_{2R_o}^{(\lambda)}\big|^p}{(2R_o)^p}+\frac{\big|u-(u)_{2R_o}^{(\lambda)}\big|^2}{\lambda^{2-p}(2R_o)^2}+\mu^p\bigg]\dx\dt\\\nonumber
  &\le
    \boldsymbol\gm\Bigg[
    \bigg(\frac{\osc_{Q_{2R_o}^{(\lambda)}}u}{2R_o}\bigg)^p
    +
    \frac1{\lambda^{2-p}} \bigg(\frac{\osc_{Q_{2R_o}^{(\lambda)}}u}{2R_o}\bigg)^2
    +\mu^p\Bigg]\\\nonumber
    &=
    \boldsymbol\gm
    \bigg(\frac{\osc_{Q_{2R_o}^{(\lambda)}}u}{2R_o}\bigg)^p
    \Bigg[
    1+\bigg(\frac{\osc_{Q_{2R_o}^{(\lambda)}}u}{2\lambda R_o}\bigg)^{2-p}
    \Bigg]
    +
    \boldsymbol\gm\mu^p.
\end{align}
In the case $1<p\le2$, we apply
Lemma~\ref{lem:osc-bound} and \eqref{lambda-gradient-bound} to get
\begin{align*}
    \osc_{Q_{2R_o}^{(\lambda)}}u
    &\le 
    2\boldsymbol\gm R_o\Big(\|Du\|_{L^\infty(Q_{2R_o}^{(\lambda)})}+\mu+\lambda\Big)\\
     &\le 
    2\boldsymbol\gm R_o\Big(\|Du\|_{L^\infty(\tilde z+Q_{R_2})}+\mu+\lambda\Big)\\
    &\le 
    2\boldsymbol\gm R_o\Big( \frac{2}{\boldsymbol\gm_o} +1\Big) \lambda \le 6\boldsymbol\gm R_o\lm.
\end{align*}
This allows us to bound the right-hand side of \eqref{energy-bound-Ro} further and obtain that
\begin{align*}
    \biint_{Q_{R_o}^{(\lambda)}}\big(\mu^2+|Du|^2\big)^{\frac{p}{2}}\,\dx\dt
    \le
   \boldsymbol\gm\Bigg[ 
    \bigg(\frac{\osc_{Q_{2R_o}^{(\lambda)}}u}{2R_o}\bigg)^p
    +
    \mu^p\Bigg].
\end{align*}
In the case $p>2$, however, we use Young's inequality on the right-hand side of \eqref{energy-bound-Ro} with exponents $\frac{p}{p-2}$ and $\frac{p}{2}$, and get 
\begin{align*}
    \biint_{Q_{R_o}^{(\lambda)}}\big(\mu^2+|Du|^2\big)^{\frac{p}{2}}\dx\dt
    \le
    \big(\eps^{\frac{1}{p'}}\lambda\big)^p+\boldsymbol\gm\eps^{-\frac{(p-1)(p-2)}{2}}
    \Bigg[
    \bigg(\frac{\osc_{Q_{2R_o}^{(\lambda)}}u}{2R_o}\bigg)^p
    +
    \mu^p
    \Bigg].
\end{align*}
Joining the estimates from both cases in \eqref{Gamma-bound} gives
\begin{align*}
  |\Gamma_{z_o}|
  &\le
    \boldsymbol\gm 
    \Bigg[ 
    \epsilon^{\frac1{p'}}\lambda
    +
    \epsilon^{-\theta_4}\bigg(
    \frac{\osc_{Q_{2R_o}^{(\lambda)}}u}{2R_o}
    +
    \mu
    \bigg)
    \Bigg]\\
    &\le 
    \boldsymbol\gm \bigg[ 
    \epsilon^{\frac1{p'}}\lambda
    +
    \epsilon^{-\theta_4}\Big(
    \frac{\boldsymbol\omega}{2R_o}
    +
    \mu
    \Big)
    \bigg]
\end{align*}
with a positive constant $\theta_4=\theta_4(N,p,C_o,C_1,\alpha)\ge\theta_3$.
Here we used the abbreviation $\boldsymbol\omega = \osc_{\tilde z+Q_{2\rho}}u$,  the inclusion $Q_{2R_o}^{(\lambda)}=z_o+Q_{2R_o}^{(\lambda)} \subset\tilde z+ Q_{R_2}$ and $R_2\le 2\rho$.
Recalling that $\Gamma_{z_o}$ is the Lebesgue representative of
$Du(z_o)$ and taking the supremum over $z_o\in \tilde z+Q_{R_1}$, we thus deduce
\begin{align}\label{sup-bound-Ro}
  \essup_{\tilde z+Q_{R_1}}|Du|
  &\le
    \boldsymbol\gm \left[
    \epsilon^{\frac1{p'}}\lambda
    +
    \epsilon^{-\theta_4}\Big(\frac{\boldsymbol \omega}{2R_o}+\mu\Big)
    \right]\\\nonumber
  &\le
    \boldsymbol\gm\left[
    \epsilon^{\frac1{p'}}\lambda
    +
    \epsilon^{-\theta_4}\big(1+\lambda^{\frac{2-p}{2}}\big)\frac{\boldsymbol \omega}{R_2-R_1}
    +\eps^{-\theta_4}\mu
    \right].
\end{align}
In the last step we used the definition of $R_o$
in~\eqref{choice-R_o}. 
In the case $1<p<2$, we 
apply Young's inequality to estimate the
right-hand side further, and choose $\eps\in(0,1)$ so small that $\boldsymbol\gm\eps^{\frac1{p'}}\le\frac1{4\boldsymbol\gm_o}$, where $\boldsymbol\gm_o$ is the constant from \eqref{lambda-gradient-bound}. In this way, we obtain 
\begin{equation}\label{sup-bound-p<2}
  \essup_{\tilde z+Q_{R_1}}|Du|
  \le
  \tfrac1{2\boldsymbol\gm_o}\lambda
  +
  \boldsymbol\gm\bigg[\frac{\boldsymbol \omega}{R_2-R_1}
  +
  \Big(\frac{\boldsymbol \omega}{R_2-R_1}\Big)^{\frac{2}{p}}
  +\mu
  \bigg],
\end{equation}
with a constant
$\boldsymbol\gm =\boldsymbol\gm (N,p,C_o,C_1,\alpha)$.
In the case $p\ge2$, however, we first use~\eqref{lambda-gradient-bound} 
to decrease the exponent of $|Du|$. Afterwards we apply~\eqref{sup-bound-Ro} and
obtain 
\begin{align*}
  \essup_{\tilde z+Q_{R_1}}|Du|
  &\le
  \lambda^{\frac{p-2}{p}}
  \essup_{\tilde z+Q_{R_1}}|Du|^{\frac{2}{p}}\\
  &\le
  \boldsymbol\gm \lambda^{\frac{p-2}{p}}
  \left[
    \epsilon^{\frac1{p'}}\lambda
    +
    \epsilon^{-\theta_4}\big(1+\lambda^{\frac{2-p}{2}}\big)
    \frac{\boldsymbol\omega}{R_2-R_1}
    +\eps^{-\theta_4}\mu
    \right]^{\frac{2}{p}}\\
    &\le
    \boldsymbol\gm 
    \big( \epsilon^{\frac1{p'}}\big)^\frac{2}{p}\lambda
    +
    \boldsymbol\gm \epsilon^{-\frac{2\theta_4}p}\big(\lambda^{\frac{p-2}{2}}+1\big)
    \Big(\frac{\boldsymbol\omega}{R_2-R_1}\Big)^\frac2p
    +\boldsymbol\gm \lambda^{\frac{p-2}{p}}\epsilon^{-\frac{2\theta_4}p}\mu^\frac2p
\\
     &\le
    \tfrac1{4 \boldsymbol\gm_o}\lambda
    +
     \boldsymbol\gm\big(\lambda^{\frac{p-2}{p}}+1\big)\Big(\frac{\boldsymbol\omega}{R_2-R_1}\Big)^{\frac{2}{p}}
    + \boldsymbol\gm\lambda^{\frac{p-2}{p}}\mu^{\frac{2}{p}},
\end{align*}
provided in the last estimate we choose $\eps=\eps(N,p,C_o,C_1,\alpha)\in(0,1)$ small enough to ensure
$ \boldsymbol\gm (\eps^{\frac{1}{p'}})^{\frac{2}{p}}\le\frac1{4 \boldsymbol\gm_o}$. Now, an
application of Young's inequality, again with exponents $\frac{p}{p-2}$ and $\frac{p}{2}$, yields that~\eqref{sup-bound-p<2}
holds in the case $p\ge2$ as well. 
Therefore, in view of the definition of
$\lambda$ in~\eqref{lambda-gradient-bound}, no matter the range of $p$, we always have
\begin{align*}
  \essup_{\tilde z+Q_{R_1}}|Du|
  \le
    \tfrac12\essup_{\tilde z+Q_{R_2}}|Du|
    +\boldsymbol\gm \bigg[\frac{\boldsymbol\omega}{R_2-R_1}
    +
    \Big(\frac{\boldsymbol\omega}{R_2-R_1}\Big)^{\frac{2}{p}}
    +\mu
   \bigg],
\end{align*}
 for all radii with $\rho\le R_1<R_2< 2\rho$. 
Applying the iteration lemma~\ref{lem:Giaq}, we arrive at the gradient
bound 
\begin{equation}\label{gradient-bound}
  \essup_{\tilde z+Q_{\rho}}|Du|
  \le
    \boldsymbol\gm \bigg[\frac{\boldsymbol\omega}{\rho}
    +
    \Big(\frac{\boldsymbol\omega}{\rho}\Big)^{\frac{2}{p}}
    +\mu
    \bigg],
\end{equation}
with a constant $\boldsymbol\gm =\boldsymbol\gm (N,p,C_o,C_1,\alpha)$.

For the proof of~\eqref{gradient-sup-bound-intro},
we consider a subset $\mathcal{K}\subset E_T$ such that 
\begin{equation}\label{Eq:rho-dist}
    \rho:=\tfrac14\mathrm{dist}_\mathrm{par}(\mathcal{K},\partial_{\rm par}E_T)>0.
\end{equation}
Then, it is easy to see that $\tilde z +Q_{2\rho}\Subset E_T$ for any
$\tilde z\in \mathcal{K}$.  Moreover we introduce
\begin{equation*}
     \mathcal{U}:=
     \bigcup_{\tilde z\in \mathcal K} \big( \tilde z+Q_\rho\big)
     \end{equation*}
the parabolic $\rho$-neighbourhood of $\mathcal{K}$. By the very definition of $\mathcal U$, 
estimate~\eqref{gradient-bound} implies 
\begin{equation}\label{gradient-bound-on-U}
  \sup_{\mathcal{U}}\big(\mu^2+|Du|^2\big)^{\frac12}
  \le
    \boldsymbol\gm \bigg[\frac{\osc_{E_T} u}{\rho}
    +
    \Big(\frac{\osc_{E_T} u}{\rho}\Big)^{\frac{2}{p}}
    +\mu
    \bigg]
\end{equation}
with a constant $\boldsymbol\gm = \boldsymbol\gm (N,p,C_o,C_1,\alpha)$. 
Since $\mathcal{K}\subset\mathcal{U}$,
this proves the claimed gradient estimate~\eqref{gradient-sup-bound-intro}
in the case $\mu\in(0,1]$ under the additional assumption $|Du|\in
L^\infty_{\mathrm{loc}}(E_T)$. 

\begin{remark}\upshape \label{rem:super-critical}
In the {\bf super-critical case} $p>\frac{2N}{N+2}$, we can
replace~\eqref{energy-bound-Ro} by an estimate that is independent
of the oscillation of $u$. In fact, in view of the definition of
$R_o$ in~\eqref{choice-R_o}, we obtain 
\begin{align}\label{alternative-energy-bound}
    \nonumber
     \biint_{z_o+Q_{R_o}^{(\lambda)}}&\big(\mu^2+|Du|^2\big)^{\frac{p}{2}}\,\dx\dt\\\nonumber
     &\le
       \frac{R_2^{N+2}}{R_o^{N+2}\lambda^{2-p}}
       \biint_{\tilde z+ Q_{R_2}}\big(\mu^2+|Du|^2\big)^{\frac{p}{2}}\dx\dt\\
     &=
       \Big(\frac{2R_2}{R_2-R_1}\Big)^{N+2}
       \max\Big\{\lambda^{p-2},\lambda^{\frac{N(2-p)}{2}}\Big\}
       \biint_{\tilde z+Q_{R_2}}\big(\mu^2+|Du|^2\big)^{\frac{p}{2}}\, \dx\dt.
   \end{align}
   By definition of $\lambda$ in \eqref{lambda-gradient-bound}, the integrand on the left-hand side is bounded from above by $\lambda^p$. In the case $\frac{2N}{N+2}<p<2$, we use this fact to estimate 
   \begin{align*}
     \biint_{z_o+Q_{R_o}^{(\lambda)}}&\big(\mu^2+|Du|^2\big)^{\frac{p}{2}}\,\dx\dt\\\nonumber
     &\le
     \lambda^{\frac{p(2-p)}{2}}\bigg[\biint_{z_o+Q_{R_o}^{(\lambda)}}\big(\mu^2+|Du|^2\big)^{\frac{p}{2}}\,\dx\dt\bigg]^{\frac{p}{2}}\\\nonumber
     &\le
     \Big(\frac{2R_2}{R_2-R_1}\Big)^{\frac{p(N+2)}{2}}
     \!\!\max\Big\{1,\lambda^{\frac{p(2-p)(N+2)}{4}}\Big\}
     \bigg[\biint_{\tilde z+Q_{R_2}}\!\!\big(\mu^2+|Du|^2\big)^{\frac{p}{2}}\,\dx\dt\bigg]^{\frac{p}{2}}\\\nonumber
     &\le
     \eps\lambda^p
     +
     \Big(\frac{2R_2}{R_2-R_1}\Big)^{\frac{p(N+2)}{2}}
     \bigg[\biint_{\tilde z+Q_{R_2}}\big(\mu^2+|Du|^2\big)^{\frac{p}{2}}\,\dx\dt\bigg]^{d_2}\\
    &\phantom{\le\,}
    +
     \boldsymbol\gm_\eps\Big(\frac{2R_2}{R_2-R_1}\Big)^{d_1(N+2)}
    \bigg(\biint_{\tilde z+Q_{R_2}}\big(\mu^2+|Du|^2\big)^{\frac{p}{2}}\,\dx\dt\bigg)^{d_1}
\end{align*}
where 
\[
    d_1:=\frac{2p}{p(N+2)-2N}\quad\mbox{and}\quad d_2:=\tfrac12 p
\] 
denote the {\bf scaling deficits}.
For the last estimate, we applied Young's inequality with exponents 
$\frac{4}{(2-p)(N+2)}$ and $\frac{4}{p(N+2)-2N}$.

For exponents in the super-quadratic range  $p\ge2$, we argue similarly.
In fact, we have 
\begin{align*}
     \biint_{z_o+Q_{R_o}^{(\lambda)}}&\big(\mu^2+|Du|^2\big)^{\frac{p}{2}}\,\dx\dt\\\nonumber
     &\le
     \lambda^{p(1-d_1)}\bigg[\biint_{z_o+Q_{R_o}^{(\lambda)}}\big(\mu^2+|Du|^2\big)^{\frac{p}{2}}\, \dx\dt\bigg]^{d_1}\\\nonumber
     &\le
     \Big(\frac{2R_2}{R_2-R_1}\Big)^{d_1(N+2)}
     \max\big\{\lambda^{p-2d_1},1\big\}
     \bigg[\biint_{\tilde z+Q_{R_2}}\big(\mu^2+|Du|^2\big)^{\frac{p}{2}}\,\dx\dt\bigg]^{d_1}\\\nonumber
     &\le
     \eps\lambda^p
     +
     \boldsymbol\gm_\eps\Big(\frac{2R_2}{R_2-R_1}\Big)^{\frac{p(N+2)}{2}}
    \bigg[\biint_{\tilde z+Q_{R_2}}\big(\mu^2+|Du|^2\big)^{\frac{p}{2}}\,\dx\dt\bigg]^{\frac{p}{2}}\\
    &\phantom{\le\,}+
     \Big(\frac{2R_2}{R_2-R_1}\Big)^{d_1(N+2)}
     \bigg[\biint_{\tilde z+Q_{R_2}}\big(\mu^2+|Du|^2\big)^{\frac{p}{2}}\, \dx\dt\bigg]^{d_1},
\end{align*}
for every $\epsilon\in(0,1)$. 
In the last line, we applied Young's inequality with exponents 
$\frac{p}{p-2d_1}$ and $\frac{p}{2d_1}$. 
Proceeding similarly as for the derivation of~\eqref{gradient-bound}
now leads to a sup-bound for the
gradient that is independent of the oscillation of $u$. 
More precisely, for every $p>\frac{2N}{N+2}$ we obtain the bound 
\begin{equation*}
    \essup_{\tilde z+Q_{\rho}}|Du|
    \le
    \boldsymbol\gm \sum_{k=1}^{2} \bigg[\biint_{\tilde z+Q_{2\rho}}(\mu^2+|Du|^2)^{\frac{p}{2}}\dx\dt\bigg]^{\frac{d_k}{p}},
\end{equation*}
with a constant $\boldsymbol\gm =\boldsymbol\gm (N,p,C_o,C_1,\alpha)$ and the scaling deficits $d_1$ and $d_2$.
For exponents $1<p\le\frac{2N}{N+2}$, however, the above procedure is
not applicable, since then the exponent $\frac{N(2-p)}{2}$ of $\lambda$ that
appears in~\eqref{alternative-energy-bound} is not smaller than
$p$. Therefore, it seems to be unavoidable that the right-hand side
of estimate~\eqref{gradient-bound} depends on the oscillation of
$u$ if $1<p\le\frac{2N}{N+2}$.\hfill$\Box$
\end{remark}

\subsection{Step 5: H\"older regularity of the gradient}
Our next goal is the proof of the H\"older estimate~\eqref{gradient-holder-bound-intro}. To this end, we wish to select the parameter $\lm$ to satisfy \eqref{def-lambda}, such that the Campanato type estimate \eqref{Campanato-bound-2} holds. 
Let $\rho$ be defined in \eqref{Eq:rho-dist} and $\boldsymbol\gm =\boldsymbol\gm(N,p,C_o,C_1,\alpha)$ be as in~\eqref{gradient-bound-on-U}. 
Recall also that $\boldsymbol\gm_o=\boldsymbol\gm _o(N,p,C_o,C_1)$ has been chosen in the course of deriving \eqref{sub-bound-Dw-lambda}.
Choose the parameter 
\begin{equation*}
  \lambda
  :=
  \boldsymbol\gm_o \boldsymbol\gm\bigg[\frac{\osc_{E_T} u}{\rho}
    +
    \Big(\frac{\osc_{E_T} u}{\rho}\Big)^{\frac{2}{p}}
    +\mu
    \bigg].
\end{equation*}
Consider the radii $R_2=\rho$, $R_1=\frac \rho2$, and
$R_o=\frac14 \min\{1,\lambda^{\frac{p-2}{2}}\}\rho$.
In view of~\eqref{gradient-bound-on-U}, we have
\begin{equation*}
  \boldsymbol\gm_o\Big(\mu^2+\|Du\|_{L^\infty(\tilde z+Q_{R_2})}^2\Big)^{\frac12}
  \le
  \boldsymbol\gm_o\Big(\mu^2+\|Du\|_{L^\infty(\mathcal{U})}^2\Big)^{\frac12}
  \le
  \lambda,
\end{equation*}
so that \eqref{def-lambda} is satisfied for any $\tilde
z\in\mathcal{K}$. Consequently,
the results of the preceding subsections are available for the above  parameters.    
In particular, from~\eqref{Campanato-bound-2}, we infer 
\begin{align*}
  \biint_{z_o+Q_{r}^{(\lambda)}}\big|Du-(Du)_r^{(\lambda)}\big|^p\dx\dt
  &\le
    \boldsymbol\gm\Big(\frac{r}{R_o}\Big)^{\alpha_o p}\lambda^p
\end{align*}
for every $z_o\in\mathcal{K}$ and every $r\in(0,R_o)$.
Moreover, using~\eqref{convergence-means} with
$\epsilon=1$ and estimating the integrand on the right-hand side by $\lambda^p$, we readily deduce
\begin{equation*}
  \big|\Gamma_{z_o}-(Du)_r^{(\lambda)}\big|^p
  \le
  \boldsymbol\gm\Big(\frac{r}{R_o}\Big)^{\alpha_o p}\lambda^p.
\end{equation*}
Joining the two preceding estimates gives 
\begin{equation}\label{campanato-Gamma}
  \biint_{z_o+Q_{r}^{(\lambda)}}|Du-\Gamma_{z_o}|^p\dx\dt
  \le
    \boldsymbol\gm\Big(\frac{r}{R_o}\Big)^{\alpha_o p}\lambda^p
\end{equation}
for every $z_o\in\mathcal{K}$ and every $r\in(0,R_o)$.
This results in the H\"older continuity of $Du$
in $\mathcal{K}\Subset\Omega_T$. In order to derive a quantitative
estimate,  consider two points $z_1,z_2\in \mathcal{K}$ that satisfy
\begin{equation*}
  0< d_\mathrm{par}^{(\lambda)}(z_1,z_2)\le\tfrac12\rho,
\end{equation*}
with the intrinsic parabolic distance $d_\mathrm{par}^{(\lambda)}$ defined in \eqref{intrinsic-distance}, and introduce 
$$
r:=2d_\mathrm{par}^{(\lambda)}(z_1,z_2) \quad\text{and}\quad 
  z_\ast:=\big(\tfrac12(x_1+x_2), \min\{t_1, t_2\}\big).
$$        
Applying~\eqref{campanato-Gamma} with $z_o=z_1$ and with $z_o=z_2$, we obtain
\begin{align*}
	|Du(z_1)& - Du(z_2)|^p\\\nonumber
	&=
	\biint_{z_\ast +Q_{r/2}^{(\lambda)}} 
	|\Gamma_{z_1} - \Gamma_{z_2}|^p \,\dx\dt \\\nonumber
	&\le
	c\bigg[
	\biint_{z_1+Q_{r}^{(\lambda)}} 
	|Du - \Gamma_{z_1}|^p \,\dx\dt +
	\biint_{z_2+Q_{r}^{(\lambda)}} 
	|Du - \Gamma_{z_2}|^p \,\dx\dt 
	\bigg] \\
	&\le
	\boldsymbol\gm\Big(\frac{r}{R_o}\Big)^{\alpha_o p}\lambda^p,
\end{align*}
where $\boldsymbol\gm =\boldsymbol\gm (N,p,C_o,C_1,\alpha)$.
Recalling the definition of $r$, we arrive at the estimate
\begin{equation}\label{holder-est}
    |Du(z_1) - Du(z_2)|
    \le
    \boldsymbol\gm \lambda\bigg(\frac{d_\mathrm{par}^{(\lambda)}(z_1,z_2)}{R_o}\bigg)^{\alpha_o},
\end{equation}
which holds for any $z_1,z_2\in \mathcal{K}$ with
$d_\mathrm{par}^{(\lambda)}(z_1,z_2)\le R_o$.
In the case $d_\mathrm{par}^{(\lambda)}(z_1,z_2)>R_o$,
we have the estimate
\begin{equation}\label{holder-bound-smooth}
  |Du(z_1) - Du(z_2)|
  \le
  2\|Du\|_{L^\infty(\mathcal{K})}
  \le
  \boldsymbol\gm \lambda\bigg(\frac{d_\mathrm{par}^{(\lambda)}(z_1,z_2)}{R_o}\bigg)^{\alpha_o},
\end{equation}
so that~\eqref{holder-est} holds in this case as well. In view
of the choice of $\lambda$ in~\eqref{gradient-bound-on-U} and the definition of $R_o$, this yields the
desired quantitative estimate \eqref{gradient-holder-bound-intro} in the
case $\mu\in(0,1]$ and $|Du|\in L^\infty_{\mathrm{loc}}(E_T)$.

\section{Approximation}\label{sec:approx}

In this section, we prove Theorem~\ref{theorem:schauder:intro}
by a compactness argument. Loosely speaking, the given solution to \eqref{p-laplace} with a H\"older continuous coefficient will be approximated by solutions to regularized
problems, for which the estimates from Section~\ref{sec:Schauder-Lip}
are applicable.

As in the statement of the theorem, we consider a subset
$\mathcal{K}\subset E_T$ such that
$\rho:=\tfrac14\dist_{\mathrm{par}}(\mathcal{K},\partial_\mathrm{par} E_T)>0$. 
For the approximation procedure, we introduce the subdomain
$\widetilde E:=\{x\in E\colon \dist(x,\partial E)>\eps_o\}$ for some $\eps_o\in(0,\rho)$. Next, we consider a standard mollifier $\phi\in
C^\infty_0(K_1)$ with $\int_{K_1}\phi\,\dx=1$. For a sequence
$\eps_i\in(0,\eps_o)$ with $\varepsilon_i\downarrow0$
and the scaled mollifiers
$\phi_{\varepsilon_i}(x):=\varepsilon_i^{-N}\phi(\frac{x}{\varepsilon_i})$
we define regularized coefficients by letting
\begin{equation*}
  a_i(x,t):=\int_{K_{\varepsilon_i}}a(x-y,t)\phi_{\varepsilon_i}(y)\,\dy
\end{equation*}
for every $x\in\widetilde E$ and a.e.~$t\in[0,T]$.
Using assumptions~\eqref{prop-a} of the coefficient $a(x,t)$, it is
straightforward to show that the regularized coefficients satisfy
\begin{equation}\label{prop-a-regular}
    \left\{
    \begin{array}{c}
    C_o\le a_i(x,t)\le C_1,\\[6pt]
    |a_i(x,t)-a_i(y,t)|\le C_1|x-y|^\alpha,\\[6pt]
    |D_xa_i(x,t)|\le \frac{C_1\|D\phi\|_{L^1}}{\varepsilon_i},    
    \end{array}
    \right.
\end{equation}
for every $i\in\N$, $x,\,y\in\widetilde E$, and a.e.~$t\in[0,T]$, with the
constants $0<C_o\le C_1$ and $\alpha\in(0,1)$ from~\eqref{prop-a}. In addition, we note that~\eqref{prop-a}$_2$ implies
\begin{equation}
  \label{diff-a}
  |a_i(x,t)-a(x,t)|\le C_1\eps_i^\alpha,
\end{equation}
for every $i\in\N$, and a.e.~$(x,t)\in\widetilde E_T$.
Moreover, introduce the notation
\begin{equation*}
\mu_i:=\left\{
    \begin{array}{cl}
        \varep_i & \text{for}\>\mu=0, \\[5pt]
        \mu & \text{for}\>\mu\in(0,1].
    \end{array}
    \right.
\end{equation*}

Let us consider
weak solutions $u_i\in L^p(0,T;W^{1,p}(\widetilde E))\cap
C([0,T];L^2(\widetilde E))$ to the regularized problems 
\begin{equation*}
  \left\{
  \begin{array}{c}
    \partial_tu_i-\Div\big(a_i(x,t)\big(\mu_i^2+|Du_i|^2\big)^{\frac{p-2}{2}}Du_i\big)=0
    \quad \mbox{in $\widetilde E_T$},\\[6pt]
    u_i=u\quad \mbox{on $\partial_\mathrm{par}\widetilde E_T$,}
  \end{array}
  \right.
\end{equation*}
where the boundary datum is taken in the sense of Definition~\ref{def:weak}, and the existence was discussed in Remark~\ref{Rmk:CP2}.


Apply the comparison estimate from Lemma~\ref{lem:Comp2}
with  
\begin{equation*}
    \mathbf A(x,t,\xi):=a(x,t)\big(\mu^2+|\xi|^2\big)^{\frac{p-2}{2}}\xi
    \quad\mbox{and}\quad
    \mathbf B(x,t,\xi):=a_i(x,t)\big(\mu_i^2+|\xi|^2\big)^{\frac{p-2}{2}}\xi,
\end{equation*}
(note that $\mathbf A$ fulfills \eqref{p-growth} by Lemma \ref{Monotonicity})
to obtain
\begin{align*}
    \iint_{\widetilde E_T}|Du-Du_i|^p\dx\dt
    &\le
    \boldsymbol\gm \|a_i-a\|_{L^\infty(\widetilde E_T)}^{p'}
    \iint_{\widetilde E_T}\big(\mu_i^2+|Du|^2\big)^{\frac{p}2}\dx\dt\\\nonumber
    &\phantom{\le\,}
    +
    \boldsymbol\gm \|a_i-a\|_{L^\infty(\widetilde E_T)}^{p}
     \iint_{\widetilde E_T}\big(\mu_i^2+|Du|^2\big)^{\frac{p}2}\dx\dt\\\nonumber
   &\le
     \boldsymbol\gm 
     \big(\eps_i^{\al p'}+\eps_i^{\al p}\big)\iint_{\widetilde E_T}\big(1+|Du|^2\big)^{\frac{p}2}\dx\dt,
\end{align*}
with a constant $\boldsymbol\gm  =\boldsymbol\gm  (p,C_o,C_1)$. In the last step, we
applied~\eqref{diff-a}. The preceding estimate proves
\begin{equation}\label{Lp-convergence}
  Du_i\to Du \quad\mbox{in $L^p(\widetilde E_T)$ as $i\to\infty$}.
\end{equation}

Next, applying the comparison principle in Proposition~\ref{prop:comparison-plapl} we infer
\begin{equation}
  \label{comp-principle}
  \boldsymbol \omega_i
  :=
  \osc_{\widetilde E_T}u_i
  \le
  \osc_{E_T}u
  =:
  \boldsymbol \omega
  \qquad\mbox{for every $i\in\N$.}
\end{equation}
In particular, the approximating solutions $u_i$ are bounded. 
The first and the last conditions in~\eqref{prop-a-regular} ensure that the
properties~\eqref{prop-a-diff} are satisfied for $b=a_i$, with a
constant $C_2$ that depends on $i\in\N$. This means that
Propositions~\ref{PROP:LINFTY} and \ref{PROP:LINFTY:P>2} are applicable to the solutions $u_i$,
and we infer $|Du_i|\in L^\infty_{\mathrm{loc}}(\widetilde E_T)$. Once we have this qualitative information, the results obtained in Section~\ref{sec:Schauder-Lip}, in particular inequalities \eqref{gradient-bound-on-U} and \eqref{holder-bound-smooth}, are at our disposal.
Hence, we
infer the estimates 
\begin{equation}
    \label{gradient-bound-i}
    \sup_{\mathcal{K}}|Du_i|
    \le
     \boldsymbol\gm \bigg[\frac{\boldsymbol \omega_i}{\rho}
     +
     \Big(\frac{\boldsymbol\omega_i}{\rho}\Big)^{\frac{2}{p}}
     +\mu_i
     \bigg]
     =:\boldsymbol\gm _o^{-1}\lambda_i,
\end{equation}
as well as 
\begin{equation}\label{holder-bound-i}
    |Du_i(z_1) - Du_i(z_2)|
    \le
    \boldsymbol\gm\lambda_i
    \bigg(\frac{d_\mathrm{par}^{(\lambda_i)}(z_1,z_2)}{\min\{1,\lambda_i^{\frac{p-2}{2}}\}\rho}\bigg)^{\alpha_o}
\end{equation}
for any $z_1,z_2\in\mathcal{K}$.
The constants $\boldsymbol\gm ,\boldsymbol\gm_o\ge1$ and 
$\alpha_o\in(0,1)$ in the preceding
estimates depend only on $N,p,C_o,C_1,$ and $\alpha$, and are in
particular independent of $i\in\N$. 
We note that the right-hand side of the preceding estimate is monotonically increasing in $\lambda_i$, because it can be rewritten as 
\begin{align*}
    \boldsymbol\gm\rho^{-\alpha_o}
    \Big(\max\Big\{\lambda_i^{\frac{1}{\alpha_o}},\lambda_i^{\frac{1}{\alpha_o}+\frac{2-p}{2}}\Big\}|x_1-x_2|+
    \max\Big\{\lambda_i^{\frac{1}{\alpha_o}+\frac{p-2}{2}},\lambda_i^{\frac{1}{\alpha_o}}\Big\} \sqrt{|t_1-t_2|}\Big)^{\alpha_o}.
\end{align*}
Indeed, for $1<p<2$ all  exponents of $\lm_i$ in the last displayed quantity are obviously non-negative. In the case $p>2$, the definition of $\alpha_o$ and $\alpha_\ast$ in~\eqref{def:alpha-0} and \eqref{def:alpha-ast} imply $\alpha_o\le\frac12\alpha_\ast=\frac{\alpha} {2(p-1)}$, so that
$\frac1{\alpha_o}\ge \frac{2(p-1)}{\alpha}\ge \frac{p-2}2$. Hence also in this case all exponents of $\lm_i$  are  non-negative.

We abbreviate
  \begin{equation*}
    \lambda
    :=
    \boldsymbol\gm_o\boldsymbol\gm \bigg[\frac{\osc_{E_T}u}{\rho}
    +
    \Big(\frac{\osc_{E_T}u}{\rho}\Big)^{\frac{2}{p}}
    +\mu
    \bigg],
  \end{equation*}
  and note that \eqref{comp-principle} implies $\limsup_{i\to\infty}\lambda_i\le\lambda$.
  Therefore, passing to the limit in \eqref{gradient-bound-i} and in \eqref{holder-bound-i}, we obtain
  \begin{equation*}
    \limsup_{i\to\infty}\,\sup_{\mathcal{K}}|Du_i|
    \le
    \boldsymbol\gm_o^{-1}\lambda
  \end{equation*}
  and 
  \begin{equation*}
    \limsup_{i\to\infty}\,|Du_i(z_1) -
    Du_i(z_2)|
    \le
    \bg\lambda
    \bigg(\frac{d_\mathrm{par}^{(\lambda)}(z_1,z_2)}{\min\{1,\lambda^{\frac{p-2}{2}}\}\rho}\bigg)^{\alpha_o} 
  \end{equation*}
  for every $z_1,z_2\in\mathcal{K}$. 
  In particular, the sequence $(Du_i)_{i\in\N}$ is bounded in the space
  $C^{\alpha_o,\alpha_o/2}(\mathcal{K},\R^N)$. Therefore, the theorem
  of Ascoli-Arzel\`a and the convergence~\eqref{Lp-convergence} imply
  $Du_i\to Du$ uniformly in $\mathcal{K}$, as $i\to\infty$. Hence,
  passing to the limit in the preceding estimates, we obtain 
  \begin{equation*}
    \sup_{\mathcal{K}}|Du|
    \le
    \bg_o^{-1}\lambda
    \qquad\mbox{and}\qquad
    |Du(z_1) - Du(z_2)|
    \le
    \bg\lambda
    \bigg(\frac{d_\mathrm{par}^{(\lambda)}(z_1,z_2)}{\min\{1,\lambda^{\frac{p-2}{2}}\}\rho}\bigg)^{\alpha_o}
  \end{equation*}
  for any $z_1,z_2\in\mathcal{K}$. This proves the asserted
  estimates~\eqref{gradient-sup-bound-intro}
  and~\eqref{gradient-holder-bound-intro} and completes the proof of
  Theorem~\ref{theorem:schauder:intro}.








\chapter[Regularity for the doubly non-linear equation]{Regularity for the doubly non-linear equation}\label{sec:regularity}

In this chapter we prove the main regularity Theorem~\ref{THM:REGULARITY-INTRO}. Thereby, we will take advantage of the previously established theory, in particular the Harnack inequality and the Schauder estimates. However, prior to the proof, 
we first remark on the relevant literature, and then briefly describe the optimality of the range of exponents $ 0<p-1<q<\frac{N(p-1)}{(N-p)_+}$ considered in Theorem~\ref{THM:REGULARITY-INTRO}.

\section{Remarks on the literature}\label{sec:reg:opt}
As we have already discussed in Chapter~\ref{Intro}, 
the first investigations of the gradient regularity of solutions to \eqref{doubly-nonlinear-prototype} are due to Ivanov \& Mkrtychyan \cite{Ivanov-Mk}, Ivanov \cite{Ivanov-1996}, and Savar\'e \& Vespri \cite{Savare}. 

The results of the first two authors are not easily comparable with ours, since they are given for so-called {\em regular} solutions, that is, weak solutions to a Cauchy-Dirichlet problem built as suitable limits.

In the slow diffusion regime ($p>1$ \& $q<p-1$) which we do not deal with here,  a bound for $Du^\alpha$ with $\alpha\in(0,1)$ holds for $q\le \frac{p-1}p$. Such a result is not surprising; indeed, as discussed in Chapter~\ref{Intro}, provided the waiting time is taken into account, for the porous medium equation ($p=2$ \& $0<q\le\frac12$) the bound holds for $Du^{1-q}$. 

As far as $Du$ is concerned, combining \cite[Theorem~6.1]{Ivanov-1996} and the correction pointed out in \cite[\S~5]{Ivanov-1997}, yields the claim that the range of values of $p$ and $q$ in order for $Du$ to be bounded is 
\[
p>1,\quad 0<q\le 1,\quad  \frac{(p-1)(1-q)}{q}<\frac{2\min\{1,p-1\}}{2\min\{1,p-1\}+\frac N4}.
\]
Even if we limit ourselves to the fast diffusion regime ($0<p-1<q$), which is the case we are interested in here, this does not correspond to the results of Theorem~\ref{THM:REGULARITY-INTRO}, and it is not clear how the Cauchy-Dirichlet problem Ivanov starts from, affects the interior regularity of the gradient. 


Finally, 
in \cite[Theorem~1.5]{Savare}, with a different notion of solution to \eqref{doubly-nonlinear-prototype} Savar\'e \& Vespri provided a gradient estimate
\[
|Du(x_o,t_o)|\le \frac{C u(x_o,t_o)}{\rho}
\]
for some $C=C(p,q,N)$ and for parameters $p<N$, $0<p-1< q<\frac{N(p-1)}{N-p}$.

\section{Optimality}\label{sec:reg:opt-1}
Let us now present some counterexamples which illustrate that the statement of Theorem~\ref{THM:REGULARITY-INTRO} breaks down in the cases $q=p-1$ and $q=\frac{N(p-1)}{(N-p)_+}$. As a matter of fact, the same counterexamples of Section~\ref{S:Harnack-opt} show the optimality of Theorem~\ref{THM:REGULARITY-INTRO}.

We start with $q=p-1$: in this case the function 
\[
u(x,t)=Ct^{-\frac{N}{p(p-1)}} \exp\left(-\frac{p-1}p\left(\frac{|x|^p}{pt}\right)^{\frac1{p-1}}\right)
\]
is a solution to \eqref{doubly-nonlinear-prototype} in $\R^N\times\R_+$ for arbitrary $C>0$. If we choose $t=1$ and $\rho=1$ and pick a sequence of points $x_n$ such that $x_n\to\infty$, it is easy to see that \eqref{grad-est-intro} cannot hold. This fact was first pointed out in \cite[Remark~1.6]{Savare}.

The estimate \eqref{grad-est-intro} does not hold in general when $q\ge\frac{N(p-1)}{(N-p)_+}$. Suppose the estimate were to hold for some $u_o>0$. Then letting $\rho\to\infty$ it would have implied $u$ is constant in $\rn\times\rr$. However, this contradicts the non-constant solutions in Section~\ref{S:Harnack-opt}.

For parameters
\[
N\ge2,\quad N>p,\quad q=\frac{N(p-1)}{N-p},\quad b=\left(\frac Nq\right)^q,
\]
the function
\[
u(x,t)=\left(|x|^{\frac{N(q+1)}{q(N-1)}}+e^{bt}\right)^{-\frac{N-1}{q+1}}
\]
is a non-negative solution to \eqref{doubly-nonlinear-prototype} in $\R^N\times\R$. If we take $(x_o,t_o)=(0,0)$, we have $u_o=1$; since $u$ is a solution everywhere, $\tilde\gamma Q_o$ can be taken with arbitrary $\varrho>0$.

The right-hand side of \eqref{grad-est-intro} reduces to $\frac\gamma\varrho$. As for the left-hand side, we have
\[
|Du(x,t)|=\frac Nq\frac{|x|^{\frac{N+q}{q(N-1)}}}{\left[|x|^{\frac{N(q+1)}{q(N-1)}}+e^{bt}\right]^{\frac{N+q}{q+1}}}.
\]
It is apparent that 
\[
\sup_{Q_o}|Du|\ge\sup_{K_\varrho}|Du(\cdot,0)|=\sup_{r\in(0,\varrho)} f(r),
\]
where we have set 
\[
|x|=r,\quad f(r)=\frac Nq\frac{r^{\frac{N+q}{q(N-1)}}}{\left[r^{\frac{N(q+1)}{q(N-1)}}+1\right]^{\frac{N+q}{q+1}}}.
\]
It is a matter of straightforward computations to check that if
\[
\rho>\left(\frac1{N-1}\right)^{\frac{q(N-1)}{N(q+1)}}:=\bar r,
\]
we have 
\[
\sup_{0<r<\varrho} f(r)=f(\bar r)=\gamma(N,q).
\]
Hence, $\sup_{Q_o}|Du|\ge\gamma(N,q)$, and provided $\varrho$ is chosen sufficiently large, \eqref{grad-est-intro} cannot hold.

\color{black}
\section{Proof of Theorem~\ref{THM:REGULARITY-INTRO} and Corollary~\ref{COR:GRAD-REG}}

We are now in the position to prove the main regularity results for doubly non-linear equation.

\begin{proof}[{\rm \textbf{Proof of Theorem~\ref{THM:REGULARITY-INTRO}}}]
The continuity of $u$ is only used in order to give $u(x_o,t_o)$ an unambiguous meaning; see also our previous Remark \ref{Rmk:semicontin}.
From Theorem~\ref{THM:BD:1} we infer that $u$ is locally bounded in $E_T$. 
We consider $z_o=(x_o,t_o)\in E_T$ such that $u_o:=u(x_o,t_o)>0$. By $\sigma=\sigma(N,p,q)\in(0,1)$ and 
$\boldsymbol \gm_{\rm h}=\boldsymbol\gm_{\rm h} (N,p,q)>1$ we denote the constants from Theorem~\ref{THM:HARNACK}.  
Now, let $\rho>0$ and define $\widetilde{\boldsymbol\gm}:= \frac{8\bg_{\rm h}}{\sigma}>1$ and assume that
$$
	K_{\widetilde{\boldsymbol\gm}\rho}(x_o) \times \big(t_o - u_o^{q+1-p}(\widetilde{\boldsymbol\gm}\rho)^p, t_o + u_o^{q+1-p}(\widetilde{\boldsymbol\gm}\rho)^p\big)
	\subset E_T.
$$
Then the smaller cylinder
$$ 
    K_{8\rho/\sigma^\frac1p}(x_o)\times 
	\Big(t_o- \bg_{\rm h} u_o^{q+1-p}\big(8\rho/\sigma^\frac1p\big)^p,t_o+\bg_{\rm h} u_o^{q+1-p}\big(8\rho/\sigma^\frac1p\big)^p\Big)
$$ 
is also contained in $E_T$, so that we are allowed to apply Harnack's inequality on this cylinder. Moreover, we have 
\begin{align*}
    Q_o
    &:=
    K_{\rho}(x_o)\times 
	\big(t_o- u_o^{q+1-p}\rho^p,t_o+ u_o^{q+1-p}\rho^p\big) \\
	&\,\subset
    K_{\rho/\sigma^\frac1p}(x_o)\times 
	\big(t_o- u_o^{q+1-p}\rho^p,t_o+ u_o^{q+1-p}\rho^p\big) \\
	&\,=
    K_{\rho/\sigma^\frac1p}(x_o)\times 
	\Big(t_o- \sigma u_o^{q+1-p}\big(\rho/\sigma^\frac1p\big)^p,t_o+\sigma u_o^{q+1-p}\big(\rho/\sigma^\frac1p\big)^p\Big).
\end{align*}
Therefore, by Harnack's inequality we conclude
\begin{align}\label{start-Harnack}
	\tfrac1{\boldsymbol\gm_{\rm h}} u_o
	\le
	u(x,t)
	\le
	\boldsymbol\gm_{\rm h} u_o
	\quad\mbox{for a.e.~$(x,t)\in Q_o$.}
\end{align}
Re-scaling to the symmetric unit cylinder $\mathcal{Q}_1:=K_1(0)\times(-1,1)$, i.e. 
\begin{align*}
	\tilde u(y,s)
	:=
	\tfrac1{u_o}\, u\big(x_o+\rho y, t_o+ u_o^{q+1-p}\rho^p s\big)
	\quad
	\mbox{for $(y,s)\in \mathcal{Q}_1$,}
\end{align*}
leads to a weak solution $\tilde u$ to the doubly non-linear parabolic equation
\begin{equation}\label{DN-tilde}
	\partial_t \tilde u^q - 
	\Div\big(|D\tilde u|^{p-2}D \tilde u\big)
	=
	0
	\quad\mbox{in $\mathcal{Q}_1$,}
\end{equation}
with values in $[\bg_{\rm h}^{-1} ,\bg_{\rm h}]$.
Substituting $v=\tilde u^q$, equation \eqref{DN-tilde} is equivalent to 
\begin{equation*}
	\partial_t v - 
	\Div\Big(\big(\tfrac1q\big)^{p-1}\tilde u^{(p-1)(1-q)}|Dv|^{p-2}Dv\Big)
	=
	0
	\qquad\mbox{in $\mathcal{Q}_1$}
\end{equation*}
and \eqref{start-Harnack} yields
\begin{align}\label{lower-upper-v}
	\bg_{\rm h}^{-q}
	\le
	v(y,s)
	\le
	\bg_{\rm h}^{q}
	\quad
	\mbox{for a.e.~$(y,s)\in \mathcal{Q}_1$.}
\end{align}
With the abbreviation
\begin{align*}
	a(y,s)
	:=
	\big(\tfrac1q\big)^{p-1} [\tilde u(y,s)]^{(p-1)(1-q)},
\end{align*}
the equation for $v$ can be interpreted as a parabolic $p$-Laplace type equation with
measurable coefficients $a(y,s)$ in $\mathcal Q_1$. Indeed, $v$ is a bounded weak solution to 
\begin{equation}\label{equation-for-v}
	\partial_t v - 
	\Div\big(a(y,s)|Dv|^{p-2}Dv\big)
	=
	0
	\qquad
	\mbox{in $\mathcal{Q}_1$.}
\end{equation}
For the coefficients $a$ we have  lower and upper bounds in terms of $\bg_{\rm h}$, i.e.
\begin{equation} \label{bounds-for-a}
	\big(\tfrac1q\big)^{p-1}\bg_{\rm h}^{-(p-1)|1-q|}
	\le 
	a(y,s)
	\le 
	\big(\tfrac1q\big)^{p-1}\bg_{\rm h}^{(p-1)|1-q|}
	\quad
	\mbox{for a.e.~$(y,s)\in \mathcal{Q}_1$.}
\end{equation}
Therefore, we may apply Chapter III, \S\,1, Theorem~1.1, resp. Chapter IV, \S\,1, Theorem~1.1 in \cite{DB} to deduce that $v$ is locally Hölder continuous in $\mathcal{Q}_{1}$ with some Hölder exponent $\alpha\in(0,1)$ depending only on $N,p$ and $q$. In particular, $v$ is $\alpha$-Hölder continuous in $\frac12\mathcal Q_1=\mathcal Q_{\frac12}=K_\frac12 \times (-\frac14,\frac14)$. 
Recalling \eqref{lower-upper-v} and the dependencies of $\bg_{\rm h}$, the quantitative H\"older estimates from \cite{DB} yield
\begin{align*}
    |v(y_1,s)-v(y_2,s)|
    \le
    \boldsymbol \gm(N,p,q) |y_1-y_2|^\alpha
\end{align*}
for any $y_1,y_2\in B_{1/2}$ and any $s\in (-1/4,1/4)$. This and the lower and upper bound for $v$ in \eqref{lower-upper-v} imply that $a$ is also $\alpha$-Hölder continuous with the  quantitative  estimate 
\begin{align*}
    |a(y_1,s)-a(y_2,s)|
    \le
    \boldsymbol \gm (N,p,q) |y_1-y_2|^\alpha
\end{align*}
for any $y_1,y_2\in B_{1/2}$ and any $s\in (-1/4,1/4)$. Together with \eqref{equation-for-v} and \eqref{bounds-for-a} this shows that the hypotheses \eqref{prop-a} of  Theorem~\ref{theorem:schauder:intro} (with $\mu=0$)
are fulfilled with constants $C_o,C_1$ depending only on $N,p,q$. Therefore,
the Schauder-estimates from the same theorem are at our disposal. 
More precisely, we have $Dv\in C^{\alpha_o,\alpha_o/2}_{\mathrm{loc}}(\mathcal{Q}_{1/2},\R^N)$ for some  H\"older exponent $\alpha_o\in(0,1)$, depending only on $N,p,\bg_{\rm h}$ and $\alpha$. 
Since $\bg_{\rm h}=\bg_{\rm h} (N,p,q)$ and $\alpha=\alpha(N,p,q)$, we have $\alpha_o=\alpha_o(N,p,q)$.
Furthermore, the Schauder-estimates from the same theorem provide quantitative estimates. To be more precise, there exist constants $\boldsymbol\gm=\boldsymbol\gm(N,p,q)\ge1$  and $\boldsymbol\gm_o=\boldsymbol\gm_o(N,p,q)\ge1$ such that for any subset $\mathcal{K}\subset \mathcal {Q}_{1/2}$ with 
$\rho:= \tfrac14\mathrm{dist}_\mathrm{par}(\mathcal{K},\partial_\mathrm{par} \mathcal {Q}_{1/2})>0$,
the $L^\infty$-gradient bound
\begin{equation*}
    \sup_{\mathcal{K}}|Dv|
    \le
    \boldsymbol\gm\bigg[\frac{\boldsymbol\omega}{\rho}
    +
    \Big(\frac{\boldsymbol\omega}{\rho}\Big)^{\frac{2}{p}}
    \bigg]
    =:\boldsymbol\gm_o^{-1}\lambda,
\end{equation*}
and the $\alpha_o$-H\"older gradient estimate 
\begin{equation*}
    |Dv(\mathfrak z_1) - Dv(\mathfrak z_2)|
    \le
    \boldsymbol\gm\lambda
    \bigg[\frac{d_\mathrm{par}^{(\lambda)}(\mathfrak z_1,\mathfrak z_2)}{\min\{1,\lambda^{\frac{p-2}{2}}\}\rho}\bigg]^{\alpha_o}
    \quad\mbox{for any $\mathfrak z_1,\mathfrak z_2\in\mathcal{K}$,} 
\end{equation*}
hold true. Here we defined $\boldsymbol \omega:=\osc_{\mathcal{Q}_{\frac12}} v$.
For $\mathcal K=\mathcal{Q}_{\frac14}$ we have $\rho=\frac{\sqrt{3}}{16}$ and $\boldsymbol \omega\le \bg_{\rm h}^q$. Therefore, we have
\begin{equation*}
    \lambda
    =
    \boldsymbol\gm \boldsymbol\gm_o\bigg[\frac{\boldsymbol\omega}{\rho}
    +
    \Big(\frac{\boldsymbol\omega}{\rho}\Big)^{\frac{2}{p}}
    \bigg]
    \le 
    \boldsymbol\gm(N,p,q),
\end{equation*}
so that 
\begin{equation}\label{sup-bound-Dv}
    \sup_{\mathcal{Q}_{1/4}}|Dv|
    \le
    \boldsymbol\gm(N,p,q),
\end{equation}
and 
\begin{equation}\label{osc-bound-Dv}
    \big|Dv(\mathfrak z_1) - Dv(\mathfrak z_2)\big|
    \le
    \boldsymbol\gm(N,p,q)\,d_\mathrm{par}(\mathfrak z_1,\mathfrak z_2)^{\alpha_o}\quad \mbox{for any $\mathfrak z_1,\mathfrak z_2\in \mathcal{Q}_{1/4}$.}
\end{equation}
To convert in the case $p>2$ the intrinsic parabolic distance $d_\mathrm{par}^{(\lambda)}$
 into the parabolic distance $d_\mathrm{par}$, we took advantage of the fact that we may assume $\alpha_o\le \frac{2}{p-2}$.

Once boundedness of  $Dv$ is established, we can derive a bound for the oscillation of $v$ on $\mathcal{Q}_{\frac18}$ by an application of Lemma~\ref{lem:osc-bound-0}.  Indeed, for two points $\mathfrak z_1=(y_1,s_1),\,\mathfrak z_2=(y_2,s_2)\in \mathcal{Q}_{\frac18}$ with $s_1\le s_2$ we apply Lemma~\ref{lem:osc-bound-0} (with $\mu=0$) to $v$ on the cylinder $Q=K_r(\tilde y)\times(s_1,s_2]\subset \mathcal{Q}_{1/4}$, where $\tilde y=\frac12(y_1+y_2)$ and $\frac12|y_1-y_2|\le r\le\frac18$.  The application yields
\begin{equation*}
    |v(\mathfrak z_1)-v(\mathfrak z_2)|
    \le
    \osc_{Q} v
    \le 
    \boldsymbol\gm\bigg[r\|Dv\|_{L^\infty(Q)} +
    \frac{s_2-s_1}{r} \|Dv\|_{L^\infty(Q)}^{p-1}\bigg]
    \le
    \boldsymbol\gm\bigg[r + \frac{s_2-s_1}{r} \bigg].
\end{equation*}
Now, if $\sqrt{s_2-s_1}\le|y_1-y_2|$ we choose $r=\frac12|y_1-y_2|$, so that the right-hand side is bounded by $\boldsymbol\gm|y_1-y_2|$. Otherwise, if $\sqrt{s_2-s_1}>|y_2-y_1|$ we choose $r=\min\big\{\frac18,\sqrt{s_2-s_1}\big\}$ which results in the bound  $\boldsymbol\gm\sqrt{s_2-s_1}$. Joining both cases, we obtain  
\begin{equation*}
    |v(\mathfrak z_1)-v(\mathfrak z_2)|
    \le
    \boldsymbol\gm\,d_\mathrm{par}(\mathfrak z_1,\mathfrak z_2)
    \quad \mbox{for any $\mathfrak z_1,\mathfrak z_2\in \mathcal{Q}_{\frac18}$}
\end{equation*}
for a constant $\boldsymbol\gm= \boldsymbol\gm(N,p,q)$.
In view of \eqref{lower-upper-v} this implies for any $\beta\in \R$ 
that
\begin{equation}\label{holder-v-beta}
    |v^\beta(\mathfrak z_1)-v^\beta(\mathfrak z_2)|
    \le
    \boldsymbol\gm \,\bg_{\rm h}^{q|\beta-1|} d_\mathrm{par}(\mathfrak z_1,\mathfrak z_2)
    \quad \mbox{for any $\mathfrak z_1,\mathfrak z_2\in \mathcal{Q}_{1/8}$, } 
\end{equation}
where $\gm=\gm(N,p,q,\beta)$. 

Finally, we scale back to the original solution $u$. From \eqref{holder-v-beta} with $\beta =\frac1q$ we obtain
\begin{equation*}
    |u(z_1)-u(z_2)|
    \le
    \boldsymbol\gm u_o\,
    \Bigg[\frac{|x_1-x_2|}{\rho}+\sqrt{\frac{|t_1-t_2|}{u_o^{q+1-p}\rho^p}} \,\Bigg]
      \quad 
      \mbox{for any $z_1,z_2\in \frac18 Q_o$,} 
\end{equation*}
where $\boldsymbol\gm=\boldsymbol\gm(N,p,q,)$. This proves \eqref{diff-u-intro} after replacing $\frac18\rho$ by $\rho$ and $\widetilde\bg$ by $8\widetilde\bg$. 
Next, computing the spatial gradient $Du$ in terms of $v$ gives
\begin{align*}
    Du(x,t)
    =
    \frac{u_o}{q\rho} 
    v^{\frac1q-1}\bigg(\frac{x-x_o}{\rho}\,,\, \frac{t-t_o}{ u_o^{q+1-p}\rho^p}\bigg) 
    Dv\bigg(\frac{x-x_o}{\rho}\,,\, \frac{t-t_o}{ u_o^{q+1-p}\rho^p}\bigg).
\end{align*}
Then, \eqref{lower-upper-v} and \eqref{sup-bound-Dv} imply
\begin{equation*}
    \sup_{\frac14 Q_o}|Du|
    \le
    \frac{\boldsymbol\gm u_o}{\rho}.
\end{equation*}
By the same modifications as before, this yields \eqref{grad-est-intro}. 
Combining  \eqref{sup-bound-Dv}, \eqref{osc-bound-Dv},  and  \eqref{holder-v-beta} (with $\beta =\frac1q-1$) gives
\begin{equation*}
    |Du(z_1) - Du(z_2)|
    \le
    \frac{\boldsymbol\gm u_o}{\rho}\,
    \Bigg[\,\frac{|x_1-x_2|}{\rho}+\sqrt{\frac{|t_1-t_2|}{
      u_o^{q+1-p}\rho^p}}\, \Bigg]^{\alpha_o}
\end{equation*}
for any $z_1,z_2\in\frac18 Q_o$, which implies \eqref{C1alph-intro} as before. 
\end{proof}

An immediate consequence of Theorem~\ref{THM:REGULARITY-INTRO} is as follows.
\begin{corollary}
Under the assumptions of Theorem~\ref{THM:REGULARITY-INTRO} we have 
\begin{equation*}
    |u(z_1)-u(z_2)|\le \frac{\boldsymbol\gm ru_o}{\rho}.
\end{equation*}
and
\begin{equation*}
    |Du(z_1) - Du(z_2)|
    \le 
    \boldsymbol\gm\Big(\frac{r}{\rho}\Big)^{\al_o}\,\frac{u_o}{\rho}
\end{equation*}
for any points $z_1,z_2\in K_r(x_o)\times(t_o- u_o^{q+1-p}\rho^{p-2}r^2,t_o+ u_o^{q+1-p}\rho^{p-2}r^2)$ and any radius $0<r\le \rho$.

\end{corollary}
\begin{proof}
The direct application of \eqref{diff-u-intro}, respectively \eqref{C1alph-intro} yields the claimed estimates.
\end{proof}

Next, we show how to extend the gradient-sup-bound and the $C^{1,\alpha}$-estimates to general compact subsets in $E_T$. This is exactly the assertion of Corollary~\ref{COR:GRAD-REG}. 

\begin{proof}[{\rm \textbf{Proof of Corollary~\ref{COR:GRAD-REG}}}]
Let $\widetilde\bg>1$ be the constant from Theorem \ref{THM:REGULARITY-INTRO} and   $\rho:=\tfrac1{\widetilde\bg} \mathcal M^{-\frac{q+1-p}{p}}\rho_o$. Then the definition of $\rho_o$ implies 
  \begin{equation*}
    K_{\widetilde\bg\rho}(x)\times\big(t-\mathcal M^{q+1-p}(\widetilde\bg\rho)^p,t+\mathcal M^{q+1-p}(\widetilde\bg\rho)^p\big)\Subset E_T
    \qquad\mbox{for all }(x,t)\in \mathcal{K}.
  \end{equation*}
  Let us fix two points $(x_i,t_i)\in \mathcal{K}$ and denote $u_i:=u(x_i,t_i)$ for $i=1,2$.
  If $u_i>0$, we can apply the gradient bound \eqref{grad-est-intro} on the cylinders
  \begin{equation*}
    Q_i:= K_\rho(x_i)\times
    \big(t_i-u_i^{q+1-p}\rho^p,t_i+u_i^{q+1-p}\rho^p\big)\Subset E_T,
  \end{equation*}
  which implies in particular
  \begin{align}\label{est:Du_i}
    |Du(x_i,t_i)|
    \le
    \bg\frac{u_i}{\rho}
    \le
    \bg\frac{\mathcal{M}}{\rho}
    =
    \bg\widetilde\bg\frac{\mathcal{M}^{\frac{q+1}{p}}}{\rho_o}.
\end{align}
In the case $u_i=0$, however, we have $u=0$
on the whole time slice $E\times\{t_i\}$, since otherwise,
repeated applications of Harnack's inequality would imply $u_i>0$. 
Therefore, we infer $Du(x_i,t_i)=0$, so that~\eqref{est:Du_i} is trivially satisfied also in this case.
Since $(x_i,t_i)\in \mathcal{K}$ is arbitrary, this implies the asserted gradient bound~\eqref{grad-bound-K}.

For the proof of the $C^{1,\alpha}$-estimate, we distinguish  two cases.

\noi {\bf Case 1: } $|x_1-x_2|< \rho$ and $|t_1-t_2|< \max\{u_1,u_2\}^{q+1-p}\rho^p$. 

Without loss of generality, we assume that $u_1\ge u_2$, so that 
  $|t_1-t_2|< u_1^{q+1-p}\rho^p$. Note that this implies in particular  $u_1>0$. The case $u_1=u_2=0$ is included in Case~2. In the present situation we have
  \begin{equation}\label{Case1}
    (x_2,t_2)\in Q_1:= K_\rho(x_1)\times \big(t_1-u_1^{q+1-p}\rho^p,t_1+u_1^{q+1-p}\rho^p\big).
  \end{equation}
  Therefore, the $C^{1,\alpha}$-estimate \eqref{C1alph-intro} on the cylinder $Q_1$ implies
  \begin{equation*}
    |Du(x_1,t_1)-Du(x_2,t_2)|
    \le \bg\frac{u_1}{\rho}
    \Bigg[\frac{|x_1-x_2|}{\rho}+\sqrt{\frac{|t_1-t_2|}{u_1^{q+1-p}\rho^p}}\, \Bigg]^{\alpha_o}.
  \end{equation*}
  In view of~\eqref{Case1}, the term in brackets on the
  right-hand side is bounded by $2$. Therefore, we may diminish the
  power $\alpha_o$ and replace it by 
  $\alpha_1:=\min\{\alpha_o,\frac{2}{q+1-p}\}$. The resulting term is
  increasing in $u_1$, so that we can bound it from above by replacing
  $u_1$ by $\mathcal M$. Hence, we obtain the estimate
  \begin{align*}
    |Du(x_1,t_1)-Du(x_2,t_2)|
    &\le
      \bg\frac{\mathcal{M}}{\rho}
      \Bigg[\frac{|x_1-x_2|}{\rho}+\sqrt{\frac{|t_1-t_2|}{\mathcal{M}^{q+1-p}\rho^p}}\,\Bigg]^{\alpha_1}\\
    &\le
      \bg\frac{\mathcal{M}^{\frac{q+1}{p}}}{\rho_o}
      \Bigg[\mathcal{M}^{\frac{q+1-p}{p}}\frac{|x_1-x_2|}{\rho_o}
      +\sqrt{\frac{|t_1-t_2|}{\rho_o^p}}\,\Bigg]^{\alpha_1}.
  \end{align*} 
  In the last step, we used the definition of $\rho$. This yields the
  asserted estimate~\eqref{grad-holder-K} in the first case.

  \noi {\bf Case 2: } $|x_1-x_2|\ge \rho$ or
  $|t_1-t_2|\ge\max\{u_1,u_2\}^{q+1-p}\rho^p$.

In this case, we use the gradient bound~\eqref{est:Du_i}, which implies
  \begin{equation*}
    |Du(x_i,t_i)|
    \le
    \bg\frac{u_i}{\rho}
    \le
    \bg\frac{u_i}{\rho}
    \Bigg[\frac{|x_1-x_2|}{\rho}+\sqrt{\frac{|t_1-t_2|}{u_i^{q+1-p}\rho^p}}\,\Bigg]^{\alpha_1},  
\end{equation*}
since the term in the brackets is bounded from below by one in the
current case.  
Similarly as above, we can
further estimate the right-hand side by monotonicity and replace $u_i$ by $\mathcal{M}$. In this way, we again obtain
\begin{align*}
  |Du(x_1,t_1)-Du(x_2,t_2)|
  &\le
  |Du(x_1,t_1)|+|Du(x_2,t_2)|\\
  &\le
  \bg\frac{\mathcal{M}^{\frac{q+1}{p}}}{\rho_o}
  \Bigg[\mathcal{M}^{\frac{q+1-p}{p}}\frac{|x_1-x_2|}{\rho_o}
  +\sqrt{\frac{|t_1-t_2|}{\rho_o^p}}\,\Bigg]^{\alpha_1}.
\end{align*}
This establishes the claim~\eqref{grad-holder-K} in the second case as well. 
\end{proof}

\section{Decay estimates at the extinction time
}\label{sec:ext-time}

In the fast diffusion range solutions might become extinct abruptly. This means that for some $T>0$ we have $u(\cdot,T)=0$ a.e. in $E$. Such a time $T$ is called \textit{extinction time}. In this section we prove the decay estimates at the extinction time stated in Corollary~\ref{Cor:12:4}. Subsequently we derive an estimate for the extinction time of the solution to a Cauchy-Dirichlet problem in terms of the measure of the set $E$ and the $L^{q+1}$-norm of the initial values $u_o$.

\begin{proof}[{\rm \textbf{Proof of Corollary~\ref{Cor:12:4}}}]
  For $x_o\in E$ and $t_o\in(\tfrac12T,T)$, we apply the integral Harnack inequality from Theorem~\ref{Thm:bd:2} on the
  cylinder $K_{d(x_o)}(x_o)\times(T-2(T-t_o),T)\subset E_T$. Since
  $u(\cdot,T)=0$, we infer the bound
  \begin{equation*}
    u_o:=u(x_o,t_o)\le \bg \bigg[\frac{T-t_o}{d^p(x_o)}\bigg]^{\frac{1}{q+1-p}}
  \end{equation*}
  with a constant $\bg$ depending on $N,p$, and $q$. This yields the first asserted estimate. Next, we choose $\rho_o:= \bg_o^{-1} d(x_o)$, where the
  constant $\bg_o> \widetilde\bg$ is at our disposal. From the last displayed estimate we have
  \begin{equation*}
    u_o^{q+1-p}(\widetilde\bg \rho_o)^p
    \le
    \bg^{q+1-p}\frac{T-t_o}{d^p(x_o)} \bigg[\frac{\widetilde\bg d(x_o)}{\bg_o}\bigg]^p
    \le
    \tfrac12(T-t_o).
  \end{equation*}
  The last inequality follows by choosing $\bg_o$  sufficiently large.
  The above choice of $\rho_o$ therefore implies
  \begin{align*}
    \widetilde\bg Q_o&:=K_{\widetilde\bg \rho_o}(x_o)\times
    (t_o-u_o^{q+1-p}(\widetilde\bg\rho_o)^p,t_o+u_o^{q+1-p}(\widetilde\bg\rho_o)^p)\\
    &\,\subset
    K_{\widetilde\bg\rho_o}(x_o)\times
    (t_o-\tfrac12(T-t_o),t_o+\tfrac12(T-t_o))
    \Subset E_T.
  \end{align*}
  Therefore, the gradient estimate \eqref{grad-est-intro} is available on the cylinder
  $Q_o$. By combining this estimate with the already established bound for $u$, we infer 
  \begin{equation*}
    |Du(x_o,t_o)|
    \le
    \frac{\bg u_o}{\rho_o}
    \le
    \frac{\bg}{d(x_o)}\bigg[\frac{T-t_o}{d^p(x_o)}\bigg]^{\frac{1}{q+1-p}}.
  \end{equation*}
Moreover, for $r\in(0,\rho_o)$ we obtain from \eqref{diff-u-intro} that
\begin{equation*}
    \osc_{K_r(x_o)} u(\cdot,t_o) 
    \le
   \frac{\boldsymbol\gm ru_o}{\rho_o}
   \le
   \bg\frac{ r}{d(x_o)}\bigg[\frac{T-t_o}{d^p(x_o)}\bigg]^{\frac{1}{q+1-p}}
\end{equation*}
and from \eqref{C1alph-intro} that
\begin{align*}
    \osc_{K_r(x_o)} Du(\cdot,t_o) 
    \le 
    \boldsymbol\gm\Big(\frac{r}{\rho_o}\Big)^{\al_o}\,\frac{u_o}{\rho_o}
    \le 
    \frac{\bg}{d(x_o)}\bigg[\frac{r}{d(x_o)}\bigg]^{\al_o}\bigg[\frac{T-t_o}{d^p(x_o)}\bigg]^{\frac{1}{q+1-p}},
\end{align*}
which completes the proof. 
\end{proof}

Concerning the extinction time $T$ we refer to in Corollary~\ref{Cor:12:4}, an upper bound for it can be given in a simple but nevertheless interesting situation. Consider the Cauchy-Dirichlet Problem
\begin{equation*}
\left\{
\begin{array}{cl}
    \partial_t \big(|u|^{q-1}u\big) - \Div (|Du|^{p-2}Du)=0  & \quad \mbox{in  $E_T$,}\\[6pt]
    u =0 &\quad \mbox{on $\partial E\times(0,T]$,}\\[6pt]
    u(\cdot,0)= u_o &\quad\mbox{in $E$,}
\end{array}
\right.
\end{equation*}
where $T\in (0,\infty]$ and the initial datum $u_o\ge 0$ is non-negative. For finite $T<\infty$ the existence of a non-negative solution to the Cauchy-Dirichlet Problem is discussed in Remark~\ref{Rmk:CP1} of Chapter~\ref{sec:sol-compar}. The case $T=\infty$ follows by the uniqueness from Corollary \ref{Cor:Uniqueness}.

\begin{proposition}
Let $0<p-1<q<\frac{N(p-1)+p}{(N-p)_+}$,
$E\subset\R^N$  a bounded domain, and $u_o\in L^{q+1}(E)$ with $u_o\ge0$. Moreover, let
$$
    u\in C\big([0,+\infty);L^{q+1}(E)\big)
    \cap L^p\big(0,+\infty;W^{1,p}_0(E)\big)
$$ 
be the unique non-negative weak solution of the Cauchy-Dirichlet problem. There exists a finite time $T>0$ depending only on $N$, $p$, $q$, $|E|$, and $u_o$, such that 
\[
    u(\cdot,t)=0\quad\text{ for all }\,\,t\ge T.
\]
Moreover, we have
\[
0<T\le \frac {q\, \bg^p}{q+1-p}\, |E|^{\frac{\lambda_{q+1}}{N(q+1)}}\|u_o\|_{L^{q+1}(E)}^{q+1-p},
\]
where $\bg$ is a constant 
that depends on $N$ and $p$,
and $\lambda_{q+1}$ is defined in \eqref{def:lambda-r}.
\end{proposition}

\begin{proof} 
Taking $u$ as test function in the weak formulation of the Cauchy-Dirichlet Problem, \cite[Lemma~1.5]{AL} yields
\[
\|u(\cdot,t)\|_{L^{q+1}(E)}^{q+1}-\|u_o\|_{L^{q+1}(E)}^{q+1}+\frac{q+1}q\int_0^t\|Du(\cdot,\tau)\|_{L^p(E)}^p\,\d\tau=0
\]
for any $t>0$. Let us first assume that $1<p<N$; the ranges of $p$ and $q$ ensure that $q+1<\frac{Np}{N-p}=:p_\ast$, so that the H\"older and the Sobolev inequalities give
\begin{equation}\label{Eq:11:1}
\|u(\cdot,t)\|_{L^{q+1}(E)}\le |E|^{\frac{\lambda_{q+1}}{Np(q+1)}}\,\|u(\cdot,t)\|_{L^{p_*}(E)}\le\gm|E|^{\frac{\lambda_{q+1}}{Np(q+1)}}\,\|Du(\cdot,t)\|_{L^p(E)},
\end{equation}
where $\gm=\gm(N,p)$ is the optimal constant of the Sobolev inequality. 

Next, we consider
the case $p\ge N\ge1$. 
Let us first suppose $N\ge2$ and take $s\in(1,N)$ that satisfies $s_*:=\frac{Ns}{N-s}>q+1$; by H\"older's inequality, we have
\[
\|u(\cdot,t)\|_{L^{q+1}(E)}\le |E|^{\frac1{q+1}-\frac{N-s}{Ns}}\|u(\cdot,t)\|_{L^{s_*}(E)}.
\]
By the Sobolev and H\"older inequalities, we estimate
\[
\|u(\cdot,t)\|_{L^{s_*}(E)}\le\gm\|Du(\cdot,t)\|_{L^s(E)}\le\gm |E|^{\frac1s -\frac1p}\|Du(\cdot,t)\|_{L^p(E)},
\]
where again $\gm=\gm(N,s)$ is the optimal constant of the Sobolev inequality.
Joining the last two inequalities yields inequality \eqref{Eq:11:1}
in the case $p\ge N\ge2$. If $N=1$, by the Sobolev inequality we directly obtain that
\[
\|u(\cdot,t)\|_{L^{q+1}(E)}\le |E|^{\frac1{q+1}}\|u(\cdot,t)\|_{L^{\infty}(E)}\le \gm |E|^{\frac1{q+1}+1 -\frac1p}\|Du(\cdot,t)\|_{L^p(E)},
\]
and once more we end up with \eqref{Eq:11:1}.

\color{black}
Consequently, combining the previous estimates yields in any case that
\[
\|u(\cdot,t)\|_{L^{q+1}(E)}^{q+1}-\|u_o\|_{L^{q+1}(E)}^{q+1}+\frac{q+1}{q \gm^p |E|^{\frac{\lambda_{q+1}}{N(q+1)}}}\int_0^t\|u(\cdot,\tau)\|_{L^{q+1}(E)}^{p}\,\d\tau\le0
\]
for every $t>0$.
Setting $\displaystyle v(t):=\|u(\cdot,t)\|_{L^{q+1}(E)}^{q+1}$, the previous inequality can be rewritten as
\[
v(t)-v(0)+\mu \int_0^t [v(\tau)]^{\frac p{q+1}}\,\d\tau\le0,
\]
where $v(0)=\|u_o\|_{L^{q+1}(E)}^{q+1}$ and 
\[
\mu:=\frac{q+1}{q \gm^p |E|^{\frac{\lambda_{q+1}}{N(q+1)}}}.
\]
By the regularity of $u$, $v$ is a continuous function of $t$. 
Just as before, we can prove that 
\[
    v(t_2)-v(t_1)\le-\mu\int_{t_1}^{t_2}[v(\tau)]^{\frac p{q+1}}\,\d\tau
    \qquad\forall\,t_2>t_1\ge0.
\]
Since $v\ge0$ by definition, it is apparent that $v$ is strictly decreasing wherever it is positive, $\lim_{t\to \infty}v(t)=0$, and if there exists $T>0$ such that $v(T)=0$, then $v(t)=0$ for any $t\ge T$.

Now, let $w:[0,+\infty)\to\R$ be the unique solution to the Cauchy Problem
\[
\left\{
\begin{array}{c}
w'(t)=-\mu[w(t)]^{\frac p{q+1}},\\[6pt]
w(0)=v(0),
\end{array}
\right.
\]
which can be computed as 
\[
w(t)=v(0)\left[1-\frac{\mu(q+1-p)}{(q+1)[v(0)]^{\frac{q+1-p}{q+1}}}\,t\right]_+^{\frac{q+1}{q+1-p}}.
\]
We obviously have
\[
    w(t_2)-w(t_1)=-\mu\int_{t_1}^{t_2}[w(\tau)]^{\frac p{q+1}}\,\d\tau
    \qquad \forall\,t_2>t_1\ge0.
\]
Suppose there exists $\tilde t\in(0,\infty)$ such that $v(\tilde t)>w(\tilde t)$ and let
\[
t_o:=\sup\big\{t\in(0,\tilde t):\, v(t)\le w(t)\big\}.
\]
We have $v(t_o)=w(t_o)$, and for any $t\in[t_0,\tilde t\,]$ there holds
\[
v(t)-v(t_o)\le-\mu\int_{t_0}^t [v(\tau)]^{\frac p{q+1}}\,\d\tau,
\quad 
w(t)-v(t_o)=-\mu\int_{t_0}^t [w(\tau)]^{\frac p{q+1}}\,\d\tau.
\]
Subtracting from one another yields that for any $t\in(t_0,\tilde t]$
\[
v(t)-w(t)\le-\mu\int_{t_o}^t\left([v(\tau)]^{\frac p{q+1}}-[w(\tau)]^{\frac p{q+1}}\right)\,\d\tau,
\]
which is a contradiction, since the left-hand side is positive, whereas the right-hand side is negative. Hence, we necessarily conclude that 
\[
v(t)\le w(t)\qquad\forall\,t\in[0,\infty) 
\]
which gives
\[
v(t)\le v(0)\left[1-\frac{\mu(q+1-p)}{(q+1)[v(0)]^{\frac{q+1-p}{q+1}}}\,t\right]_+^{\frac{q+1}{q+1-p}}.
\]
From this, reverting to $\|u(\cdot,t)||_{L^{q+1}(E)}$, we have
\[
\|u(\cdot,t)\|_{L^{q+1}(E)}\le\|u_o\|_{L^{q+1}(E)}\left[1-\frac{q+1-p}{q \gm^p |E|^{\frac{\lambda_{q+1}}{N(q+1)}} \|u_o\|_{L^{q+1}(E)}^{q+1-p}}\,t\right]_+^{\frac1{q+1-p}},
\]
and
\[
0<T\le \frac {q\, \gm^p}{q+1-p}\, |E|^{\frac{\lambda_{q+1}}{N(q+1)}}\|u_o\|_{L^{q+1}(E)}^{q+1-p}.
\]
This proves the claim.
\end{proof}

\begin{remark}
The same estimate of the extinction time $T$ is given in \cite[Proposition~3.5]{MKS}, assuming $1<p<N$, $q\ge1$, $p-1<q<\frac{N(p-1)+p}{N-p}$. In this paper there is also an estimate from below in the same ranges for $p$ and $q$. Indeed, in \cite[Proposition~4.2 and Corollary~4.3]{MKS} Misawa \& Nakamura \& Sarkar prove that when $u_o\in W^{1,p}_0(E)$ one has
\[
T\ge\frac{q}{q+1-p}\frac{\|u_o\|_{L^{q+1}(E)}^{q+1}}{\|D u_o\|_{L^p(E)}^p}.
\]
\end{remark}


\backmatter


\bibliographystyle{amsalpha}



\begin{thebibliography}{A}


%
%

\bibitem{Acerbi-Fusco}
E.~Acerbi and N.~Fusco, 
\textit{Regularity for minimizers of non-quadratic functionals: the case $1<p<2$},
J. Math. Anal. Appl. {\bf 140} (1989), 115--135.

\bibitem{AL} H.~W.~Alt and S.~Luckhaus, 
\textit{Quasilinear elliptic-parabolic differential equations}, 
Math. Z. {\bf183} (1983), 311--341.

\bibitem{Aronson-69}
D.~G~Aronson, 
\textit{Regularity properties of flows through porous media}, 
SIAM J. Appl. Math. {\bf 17} (1969), 461--467.

\bibitem{Aronson-Caffarelli}
D.~G.~Aronson and L.~A.~Caffarelli, 
\textit{Optimal regularity for one dimensional porous medium flow}, 
Rev. Mat. Iberoamericana \textbf{2} (1986), 357--366.


\bibitem{Barenblatt}
G.~I.~Barenblatt, V.~M.~Entov and V.~M.~Ryzhik, 
 \textit{Theory of fluid flows through natural rocks},
 Kluwer Academic Publisher, Dordrecht, (1990).

\bibitem{Bear}
J.~Bear, 
\textit{Dynamics of fluids in porous media},
American Elsevier Publishing Company, 1972.

\bibitem{Benilan}
P.~B\'enilan, 
\textit{A strong regularity $L^p$ for solution of the porous media equation}, Contributions to nonlinear partial differential equations (Madrid, 1981),
Res. Notes in Math. 89, Pitman, Boston, MA, 1983, pp. 39--58.

\bibitem{BGKT}
J.~Benedikt, P.~Girg, L.~Kotrla and P.~Tak\'{a}\v{c}, 
\textit{Origin of the $p$-Laplacian and A. Missbach},
Electron. J. Differential Equations (2018), Paper No. 16, 1--17.

\bibitem{BGK}
J.~Benedikt, P.~Girg and L.~Kotrla,
\textit{Nonlinear models of the fluid flow in porous media and their methods of study},
Functional differential equations and applications,
Springer Proc. Math. Stat., 379, Springer, Singapore, (2021), 15--42.


\bibitem{bidaut}
M.~F.~Bidaut-V\'{e}ron,
\textit{Self-similar solutions of the {$p$}-{L}aplace heat equation:
the fast diffusion case},
Pacific J. Math. {\bf227} (2006), no. 2, 201--269.

\bibitem{Boegelein-Duzaar:p(z)}
V.~B\"ogelein and F.~Duzaar, 
\textit{H\"older estimates for parabolic $p(x,t)$-Laplacian systems},
Math. Ann. 354 (2012), no. 3, 907–938.

\bibitem{BDL-21} 
V.~B\"ogelein,  F.~Duzaar and N.~Liao,
\textit{On the H\"older regularity of signed solutions to a doubly nonlinear equation}, 
J. Funct. Anal. {\bf281} (2021), no. 9, Paper No. 109173, 58 pp.

\bibitem{BDLS-22}
V.~B\"ogelein,  F.~Duzaar, N.~Liao and L.~Schätzler,
\textit{On the Hölder regularity of signed solutions to a doubly nonlinear equation. Part II},
Rev. Mat. Iberoam. DOI 10.4171/RMI/1342.

\bibitem{BDLS-boundary}
V.~Bögelein, F.~Duzaar, N.~Liao and C.~Scheven, 
\textit{Boundary regularity for parabolic systems in convex domains}, 
J. Lond. Math. Soc. (2) {\bf105} (2022), no. 3, 1702--1751.

\bibitem{BDLS-Tolksdorf}
V.~Bögelein, F.~Duzaar, N.~Liao and C.~Scheven,
\textit{Gradient Hölder regularity for degenerate parabolic systems},
Nonlinear Anal. {\bf225} (2022), no. 113119, 61 pp. 

\bibitem{BDM-13} 
V.~Bögelein, F.~Duzaar and P.~Marcellini, 
\textit{Parabolic systems with $p,q$-growth: a variational approach}, 
Arch. Ration. Mech. Anal. {\bf210} (2013), no. 1, 219--267.

\bibitem{BDMS-18} 
V.~Bögelein, F.~Duzaar, P.~Marcellini and C.~Scheven, 
\textit{Doubly nonlinear equations of porous medium type}, 
Arch. Ration. Mech. Anal. {\bf229} (2018), no. 2, 503--545.

\bibitem{BDMS-survey}
V.~Bögelein, F.~Duzaar, P.~Marcellini and C.~Scheven, 
\textit{A variational approach to doubly nonlinear equations}, 
Atti Accad. Naz. Lincei Rend. Lincei Mat. Appl. {\bf29} (2018), no. 4, 739--772.

\bibitem{BHSS}
V.~Bögelein, A.~Heran, L.~Schätzler and T.~Singer,
\textit{Harnack's inequality for doubly nonlinear equations of slow diffusion type},
Calc. Var. Partial Differential Equations {\bf60} (2021), no. 6, Paper No. 215, 35 pp.

\bibitem{BRM2}
G.~I.~Barenblatt, P.~J.~M.~Monteiro and C.~H.~Rycroft, 
\textit{On a boundary layer problem related to the gas flow in shales}, 
J. Eng. Math. {\bf84} (2014), 11--18. 

\bibitem{CVW}
L.~A.~Caffarelli, J.~L.~V\'azquez and N.~I.~Wolanski, 
\textit{Lipschitz continuity of solutions and interfaces of the N-dimensional porous medium equation}, 
Indiana Univ. Math. J. {\bf36} (1987), no. 2, 373--401.

\bibitem{Chen} 
Y.~Z.~Chen, 
\textit{H\"older continuity of the gradient of solutions of nonlinear degenerate parabolic systems}, 
Acta Math. Sinica (N.S.) {\bf2}  (1986), no. 4, 309--331.

\bibitem{Chen-DB-89} 
Y.~Z.~Chen and E.~DiBenedetto, 
\textit{Boundary estimates for solutions of nonlinear degenerate parabolic systems}, 
J. Reine Angew. Math. {\bf395} (1989), 102--131.

\bibitem{Choe:1991} 
H.~Choe,
\textit{H\"older regularity for the gradient of solutions of
certain singular parabolic systems},
Comm. Partial Differential Equations {\bf16} (1991), no. 11,
  1709--1732.

\bibitem{DB-1d}
E.~DiBenedetto, 
\textit{Regularity results for the porous media equation}, 
Ann. Mat. Pura Appl. (4) {\bf121} (1979), 249--262.

\bibitem{DB} 
E.~DiBenedetto,
\textit{Degenerate parabolic equations},
Universitext, Springer-Verlag, New York, 1993. 

\bibitem{DB-mech} 
E.~DiBenedetto,  
\textit{Classical mechanics. Theory and mathematical modeling}, 
Cornerstones, Birkh\"auser/Springer, New York, 2011.

\bibitem{DiBenedetto-Friedman} 
E.~DiBenedetto and A.~Friedman, 
\textit{H\"older estimates for nonlinear degenerate parabolic systems}, 
J. Reine Angew. Math. {\bf357} (1985), 1--22.

\bibitem{DiBenedetto-Friedman2}
E.~DiBenedetto and A.~Friedman, 
\textit{Regularity of solutions of nonlinear degenerate parabolic systems}, 
J. Reine Angew. Math. {\bf349} (1984), 83--128.

\bibitem{DiBenedetto-Friedman3}
E.~DiBenedetto and A.~Friedman, 
\textit{Addendum to: ``H\"older estimates for nonlinear degenerate parabolic systems''}, 
J. Reine Angew. Math. {\bf363} (1985), 217--220.

\bibitem{DBGV-mono} 
E.~DiBenedetto, U.~Gianazza and V.~Vespri, 
\textit{Harnack's inequality for degenerate and singular parabolic 
equations}, 
Springer Monographs in Mathematics, Springer-Verlag, New York, 2012.

\bibitem{DBKV}
E.~DiBenedetto, Y.~Kwong and V.~Vespri, 
\textit{Local space-analyticity of solutions of certain singular parabolic equations}, Indiana Univ. Math. J. {\bf 40} (1991), no. 2, 741--765.

\bibitem{DB-RA} E.~DiBenedetto, 
\textit{Real analysis}, 
Second edition,  Birkh\"auser/Springer, New York, 2016.

\bibitem{FSV-13} 
S.~Fornaro,  M.~Sosio and V.~Vespri,  
\textit{Energy estimates and integral Harnack inequality for some doubly nonlinear singular parabolic equations}, 
Recent trends in nonlinear partial differential equations. I. Evolution problems, 179--199, Contemp. Math., 594, Amer. Math. Soc., Providence, RI, 2013.

\bibitem{FSV-14} 
S.~Fornaro,  M.~Sosio and V.~Vespri, 
\textit{$L^r$--$L^{\infty}$ estimates and expansion of positivity for a class of doubly non linear singular parabolic equations}, 
Discrete Contin. Dyn. Syst. Ser. S \textbf{7} (2014), no. 4, 737--760.

\bibitem{FSV-15} 
S.~Fornaro,  M.~Sosio and V.~Vespri, 
\textit{Harnack type inequalities for some doubly nonlinear singular parabolic equations}, 
Discrete Contin. Dyn. Syst. \textbf{35} (2015), no. 12, 5909--5926.

\bibitem{Gianazza-Siljander}
U.~Gianazza and J.~Siljander, 
\textit{Local bounds of the gradient of weak solutions to the porous medium equation}, 
Partial Differ. Equ. Appl. {\bf4} (2023), no. 2, Paper No. 8, 35 pp. 

\bibitem{GiaquintaModica:1986-a} 
M.~Giaquinta and G.~Modica,  
\textit{Remarks on the regularity of the minimizers of certain degenerate functionals},
Manuscripta Math. {\bf 57} (1986), no. 1, 55--99.

\bibitem{Giusti}
E.~Giusti,
\textit{Direct methods in the calculus of variations},
World Scientific, Singapore, 2003.

\bibitem{Henriques-22} 
E.~Henriques, 
\textit{Expansion of positivity to a class of doubly nonlinear parabolic equations}, 
Electron. J. Qual. Theory Differ. Equ. (2022), Paper No. 15, 24 pp.

\bibitem{iagar}
R.~G.~Iagar, A.~S\'{a}nchez and J.~L.~V\'{a}zquez, 
\textit{Radial equivalence for the two basic nonlinear degenerate diffusion equations},
J. Math. Pures Appl. (9) \textbf{89} (2008), no. 1, 1--24.

\bibitem{Ivanov-1989}
A.~V.~Ivanov,
\textit{Hölder estimates for quasilinear doubly degenerate parabolic equations}, 
Zap. Nauchn. Sem. Leningrad. Otdel. Mat. Inst. Steklov. (LOMI) 171 (1989), Kraev. Zadachi Mat. Fiz. i Smezh. Voprosy Teor. Funktsiĭ. 20, 70–105, 185; translation in J. Soviet Math. 56 (1991), no. 2, 2320--2347.

\bibitem{Ivanov-1994}
A.~V.~Ivanov,
\textit{H\"older estimates for equations of slow and normal diffusion type}, 
Zap. Nauchn. Sem. S.-Peterburg. Otdel. Mat. Inst. Steklov. (POMI) 215 (1994), Differentsial'naya Geom. Gruppy Li i Mekh. 14, 130--136, 311; translation in J. Math. Sci. (New York) 85 (1997), no. 1, 1640--1644.


\bibitem{Ivanov-1995-2}
A.~V.~Ivanov,
\textit{Maximum modulus estimates for generalized solutions to doubly nonlinear parabolic equations}, 
Zap. Nauchn. Sem. S.-Peterburg. Otdel. Mat. Inst. Steklov. (POMI) 221 (1995), Kraev. Zadachi Mat. Fiz. i Smezh. Voprosy Teor. Funktsiĭ. 26, 83–113, 257; translation in J. Math. Sci. (New York) 87 (1997), no. 2, 3322--3342. 

\bibitem{Ivanov-1995-3}
A.~V.~Ivanov,
\textit{Hölder estimates for a natural class of equations of fast diffusion type},  
Zap. Nauchn. Sem. S.-Peterburg. Otdel. Mat. Inst. Steklov. (POMI) 229 (1995), Chisl. Metody i Voprosy Organ. Vychisl. 11, 29--62, 322; translation in J. Math. Sci. (New York) 89 (1998), no. 6, 1607--1630.

\bibitem{Ivanov-1996}
A.~V.~Ivanov,
\textit{Gradient estimates for doubly nonlinear parabolic equations}, 
Zap. Nauchn. Sem. S.-Peterburg. Otdel. Mat. Inst. Steklov. (POMI) 233 (1996), Kraev. Zadachi Mat. Fiz. i Smezh. Vopr. Teor. Funkts. 27, 63--100, 256; reprinted in J. Math. Sci. (New York) 93 (1999), no. 5, 661--688.

\bibitem{Ivanov-1997}
A.~V.~Ivanov, 
\textit{The regularity theory for $(m,l)$-Laplacian parabolic equation}, 
Zap. Nauchn. Sem. POMI {\bf 243} 1997, 87--110.

\bibitem{Ivanov-Mk} 
A.~V.~Ivanov and P.~Z.~Mkrtychyan,
\textit{A weighted estimate of the gradient for nonnegative generalized solutions of quasilinear doubly degenerate parabolic equations}, (Russian. English summary) 
Zap. Nauchn. Sem. Leningrad. Otdel. Mat. Inst. Steklov. (LOMI) 181 (1990), Differentsial'naya Geom. Gruppy Li i Mekh. 11, 3–23, 186; translation in J. Soviet Math. 62 (1992), no. 2, 2605--2619.


\bibitem{JX} 
T.~Jin and J.~Xiong, 
\textit{Regularity of solutions to the Dirichlet problem for fast diffusion equations}, arXiv:2201.10091

\bibitem{K} A.~S.~Kalashnikov,
\textit{Some problems of the qualitative theory of non-linear degenerate second-order parabolic equations},
Uspekhi Mat. Nauk \textbf{42} (1987), no. 2, 135--176.

\bibitem{Kr}
S.~A.~Khristianovich, 
\textit{Motion of ground water that does not satisfy Darcy's law}, 
Prikl. Mat. Mekh. \textbf{4} (1940), no. 1, 33--52.

\bibitem{king1990}
J.~R.~King, 
\textit{Exact similarity solutions to some nonlinear diffusion equations}, 
J. Phys. A: Math. Gen. \textbf{23} (1990), 3681--3697.

\bibitem{king}
J.~R.~King, 
\textit{Self-similar behaviour for the equation of fast nonlinear diffusion}, 
Phil. Trans. R. Soc. Lond. A \textbf{343} (1993), 337--375.

\bibitem{kosov}
A.~A.~Kosov and \`E.~I.~Semenov,
\textit{Exact Solutions of the Nonlinear Diffusion Equation},
Siberian Math. J. \textbf{60} (2019), no. 1, 93--107.

\bibitem{Kuusi-Mingione}
T.~Kuusi and G.~Mingione, 
\textit{New perturbation methods for nonlinear parabolic problems},
J. Math. Pures Appl. (9) \textbf{98} (2012), no. 4, 390--427.

\bibitem{KSU-12} 
T.~Kuusi, J.~Siljander and J.~M.~Urbano, 
\textit{Local H\"older continuity for doubly nonlinear parabolic equations}, 
Indiana Univ. Math. J. \textbf{61} (2012), no. 1, 399--430. 



\bibitem{Liao-JMPA-21} 
N.~Liao,  
\textit{Regularity of weak supersolutions to elliptic and parabolic equations: Lower semicontinuity and pointwise behavior}, 
J. Math. Pures Appl. (9) \textbf{147} (2021), 179--204.

\bibitem{Liao-Schaetzler}
N.~Liao and L.~Schätzler, 
\textit{On the Hölder regularity of signed solutions to a doubly nonlinear equation. Part III}, 
Int. Math. Res. Not. IMRN {\bf2022}, no. 3, 2376--2400.

\bibitem{LL-19}
E.~Lindgren and P.~Lindqvist, 
\textit{On a comparison principle for Trudinger's equation}, 
Adv. Calc. Var. \textbf{15} (2022), no. 3, 401--415.

\bibitem{Misawa-Schauder}
M.~Misawa, 
\textit{Local Hölder regularity of gradients for evolutional p-Laplacian systems}, 
Ann. Mat. Pura Appl. (4) \textbf{181} (2002), 389--405.

\bibitem{MKS} 
M.~Misawa, K.~Nakamura and M.~A.~H.~Sarkar,
\textit{A ﬁnite time extinction proﬁle and optimal decay for a fast diﬀusive doubly nonlinear equation},
Nonlinear Diﬀer. Equ. Appl. \textbf{30} (2023), no. 43, 1--48.

\bibitem{MRB}
P.~J.~M.~Monteiro, C.~H.~Rycroft and G.~I.~Barenblatt,
\textit{A mathematical model of fluid and gas flow in nanoporous media},
Proc. Natl. Acad. Sci. USA \textbf{109} (2012),  no. 50, 20309--20313.

\bibitem{moser-71} 
J.~Moser, 
\textit{On a pointwise estimate for parabolic differential equations}, 
Comm. Pure Appl. Math. {\bf24} (1971), 727--740.

\bibitem{Neuman}
S.~P.~Neuman, 
\textit{Theoretical derivation of Darcy's law},
Acta Mech. \textbf{25} (1977), 153--170.

\bibitem{yamada}
T.~Oka and T.~Yamada, 
\textit{Topology optimization method with nonlinear diffusion},
Comput. Methods Appl. Mech. Engrg. \textbf{408} (2023), Paper No. 115940, 24 pp. 

\bibitem{Otto} F.~Otto,
\textit{$L^1$-contraction and uniqueness for quasilinear elliptic-parabolic equations},
J. Differential Equations \textbf{131} (1996), 20--38.

\bibitem{peletier}
M.~A.~Peletier and H.~F.~Zhang,
\textit{Self-similar solutions of a fast diffusion equation that do not conserve mass},
Differential Integral Equations \textbf{8} (1995), no. 8, 2045--2064.

\bibitem{Porzio-Vespri}
M.~M.~Porzio and V.~Vespri, 
\textit{Hölder estimates for local solutions of some doubly nonlinear degenerate parabolic equations},
J. Differential Equations \textbf{103} (1993), no. 1, 146--178. 

\bibitem{Savare}
G.~Savar\'e and V.~Vespri,
\textit{The asymptotic profile of solutions of a class of doubly nonlinear equations}, Nonlinear Anal. \textbf{22} (1994), no. 12, 1553--1565. 

\bibitem{Schatzler-1}
L.~Schätzler,
\textit{The obstacle problem for singular doubly nonlinear equations of porous medium type}, 
Atti Accad. Naz. Lincei Rend. Lincei Mat. Appl. \textbf{31} (2020), no. 3, 503--548.

\bibitem{Schatzler-2}
L.~Schätzler,
\textit{The obstacle problem for degenerate doubly nonlinear equations of porous medium type}, 
Ann. Mat. Pura Appl. (4) \textbf{200} (2021), no. 2, 641--683.


\bibitem{serrin}
J.~Serrin, 
\textit{Mathematical principles of classical fluid mechanics}, 
Herausgegeben von S. Flügge. Mitherausgeber C. Truesdell. Handbuch der Physik, Bd. 8/1, Str\"omungsmechanik I. Springer-Verlag, Berlin-G\"ottingen-Heidelberg, 1959, 125--263.

\bibitem{Showalter}
R.~E.~Showalter, 
\textit{Monotone operators in Banach space and nonlinear partial differential equations},
Mathematical Surveys and Monographs, 49. American Mathematical Society, Providence, RI, 1997. 

\bibitem{SW1}
R.~E.~Showalter and N.~J.~Walkington,
\textit{Diffusion of fluid in a fissured medium with microstructure},
SIAM J. Math. Anal. \textbf{22} (1991), 1702--1722.

\bibitem{SW2} 
R.~E.~Showalter and N.~J.~Walkington.
\textit{Micro--structure models of diffusion in fissured media}, 
J.  Math. Anal. Appl. \textbf{155} (1991), 1--20.

\bibitem{Urbano-08}
J.~M.~Urbano,   
\textit{The method of intrinsic scaling. A systematic approach to regularity for degenerate and singular PDEs}, 
Lecture Notes in Mathematics, 1930. Springer-Verlag, Berlin, 2008.

\bibitem{Vazquez}
J.~L.~V\'azquez, 
\textit{The porous medium equation. 
Mathematical theory}, Oxford Mathematical Monographs. The Clarendon Press, Oxford University Press, Oxford, 2007.

\bibitem{Vespri-Vestberg}
V.~Vespri and M.~Vestberg, 
\textit{An extensive study of the regularity properties of solutions to doubly singular equations},
Adv. Calc. Var. \textbf{15} (2022), no. 3, 435--473. 

\bibitem{WZYL-01}
Z.~Wu, J.~Zhao, J.~Yin and  H.~Li, 
\textit{Nonlinear diffusion equations}, 
World Scientific Publishing Co., Inc., River Edge, NJ, 2001.

\end{thebibliography}
\printindex

\end{document}